\documentclass[12pt, A4]{book}
\usepackage{url}
\usepackage[a4paper]{geometry}

\usepackage{makeidx}
\makeindex

\usepackage{amsbsy}
\let\oldmarginpar\marginpar
\renewcommand\marginpar[1]{\-\oldmarginpar[\raggedleft\footnotesize #1]%
{\raggedright\tiny #1}}

\usepackage{amsmath,amssymb,amsthm,mathrsfs,color}

\usepackage{bbm}

\usepackage{graphicx}
\usepackage{ifpdf}

\usepackage[small,nohug,heads=littlevee]{diagrams}
\usepackage[all]{xy}
\diagramstyle[labelstyle=\scriptstyle]
\usepackage{multicol,lipsum}
\usepackage[mathscr]{euscript}
\usepackage{marginnote}

\usepackage{stmaryrd}

\usepackage{hyperref}
\hypersetup{colorlinks,%
    citecolor=red,%
    filecolor=black,%
    linkcolor=black,%
    urlcolor=black}

\begin{document}
\setlength{\unitlength}{2.5cm}

\newtheorem{thm}{Theorem}[section]
\newtheorem{lm}[thm]{Lemma}
\newtheorem{prop}[thm]{Proposition}
\newtheorem{cor}[thm]{Corollary}

\theoremstyle{definition}
\newtheorem{dfn}[thm]{Definition}
\newtheorem{eg}[thm]{Example}
\newtheorem{rmk}[thm]{Remark}

\newcommand{\F}{\mathbf{F}}
\newcommand{\N}{\mathbf{N}}
\newcommand{\R}{\mathbf{R}}
\newcommand{\C}{\mathbf{C}}
\newcommand{\Z}{\mathbf{Z}}
\newcommand{\Q}{\mathbf{Q}}
\newcommand{\T}{\mathbf{T}}
\newcommand{\K}{\mathbf{K}}
\newcommand{\LL}{\mathcal{L}}

\newcommand{\G}{\text{G}}
\newcommand{\re}{\text{Re}}
\newcommand{\im}{\text{Im}}
\newcommand{\gal}{\text{Gal}}
\newcommand{\ke}{\text{Ker}}
\newcommand{\ma}{\text{Max}}
\newcommand{\Spec}{\text{Spec}}
\newcommand{\Pic}{\text{Pic}}
\newcommand{\ord}{\text{ord}}
\newcommand{\app}{\thickapprox}
\newcommand{\deh}{H_\mca{D}}
\newcommand{\moh}{H_\mca{M}}
\newcommand{\ab}{\text{ab}}
\newcommand{\Mp}{\text{Mp}}
\newcommand{\Sp}{\text{Sp}}
\newcommand{\GL}{\text{GL}}
\newcommand{\PGL}{\text{PGL}}
\newcommand{\SL}{\text{SL}}
\newcommand{\Ind}{\text{Ind}}
\newcommand{\Res}{\text{Res}}
\newcommand{\vs}{\vec{s}}
\newcommand{\Hom}{\text{Hom}}
\newcommand{\msc}[1]{\mathscr{#1}}
\newcommand{\mfr}[1]{\mathfrak{#1}}
\newcommand{\mca}[1]{\mathcal{#1}}
\newcommand{\mbf}[1]{\mathbbmss{#1}}

\newcommand{\dual}[1]{{}^L\wt{#1}}
\newcommand{\cd}[1]{{\wt{#1}}^\vee}
\newcommand{\noa}{\mathbf{n}_o^\alpha}
\newcommand{\koa}{\mathbf{k}_o^\alpha}
\newcommand{\Hs}{\textsf{Hs}}
\newcommand{\Inc}{\textsf{Inc}}
\newcommand{\CExt}{\textsf{CExt}}
\newcommand{\Bis}{\textsf{Bis}}
\newcommand{\Rec}{\text{Rec}}
\newcommand{\s}{\mathbf{s}}
\newcommand{\w}{\mathbf{w}}
\newcommand{\A}{\mbf{A}_F}
\newcommand{\p}{\mathbf{p}}
\newcommand{\q}{\mathbf{q}}
\newcommand{\WD}{\text{WD}}
\newcommand{\W}{\text{W}}

\newcommand{\cu}[1]{\textsc{\underline{#1}}}
\newcommand{\set}[1]{\bigl\{#1\bigr\}}
\newcommand{\wt}[1]{\overline{#1}}
\newcommand{\angb}[2]{\langle #1, #2 \rangle}
\newcommand{\seq}[3]{\xymatrix{
#1 \ar@{^(->}[r] & #2 \ar@{>>}[r] & #3
}}
\newcommand{\wm}[1]{\wt{\mbf{#1}}}
\newcommand{\elt}[1]{\pmb{\big[} #1\pmb{\big]} }

\title{\textsc{The Gindikin-Karpelevich Formula and Constant Terms of Eisenstein Series for Brylinski-Deligne Extensions}}
\author{
 \\  \textsc{\Large Fan GAO}
 \\  (B.Sc., NUS)
\\  \\  \\  \\ \\ \\ \\  \\ \\ \\   \textsc{\Large A Thesis Submitted} 
 \\ \textsc{\Large for the Degree of Doctor of Philosophy} 
  \\  \textsc{\Large Department of Mathematics} 
\\ \textsc{\Large National University of Singapore}   \\ \\ \textsc{\Large 2014}
}
\date{}
\maketitle

\newpage

\setcounter{page}{1}
\pagenumbering{roman}

\chapter*{}
\centerline{\textsc{\Large Declaration}}
\vspace{1cm}

I hereby declare that the thesis is my original work and it has been written by me in its entirety. I have duly acknowledged all the sources of information which have been used in the thesis.

This thesis has also not been submitted for any degree in any university previously.

\vspace{2cm}

\centerline{\underline{\hspace{3cm}}}

\centerline{Fan Gao}

\centerline{14/Oct/2014}

\newpage
\chapter*{Acknowledgement}
I would like to take this opportunity to thank those whose presence has helped make this work possible.

First and foremost, I am deeply grateful to my supervisor Professor Wee Teck Gan, for numerous discussions and inspiring conversations. I would like to thank Prof. Gan for his patience, guidance and encouragement through the whose course of study, and also for sharing some of his insightful ideas. The suggestions and corrections to an earlier version by Prof. Gan have been very helpful in improving the exposition of this thesis. The gratitude I owe not only arises from the formal academic supervision that I receive; more importantly, it has been due to the sense of engagement in the enterprise of modern number theory that Prof. Gan has bestowed his students with, by showing in an illuminating way how problems in mathematics could be approached.

I would like to thank Professor Martin Weissman for generously sharing his letter to P. Deligne \cite{We12}
and his preprint \cite{We14}. Many discussions with Prof. Weissman have been rewarding and very helpful. I am grateful for his pioneer work in \cite{We09}-\cite{We14}, without which this dissertation would not be possible.
In parallel, I would also like to thank AIM for their support for the 2013 workshop ``Automorphic forms and harmonic analysis on covering groups'' organized by Professors Jeffrey Adams, Wee Teck Gan and Gordan Savin, during which many experts have generously shared their insights on the subjects.

Meanwhile, it is my pleasure to thank Professor Chee Whye Chin, who has always been generous to share his knowledge on mathematics, including but not restricted to arithmetic geometry. I would like to thank him for initiating my interest in number theory from the undergraduate days, and also for the efforts he devoted to in a series of courses during which I benefited tremendously from his neat and clear expositions. I would also like to thank Professor Jon Berrick, who always gives excellent illustrations of how to think about and write mathematics nicely from his courses, for sharing his broad perspectives on the subjects of topology and $K$-theory.

My sincere thanks are due to Professor Chen-Bo Zhu and Professor Hung Yean Loke for many enlightening and helpful conversations on both academic and non-academic affairs, also for their efforts devoted to the SPM program and the courses therein.  Meanwhile, I thank Professor Yue Yang for sharing in a series of his courses the joy of mathematical logic, the content and theorems of which still remain like magic to me. I also benefit from various courses from Prof. Ser Peow Tan, Prof. De-Qi Zhang, Prof. Denny Leung, Prof. Seng Kee Chua, Prof. Graeme Wilkin and Prof. Jie Wu. It has been a privilege to be able to talk to them, and I thank these professors heartedly.

Mathematics would not have been so fun if without the presence of my friends and the time we have shared together. I would like to thank Minh Tran, Colin Tan, Jia Jun Ma, Heng Nan Hu, Jun Cai Lee, Jing Zhan Lee, Zhi Tao Fan, Wei Xiong, Jing Feng Lau, Heng Fei Lv, Cai Hua Luo and the fellows in my office. At the same time, thanks are due to the staff of the general office of mathematics, for their constant support and help.

Last but not the least, I am much grateful for my wife Bo Li for her love and support throughout. It has been entertaining to discuss with her on problems in mathematics as well as statistics. Moreover, the support and encouragement of my parents and parents-in-law have been crucial in the whole course of my study and in the preparation of this dissertation. I would like to thank my family, to whom I owe my truly deep gratitude.

\newpage
\chapter*{Notations and Terminology}
\

$F$: a number field or a local field with finite residue field of size $q$ in the nonarchimedean case.

$\text{Frob}$ or $\text{Frob}_v$: the geometric Frobenius class of a local field. \index{Frobenius}

$I$ or $I_v$: the inertia group of the absolute Galois group of a local field. \index{inertia group}

$\msc{O}_F$: the ring of integers of $F$.

$\mbf{G}_\text{add}$ and $\mbf{G}_\text{mul}$: the additive and multiplicative group over $F$ respectively.

$\mbf{G}$: a general split reductive group (over $F$) with root datum $(X, \Psi, Y, \Psi^\vee)$. We fix a set positive roots $\Psi^+\subseteq \Psi$ and thus also a set of simple roots $\Delta \subseteq \Psi$. Let $\mbf{G}^{sc}$ be the simply connected cover of the derived subgroup $\mbf{G}^{sc}$ of $\mbf{G}$ with the natural map denoted by $\xymatrix{\Phi: \mbf{G}^{sc} \ar[r] & \mbf{G}} $

We fix a Borel subgroup $\mbf{B}=\mbf{T} \mbf{U}$ of $\mbf{G}$ and also a Chevalley system of  \'epinglage for $(\mbf{G}, \mbf{T}, \mbf{B})$ (cf. \cite[\S 3.2.1-2]{BrTi84}), from which we have an isomorphism $\mbf{e}_\alpha: \mbf{G}_\text{add} \to \mbf{U}_\alpha$ for each $\alpha \in \Psi$ with associated root subgroup $\mbf{U}_\alpha$. Moreover, for each $\alpha\in \Psi$, there is the induced morphism $\xymatrix{\varphi_\alpha: \mbf{SL}_2 \ar[r] & \mbf{G}}$ which restricts to $\mbf{e}_{\pm \alpha}$ on the upper and lower triangular subgroup of unipotent matrices of $\mbf{SL}_2$.

$\mbf{T}$: a maximally split torus of $\mbf{G}$ with character group $X$ and cocharacter group $Y$.

$Q$: an integer-valued Weyl-invariant quadratic form on $Y$ with associated symmetric bilinear form
$$B_Q(y_1, y_2):=Q(y_1+y_2) -Q(y_1) -Q(y_2).$$

In general, notations will be explained the first time they appear in the text.

\vspace{1cm}

``character": by a \emph{character} of a group we just mean a continuous homomorphism valued in $\C^\times$, while a \emph{unitary character} refers to a character with absolute value 1.

``section'' and ``splitting": \index{section} \index{splitting}
for an exact sequence $\seq{A}{B}{C}$ of groups we call any map $\xymatrix{s: C \ar[r] & B}$ a \emph{section} if its post composition with the last projection map on $C$ is the identity map on $C$. We call $s$ a \emph{splitting} if it is a homomorphism, and write $\mfr{S}(B,C)$ for all splittings of $B$ over $C$, which is a torsor over $\Hom(C,A)$ when the extension is central.

``push-out'': \index{push-out}
for a group extension $\seq{A}{B}{C}$ and a homomorphism $f: A \to A'$ whose image is a normal subgroup of $A'$, the push-out $f_* B$ (as a group extension of $C$ by $A'$) is given by
$$f_* B:= \frac{A' \times B}{\langle (f(a), i^{-1}(a)): a\in A \rangle},$$
whenever it is well-defined. Here $\xymatrix{i: A \ar@{^(->}[r] & B}$ is the inclusion in the extension. For example, if $f$ is trivial or both $i$ and $f$ are central, i.e. the image of the map lies in the center of $B$ and $A'$ respectively, then $f_*B$ is well-defined.

``pull-back'': \index{pull-back}
for a group extension $\seq{A}{B}{C}$ and a homomorphism $h: C' \to C$, the pull-back $h^* B$ is the group
$$h^*B:= \set{(b, c'): q(b)=h(c')} \subseteq B \times C',$$
where $\xymatrix{q: B \ar@{>>}[r] & C }$ is the quotient map. The group $h^*B$ is an extension of $C'$ by $A$.

\newpage

\chapter*{Summary}
We work in the framework of the Brylinski-Deligne (BD) central covers of general split reductive groups. To facilitate the computation, we use an incarnation category initially given by M. Weissman which is equivalent to that of Brylinski-Deligne. 

Let $F$ be a number field containing $n$-th root of unity, and let $v$ be an arbitrary place of $F$. The objects of main interest will be the topological covering groups of finite degree arising from the BD framework, which are denoted by $\wt{G}_v$ and $\wm{G}(\A)$ in the local and global situations respectively. The aim of the dissertation is to compute the Gindikin-Karpelevich (GK) coefficient which appears in the intertwining operator for global induced representations from parabolic subgroups $\wm{P}(\A)=\wm{M}(\A)\mbf{U}(\A)$ of general BD-type covering groups $\wm{G}(\A)$. The result is expressed in terms of naturally defined elements without assuming $\mu_{2n}\subseteq F^\times$, and thus could be considered as a refinement of that given by McNamara etc.

Moreover, using the construction of the $L$-group ${}^L\overline{G}$ by Weissman for the global covering $\wm{G}(\A)$, we define partial automorphic $L$-functions for covers $\wm{G}(\A)$ of BD type. We show that the GK coefficient computed can be interpreted as Langlands-Shahidi type partial $L$-functions associated to the adjoint representation of ${}^L\overline{M}$ on a certain subspace $\wm{\mfr{u}}^\vee \subset \wt{\mfr{g}}^\vee$ of the Lie algebra of $\wt{}^L\wt{G}$. Consequently, we are able to express the constant term of Eisenstein series of BD covers, which relies on the induction from parabolic subgroups as above, in terms of certain partial $L$-functions of Langlands-Shahidi type.

The interpretation relies crucially on the local consideration. Therefore, along the way, we discuss properties of the local $L$-group ${}^L\overline{G}_v$ for $\wt{G}_v$, which by the construction of Weissman sits in an exact sequence $\xymatrix{\overline{G}^\vee \ar@{^(->}[r] & {}^L\overline{G}_v \ar@{>>}[r] & \W_{F_v} }$. For instance, in general ${}^L\overline{G}_v$ is not isomorphic to the direct product $\overline{G}^\vee \times \W_{F_v}$ of the complex dual group $\overline{G}^\vee$ and the Weil group $\W_{F_v}$. There is a close link between splittings of ${}^L\overline{G}_v$ over $\W_{F_v}$ which realize such a direct product and Weyl-group invariant genuine characters of the center $Z(\wt{T}_v)$ of the covering torus $\wt{T}_v$ of $\wt{G}_v$. In particular, for $\wt{G}_v$ a cover of a simply-connected group there always exist Weyl-invariant genuine characters of $Z(\wt{T}_v)$. We give a construction for general BD coverings with certain constraints. In the case of BD coverings of simply-laced simply-connected groups, our construction agrees with that given by G. Savin. It also agrees with the classical double cover $\wm{Sp}_{2r}(F_v)$ of $\mbf{Sp}_{2r}(F_v)$. Moreover, the discussion for the splitting of ${}^L\wt{G}_v$ in the local situation could be carried over parallel for the global ${}^L\wt{G}$ as well.

In the end, for illustration purpose we determine the residual spectrum of general BD coverings of $\mbf{SL}_2(\A)$ and $\mbf{GL}_2(\A)$. In the case of the classical double cover $\wm{Sp}_4(\A)$ of $\mbf{Sp}_4(\A)$, it is also shown that the partial Langlands-Shahidi type $L$-functions obtained here agree with what we computed before in another work, where the residual spectrum for $\wm{Sp}_4(\A)$ is determined completely.

\newpage
\tableofcontents

\newpage
\setcounter{page}{1}
\pagenumbering{arabic}
\chapter{Introduction}

It has been one of the central themes in the theory of automorphic forms for a \emph{split} reductive group $\mbf{G}$ to determine completely the spectral decomposition of $L^2(\mbf{G}(F)\backslash \mbf{G}(\A))$, where $F$ is a number field (or in general a global field) and $\A$ its adele ring. In rough terms, the space $L^2(\mbf{G}(F)\backslash \mbf{G}(\A))$ carries a $\mbf{G}(\A)$ action and is endowed with a representation of the group, thus the notion of automorphic representations as its constituents. It is important to be able to construct automorphic representations. Moreover, one would like to give arithmetic parametrization of such automorphic representations. All these are integrated in the enterprise of the Langlands program, which has successfully weaved different disciplines of mathematics together and proved to be a cornerstone of modern number theory (cf. \cite{Gel84}, \cite{BaKn97}).

The profound theory of Eisenstein series as developed by Langlands in \cite{Lan71} is a fundamental tool for the study of the above problem regarding the spectral decomposition of $\mbf{G}(\A)$. It enables us to answer part of the above question and provides an inductive machinery that reduces the question to the understanding of the subset of so-called cuspidal automorphic representations. More precisely, using Eisenstein series, the continuous and residual spectrum in the spectral decomposition of $L^2(\mbf{G}(F)\backslash \mbf{G}(\A))$ could be understood in terms of cuspidal representations of Levi subgroups $\mbf{M}$ of $\mbf{G}$.

The residual spectrum arises from taking residues of Eisenstein series. In this way, $L$-functions appear naturally in determining the residual spectrum, as observed by Langlands (\cite{Lan71}). The poles of such $L$-functions, which are further determined by the inducing cuspidal representation on the Levi, give precise information on the location and space of such desired residues. It is in this sense that $L$-functions play an essential role in determining the residual spectrum of $\mbf{G}(\A)$. The properties of such $L$-functions could also in turn be derived from those of the Eisenstein series formed, e.g. meromorphic continuation and crude functional equation. The theory is developed and completed to some extent by various mathematicians, notably Langlands and Shahidi, and thus bears the name Langlands-Shahidi method (cf. \cite{CKM04}, \cite{Sha10}).

Moreover, to determine completely the residual spectrum, there are local considerations and thus a good understanding of local representation theory is necessary. Such interplay between global and local problems is not surprising at all. It should be mentioned that for the parametrization problem, J. Arthur (cf. \cite{Art89}) has proposed a conjectural classification of $L^2(\mbf{G}(F)\backslash \mbf{G}(\A))$, which could be viewed as a refined Langlands parametrization for automorphic representation in the view of the spectral decomposition. The conjecture could also be formulated for the double covering of $\mbf{Sp}_{2r}(\A)$ as in \cite{GGP13}.

From the spectral theory of automorphic forms for linear algebraic groups, it is natural to wonder about what could be an analogous theory for covering groups. To start with, we will concentrate on the Brylinski-Deligne type coverings $\wm{G}$  (cf. \cite{BD01})  of general reductive group $\mbf{G}$ and determine the Langlands-Shahidi type partial $L$-functions which appear naturally in the course of determining the residues of Eisenstein series.

\section{Covering groups and $L$-groups}

Covering groups of linear algebraic groups, especially those of algebraic nature, arise naturally. For example, the beautiful construction of Steinberg (cf. \cite{Ste62}) dated back to 1962 gives a simple description of the universal coverings of certain simply-connected groups. Since then, there have been investigations of covering groups by many mathematicians such as Moore, Matsumoto and Deligne, to mention a few. Connections with arithmetic have been discovered and developed. There have also been close relations between automorphic forms on covering groups and those on linear groups since the seminal paper of Shimura (\cite{Shi73}), which concerns automorphic representation of the double cover $\wm{Sp}_2(\A)$ of $\mbf{Sp}_2(\A)$.

In some sense, these could be viewed as efforts to establish a Langlands program for covering groups. As for more examples, we can mention the works by Flicker-Kazhdan (cf.\cite{FlKa86}), Kazhdan-Patterson (\cite{KaPa84}, \cite{KaPa86}), Savin (\cite{Sav04}) and many others. Such works, despite their success in treating aspects of the theory, usually focus on particular coverings rather than a general theory.

However, recently Brylinski-Deligne has developed quite a general theory of covering groups of algebraic nature in their influential paper \cite{BD01}. In particular, they classified multiplicative $\mbf{K}_2$-torsors $\wm{G}$ (equivalently in another language, central extensions by $\mbf{K}_2$) over an algebraic group $\mbf{G}$ in the Zariski site of $\Spec(F)$:
$$\seq{\mbf{K}_2}{\wm{G}}{\mbf{G}}.$$
The extension has kernel the sheaf $\mbf{K}_2$ defined by Quillen. In fact, they actually work over general schemes and not necessarily $\Spec(F)$, but for our purpose we take this more restrictive consideration. 

There are two features among others which make the Brylinski-Deligne extension distinct:
\begin{enumerate}
\item The classification of the $\mbf{K}_2$-torsors above is functorial in terms of combinatorial data. Thus, it could be viewed as a generalization of the classification of connected reductive group by root data.
\item The category is encompassing. From $\wm{G}$, we obtain the local topological extension $\wt{G}_v:=\wm{G}(F_v)$
$$\seq{\mu_n}{\wt{G}_v}{\mbf{G}(F_v)}$$
as well as global
$$\seq{\mu_n}{\wm{G}(\A)}{\mbf{G}(\A)},$$
where we have assumed $\mu_n \subseteq F$. Though topological coverings which arise in this way do not exhaust all existing ones, such Brylinski-Deligne type does contain all classically interesting examples which are of concern to us. We could mention for example coverings for split and simply-connected $\mbf{G}$ and the Kazhdan-Patterson type extension for $\mbf{G}=\mbf{GL}_n$ (cf. \cite[\S 13.2]{GaG14}).
\end{enumerate}

In the arithmetic classification of automorphic representations and local representations in terms of Galois representations, and more generally in the formation of Langlands functoriality  which has been established for several cases, a crucial role is played by the $L$-group ${}^L\mbf{G}$ of $\mbf{G}$. However, the construction of $L$-group classically is restricted only to connected reductive linear algebraic groups (cf. \cite{Bor79}).

Due to the algebraic nature of BD extensions, it is expected that the theory of automorphic forms and representations of such coverings could be developed in line with the linear algebraic case. For this purpose, a global $L$-group ${}^L\wt{G}$ and more importantly for our purpose its local analog ${}^L\wt{G}_v$ are indispensable. The latter should fit in the exact sequence
$$\seq{\wt{G}^\vee}{{}^L\wt{G}_v}{\W_{F_v}},$$
where $\wt{G}^\vee$ is the pinned complex dual group of $\wt{G}_v$ and $\W_{F_v}$ the Weil group (cf. \cite{Tat79}) of $F_v$.

There has already been a series of work in this direction starting with P. McNamara and M. Weissman (cf. \cite{McN12}, \cite{We09}, \cite{We13}, \cite{We14}). In the geometric setting, one may refer to the work of Reich (\cite{Re11}) and Finkelberg-Lysenko (\cite{FiLy10}). In particular, McNamara gave the definition of the root data of  $\wt{G}^\vee$ in order to interpret the established Satake isomorphism for $\wt{G}_v$. The root data of $\wt{G}^\vee$ rely on the degree $n$ and the root data of $\mbf{G}$, modified using the combinatorics associated with $\wm{G}$ in the BD classification. Therefore it is independent of the place $v\in |F|$, and this justifies the absence of $v$ in the notation $\wt{G}^\vee$ we use.

Since we assume $\mbf{G}$ split, it is inclined to take ${}^L\wt{G}_v$ to be just the product of $\wt{G}^\vee \times \W_{F_v}$. However, Weissman firstly realized that such approach could be insufficient, especially in view of the role that ${}^L\wt{G}_v$ should play in the parametrization of genuine representations of $\wt{G}_v$.

In his Crelle's paper \cite{We13}, Weissman gave a construction of the $L$-group ${}^L\wt{G}_v$ of certain BD covers using the language of Hopf algebras. Later in a letter to Deligne (\cite{We12}), he gave a simple construction for all BD covers utilizing the combinatorial data associated with the $\mbf{K}_2$-torsor $\wm{G}$. In a recent work \cite{We14}, the construction is realized for covering of not necessarily split $\mbf{G}$.

The insight of \cite{We13} is that an $L$-parameter is just a splitting of ${}^L\wt{G}_v$. More generally, a Weil-Deligne parameter is just a continuous group homomorphism $\xymatrix{\text{WD}_{F_v} \ar[r] & {}^L\wt{G}_v}$ such that the following  diagram commutes:
$$\xymatrix{
\WD_{F_v} \ar[r] \ar@{>>}[rd] & {}^L\wt{G}_v \ar@{>>}[d] \\
& \W_{F_v}.
}$$
Here $\WD_{F_v}:=\mbf{SL}_2(\C)\times \W_{F_v}$ is the Weil-Deligne group and the diagonal map the projection onto its second component. Moreover, the key is that \emph{even if} the group ${}^L\wt{G}_v$ as an extension
$$\seq{\wt{G}^\vee}{{}^L\wt{G}_v}{\W_{F_v}}$$
is isomorphic to the direct product $\wt{G}^\vee \times \W_{F_v}$, it is \emph{not canonically} so. This reflects the fact, locally for instance, that there is no canonical genuine representation of $\wt{G}_v$. In fact, one could show by examples that in general ${}^L\wt{G}_v$ is only a semidirect product of $\wt{G}^\vee$ and $\W_{F_v}$, see \cite{GaG14}.

To be brief, the work of Weissman has supplied us the indispensable local ${}^L\wt{G}_v$ and global ${}^L\wt{G}$ for any further development of the theory of automorphic forms on Brylinski-Deligne covers.

\section{Main results}

We assume only (the necessary ) $\mu_n \subseteq F^\times$ as opposed to $\mu_{2n}\subseteq F^\times$ in most literature on covering groups, and consider covering groups $\wm{G}(\A)$ for $\wm{G}$ arising from the Brylinski-Deligne framework and Eisenstein series induced from genuine cuspidal representation on parabolic $\wm{P}=\wm{M}\mbf{U}$. The global analysis for the spectral decomposition for more general central covering groups is carried out in the book \cite{MW95}, which also contains details of how Eisenstein series play the fundamental role in the spectral decomposition.

In order to carry out the computation, we first introduce an incarnation category which is equivalent to the Brylinski-Delgine category of multiplicative $\mbf{K}_2$-torsors over $\mbf{G}$. The definition is motivated from Weissman's paper and generalized properly here.

The aim is to compute the GK formula and interpret it as partial $L$-functions appearing in the constant term of such Eisenstein series. The knowledge of poles of the completed $L$-functions, which is yet to be fully understood even in the linear algebraic case, together with local analysis determine completely the residual spectrum $L^2_\text{res}(\mbf{G}(F)\backslash \wm{G}(\A))$.

To explain the idea which is essentially the classical one, we note that the constant term of Eisenstein series can be written as global intertwining operators which decompose into local ones. These local intertwining operators enjoy the cocycle relation, which enables us to compute by reduction to the rank one case. The outcome is the analogous Gindikin-Karpelevich formula for intertwining operators at unramified places. The GK formula gives the coefficient in terms of the inducing unramified characters, and therefore the constant term takes a form involving the global inducing data.

We note that the GK formula has been computed in \cite{McN11} using crystal basis decomposition of the integration domain. Recently, as a consequence of the computation of the Casselman-Shalika formula, McNamara also computed the GK formula in \cite{McN14}. However, our computation is carried along the classical line and removes the condition that $2n$-th root of unity lies in the field. More importantly, our GK formula is expressed in naturally defined elements. It is precisely this fact which enables us to give an interpretation in terms local Langlands-Shahidi $L$-functions.

Now we explain more on the dual side. The construction of ${}^L\wt{G}_v$ in \cite{We14} could be recast using the languages in the incarnation category. It is important that the construction is functorial with respect to Levi subgroups of $\wt{G}_v$. In particular, if $\wt{M}_v$ is a Levi of $\wt{G}_v$, there is a natural map from $\xymatrix{{}^L\varphi: {}^L\wt{M}_v \ar[r] & {}^L\wt{G}_v}$ such that the diagram
$$\xymatrix{
\wt{M}^\vee \ar@{^(->}[r] \ar@{^(->}[d]^-{\varphi^\vee} &{}^L\wt{M}_v \ar@{>>}[r] \ar[d]^-{{}^L\varphi} &\W_{F_v} \ar@{=}[d] \\
\wt{G}^\vee \ar@{^(->}[r]   &{}^L\wt{G}_v \ar@{>>}[r]  &\W_{F_v}
}$$
commutes, where $\varphi^\vee$ is the natural inclusion by construction of the dual group. In the case $\wt{M}_v=\wt{T}_v$ is the covering torus of $\wt{G}_v$, we have an explicit description of the map in terms of the incarnation language. This turns out to be essential for our interpretation of GK formula as local Langlands-Shahidi type $L$-functions later.

After recalling the construction of $L$-groups, we discuss the problem whether ${}^L\wt{G}_v$ is isomorphic to the direct product $\wt{G}^\vee \times \W_{F_v}$, and refer to \cite{GaG14} for a discussion of more properties of the $L$-group. Thus here the question is equivalent to whether there exist splittings of ${}^L\wt{G}_v$ over $\W_{F_v}$ which take values in the centralizer $Z_{{}^L\wt{G}_v}(\wt{G}^\vee)$ of $\wt{G}^\vee$ in ${}^L\wt{G}_v$.  It is shown that such splittings, which we call admissible, could arise from certain characters of $Z(\wt{T}_v)$, which we call \emph{qualified}. There is even a subclass of qualified characters of $Z(\wt{T}_v)$ which we name as  \emph{distinguished} characters. In the simply-connected case, there is no obstruction to the existence of distinguished characters, while in general there is. One property of qualified characters is that they are Weyl-invariant, which holds in particular for distinguished characters. This could be considered as a generalization of \cite[Cor. 5.2]{LoSa10}, where the authors use global methods to show the Weyl-invarance of certain unramified principal series for degree two covers of simply-connected groups. We give an explicit construction of distinguished characters and show that they agree with those in the case of double cover $\wm{Sp}_{2r}(F_v)$ (cf. \cite{Rao93} \cite{Kud96}) and simply-laced simply-connected case treated by Savin (cf. \cite{Sav04}). 

As a consequence of the discussion above on the admissible splittings of ${}^L\wt{G}_v$ applied to the case $\wt{G}_v=\wt{T}_v$, we obtain a local Langlands correspondence (LLC) for covering tori. More precisely, any genuine character $\wt{\chi}$ of $Z(\wt{T}_v)$ is qualified and gives rise to a splitting $\rho_{\wt{\chi}}$ of ${}^L\wt{T}_v$ over $\W_{F_v}$. Coupled with the Stone von-Neumann theorem which gives a bijection between isomorphism classes of irreducible genuine characters $\Hom_\epsilon(Z(\wt{T}_v), \C^\times)$ of $Z(\wt{T}_v)$ and irreducible representations $\text{Irr}_\epsilon(\wt{T}_v)$ of $\wt{T}_v$, this correspondence could be viewed with $\Hom_\epsilon(Z(\wt{T}_v), \C^\times)$ replaced by $\text{Irr}_\epsilon(\wt{T}_v)$. In the case $n=1$, it recovers the LLC for linear tori.

Back to the case of general $\wt{G}_v$, because of the compatibility between ${}^L\wt{T}_v$ and ${}^L\wt{G}_v$ above given by ${}^L\varphi$, the splitting $\rho_{\wt{\chi}}$ could be viewed as a splitting of ${}^L\wt{G}_v$ over $\W_{F_v}$. On the $L$-group ${}^L\wt{G}_v$ we could define the adjoint representation
$$\xymatrix{
Ad: \quad {}^L\wt{G}_v \ar[r]  & GL(\wt{\mfr{g}}^\vee).
}$$
Our local Langlands correspondence for representations of covering tori gives locally a splitting of ${}^L\wt{G}_v$ over $\W_{F_v}$ which arises from a splitting $\xymatrix{{}^L\wt{T}_v \ar@{>>}[r] &\W_{F_v}}$. We express the GK formula for unramified principal series in terms of the composition $Ad\circ {}^L\varphi \circ \rho_{\wt{\chi}}$.

The discussion can be carried in parallel for the global setting; more importantly, we have local and global compatibility. For example, one can consider similarly admissible splittings of the global ${}^L\wt{G}$; there is the adjoint representation of the ${}^L\wt{G}$, which by restriction to ${}^L\wt{G}_v$ is just the adjoint representation of ${}^L\wt{G}_v$ above. 

We will define automorphic (partial) $L$-function of an automorphic representation $\wt{\sigma}$ of $\wm{H}(\A)$ of BD type associated with a finite dimensional representation $\xymatrix{R: {}^L\wt{H} \ar[r] & GL(V)}$. In particular, we are interested in the case where $\wm{H}=\wm{M}$ is a Levi of $\wm{G}$ and that $R$ is the adjoint representation of ${}^L\wt{M}$ on a certain subspace $\wt{\mfr{u}}^\vee$ of the Lie algebra $\wt{\mfr{g}}^\vee$. 

In view of this, the constant term of Eisenstein series for induction from general parabolics can be expressed in terms of certain Langlands-Shahidi type $L$-functions, by combining the formula from the unramified places. We work out the case for maximal parabolic, and the general case is similar despite the complication in notations.

As simple examples, we will determine the residual spectra of arbitrary degree BD covers $\wm{SL}_2(\A)$ and $\wm{GL}_2(\A)$ of $\mbf{SL}_2(\A)$ and $\mbf{GL}_2(\A)$ respectively. We also compute the partial $L$-functions appearing in the constant terms of Eisenstein series for induction from maximal parabolic of the double cover $\wm{Sp}_4(\A)$. It is shown to agree with that given in \cite{Gao12}.

In the end, we give brief discussions on immediate follow-up or future work that we would like to carry out. For instance, we would like to explore in details the Kazhdan-Patterson covers (cf. \cite{KaPa84}) from the BD-perspective. Also since the construction of ${}^L\wt{G}$ by Weissman is actually for $\mbf{G}$ not necessarily split, one can readily implement the computation here with proper modifications and expect same Langlands-Shahidi $L$-function appears. Moreover, to determine the residual spectrum of $\wm{G}(\A)$, a natural step in the sequel would be to develop a theory of local $L$-functions, which in the case of metaplectic extension $\wm{Sp}_{2r}(F_v)$ has been covered by the work of Szpruch. In his thesis, the Langlands-Shahidi method is extended to such groups for generic representations. Moreover, such a theory of local $L$-functions would lay foundations for the theory of converse theorems, which perhaps could be used to provide links between these completed Langlands-Shahidi $L$-functions arising from BD covering groups and those from linear algebraic groups.

\chapter{The Brylinski-Deligne extensions and their $L$-groups}
\chaptermark{BD extensions and their $L$-groups}

\section{The Brylinski-Deligne extensions and basic properties}
In this section, let $F$ be a number field or its localization. We will be more specific when the context requires so. Let $\mbf{G}$ be a split reductive group over $F$ with root data $(X, \Psi, Y, \Psi^\vee)$. We also fix a set of simple roots $\Delta\subseteq \Psi$. 

In their seminal paper \cite{BD01}, Brylinski and Degline have studied a certain category of central extensions of $\mbf{G}$ and given a classification of such objects in terms of combinatorial data. We will recall in this section the main results of that paper and state some properties which are important for our consideration later.

A central extension $\wt{\mbf{G}}$ of $\mbf{G}$ by $\mbf{K}_2$ is an extension in the category of sheaves of groups on the big Zariski site over $\Spec(F)$. It is written in the form
$$\xymatrix{
\mbf{K}_2 \ar@{^(->}[r] & \wt{\mbf{G}} \ar@{>>}[r] & \mbf{G}.
}$$
The category of such central extensions of $\mbf{G}$ is denoted by $\CExt(\mbf{G},\mbf{K}_2)$.

Any $\wt{\mbf{G}} \in \CExt(\mbf{G},\mbf{K}_2)$ gives an exact sequence of $F'$-rational points for any field extension $F'$ of $F$:
$$\xymatrix{
\mbf{K}_2(F') \ar@{^(->}[r] & \wt{\mbf{G}}(F') \ar@{>>}[r] & \mbf{G}(F').
}$$
The left exactness follows from the the fact that the extension $\wt{\mbf{G}}$ is an extension of sheaves, while the right exactness at last term is due to the vanishing of $\text{H}_\text{Zar}^1(F',\mbf{K}_2)$, an analogue of Hilbert Theorem 90.

We will recall the classification of such extensions for $\mbf{G}$ being a torus, a semi-simple simply-connected group and a general reductive group in the sequel.

\subsection{Central extensions of tori} \label{torsor/t}
Let $\mbf{T}$ be a split torus with character group $X=X(\mbf{T})$ and cocharacter group $Y=Y(\mbf{T})$. The category $\CExt (\mbf{T},\mbf{K}_2)$ of central extensions of $\mbf{T}$ by $\mbf{K}_2$ is described as follows.

\begin{thm} \label{torsor/t}
Let $\mbf{T}$ be a split torus over $F$. The category of central extensions  $\CExt(\mbf{T},\mbf{K}_2)$ is equivalent to the category of pairs $(Q, \mca{E})$, where $Q$ is a quadratic form on $Y$ and $\mca{E}$ is a central extension of $Y$ by $F^\times$
$$\seq{F^\times}{\mca{E}}{Y}$$
such that the commutator $\xymatrix{ Y\times Y \ar[r] & F^\times}$ is given by
$$\xymatrix{
[-, -]: \quad  (y_1, y_2) \ar@{|->}[r] & (-1)^{B_Q(y_1, y_2)}.
}$$
Here $B_Q$ is the symmetric bilinear form associated with $Q$, i.e. $B_Q(y_1,y_2)=Q(y_1+y_2)-Q(y_1)-Q(y_2)$.
\end{thm}

Note that the commutator, which is defined on the group $\mca{E}$, descends to $Y$ since the extension is central. For any two pairs $(Q, \mca{E})$ and $(Q', \mca{E}')$, the group of morphisms exists if and only if $Q=Q'$, in which case it is defined to consist of the isomorphisms between the two extensions $\mca{E}$ and $\mca{E}'$.

To recall the functor $\xymatrix{
\CExt(\mbf{T},\mbf{K}_2) \ar@{~>}[r] & \set{(Q,\mca{E})}
}$, assume we are given with $\wm{T}\in \CExt(\mbf{T}, \mbf{K}_2)$. The quadratic from $Q$ thus obtained does not allow for a simple description, and we refer to \cite[\S 3.9-3.11]{BD01} for the details. However, the description of $\mca{E}$ is relatively simple and we reproduce it here.

Start with $\wm{T} \in \CExt(\mbf{T},\mbf{K}_2)$ over $F$. Taking the rational points of the Laurent field $F(\!(\tau)\!)$ gives
$$\seq{\mbf{K}_2(F(\!(\tau )\!) )}{\wt{\mbf{T}}(F(\!(\tau)\!))}{\mbf{T}(F(\!(\tau)\!))}. $$

Pull-back by $\xymatrix{Y \ar[r] & \mbf{T}(F(\!(\tau)\!) )}$ which sends $y\in Y$ to $y\otimes \tau \in \mbf{T}(F(\!(\tau)\!) )$, and then push-out by the tame symbol $\xymatrix{\mbf{K}_2(F(\!(\tau)\!) ) \ar[r] & F^\times}$ give the extension $\mca{E}$ over $Y$ by $F^\times$. Here the tame symbol is defined to be
$$\set{f, g} \mapsto (-1)^{\text{val}(f)\text{val}(g)} \frac{f^{\text{val}(g)}}{g^{\text{val}(f)}}(0).$$
In particular, $\set{a, \tau} \in \mbf{K}_2(F(\!(\tau )\!) )$ is sent to $a$ for all $a\in F^\times$. This process describes the construction of $\mca{E}$. 

For convenience, for any lifting $\wt{y\otimes \tau} \in \wm{T}(F(\!(\tau)\!))$ of $y\otimes \tau\in \mbf{T}(F(\!(\tau)\!))$, we write 
\begin{equation} \label{[] notation}
\elt{\wt{y\otimes \tau}}:= \text{ the image of } \wt{y\otimes \tau} \text{ in } \mca{E}.
\end{equation}

Moreover, any $\wm{T} \in \CExt(\mbf{T},\mbf{K}_2)$ is isomorphic to $\mbf{K}_2 \times_D \mbf{T}$, where $D$ is a (not necessarily symmetric) bilinear form on $Y$ such that $D(y_1, y_2) + D(y_2, y_1)=B_Q(y_1, y_2)$. The trivialized torsor $\mbf{K}_2 \times_D \mbf{T}$ is thus endowed with a multiplicative structure described as follows.

Write $D=\sum_i x_1^i\otimes x_2^i \in X\otimes_\Z X$. Then the cocycle of $\mbf{K}_2(F') \times_D \mbf{T}(F')$, for $F'$ any field extension of $F$, is given by
\begin{equation} \label{torus law}
\sigma_D(t_1, t_2)=\prod_i \set{x_1^i(t_1), x_2^i(t_2)},\quad  t_1, t_2 \in \mbf{T}(F').
\end{equation}

Now it is easy to check that the commutator of $\mca{E}$ is given by the formula $[y_1, y_2]=(-1)^{B_Q(y_1, y_2)}$.

\subsection{Central extensions of semi-simple simply-connected groups} \label{torsor/sc}
Let $\mbf{G}$ be a split semi-simple simply-connected group over $F$ with root data $(X, \Psi, Y, \Psi^\vee)$. Let $\mbf{T}$ be a maximal split torus of $\mbf{G}$ with character group $X$ and cocharacter group $Y$. By the perfect pairing of $X$ and $Y$, $\text{Sym}^2(X)$ is identified with integer-valued quadratic forms on $Y$. Let $W$ be the Weyl-group of $\mbf{G}$. We have the following classification theorem for $\CExt(\mbf{G}, \mbf{K}_2)$.

\begin{thm} \label{torsor/sc}
The category $\CExt(\mbf{G}, \mbf{K}_2)$ is rigid, i.e., any two objects have at most one morphism between them. The set of isomorphism classes is classified by $W$-invariant integer-valued quadratic forms $\xymatrix{Q: Y\ar[r] & \Z}$, i.e., by $Q\in \text{Sym}^2(X)^W.$
\end{thm}

A special case of $\mbf{G}$ is when it is almost simple. In this case we can identify
$$\xymatrix{
\text{Sym}^2(X)^W \ar[r] &\Z, \quad Q \ar@{|->}[r] & Q(\alpha^\vee),
}$$
where $\alpha^\vee \in \Psi^\vee$ is the short coroot associated to any long root. The fact that $Q(\alpha^\vee)$ for short coroot uniquely determines the quadratic form $Q$ follows from the following easy fact.

\begin{lm} \label{BQ<>}
For any $\alpha^\vee \in \Psi^\vee$ and $y\in Y$,
$$B_Q(\alpha^\vee, y)=Q(\alpha^\vee) \cdot \angb{\alpha}{y},$$
where $\angb{-}{-}$ denotes the paring between $X$ and $Y$.
\end{lm}
\begin{proof}
The Weyl invariance property of $B_Q$ follows from that of $Q$, and it gives
\begin{align*}
B_Q(\alpha^\vee, y) &= B_Q(s_{\alpha^\vee}(\alpha^\vee), s_{\alpha^\vee}(y)) \\
 &=B_Q(-\alpha^\vee, y- \angb{\alpha}{y} \alpha^\vee) \\
&= -B_Q(\alpha^\vee, y) + 2\angb{\alpha}{y} \cdot Q(\alpha^\vee).
\end{align*}
The claim follows.
\end{proof}

\begin{eg}
The classical metaplectic double cover arises from a central extension $\wm{Sp}_{2r}$ over $\mbf{Sp}_{2r}$ of this type. Let $\alpha_1^\vee, \alpha_2^\vee, ..., \alpha_r^\vee$ be the simple coroots of $\mbf{Sp}_{2r}$ with $\alpha_1^\vee$ the unique short one. Let $Q$ be the unique Weyl invariant quadratic form on $Y$ with $Q(\alpha_1^\vee)=1$, see also \cite[pg. 7-8]{BD01}. This gives the desired $\wm{Sp}_{2r}$ according to the above classification theorem.
\end{eg}

\subsection{Central extensions of general split reductive groups} \label{torsor/red}
Now we fix a split reductive group $\mbf{G}$ and a maximal split torus $\mbf{T}$ over $F$ with root data $(X, \Psi, Y, \Psi^\vee)$. The classification of $\CExt(\mbf{G}, \mbf{K}_2)$ relies on more data in the description. It is a combined result from both the classifications of $\mbf{K}_2$-torsors over tori and of semisimple simply-connected groups in \ref{torsor/t} and \ref{torsor/sc} respectively. The following is the main result by Brylinski and Deligne in the split case.

\begin{thm} \label{torsor/red}
Let $\mbf{G}$ be a split connected reductive group over $F$ with maximal split torus $\mbf{T}$. Let $X$ and $Y$ be the character and cocharacter groups of $\mbf{T}$ respectively. The category $\CExt(\mbf{G}, \mbf{K}_2)$ is equivalent to the category specified by the triples $(Q, \mca{E}, \phi)$ with the following properties:

The $Q$ is a Weyl invariant quadratic form on $Y$ and $\mca{E}$ a central extension
$$\seq{F^\times}{\mca{E}}{Y},$$
such that the commutator $\xymatrix{[-, -]: Y\times Y \ar[r] & F^\times}$ is given by
$$[y_1, y_2] =(-1)^{B_Q(y_1,y_2)}.$$

Let $\xymatrix{\Phi: \mbf{G}^{sc} \ar[r] & \mbf{G}^{der} \ar[r] & \mbf{G} }$ be the natural composition, where $\mbf{G}^{sc}$ is the simply connected cover of the derived group $\mbf{G}^{der}$ of $\mbf{G}$. Let $\mbf{T}^{sc} =\Phi^{-1} (\mbf{T})$ be a maximal split torus of $\mbf{G}^{sc}$ with cocharacter group $Y^{sc} \subseteq Y$. The restriction $Q|_{Y^{sc}}$ gives an element $\wm{G}^{sc}\in \CExt (\mbf{G}^{sc}, \mbf{K}_2)$ unique up to unique isomorphism by Theorem \ref{torsor/sc}, which by further pull-back to the torus $\mbf{T}^{sc}$ gives a central extension $\wm{T}^{sc}$ by $\mbf{K}_2$. Therefore, we have from Theorem \ref{torsor/t} a corresponding central extension
$$\seq{F^\times}{\mca{E}^{sc}}{Y^{sc}}.$$
The requirement on $\phi$ is that it is a morphism from $\mca{E}^{sc}$ to $\mca{E}$ such that the following diagram commute:
$$\xymatrix{
F^\times \ar@{^(->}[r]  \ar@{=}[d] &\mca{E}^{sc}  \ar@{>>}[r]  \ar[d]^-\phi  & Y^{sc} \ar@{^(->}[d] \\
F^\times \ar@{^(->}[r]  &\mca{E}  \ar@{>>}[r]  & Y.
}$$

Homomorphisms between two triples $(Q_1, \mca{E}_1, \phi_1)$ and $(Q_2, \mca{E}_2, \phi_2)$ exist only for $Q_1=Q_2$, in which case they are defined to be the homomorphisms between $\mca{E}_1$ and $\mca{E}_2$ which respect the above commutative diagram.
\end{thm}

In fact, the theorem stated here could be strengthened, since the category $\CExt(\mbf{G}, \mbf{K}_2)$ and the category consisting of $(Q,\msc{E}, \phi)$ are both commutative Picard categories with respect to the Baer sum operation. 

In general, for any group $C$ and an abelian group $A$, we view $A$ as a trivial $C$-module. Then, the second cohomology group $H^2(C, A)$ classifies the isomorphism classes of central extensions of $C$ by $A$. The group law on $H^2(C, A)$ is then realized as the Baer sum.

Recall the definition of Baer sum. \index{Baer sum}
Let $\seq{A}{E_i}{C}$ be two central extensions with $A$ an abelian group (written multiplicatively say). Let
$$\xymatrix{\delta: C \ar[r] & C \times C }$$
be the diagonal map and let 
$$\xymatrix{m: A \times A \ar[r] & A }$$ be the multiplication map. 
From $E_1$ and $E_2$, we obtain the extension $E_1\times E_2$ of $C\times C$ by $A\times A$ by forming the Cartesian product. Then by definition the Baer sum of $E_1$ and $E_2$ is given by
$$E_1\oplus_B E_2 :=\Big(\delta^*\circ m_* (E_1\times E_2) \simeq m_* \circ \delta^* (E_1\times E_2) \Big).$$

Thus for example, the additive structure for the category $\set{(Q,\mca{E}, \phi)}$ is given such that the sum of two $(Q_i,\mca{E}_i, \phi_i)$ for $i=1, 2$ is by definition
$$(Q_1+ Q_2, \ \mca{E}_1 \oplus_B \mca{E}_2, \ \phi_1 \oplus_B \phi_2),$$
where $\phi_1\oplus_B \phi_2$ is the obvious induced map.

\begin{thm}[{\cite[\S 7]{BD01}}] \label{Pic equiv}
The equivalence of the two categories in Theorem \ref{torsor/red} respects the Picard structure, i.e., it establishes an equivalence between the two commutative Picard categories.
\end{thm}

\subsection{The Brylinski-Deligne section} \label{BD section}
Assume $\mbf{G}$ reductive and $\mbf{G}^{sc}$ the semisimple simply-connected group as before. From the $\mbf{K}_2$-torsor $\wt{\mbf{T}}$ over $\mbf{T}$ we have constructed the extension $\mca{E}$. By pull-back, any $\mbf{K}_2$-torsor $\wt{\mbf{G}}$ gives a $\mbf{K}_2$-torsor $\wm{G}^{sc}$ over $\mbf{G}^{sc}$. Further restriction gives the covering $\wm{T}^{sc}$ of $\mbf{T}^{sc} \subseteq \mbf{G}^{sc}$.  Similarly one obtains $\mca{E}^{sc}$ in the same way starting from $\wm{T}^{sc}$. It is possible to characterize $\wt{\mbf{T}}^{sc}$ and the corresponding extension 
$$\seq{F^\times}{\mca{E}^{sc}}{Y^{sc}}$$
given by Theorem \ref{torsor/t} which arise in this way, among general $\mbf{K}_2$-torsors over $\mbf{T}^{sc}$ associated with the same $Q$ (cf. \cite[\S 11]{BD01}).

For the time being, we also denote by $\xymatrix{\Phi: \wm{G}^{sc} \ar[r] & \wm{G}}$ the natural pull-back map, and which restricts to the tori to give the middle map of the following diagram 
$$\xymatrix{
\mbf{K}_2 \ar@{^(->}[r] \ar@{=}[d] & \wm{T}^{sc} \ar@{>>}[r] \ar[d]^-{\Phi} &\mbf{T}^{sc} \ar[d]^-{\Phi} \\
\mbf{K}_2 \ar@{^(->}[r] &\wt{\mbf{T}} \ar@{>>}[r] &\mbf{T}.
}$$

Let $\alpha \in \Psi$, and let $\wm{T}_\alpha^{sc}$ be the pull-back of $\wm{G}^{sc}$ to the one-dimensional torus $\mbf{T}_\alpha^{sc} \subseteq \mbf{T}^{sc}$. What is important to us is that $\wm{T}_\alpha^{sc}$ is endowed with a natural section over $\mbf{T}_\alpha^{sc}$, which depends on the \'epinglage we fix for $\mbf{G}$.

\subsubsection{The Bylinski-Deligne section of $\wm{T}^{sc}_\alpha$} \index{Brylinski-Deligne section}
Recall that the extension $\wt{\mbf{G}}$ splits uniquely over the unipotent subgroup $\mbf{U}\subseteq \mbf{G}$, and this splitting is $\mbf{B}=\mbf{T}\mbf{U}$-equivariant (cf. \cite[Prop. 11.3]{BD01}). Here $\mbf{U}$ could be viewed as the unipotent radical of the Borel  subgroup $\mbf{B}^{sc}=\mbf{T}^{sc} \mbf{U}$ of $\wm{G}^{sc}$, which splits uniquely in $\wm{G}^{sc}$ and thus is compatible with the map $\xymatrix{\Phi: \wm{G}^{sc} \ar[r] & \wm{G}}$. We denote by $\wm{e}\in \wm{U}$ the image of this splitting of any element $\mbf{e} \in \mbf{U}$, and no confusion will arise on the context of such definition, i.e. with respect to $\mbf{G}$ or $\mbf{G}^{sc}$.

We fix a Chevalley system of e\'pinglage for $(\mbf{G}, \mbf{T}, \mbf{B})$ (cf. \cite[\S 3.2.1-2]{BrTi84}). In particular, for each $\alpha\in \Psi$ with associated root subgroup $\mbf{U}_\alpha$, we have a fixed isomorphism $\xymatrix{\mbf{e}_\alpha: \mbf{G}_\text{add} \ar[r] & \mbf{U}_\alpha}$. Also, there is the induced morphism $\xymatrix{\varphi_\alpha: \mbf{SL}_2 \ar[r] & \mbf{G}}$. In fact, the data for the e\'pinglage above gives an e\'pinglage of $(\mbf{G}^{sc}, \mbf{T}^{sc}, \mbf{B}^{sc})$, and we still denoted by $\xymatrix{\varphi_\alpha: \mbf{SL}_2 \ar@{^(->}[r] & \mbf{G}^{sc}}$ the induced morphism which is an injection in this case.

Let $a\in \mbf{G}_\text{mul}$, consider $\mbf{e}_+(a), \mbf{e}_-(a), \mbf{w}_o(a)$ of $\mbf{SL}_2$ as follows:
\begin{equation*}
    \mbf{e}_+(a) = \left(
      \begin{array}{cccc}
        1 & a \\
        0 & 1
      \end{array} \right),  \quad
\mbf{e}_-(a) = \left(
      \begin{array}{cccc}
        1 & 0 \\
        -a & 1
      \end{array} \right),
\end{equation*}
\begin{equation*}
\mbf{w}_o(a) =\mbf{e}_+(a)\mbf{e}_-(a^{-1})\mbf{e}_+(a)=\left(
      \begin{array}{cccc}
        0 & a \\
        -a^{-1} & 0
      \end{array} \right), \quad
\mbf{h}_o(a) = \mbf{w}_o(a) \mbf{w}_o(-1)= \left(
      \begin{array}{cccc}
        a & 0 \\
        0  &  a^{-1}
      \end{array} \right).
\end{equation*}

By the Tits trijection (cf. \cite[\S 11]{BD01}) we mean the triple $\mbf{e}_\alpha(a), \mbf{e}_{-\alpha}(a^{-1}), \mbf{w}_\alpha(a)  \in \mbf{G}^{sc}$ given by
$$\mbf{e}_\alpha(a)=\varphi_\alpha(\mbf{e}_+(a)),\quad  \mbf{e}_{-\alpha}(a^{-1})=\varphi_\alpha(\mbf{e}_-(a^{-1})), \quad \mbf{w}_\alpha(a):=\varphi_\alpha(\mbf{w}_o(a)).$$
We also write $\mbf{h}_\alpha(a):=\varphi_\alpha(\mbf{h}_o(a))$ and thus $\mbf{h}_\alpha(a)=\mbf{w}_\alpha(a) \mbf{w}_\alpha(-1)$.

Now we can proceed to describe the Brylinski-Deligne (BD) section $\wm{h}_\alpha^{[b]}$ (which depends on $b\in \mbf{G}_\text{mul}$) of $\wm{T}^{sc}_\alpha$ over $\mbf{T}^{sc}_\alpha \simeq \mbf{G}_\text{mul}$.

In particular, we describe the BD section at the level of $F'$-rational points, where $F'/F$ is a field extension. That is, for any $b\in \mbf{G}_\text{mul}(F')=(F')^\times$, we have the BD section $\wm{h}_\alpha^{[b]}$:
$$\xymatrix{
\mbf{K}_2(F') \ar@{^(->}[r] & \wm{T}_\alpha^{sc}(F') \ar@{>>}[r] & \mbf{T}_\alpha^{sc}(F') \ar@/_1.3pc/[l]_-{\wm{h}_\alpha^{[b]}}.
}$$

Recall the definition of $\wm{h}_\alpha^{[b]}$ as follows. For any $a\in (F')^\times$, first define a lifting $\wm{w}_\alpha(a) \in \wm{G}^{sc}(F')$  of the element $\mbf{w}_\alpha(a) \in N(\mbf{T})(F')$ by
$$\xymatrix{
\mbf{w}_\alpha(a) \ar@{|->}[r] & \wm{w}_\alpha(a):= \wm{e}_\alpha(a)\cdot \wm{e}_{-\alpha}(a^{-1}) \cdot \wm{e}_\alpha(a).
}$$ 

The BD section $\wm{h}_\alpha^{[b]}(a)$ of $\wt{\mbf{T}}^{sc}(F')$ over $\mbf{T}^{sc}(F')$ is then by definition (cf. \cite[\S 11.1]{BD01})
$$\wm{h}_\alpha^{[b]}(a):= \wm{w}_\alpha(ab) \cdot \wm{w}_\alpha(b)^{-1}.$$

Two important properties of this section are
\begin{align} 
\wm{h}^{[b]}_\alpha(a) \cdot \wm{h}^{[b]}_\alpha(c) & =\wm{h}^{[b]}_\alpha(ac) \cdot \set{a, c}^{Q(\alpha^\vee)},  \label{h ppty1} \\
\wm{h}^{[db]}_\alpha(a) &= \wm{h}^{[b]}_\alpha(a) \cdot \set{d, a}^{Q(\alpha^\vee)}.  \label{h ppty2}
\end{align}

This section described above gives rise to an inherited lifting into $\mca{E}^{sc}$ of the one-dimension lattice $Y_\alpha^{sc} \subseteq Y^{sc}$ spanned by $\alpha^\vee \in \Psi^\vee$. More precisely, apply the case $F'=F(\!(\tau)\!)$ and pick any nonzero $f\in F(\!(\tau)\!)$. Let $\wm{h}_\alpha^{[f]}$ be the BD section $\wm{T}_\alpha^{sc}(F(\!(\tau)\!))$ over $\mbf{T}_\alpha^{sc}(F(\!(\tau)\!))$. It induces a section over $Y_\alpha^{sc}(\tau) \subseteq \mbf{T}^{sc}(F(\!(\tau)\!))$. Thus, we obtain an inherited section of $\mca{E}^{sc}$ over $Y_\alpha^{sc}$, which is still denoted by $\wm{h}_\alpha^{[f]}$. Write $\mca{E}_\alpha^{sc}$ for the pull-back of $\mca{E}^{sc}$ via $Y_\alpha^{sc}\subseteq Y^{sc}$, then we have the section
$$\xymatrix{
F^\times \ar@{^(->}[r] & \mca{E}_\alpha^{sc} \ar@{>>}[r] & Y_\alpha^{sc} \ar@/_1.2pc/[l]_-{\wm{h}_\alpha^{[f]}}.
}$$

Recall that for any lifting $\wt{y\otimes \tau} \in \wm{T}^{sc}(F(\!(\tau)\!))$ of $y\otimes \tau\in \mbf{T}^{sc}(F(\!(\tau)\!))$, we have denoted by $\elt{\wt{y\otimes \tau}}\in \mca{E}^{sc}$ its image in $\mca{E}^{sc}$, as in (\ref{[] notation}). In particular for any $k\in \Z$,
\begin{equation}
\elt{\wm{h}^{[f]}_\alpha(\tau^k)}:=\text{ the image of } \wm{h}^{[f]}_\alpha(\tau^k) \in \wm{T}^{sc}(F(\!(\tau)\!)) \text{ in } \mca{E}^{sc}.
\end{equation}

\begin{dfn} \label{dfn BD section}
Consider 
$$\wm{h}_\alpha^{[f]}(k\cdot \alpha^\vee):= \elt{\wm{h}^{[f]}_\alpha(\tau^k)},$$
which is a lifting of $k\alpha^\vee \in Y_\alpha^{sc}$. We call it the Brylinski-Deligne section of $\mca{E}_\alpha^{sc}$ over $Y_\alpha^{sc}$.
\end{dfn}

\subsubsection{Rigidifying $\mca{E}^{sc}$}
Let $\mca{E}_Q^{sc}$ be the abstract group generated by $\set{a}_{a\in F^\times} \cup \set{\gamma_\alpha}_{\alpha^\vee \in \Delta^\vee}$ subject to the conditions:
\begin{align*}
& (i) \quad  F^\times \text{ is contained in the center of } \mca{E}_Q^{sc}, \\
& (ii) \quad [\gamma_\alpha, \gamma_\beta]=(-1)^{B_Q(\alpha^\vee, \beta^\vee)} \text{ for any } \alpha^\vee, \beta^\vee \in \Delta^\vee.
\end{align*} 
We obtain the exact sequence $\seq{F^\times}{\mca{E}^{sc}_Q}{Y}$, which is uniquely determined by requiring that $a \in F^\times$ sent to the generator $a$ of $\mca{E}^{sc}_Q$ and $\gamma_\alpha$ to $\alpha^\vee$.

Thus it is possible to rigidify the extension $\mca{E}^{sc}$ obtained above by using the unique isomorphism given by
$$\xymatrix{
\mca{E}^{sc} \ar[r]^-\simeq & \mca{E}_Q^{sc}, \quad \wm{h}^{[1]}_\alpha(\alpha^\vee) \ar@{|->}[r] & \gamma_\alpha,
}$$
where $\wm{h}^{[1]}_\alpha(\alpha^\vee)$ is just $\wm{h}^{[f]}_\alpha(\alpha^\vee)$ for $f=1$.

\section{Incarnation functor and an equivalent category}

\subsection{Equivalence between the incarnation category and the BD category}

By adapting and modifying the definition in \cite{We13}, one can define a Picard category $\Bis_{\mbf{G}}=\bigsqcup_Q \Bis_{\mbf{G}}^Q$ and an incarnation functor 
$$\xymatrix{
\Inc_{\mbf{G}}: \quad \Bis_{\mbf{G}} \ar@{~>}[r] &\CExt(\mbf{G},\mbf{K}_2),
}$$
which gives an equivalence of Picard categories. That is, $\Inc_{\mbf{G}}$ is fully faithful and essentially surjective, namely,  surjective onto the isomorphism classes of $\CExt(\mbf{G},\mbf{K}_2)$. Moreover, it respects the Picard structures on both categories. Because of this equivalence, we can concentrate and work in the category $\Bis_{\mbf{G}}$.

It is important for us to have a description of $\Bis_{\mbf{G}}^Q$, and so we recall the definition here.

\begin{dfn} \index{incarnation category}
The category $\Bis_{\mbf{G}}^Q$ consists of pairs $(D, \eta)$, where $D$ is a $\Z$-valued bilinear (not necessarily symmetric) form on $Y$ such that $D(y_1, y_2) + D(y_2, y_1) =B_Q(y_1, y_2)$ and $\xymatrix{\eta: Y^{sc} \ar[r] & F^\times}$ a group homomorphism. In particular, 
$$D(y, y)=Q(y).$$

We call $D$ a bisector of $Q$. \index{bisector}
Morphisms of pairs $(D_i, \eta_i)$ for $i=1, 2$ consist of maps $\xymatrix{H: Y \ar[r] & F^\times}$ (not necessarily a homomorphism) such that
\begin{align*}
& (i) \quad  (-1)^{D_2(y_1,y_2)- D_1(y_1,y_2)}=H(y_1+y_2) \cdot H(y_1)^{-1} \cdot H(y_2)^{-1},\\
& (ii) \quad  \eta_2(\alpha^\vee)/\eta_1(\alpha^\vee)= H(\alpha^\vee) \text{ for all } \alpha^\vee \in \Delta^\vee.
\end{align*}

The composition of two morphisms is given by multiplication, i.e., $H_1\circ H_2(y)=H_1(y)\cdot H_2(y)$. 
\end{dfn}

\begin{rmk}
Note that in $(ii)$ we do not require $\eta_2/\eta_1 =H|_{Y^{sc}}$ which will enforce $H|_{Y^{sc}}$ to be a homomorphism, which is unnecessary and in fact not true for the desired equivalence of categories mentioned. Also the definition given above is not exactly the same as in \cite{We13}, where $D$ is assumed with some fairness condition. The morphism between two objects here is also defined in a less restrictive way.
\end{rmk}

It is shown in \cite[Prop. 2.4]{We13} that for $(D_i, \eta_i), i=1, 2$ associated with the same $Q$, there always exists $H$ satisfying $(i)$. Consequently, we see that up to isomorphism we could always fix a base $D$ and allow $\eta$ to be varied. More precisely, we have the following.

\begin{eg} \label{D1 to D2}
Let $D_1, D_2$ be two bisectors of $Q$. Then $(D_1, \eta_1)$ for any $\eta_1$ is isomorphic to $(D_2, \eta_2)$ for some $\eta_2$. We explain how the $\eta_2$ can be obtained. Pick $\xymatrix{H: Y\ar[r] & F^\times}$ such that the property $(i)$ is satisfied with respect to $D_1$ and $D_2$. Define $\eta_2$ to be such that
$$\eta_2(\alpha^\vee)/\eta_1(\alpha^\vee) = H(\alpha^\vee) \text{ for all } \alpha^\vee \in \Delta^\vee.$$
Then $(D_1,\eta_1)\simeq (D_2, \eta_2)$ for $\eta_2$ obtained in this way.

On the other hand, we observe that not every $(D_2, \eta)$ is isomorphic to $(D_1, \mbf{1})$ for some $D_1$. Suppose on the contrary there is a $H$ which realizes an isomorphism from $(D_1, \mbf{1})$ to $(D_2, \eta)$. Then property $(i)$ implies $H(ky)=H(y)^k$ for all $k\in \Z, y\in Y$. If $\alpha^\vee =ky$ for some $y\in Y$, this implies necessarily by $(b)$
$$\eta(\alpha^\vee)=H(y)^k \in (F^\times)^k.$$
However, in general $\eta$ may not satisfy such a condition. A concrete example is for $\mbf{G}=\mbf{PGL}_2$ with $Q=0$ with $\eta(\alpha^\vee)\in F^\times\backslash (F^\times)^2$ where $\alpha^\vee \in Y$ is the coroot.
\end{eg}

\begin{eg}
Assume that $\mbf{G}$ has a simply-connected derived group $\mbf{G}^\text{der}$. Then $Y/Y^{sc}$ is a free $\Z$-module. Let $(D, \eta)\in \Bis_{\mbf{G}}^Q$, then $(D,\eta) \simeq (D, \mbf{1})$. In fact, any $H\in \Hom(Y, F^\times)$ extending $\eta$ will provide a morphism from $(D,\mbf{1})$ to $(D, \eta)$.

We also claim that for such $\mbf{G}$ any  two $(D_i, \eta_i) \in \Bis_{\mbf{G}}^Q, i=1,2$ are isomorphic. This can be seen from the composition of isomorphisms:
\begin{equation} \label{D-eta free}
\xymatrix{
(D_1, \eta_1) \ar[r] & (D_1, \mbf{1}) \ar[r]^-H & (D_2, \eta_H) \ar[r] & (D_2, \mbf{1}) \ar[r] & (D_2, \eta_2),  
}
\end{equation}
\noindent where the second isomorphism exists and depends on $H$ as in previous example.
\end{eg}

The incarnation functor from $\Bis_{\mbf{G}}^Q$ to $\CExt(\mbf{G},\mbf{K}_2)$ is realized by first defining a functor from $\Bis_{\mbf{G}}^Q$ to the category of $\set{(Q, \mca{E}, \phi)}$, in which the target object of $(D, \eta)$ is denoted as $(Q, \mca{E}_D, \phi_{D,\eta})$. 

The extension $\mca{E}_D$ is described as $F^\times \times_D Y$ with group law given by
$$(a, y_1) \cdot (b, y_2)= \big( ab\cdot (-1)^{D(y_1, y_2)}, y_1+y_2\big).$$

The map $\xymatrix{\phi_{D,\eta}: \mca{E}^{sc}_Q \ar[r] & \mca{E}}$ is the one uniquely determined by
$$\gamma_\alpha \mapsto (\eta(\alpha^\vee), \alpha^\vee) \text{ for all } \alpha^\vee \in \Delta^\vee.$$

\begin{lm}
The map $\phi_{D,\eta}$ is a homomorphism.
\end{lm}
\begin{proof}
By the definition of $\mca{E}^{sc}_Q$ using generators and relations, it suffices to check that for $\alpha, \beta\in \Delta$ the equality
$$\phi_{D,\eta}([\gamma_\alpha, \gamma_\beta])=[\phi_{D,\eta}(\gamma_\alpha), \phi_{D,\eta}(\gamma_\beta)].$$
Now the left hand side is $\phi_{D,\eta}\big((-1)^{B_Q(\alpha^\vee, \beta^\vee)}\big)=(-1)^{B_Q(\alpha^\vee, \beta^\vee)}$. On the other hand, the right hand side is $[\alpha^\vee, \beta^\vee]=(-1)^{B_Q(\alpha^\vee, \beta^\vee)}$.

This completes the proof.
\end{proof}

Hence, the incarnation functor is well-defined. Now we show it gives an equivalence of categories. Any morphism $H\in \Hom\big((D_1, \eta_1), (D_2, \eta_2) \big)$ gives a map  $\Inc_\mbf{G}(H)$ from $\mca{E}_{D_1}$ to $\mca{E}_{D_2}$ given by
$$\xymatrix{
\Inc_\mbf{G}(H): \quad (b, y)_{\mca{E}_{D_1}} \ar@{|->}[r] & (b\cdot H(y), y)_{\mca{E}_{D_2}}.
}$$

It is easy to see that $\Inc_{\mbf{G}}$ is essentially sujective. To show the equivalence between the category of $\Bis_{\mbf{G}}^Q$ and $\set{(Q, \mca{E}, \phi)}$, it suffices to prove the following.

\begin{prop}
The functor $\xymatrix{ H \ar@{|~>}[r] & \Inc_\mbf{G}(H)} $ gives an isomorphism
$$\Hom\big((D_1, \eta_1), (D_2, \eta_2) \big) \simeq \Hom\big( (\mca{E}_{D_1}, \phi_{D_1, \eta_1}), (\mca{E}_{D_2}, \phi_{D_2,\eta_2}) \big).$$
\end{prop}
\begin{proof}
Property (i) of $H$ implies that $\Inc_\mbf{G}(H)$ is a homomorphism. Now, we check
$$\phi_{D_2,\eta_2} =\Inc_\mbf{G}(H) \circ \phi_{D_1,\eta_1}.$$
For any $\gamma_\alpha \in \mca{E}_Q$,
\begin{align*}
\Inc_\mbf{G}(H) \circ \phi_{D_1,\eta_1}(\gamma_\alpha)= &\Inc_\mbf{G}(H) \big( (\eta_1(\alpha^\vee), \alpha^\vee)_{\mca{E}_{D_1}} \big) \\
= & ( \eta_1(\alpha^\vee) \cdot H(\alpha^\vee), \alpha^\vee)_{\mca{E}_{D_2}} \\
=& (\eta_2(\alpha^\vee), \alpha^\vee)_{\mca{E}_{D_2}} \\
=& \phi_{D_2,\eta_2}(\gamma_\alpha).
\end{align*}

It is easy to see from definition that if $\Inc_{\mbf{G}}(H_1)=\Inc_{\mbf{G}}(H_2)$, then $H_1=H_2$. Moreover, one can check easily that any morphism in $\Hom\big( (\mca{E}_{D_1}, \phi_{D_1, \eta_1}), (\mca{E}_{D_2}, \phi_{D_2,\eta_2}) \big)$ arises in this way, i.e., is equal to $\Inc_\mbf{G}(H)$ for some $H$.
\end{proof}

On the category $\Bis_\mbf{G}=\bigsqcup_Q \Bis_{\mbf{G}}^Q$ there is  a commutative Picard structure. For $(D_i, \eta_i) \in \Bis_\mbf{G}^{Q_i}, i=1, 2$, define the sum of the two to be
$$(D_1 + D_2, \eta_1 \cdot \eta_2).$$ 

With respect to this structure on $\Bis_\mbf{G}$, the functor $\xymatrix{ (D, \eta) \ar@{|~>}[r] & (Q, \mca{E}_D, \phi_{D,\eta}) }$ gives an equivalence between the two Picard categories. We consider the following composition (still denoted by $\Inc_\mbf{G}$)
\begin{equation} \label{inc functor}
\xymatrix{
\Inc_\mbf{G} : \quad \Bis_\mbf{G} \ar@{~>}[r] & \set{(Q, \mca{E}, \phi)} \ar@{~>}[r] & \CExt(\mbf{G}, \mbf{K}_2),
}
\end{equation}
where the second functor is a quasi-inverse of the Brylinski-Deligne classification functor, well-defined up to natural equivalence. 

To summarize,
\begin{prop}
The incarnation functor $\Inc_\mbf{G}$ establishes an equivalence of commutative Picard categories between $\Bis_\mbf{G}$ and $\CExt(\mbf{G}, \mbf{K}_2)$.
\end{prop}

\subsection{Description of $\wt{\mbf{G}}_{D,\eta}$}

It follows from the above equivalence of categories that no information is lost when working with $(D, \eta)$. From now, we fix a quasi-inverse functor $\xymatrix{\set{(Q, \mca{E}, \phi)} \ar@{~>}[r] &\CExt(\mbf{G}, \mbf{K}_2)}$ in (\ref{inc functor}). We will work with $\wt{\mbf{G}}$ which is incarnated by $(D, \eta)$, and write it as $\wt{\mbf{G}}_{D, \eta}$.

It is desirable and in fact crucial to have a more precise description of $\wm{G}_{D,\eta}$, including some structure facts useful for later considerations. Note that the following diagram of functors commutes (upto to equivalence of categories):
\begin{equation} \label{inc comp}
\xymatrix{
(D, \eta) \ar@{|~>}[rr] \ar@{|~>}[rd] & & \wm{G}_{D,\eta}:=\Inc_\mbf{G}\big((D,\eta)\big) \ar@{|~>}[ld]^-{\text{ BD classif.}} \\
& (Q, \mca{E}_D, \phi_{D, \eta}).
}
\end{equation}

That is, given $(D, \eta)$, we may assume that $\wm{G}_{D,\eta}$ gives rise to $(Q,\mca{E}_D, \phi_{D,\eta})$ from the BD classification.

First of all, we consider the case $\mbf{G}=\mbf{T}$ and a $\mbf{K}_2$-torsor $\wm{T}_D$ which is incarnated by $D$ (there is no $\eta$ for torus). Equivalently, there is a section of $\wm{T}_D$ over $\mbf{T}$ and therefore a bisector $D$ as above, with respect to which we can write $\wm{T}_D=\mbf{K}_2 \times_D \mbf{T}$ with the group law given by (\ref{torus law}). In particular, for any field extension $F'$ of $F$, one can write $\wm{T}_D(F')=\mbf{K}_2(F') \times_D \mbf{T}(F')$ with the group law given by:
\begin{align}
& i) \quad \big[ y_1\otimes a_1, y_2\otimes a_2 \big] =\set{a_1, a_2}^{B_Q(y_1, y_2)}, \label{tlaw1}\\
& ii)  \quad (1, y_1 \otimes a) \cdot (1, y_2 \otimes a) =\big(\set{a, a}^{D(y_1, y_2)}, (y_1+y_2)\otimes a \big),  \label{tlaw2}\\
& iii) \quad (1, y \otimes a_1) \cdot (1, y \otimes a_2) =\big(\set{a_1, a_2}^{Q(y)}, y\otimes (a_1 a_2) \big). \label{tlaw3}
\end{align}

Now back to the case of a general reductive group $\mbf{G}$. We will describe properties of $\wm{G}_{D,\eta}$.

The covering torus of $\wm{G}_{D,\eta}$ is incarnated by $D$ as above, for which we have written as $\wm{T}_D$. Note that implicitly the incarnation depends on a certain section over $\mbf{T}$, although notationally we are only using the resulting bilinear from $D$. We have assumed that the extension $\wm{G}_{D,\eta}$ gives rise to the Brylinski-Deligne data $Q$ and $\mca{E}_D$.

Recall the pull-back $\wm{T}^{sc}$ of $\wm{T}_D$ to $\mbf{T}^{sc}$ and the natural pull-back map $\Phi$ from $\wm{T}^{sc}$ to $\wm{T}_D$. Here $\wm{T}^{sc}$ inherits a description in terms of $D$; however, since this fact is never used, we will refrain from considering it. For the same reason, we use the notation $\wm{T}^{sc}$ without the subscript $D$.

We have the sheaves of extensions 
$$\xymatrix{
\mbf{K}_2 \ar@{^(->}[r] \ar@{=}[d] & \wm{T}^{sc} \ar@{>>}[r] \ar[d]^-{\Phi} &\mbf{T}^{sc} \ar[d]^-{\Phi} \\
\mbf{K}_2 \ar@{^(->}[r] &\wt{\mbf{T}}_D \ar@{>>}[r] &\mbf{T}.
}$$

Consider the Brylinski-Deligne lifting $\wm{h}^{[b]}_\alpha(a) \in \wm{T}^{sc}(F')$ of $\mbf{h}_\alpha(a) \in \mbf{T}^{sc}(F')$ as in section \ref{BD section}, which depends on $b\in (F')^\times$. We are interested in the element $\Phi\big( \wm{h}^{[b]}_\alpha(a)\big)$ for $\alpha\in \Delta$ expressed in terms of $\mbf{K}_2(F') \times_D \mbf{T}_D(F')$.

We have the following explicit description for $\Phi$ at the level of $F'$-rational points.
\begin{prop}
Keep notations as above. Then for all $\alpha \in \Delta$, we have
\begin{equation} \label{Phi}
\Phi\big(\wm{h}_\alpha^{[b]}(a)\big)= \big( \set{b^{Q(\alpha^\vee)}\cdot \eta(\alpha^\vee), a}, \mbf{h}_\alpha(a) \big),
\end{equation}
where $a, b \in (F')^\times$ are nonzero elements of $F'$. Here the right hand side is written in terms of $\mbf{K}_2(F') \times_D \mbf{T}(F')$.
\end{prop}
\begin{proof}
Fix $\alpha^\vee \in \Delta^\vee$. Recall the property $\wm{h}^{[b]}_\alpha(a)=\wm{h}^{[1]}_\alpha(a) \cdot \set{b, a}^{Q(\alpha^\vee)} $ from (\ref{h ppty2}). It follows that it suffices to show (\ref{Phi}) for $b=1$ i.e. $\Phi (\wm{h}_\alpha^{[1]}(a))= \big( \set{\eta(\alpha^\vee), a}, \mbf{h}_\alpha(a) \big)$.

Note that for fixed $\alpha\in \Delta$, $\Phi$ takes a general form
$$\Phi(\wm{h}_\alpha^{[1]}(a))= \big( \aleph_\alpha(a), \mbf{h}_\alpha(a) \big),$$
where $\aleph_\alpha \in \textsf{Hom}_\text{Zar}(\mbf{G}_\text{mul}, \mbf{K}_2)$ is a (sheaf) homomorphism of abelian sheaves for the big Zariski site, by property (\ref{h ppty1}). In particular, 
$$\xymatrix{
a \in (F')^\times \ar@{|->}[r] & \aleph_\alpha(a) \in \mbf{K}_2(F')
}$$
is a homomorphism from $\mbf{G}_\text{mul}(F')$ to $\mbf{K}_2(F')$ for any field extension $F'/F$. Note that since $\mbf{G}_\text{mul}$ is represented by $F[t, t^{-1}]$, we have $\textsf{Mor}_\text{Zar}(\mbf{G}_\text{mul}, \mbf{K}_2)\simeq \mbf{K}_2\big(F[t, t^{-1}]\big)$ (the algebraic-geometric morphisms which do not necessarily respect the group structure) by Yoneda's lemma.

However, by \cite[\S 3.7-3.8]{BD01}, or in more details \cite[Thm 1.1]{Blo78}, we have the following commutative diagram:
$$\xymatrix{
\mbf{K}_1(F) \ar[r]^-\simeq  \ar@{^(->}[d]_-{-\otimes  t } & \textsf{Hom}_\text{Zar}(\mbf{G}_\text{mul}, \mbf{K}_2) \ar@{^(->}[d] \\
\mbf{K}_2\big(F[t, t^{-1}]\big)  \ar[r]^-\simeq & \textsf{Mor}_\text{Zar}(\mbf{G}_\text{mul}, \mbf{K}_2) .
}$$

That is, the left vertical cup product with $t \in \mbf{K}_1\big(F[t, t^{-1}]\big)$ induces an isomorphism of the top row. More precisely, every $\aleph_\alpha \in \textsf{Hom}_\text{Zar}(\mbf{G}_\text{mul}, \mbf{K}_2)$ arises from a certain $\lambda_\alpha \in F^\times=\mbf{K}_1(F)$, which at the level $F'$-points is given by
$$\aleph_\alpha(a)=\set{\lambda_\alpha, a}, a\in F'.$$

However, since we have assumed that $\wm{G}_{D,\eta}$ gives rise to $\phi_{D,\eta}$ from the Brylinski-Deligne classification, we have that $\Phi$ realizes $\phi_{D,\eta}$ by passing to $\mca{E}^{sc}$ and $\mca{E}$. More precisely, once we identify $\gamma_\alpha$ with $\elt{\wm{h}_\alpha^{[1]}(\tau)}$, we must have
\begin{equation} \label{Phi to phi}
\phi_{D,\eta}\big(\elt{\wm{h}_\alpha^{[1]}(\tau)}\big)=\elt{\Phi\big(\wm{h}^{[1]}_\alpha(\tau)\big)} \in \mca{E} .
\end{equation}

Tracing through the definition of tame symbols, it gives $\phi_{D,\eta}(\gamma_\alpha)=(\lambda_\alpha, \alpha^\vee) \in \mca{E}$ for $\alpha \in \Delta$. But by definition $\phi_{D,\eta}(\gamma_\alpha)=(\eta(\alpha^\vee), \alpha^\vee)$. Therefore, $\lambda_\alpha=\eta(\alpha^\vee)$ for all simple root $\alpha\in \Delta$, and this completes the proof.
\end{proof}

To summarize, for $\wm{G}_{D,\eta}$ incarnated by $(D, \eta)$, its covering torus is incarnated by $D$ with the group law given by (\ref{tlaw1})-(\ref{tlaw3}). Moreover, if we write $\Phi_{D,\eta}$ for $\Phi$ instead to emphasize the dependence of the target expressed using $(D, \eta)$ , then we have
\begin{equation} \label{Phi-eta}
\Phi_{D,\eta}\big(\wm{h}_\alpha^{[b]}(a)\big)= \big( \set{b^{Q(\alpha^\vee)}\cdot \eta(\alpha^\vee), a}, \mbf{h}_\alpha(a) \big) \text{ for all } \alpha^\vee \in \Delta^\vee.
\end{equation}

\section{Finite degree topological covers: local and global}
The goal of this section is to introduce the local and global covering groups which arise from BD framework and to give a brief discussion on some of their properties.

\subsection{Local topological central extensions of finite degree}
Assume first that $F$ is a local field with residual characteristic $p$. Let $n\in \N_{\ge 1}$ be a natural number and assume $\mu_n \subseteq F$. The Hilbert symbol \index{Hilbert symbol}
$$\xymatrix{ (-,-)_n: \quad F^\times \times F^\times \ar[r] &\mu_n}$$
descends to a bilinear from on $F^\times/(F^\times)^n$ and factors through $\mbf{K}_2(F)$. In the tame case, i.e. when $\text{gcd}(n, p)=1$, we have the formula
$$(x, y)_n=\kappa(x,y)^\frac{q-1}{n}, \quad \kappa(x,y)=\wt{(-1)}^{\text{val}(x)\text{val}(y)} \wt{\bigg( \frac{x^{\text{val}(y)}}{y^{\text{val}(x)})} \bigg)} \in \mathbf{f},$$
where $\wt{(-)}$ denotes the reduction of $\msc{O}_F$ into the residue field $\mathbf{f}$ of $F$.

Let $\wt{\mbf{G}} \in \CExt(\mbf{G}, \mbf{K}_2)$ be a central extension over $F$. It gives a central extension of $F$-groups. From the push-out by the $n$-th Hilbert symbol we get a central extension of $G=\mbf{G}(F)$:
$$\xymatrix{
\mbf{K}_2(F) \ar@{^(->}[r] \ar[d]^-{(-,-)_n} &\wt{\mbf{G}}(F)  \ar@{>>}[r] \ar[d]  &\mbf{G}(F) \ar@{=}[d] \\
\mu_n \ar@{^(->}[r] &\wt{G}  \ar@{>>}[r]  & G.
}$$

Consider the category $\CExt(\mbf{G}(F),\mu_n)$ of central extensions of $\mbf{G}(F)$ by $\mu_n$. Then we have defined a functor
$$\xymatrix{
\Hs:  \quad \CExt(\mbf{G},\mbf{K}_2) \ar@{~>}[r] & \CExt(\mbf{G}(F),\mu_n).
}$$

For any morphism $\xymatrix{\mbf{G}' \ar[r] & \mbf{G}}$, let $\wm{G}'$ be the pull-back of $\wm{G}$ to $\mbf{G}'$ and write $\xymatrix{r: \wt{\mbf{G}} \ar@{|->}[r] &\wm{G}' }$ for the map. Then the following diagram commutes
$$\xymatrix{
\CExt(\mbf{G},\mbf{K}_2) \ar@{~>}[r]^-\Hs \ar@{~>}[d]^-r    &\CExt(\mbf{G}(F),\mu_n) \ar@{~>}[d]^-{\Hs(r)}\\ 
 \CExt(\mbf{G}',\mbf{K}_2) \ar@{~>}[r]^-\Hs    &\CExt(\mbf{G}'(F),\mu_n),
}$$
where $\Hs(r)$ is the induced map on the central extension of $F$-points.

\subsubsection{When $r$ is induced from the inclusion $\mbf{T} \to \mbf{G}$}

Let $\wm{G}=\wt{\mbf{G}}_{D,\eta} \in \CExt(\mbf{G},\mbf{K}_2)$ be incarnated by $(D,\eta) \in \Bis_{\mbf{G}}^Q$, we call $\wt{G}=\Hs(\wm{G})$ incarnated by $(D,\eta)$ also. Let $\wt{\mbf{T}}$ be the restriction (i.e. pull-back) of $\wt{\mbf{G}}$ to $\mbf{T}$.

Now consider the torus $T:=\mbf{T}(F)\subseteq G$, the group $\wt{G}$ arising from $\wt{\mbf{G}}$ gives a central extension $\wt{T}$ of $T$ by pull-back:
$$\seq{\mu_n}{\wt{T}}{T}.$$

Equivalently, this $\wt{T}$ is also the $F$-points of $\wm{T}=r(\wt{\mbf{G}}) \in \CExt(\mbf{T}, \mbf{K}_2)$. Note $T=Y\bigotimes_\Z F^\times$. Since $\wm{T}$ is incarnated by $D$, it follows that $\wt{T}$ can be described as $\mu_n \times_D T$ with the group law inherited from (\ref{tlaw1})-(\ref{tlaw2}):
\begin{align}
& (i)\quad   \big[(\zeta_1, y_1\otimes a), (\zeta_2, y_2 \otimes b)\big]=(a, b)_n^{B_Q(y_1, y_2)}, \label{Ftlaw1}\\
& (ii)\quad   (1, y_1\otimes a)\cdot (1, y_2 \otimes a)=\big( (a, a)_n^{D(y_1, y_2)}, (y_1 + y_2)\otimes a \big), \label{Ftlaw2}\\
& (iii)\quad   (1, y\otimes a)\cdot (1, y\otimes b)=\big( (a, b)_n^{Q(y)}, y\otimes (ab) \big). \label{Ftlaw3}
\end{align}

\subsubsection{When $r$ is induced from $\varphi_\alpha: \mbf{SL}_2 \to \mbf{G}^{sc}$ }

Let $\alpha \in \Psi$ be any root of $\mbf{G}$ and $\xymatrix{\varphi_\alpha: \mbf{SL}_2 \ar[r] & \mbf{G}^{sc}}$ the map in section \ref{BD section}. We may also denote by $\varphi_\alpha$ the composition $\xymatrix{\mbf{SL}_2 \ar[r] & \mbf{G}^{sc} \ar[r] & \mbf{G}}$. For any $\wt{\mbf{G}} \in \CExt(\mbf{G} ,\mbf{K}_2)$, we get $\wt{\mbf{SL}}_2^\alpha=\varphi_\alpha^*(\wt{\mbf{G}})$. By Brylinski-Deligne classification as in Theorem \ref{torsor/sc}, the extension $\wt{\mbf{SL}}_2^\alpha$ is determined up to a unique isomorphism by the quadratic form $Q_{\alpha^\vee}:=Q|_{\Z\alpha^\vee}$, or equivalently by $Q(\alpha^\vee)\in \Z$. The map $\varphi_\alpha$ induces a central extension $\wt{SL}_2^\alpha:=\wm{SL}_2^\alpha(F)$ of $SL_2:=\mbf{SL}_2(F)$ by $\mu_n$ and we have
$$\xymatrix{
\mu_n \ar@{^(->}[r] & \wt{G}  \ar@{>>}[r] &G \\
\mu_n \ar@{=}[u] \ar@{^(->}[r] &\wt{SL}_2^\alpha  \ar[u]_-{\Hs(r)} \ar@{>>}[r]  & SL_2 \ar[u]\ .
}$$

If $\wt{G}$ is incarnated by $(D, \eta)$, then $\wt{SL}_2^\alpha$ is incarnated by $D|_{\Z\alpha^\vee}$, i.e., by $Q_{\alpha^\vee}$. Now the central extension $\wt{SL}_2^\alpha$ is isomorphic by a unique isomorphism to the extension described by Matsumuto (cf. \cite{Mat69}). We can write $ \wt{SL}_2^\alpha \simeq \mu_n\times_{Q_{\alpha^\vee}} SL_2$ with respect to the cocycle  given by
$$\sigma(g_1, g_2)=\Bigg( \frac{\mathrm{x}(g_1g_2)}{\mathrm{x}(g_2)}, \frac{\mathrm{x}(g_1g_2)}{\mathrm{x}(g_1)} \Bigg)_n^{Q(\alpha^\vee)},$$
where
\begin{equation*}
\mathrm{x} \left(
\begin{array}{cccc}
a & b\\
c & d
\end{array}
\right)
=
\begin{cases}
c &\text{if } c\ne 0,\\
d &\text{else}.
\end{cases}
\end{equation*}

This cocycle, when restricted to the one dimensional torus $T$ of $SL_2$, gives a description of the covering torus $\wt{T}^\alpha \subseteq \wt{SL}_2^\alpha$ as $\mu_n \times_{Q_{\alpha^\vee}} F^\times$. Such description  agrees with the group law given by (\ref{Ftlaw1})-(\ref{Ftlaw3}), which states exactly as
$$\big( \zeta_1, \mbf{h}_\alpha(a) \big) \cdot \big(\zeta_2, \mbf{h}_\alpha(c) \big)=\big(\zeta_1\zeta_2 (a, c)_n^{Q(\alpha^\vee)}, \mbf{h}_\alpha(ac) \big).$$

That is, the cocycle $\sigma$ on $SL_2$ given by Matsumuto above is an extension to the whole group $SL_2$ of the cocycle on $T$ specified from incarnation.

\subsection{Local splitting properties}

\subsubsection{Splitting over a maximal compact group}
Continue to assume $F$ a local field, and we fix the e\'pinglage for $(\mbf{G}, \mbf{T}, \mbf{B})$  as before. The Bruhat-Tits building of $G=\mbf{G}(F)$ over $F$ has a hyperspecial point determined by the e\'pinglage, which gives an associated group scheme $\mathbf{G}$ over $\msc{O}_F$ with generic fibre $\mbf{G}$ via the Bruhat-Tits theory, i.e., we have $\mbf{G}=\mathbf{G} \times_{\msc{O}_F} F$. 

Let $K=\mathbf{G}(\msc{O}_F)$, which is a maximal compact subgroup of $G$. We are interested in the case when $K$ splits into $\wt{G}$. 

More precisely, assume $n$ prime to the residue characteristic of $F$, the Hilbert symbol $(-,-)_n$ becomes a power of the tame symbol, and it gives a degree $n$ central cover $\wt{G}$ of $G$. We are interested in the case when there exists a splitting $s_K$ of $K$ into $\wt{G}$:
$$\xymatrix{
\mu_n \ar@{^(->}[r] &\wt{G}  \ar@{>>}[r]  & G \\
  & & K \ar@{^(-->}[lu]^-{s_K} \ar@{^(->}[u] \ .
}$$

\begin{dfn} \label{G-bar unram} \index{unramified!covering group}
The group $\wt{G}$ called $s_K$-unramified (or simply unramified) if $\text{gcd}(n, p)=1$ and there exists a splitting $s_K$ of $K$ into $\wt{G}$.
\end{dfn}

Note that the $\mbf{K}_2$-torsor $\wt{\mbf{G}}$ defined over $F$ may not be the base change of some $\mathbf{K}_2$-torsor $\wt{\mathbf{G}}$ over $\Spec(\msc{O}_F)$. Otherwise suppose $\wt{\mbf{G}}=\wt{\mathbf{G}} \times_{\msc{O}_F} F$, since the tame Hilbert symbol (if we assume $\text{gcd}(n, d)=1$) vanishes on $\mathbf{K}_2(\msc{O}_F)$, the existence of the splitting of $K$ into $\wt{G}$ is then automatic (cf. \cite[\S 10.7]{BD01}). This suggest that the existence of splitting of $K$, relies on other data besides the condition $\text{gcd}(n, d)=1$. We refer to the recent work of Weissman (cf. \cite{We11}, \cite{We14-1}) on integral models for  Brylinski-Deligne covering groups over $F$. 

In the language of incarnations, if one starts in general with $\wt{\mbf{G}}_{D,\eta}$, then the existence of splitting $s_K$ as above holds only conditionally. In fact, one sufficient condition is that the homomorphism $\eta$ has image in the units $\msc{O}_F^\times$ of $F^\times$:
$$\xymatrix{
\eta: & Y^{sc} \ar[r] \ar@{-->}[d] & F^\times \\
        & \msc{O}_F^\times \ar@{^(->}[ru] .
}$$
In particular, the condition $\text{gcd}(n, q)=1$ is not sufficient to guarantee the splitting of $K$ in general. For more detailed discussions, see \cite[\S 4]{GaG14}.

\subsubsection{Unipotent splitting of $\wt{G}$}

In section \ref{BD section}, we have seen that the unipotent subgroup $\mbf{U}$ of the Borel subgroup $\mbf{B}^{sc}$ of $\mbf{G}^{sc}$ splits in $\wm{G}^{sc}$, and the splitting is $\mbf{B}^{sc}$-equivariant. In fact, the same holds with $\mbf{G}^{sc}$ replaced by $\mbf{G}$ (cf. \cite[Prop. 11.3]{BD01}): there exists a unique $\mbf{B}$-equivariant splitting of $\mbf{U}$ into $\wm{G}$. Taking $F$-rational points, we obtain a $B$-equivariant splitting of  $\mbf{U}(F)$ into $\wt{G}$. It is important that such splitting is unique, see also \cite[App. A]{MW95}.

In fact, the above splitting on the unipotent radical of the Borel subgroup could be extended to a section of the set $G_u$ of unipotent elements of $G$ with certain properties. As before, the group $\wt{G}$ acts on $G_u$ by conjugation, which descends to $G$. We have the following useful fact, whose proof we refer to \cite[Prop. 2.2.1]{Li12}. 

\begin{prop} \label{unip splitting}
There exists uniquely a continuous set section $\xymatrix{i_u: G_u \ar[r] & \wt{G}}$ such that
\begin{enumerate} 
\item[(i)] for all unipotent subgroup $U$ of $G$, the restriction $i_u|_U$ is a homomorphism (i.e. a splitting), and
\item[(ii)] $i_u$ is $G$-conjugation invariant. That is, $i_u(g^{-1}u g) =g^{-1} i_u(u) g$ for all $g\in G$.
\end{enumerate}
\end{prop}

Note that in section \ref{BD section}, we have used the notation $\wm{e}_\alpha(-)$ for the splitting of $\mbf{U}(F)$ into $\wm{G}^{sc}(F)$:
$$\xymatrix{
\mbf{e}_\alpha(a) \ar@{|->}[r] &\wm{e}_\alpha(a).
}$$
We also recall that we have denoted by $\Phi$ for the natural map $\xymatrix{\mbf{G}^{sc}(F) \ar[r] & \mbf{G}(F)}$, and also $\Phi_{D,\eta}$ for the induced natural pull-back map $\xymatrix{\wm{G}^{sc}(F) \ar[r] & \wm{G}(F)}$. Write
$$e_\alpha(a):=\Phi(\mbf{e}_\alpha(a)), \quad \wt{e}_\alpha(a):=\Phi_{D,\eta}(\wm{e}_\alpha(a)).$$

Since $\Phi$ is injective on the unipotent elements, the map
$$\xymatrix{
e_\alpha(a) \ar@{|->}[r] &\wt{e}_\alpha(a),
}$$
is a splitting of $\mbf{U}(F)$ into $G$.

As a consequence of above proposition,
\begin{cor} \label{sK=iu}
For $\alpha \in \Psi$, the unique splitting $i_u$ of the unipotent subgroup $U_\alpha$ is just given by
$$\xymatrix{
e_\alpha(a) \ar@{|->}[r] &\wt{e}_\alpha(a),
}$$
which is $G$-equivariant. 

Moreover, assume $\wt{G}$ $s_K$-unramified. Let $U:=\mbf{U}(F)$ be the unipotent radical of some parabolic $P\subseteq G$, then
$$s_K|_{K\cap U}=i_u|_{K\cap U}.$$
\end{cor}
\begin{proof}
We only need to show the second assertion. Note that $K\cap U$ is a pro-$p$ group. The two splittings $s_K|_{K\cap U}$ and $i_u|_{K\cap U}$ differ by a homomorphism from $K\cap U$ to $\mu_n$, which has to be trivial since $n$ is prime to $p$.
\end{proof}

\subsection{Global topological central extensions of finite degree} \label{G-bar(A)}

Back to the global situation now and assume $F$ is a number field with $\mu_n\subseteq  F^\times$. Write $F_v$ for the completion of $F$ with respect to any place $v\in |F|$. Let $\A$ be the adele ring of $F$. Let $\wm{G}_{D,\eta}\in \CExt(\mbf{G}, \mbf{K}_2)$ be a $\mbf{K}_2$-torsor over $F$.

In \cite{BD01}, it is shown that one has the following inherited data:
\begin{enumerate}
\item[$\bullet$] For all $v$, a central extension $\seq{\mu_n}{\wt{G}_v}{G_v}$. For almost all $v$, there is a splitting $s_{K_v}$ of the group $K_v:=\mbf{G}(\msc{O}_v)$, which is well-defined, into $\wt{G}_v$. In this case, $K_v$ is equal to $\mathbf{G}_v(\msc{O}_v)$, where $\mathbf{G}_v$ is the integral model of $\mbf{G}_v:=\mbf{G}\times_F F_v$ given by the Bruhat-Tits theory.
 
\item[$\bullet$] An adelic group $\wt{\mbf{G}}(\A):=\prod'_v \wt{G}_v \Big/ (\bigoplus_v \mu_n)^o$, where $\prod'_v \wt{G}_v$ is the restricted product with respect to the $K_v$'s and $(\bigoplus_v \mu_n)^o:=\set{\bigoplus \xi_v: \prod_v \xi_v=1}$. For $v\in |F|$, there is the natural inclusion
$$\xymatrix{
\mu_n \ar@{^(->}[r] & \wt{\mbf{G}}(\A) \ar@{>>}[r] & \mbf{G}(\A) \\
\mu_n \ar@{^(->}[r] \ar@{=}[u] & \wt{G}_v \ar@{^(->}[u] \ar@{>>}[r] & G_v \ar@{^(->}[u]\ .
}$$
\item[$\bullet$] There is a natural splitting of $\mbf{G}(F)$ into $\wm{G}(\A)$,
$$\xymatrix{
\mu_n \ar@{^(->}[r] & \wt{\mbf{G}}(\A) \ar@{>>}[r] & \mbf{G}(\A) \\
       & & \mbf{G}(F) \ar@{^(->}[lu] \ar@{^(->}[u],
}$$
which allows us to define the notion of automorphic forms on $\wm{G}(\A)$.
\item[$\bullet$] For any unipotent subgroup $\mbf{U}$ of a parabolic $\mbf{P}=\mbf{M}\mbf{U}$ of $\mbf{G}$, there is a unique $\mbf{P}(F)$-equivariant splitting of $\mbf{U}(\A)$ into $\wm{G}(\A)$
$$\xymatrix{
\mu_n \ar@{^(->}[r] & \wt{\mbf{G}}(\A) \ar@{>>}[r] & \mbf{G}(\A) \\
       & & \mbf{U}(\A) \ar@{^(->}[lu] \ar@{^(->}[u],
}$$

such that its restriction to $\mbf{U}(F)$ is the fixed splitting of $\mbf{G}(F)$ restricted to $\mbf{U}(F)$.
\end{enumerate}

\section{Dual groups and $L$-groups for topological extensions}
The references are \cite{FiLy10}, \cite{McN12}, \cite{We13} and \cite{We14}. We will only recall here some facts important to us and refer to the original papers for details.

\subsection{The dual group $\wt{\mbf{G}}^\vee$ \`a la Finkelberg-Lysenko-McNamara-Reich}
Let $\mbf{G}$ be a split reductive group over $F$ with root datum $(X, \Psi, Y, \Psi^\vee)$. We have also fixed $\Delta \subseteq \Psi$, a set of simple roots. For any $\wt{\mbf{G}}$ in the Brylinski-Deligne framework, it gives rise to local and global topological degree $n$ covers, depending on whether $F$ is local or global. Several authors have defined a group scheme $\wt{\mbf{G}}^\vee$ over $F$, whose complex points $\wt{G}^\vee=\wt{\mbf{G}}^\vee(\C)$ would be called the complex dual group for the degree $n$ topological covers. In fact $\wt{\mbf{G}}^\vee$ could be defined over a much smaller ring (cf. \cite{We13} \cite{We14}); but for simplicity here, we will refrain from stating the construction in its most general form.

First we describe the root datum of $\wt{\mbf{G}}^\vee$. Let $Q$ be the quadratic forms associated with $\wt{\mbf{G}}$. Define
$$Y_{Q,n}=\set{y\in Y| \ B_Q(y, y')\in n\Z \text{ for all } y'\in Y}.$$
It is a sublattice of $Y$ and clearly contains $nY$. For any $\alpha\in \Psi$, we also write
$$n_\alpha =\frac{n}{\text{gcd}(Q(\alpha^\vee),n)}, \quad \alpha_{[n]}^\vee=n_\alpha \alpha^\vee.$$

Let $Y_{Q,n}^{sc}$ be the lattice generated by $\alpha^\vee_{[n]}, \alpha^\vee \in \Psi^\vee$ with relations inherited from $Y$ (cf. \cite[Lm. 11.5]{BD01}).  Now consider the quadruple given by
\begin{align*}
Y_1 &= Y_{Q,n}, \\
\Psi_1^\vee &=\set{\alpha_{[n]}^\vee: \alpha^\vee \in \Psi^\vee}, \\
X_1 &= \text{Hom}(Y_1,\Z) \subseteq X\otimes \Q,\\
\Psi_1 &=\set{n_\alpha^{-1}\alpha: \alpha\in \Psi}.
\end{align*}

\begin{thm}
The quadruple $(Y_1, \Psi_1^\vee, X_1, \Psi_1)$ forms a root datum.
\end{thm}
\begin{proof}
See \cite[\S 13.11]{McN12}.
\end{proof}

Define the dual group $\wt{\mbf{G}}^\vee$ to be the split pinned reductive group associated with this root datum. Let $\wt{G}^\vee=\wt{\mbf{G}}^\vee(\C)$. In particular, if $\mbf{G}=\mbf{T}$ is a torus, then $\wt{T}^\vee=X_1\otimes \C$.

\subsection{Local $L$-group \`a la Weissman}
Let $F$ be a number field or a local field obtained as the localization of a number field. Let $(Q, \mca{E}, \phi)$ be the BD classification data associated with $\wm{G} \in \CExt(\mbf{G}, \mbf{K}_2)$, then the construction of $\wm{G}^\vee$ uses the data $n$ and $B_Q$ alone. The construction of the $L$-group for covering groups is due to M. Weissman and utilizes the full data $(Q, \mca{E}, \phi)$. In \cite{We13} and moreover \cite{We14}, he has constructed a proalgebraic group scheme ${}^L\wt{\mbf{G}}$ over $\Z[\mu_n]$. For our purpose, it suffices to have ${}^L\wm{G}$ defined over certain subfield of $\C$ which sits in the exact sequence
$$\seq{\wt{\mbf{G}}^\vee}{ {}^L\wt{\mbf{G}} }{\mbf{W}_F}.$$
Here $\mbf{W}_F$ is the constant group scheme over the field of definition whose complex point is the global Weil group $\W_F$ if $F$ is a number field; and it could be the local Weil group $\W_F$ or local Weil-Deligne group $\WD_{F}=\mbf{SL}_2(\C) \times \W_F$ if $F$ is a local field.

The freedom for different candidates in the place of $\mbf{W}_F$ in the local case is due to the fact that the essential part of the construction of ${}^L\wm{G}$ is a 
fundamental extension over $F^\times$. However, for our interest we will concentrate on $\mbf{W}_F$ associated with the local and global Weil group respectively.

The constructions for the dual group $\wt{\mbf{G}}^\vee$ and the $L$-group ${}^L\wt{\mbf{G}}$ are both compatible with Levi subgroups. More precisely, consider any Levi subgroup $\mbf{H}$ of $\mbf{G}$ embedded in the latter by $\varphi$. It induces a morphism $\xymatrix{\Hs(\varphi): \wt{H} \ar[r] &\wt{G}}$. Then the construction of dual group and $L$-group gives the commutative diagram
$$\xymatrix{
\wt{\mbf{G}}^\vee \ar@{^(->}[r]  &{}^L\wt{\mbf{G}}  \ar@{>>}[r] &\mbf{W}_F \\
\wt{\mbf{H}}^\vee \ar@{^(->}[u]^-{\varphi^\vee} \ar@{^(->}[r]  &{}^L\wt{\mbf{H}} \ar@{^(->}[u]^-{{}^L \varphi} \ar@{>>}[r] &\mbf{W}_F \ar@{=}[u]
}$$

In the case of $\mbf{H}=\mbf{T}\subseteq \mbf{G}$, by taking $\C$-point we obtain
$$\xymatrix{
\wt{G}^\vee \ar@{^(->}[r]  &{}^L\wt{G}  \ar@{>>}[r] & \mbf{W}_F(\C) \\
\wt{T}^\vee \ar@{^(->}[u]^-{\varphi^\vee} \ar@{^(->}[r]  &{}^L\wt{T} \ar@{^(->}[u]^-{{}^L \varphi} \ar@{>>}[r] &\mbf{W}_F(\C) \ar@{=}[u] \ .
}$$

In the remaining part of this section, we assume $F$ a local field. We will recall in details the construction of the extension of ${}^L\wt{G}$ in \cite{We14} for general BD type extension $\wt{G}$:
$$\seq{\wt{G}^\vee}{{}^L\wt{G}}{\W_F}.$$

Since we work with $\wt{G}$ which is incarnated by $(D,\eta)$, the description in \cite{We14} can be carried out in this language. The construction of ${}^L\wt{G}$ relies on two abelian extensions $E_i$ for $i=1, 2$ by the center $Z(\wt{G}^\vee)$:
$$\xymatrix{
Z(\wt{G}^\vee) \ar@{^{(}->}[r]  &E_i \ar@{>>}[r] & F^\times.
}$$

Now we describe the two extensions $E_i$. Note $Z(\wt{G}^\vee)=\Hom(Y_{Q,n}/Y_{Q,n}^{sc}, \C^\times)$. 

\subsubsection{The extension $E_1$}
The extension $E_1$ is given by the cocycle
\begin{equation}
\begin{diagram}
F^\times \times F^\times &\rTo &\Hom(Y_{Q,n}/Y_{Q,n}^{sc}, \C^\times) \\
(a,b) &\rMapsto &\big( y \mapsto (a, b)_n^{Q(y)} \big).
\end{diagram}
\end{equation}

Thus, we write $E_1=Z(\wt{G}) \times_Q F^\times$ with the group law given above. Note that since $(a, b)_n^{Q(y)}\in \mu_2$ since $2n|Q(y)$ for $y\in Y_{Q,n}$, the cocycle actually takes values in $\Hom(Y_{Q,n}/Y_{Q,n}^{sc}, \mu_2)$.

\subsubsection{The extension $E_2$}

For $E_2$, consider the extension $\mca{E}$ and $\mca{E}^{sc}$ associated to the covering tori $\wm{T}$ and $\wm{T}^{sc}$ of $\wt{\mbf{G}}$ and $\wt{\mbf{G}}^{sc}$ respectively. 

For convenience, we write $F^\times/n=F^\times/(F^\times)^n$. Consider the pull-back of $\mca{E}^{sc}$ by $Y_{Q,n}^{sc}$ and the followed-up push-out by the quotient map $\xymatrix{F^\times \ar[r] & F^\times/n}$, we obtain a group $\mca{E}_{Q,n}^{sc}$ which sits in an exact sequence
$$\xymatrix{
F^\times/n \ar@{^(->}[r] &\mca{E}_{Q,n}^{sc} \ar@{>>}[r]  &Y_{Q,n}^{sc}.
}$$

Fix an arbitrary nonzero $f\in F(\!(\tau)\!)$.  We define a section $\s$ of $\mca{E}_{Q,n}^{sc}$ over $Y_{Q,n}^{sc}$ as follows. For every element $\alpha^\vee_{[n]} \in Y^{sc}_{Q,n}$ with $\alpha\in \Psi$, it is given by
$$\xymatrix{
\s: \quad \alpha^\vee_{[n]} \ar@{|->}[r] & \wm{h}_\alpha^{[f]}(n_\alpha \alpha^\vee), \quad  \alpha\in \Psi.
}$$
Recall from Definition \ref{dfn BD section}  $\wm{h}_\alpha^{[f]}(n_\alpha \alpha^\vee)=\elt{\wm{h}_\alpha^{[f]}(\tau^{n_\alpha})} \in \mca{E}^{sc}$ and it clearly lies over $\alpha^\vee_{[n]}$, and by abuse of notation we also use it to denote its image in $\mca{E}_{Q,n}^{sc}$. Because of the property (cf. (\ref{h ppty2})) $$\wm{h}_\alpha^{[f]}(\tau^{n_\alpha})=\wm{h}_\alpha^{[1]}(\tau^{n_\alpha}) \cdot \set{f, \tau^{n_\alpha}}^{Q(\alpha^\vee)},$$

\noindent it follows that $\elt{\wm{h}_\alpha^{[f]}(\tau^{n_\alpha})} \in \mca{E}_{Q,n}^{sc}$ is independent of $f$, which justifies the absence of $f$ in the notation $\s$ we use. Therefore we may omit the superscript in both $\wm{h}_\alpha^{[f]}(n_\alpha \alpha^\vee)$ and $\elt{\wm{h}_\alpha^{[f]}(\tau^{n_\alpha})}$, and instead just write  $\wm{h}_\alpha(n_\alpha \alpha^\vee)$,  $\elt{\wm{h}_\alpha(\tau^{n_\alpha})}$ respectively. For computational convenience, one may take $f=1$ and there is no loss of generality because of the independence of $f$ in general.

What is more important is the following
\begin{prop} \label{s on sc}
The section $\xymatrix{\s: \alpha^\vee_{[n]} \ar@{|->}[r] & \wm{h}_\alpha(n_\alpha \alpha^\vee)}$ for all $\alpha\in \Psi$ above gives a well-defined splitting (i.e. a homomorphism) of the sequence:
$$\xymatrix{
F^\times/n \ar@{^(->}[r] &\mca{E}_{Q,n}^{sc} \ar@{>>}[r]  &Y_{Q,n}^{sc} \ar@/_1.2pc/[l]_-{\s}.
}$$
\end{prop}
\begin{proof}
For all $\alpha\in \Psi$, write $\alpha_{[n]}:=\alpha/n_\alpha$. The lattice $Y_{Q,n}^{sc}$ is generated by $\alpha^\vee_{[n]}$ for all $\alpha^\vee \in \Psi^\vee$ with the following relation (cf. \cite[\S 11.5]{BD01}):
$$s_{\alpha_{[n]}}(\beta_{[n]})^\vee = \beta_{[n]}^\vee -\angb{\alpha_{[n]}}{\beta_{[n]}^\vee} \alpha_{[n]}^\vee,$$
which is equivalent to $n_\beta s_\alpha(\beta)^\vee =n_\beta \beta^\vee -n_\beta \angb{\alpha}{\beta^\vee}\alpha^\vee$. Note that Lemma \ref{BQ<>} is valid for general $\wm{G}$ of BD type, and it gives
\begin{align*}
n_\beta \cdot \angb{\alpha}{\beta^\vee} Q(\beta^\vee) &= n \cdot \angb{\alpha}{\beta^\vee} \\
&= n_\alpha \cdot Q(\alpha^\vee) \angb{\alpha}{\beta^\vee} \\
&= n_\alpha \cdot B(\alpha^\vee, \beta^\vee) \quad  \text{ by Lemma } \ref{BQ<>} \\
&=n_\alpha \cdot \angb{\beta}{\alpha^\vee} Q(\beta^\vee).
\end{align*} 
Thus it follows
\begin{equation} \label{n-<> dual}
n_\beta \cdot \angb{\alpha}{\beta^\vee} = n_\alpha \cdot \angb{\beta}{\alpha^\vee}.
\end{equation}

We see that it suffices to show
$$\s\big( n_\beta s_\alpha(\beta)^\vee \big)=\s\big(n_\beta \beta^\vee\big) \cdot \s\big(n_\alpha \alpha^\vee \big)^{-\angb{\beta}{\alpha^\vee} }.$$

Write $\gamma=s_\alpha(\beta)$, then $Q(\gamma^\vee)=Q(\beta^\vee)$ and $n_\gamma=n_\beta$. We have $\s\big( n_\beta s_\alpha(\beta)^\vee \big)=\s\big( n_\gamma \gamma^\vee \big)=\elt{\wm{h}_\gamma(\tau^{n_\gamma})}$ which by \cite[\S 11.3]{BD01} is equal to
\begin{align*}
&\elt{\wm{h}_\beta(\tau^{n_\beta})} \cdot \elt{\wm{h}_\alpha(\tau^{-n_\beta \angb{\alpha}{\beta^\vee}})} \\
=&\elt{\wm{h}_\beta(\tau^{n_\beta})} \cdot \elt{\wm{h}_\alpha(\tau^{-n_\alpha \angb{\beta}{\alpha^\vee}})} \\
=&\elt{\wm{h}_\beta(\tau^{n_\beta})} \cdot \elt{\wm{h}_\alpha(\tau^{n_\alpha})}^{- \angb{\beta}{\alpha^\vee}} \cdot (-1)^{n_\alpha^2 \varepsilon(-\angb{\beta}{\alpha^\vee}) Q(\alpha^\vee)} \quad \text{ by } \cite[(11.6.1)]{BD01}\ ,
\end{align*}
where $\varepsilon(N)=N(N-1)/2$. Note $n| n_\alpha Q(\alpha^\vee)$, and hence $(-1)^{n_\alpha^2 \varepsilon(-\angb{\beta}{\alpha^\vee}) Q(\alpha^\vee)} \in (F^\times)^n$. It follows $\elt{\wm{h}_\gamma(\tau^{n_\gamma})}=\elt{\wm{h}_\beta(\tau^{n_\beta})} \cdot \elt{\wm{h}_\alpha(\tau^{n_\alpha})}^{- \angb{\beta}{\alpha^\vee}}$ in $\mca{E}_{Q,n}^{sc}$, which is an extension with kernel $F^\times/n$. This exactly shows that $\s$ is a homomorphism and the proof is completed.
\end{proof}

Back to the main discussion, we could obtain on the other hand a group $\mca{E}_{Q,n}$ as composition of the pull-back of $\mca{E}$ by $\xymatrix{Y_{Q,n} \ar@{^(->}[r] & Y}$ and the push-out by the quotient map $\xymatrix{F^\times \ar[r] & F^\times/n}$. It is an extension of $Y_{Q,n}$ by $F^\times/n$:
$$\seq{F^\times/n}{\mca{E}_{Q,n}}{Y_{Q,n}}.$$
There is the inherited map $\xymatrix{\mca{E}_{Q,n}^{sc} \ar[r] & \mca{E}_{Q,n}}$ still denoted by $\phi_{D,\eta}$.

To summarize, we have the following commutative diagram with a canonical splitting $\s$ for the top row.
\begin{equation}
\xymatrix{
F^\times/n \ar@{^(->}[r] \ar@{=}[d] &\mca{E}_{Q,n}^{sc} \ar@{>>}[r] \ar[d]^-{\phi_{D,\eta}} &Y_{Q,n}^{sc} \ar@{^(->}[d] \ar@/_1.5pc/[l]_-{\s}\\
F^\times/n \ar@{^(->}[r] &\mca{E}_{Q,n} \ar@{>>}[r]  &Y_{Q,n}.
}
\end{equation}

As a simple consequence of above discussion
\begin{cor} \label{phi}
The splitting $\xymatrix{\s: Y_{Q,n}^{sc} \ar[r] & \mca{E}_{Q,n}^{sc} }$ gives rise to a splitting of $Y_{Q,n}^{sc}$ into $\mca{E}_{Q,n}$ which takes the explicit form for $\alpha\in \Delta$ as
$$\phi_{D,\eta} \circ \s(\alpha^\vee_{[n]})= \big(\eta(\alpha^\vee_{[n]}), \alpha^\vee_{[n]} \big), \quad \alpha^\vee \in \Delta^\vee,$$
where the right hand side is written in the form with respect to $\mca{E}_{Q,n}=F^\times \times_D Y_{Q,n}$.
\end{cor}
\begin{proof}
For $\alpha^\vee \in \Delta^\vee$, direct computation gives
\begin{align*}
\phi_{D,\eta} \circ \s(\alpha^\vee_{[n]}) =& \phi_{D,\eta}\big(\elt{\wm{h}_\alpha(\tau^{n_\alpha})}\big) \\
=& \phi_{D,\eta} \big(\elt{\wm{h}_\alpha(\tau)}^{n_\alpha} \cdot (-1)^{\varepsilon(n_\alpha)Q(\alpha^\vee)} \big)  \text{ by } \cite[(11.6.1)]{BD01} \\
=&\big((\eta(\alpha^\vee), \alpha^\vee)\big)^{n_\alpha} \cdot (-1)^{\varepsilon(n_\alpha)Q(\alpha^\vee)}\\
=& \big(\eta(n_\alpha \alpha^\vee), n_\alpha\alpha^\vee\big),
\end{align*}
which is the desired result.
\end{proof}

The splitting of the row over $Y_{Q,n}^{sc}$ gives
\begin{equation} \label{germ: E_2}
\xymatrix{
F^\times/n \ar@{^(->}[r] \ar@{=}[d] &\mca{E}_{Q,n} \ar@{>>}[r] \ar@{>>}[d] &Y_{Q,n} \ar@{>>}[d] \\
F^\times/n \ar@{^(->}[r] &\mca{E}_{Q,n}\big/ \phi_{D,\eta}\circ \s(Y_{Q,n}^{sc}) \ar@{>>}[r] &Y_{Q,n}/Y_{Q,n}^{sc}.
}
\end{equation}

By taking $\Hom(-, \C^\times)$, we further consider the pull-back by the map $h: F^\times \to \Hom(F^\times/n, \C^\times)$ given by the Hilbert symbol $a\mapsto h_a$ with $h_a(b)=(b, a)_n$. Then by definition $E_2$ is the pull-back of $\Hom(\mca{E}_{Q,n}\big/ \phi_{D,\eta}\circ \s(Y_{Q,n}^{sc}), \C^\times)$ by $h$:

$$\xymatrix{
Z(\wt{G}^\vee)  \ar@{^(->}[r] &\Hom(\mca{E}_{Q,n}\big/ \phi_{D,\eta}\circ \s(Y_{Q,n}^{sc}), \C^\times) \ar@{>>}[r] & \Hom(F^\times/n, \C^\times) \\
Z(\wt{G}^\vee) \ar@{=}[u]  \ar@{^(->}[r] & E_2 \ar[u] \ar@{>>}[r] & F^\times \ar[u]^-{h}.
}$$
\subsubsection{The fundamental extension $E_{\wt{G}}$ and the group ${}^L\wt{G}$}

\begin{dfn}
Consider the Baer sum $E_1\oplus_B E_2$, which we denote by $E_{\wt{G}}$. We call $E_{\wt{G}}$ the fundamental extension associated with $\wt{G}$. \index{fundamental extension}
\end{dfn}

For fixed $n$, the construction $\xymatrix{\wt{G} \ar@{|~>}[r] & E_{\wt{G}} }$ could be viewed as a construction of the abelian $E_{\wt{G}}$ in the category $\textsf{Ab}(F^\times, Z(\wt{G}^\vee))$ of abelian extensions of $F^\times$ by $Z(\wt{G})$, which starts with the category of incarnations $\CExt(\wm{G}, \mbf{K}_2)$. The category $\textsf{Ab}(F^\times, Z(\wt{G}^\vee))$ is clearly a commutative Picard category with respect to the Baer sum.

\begin{prop}[{\cite{We14}}]
Fix $n$. The construction of $E_{\wt{G}}$ from $ \wm{G}\in \CExt(\wm{G}, \mbf{K}_2)$ is a functor of Picard categories
$$\xymatrix{
\CExt(\wm{G}, \mbf{K}_2) \ar@{~>}[r] & \textsf{Ab}(F^\times, Z(\wt{G}^\vee)).
}$$
Therefore, the composition $\xymatrix{(D,\eta) \ar@{|~>}[r] & \wt{G}_{D,\eta} \ar@{|~>}[r] & E_{\wt{G}_{D,\eta}} }$ also gives a functor of Picard categories:
$$\xymatrix{
\Bis_{\wm{G}} \ar@{~>}[r] & \textsf{Ab}(F^\times, Z(\wt{G}^\vee)).
}$$
\end{prop}

Recall that $\W_F$ is the Weil group of $F$. Denote by $\Rec$ the composition
$$\xymatrix{
\Rec: \quad  \W_F \ar[r] & F^\times,
}$$
which is induced from the Artin recirpocity map $\xymatrix{\W_F^{ab} \ar[r] &F^\times}$ which sending the class of geometric Frobenius to the class of uniformizer (modulo $\msc{O}_F^\times$) of $F^\times$.

We obtain the pull-back $\Rec^*(E_{\wt{G}})$:
$$\seq{Z(\wt{G})}{\Rec^*(E_{\wt{G}})}{\W_F}.$$

Then,

\begin{dfn} \index{$L$-group ! local}
Let $\wm{G} \in \CExt(\mbf{G}, \mbf{K}_2)$ be a BD extension, and let $\wt{G}\in \CExt(G, \mu_n)$ be the resulting topological extension. Consider the natural inclusion $\xymatrix{ j^{\wt{G}^\vee}: Z(\wt{G}^\vee) \ar@{^(->}[r] & \wt{G}^\vee}$. The $L$-group ${}^L\wt{G}$ of $\wt{G}$ is defined to be the push-out $j^{\wt{G}}_*$ of $\Rec^*(E_{\wt{G}})$:
$${}^L\wt{G}:=j^{\wt{G}^\vee}_* \circ \Rec^*(E_{\wt{G}})=\frac{\wt{G}^\vee \times \Rec^*(E_{\wt{G}}) }{\nabla Z(\wt{G}^\vee)},$$
where $\xymatrix{\nabla: Z(\wt{G}^\vee) \ar@{^(->}[r] & \wt{G}^\vee \times Z(\wt{G}^\vee)}$ is the anti-diagonal embedding. Also, we may use the equivalent $\Rec^*\circ j_*^{\wt{G}^\vee}(E_{\wt{G}})$ for the definition of ${}^L\wt{G}$.
\end{dfn}

\begin{rmk}
The definition of of $E_{2,\wt{G}}$ here differs slightly from that of \cite{We14}, where one considers the section of $\mca{E}_{Q,n}^{sc}$ over $Y_{Q,n}^{sc}$ given by
$$\xymatrix{
\s': \quad \alpha^\vee_{[n]} \ar@{|->}[r] & \wm{h}_\alpha^{[f]}(\alpha^\vee)^{n_\alpha}.
}$$
Similar proof shows that this gives a splitting, and in particular is independent of $f$ as well. However, the splitting $\s$ we defined gives exactly the expected match-up between the eigenvalue of adjoint action and the GK coefficient, whereas for $\s'$ there would be an additional term of $(-1)$-powers in the expression.
\end{rmk}
\subsubsection{Functoriality for Levi subgroups}
The group $\wt{G}$ gives rise to the Levi $\wt{M}$ by restriction of the Brylinski-Deligne data, and we obtain ${}^L\wt{G}$ and ${}^L\wt{M}$. It is shown in \cite{We14} (see also \cite[\S 5.5]{GaG14}) that there is a canonical morphism ${}^L\varphi$ extending the inclusion $\xymatrix{\varphi^\vee: {}^L\wt{M} \ar[r] & {}^L\wt{G}}$ such that the diagram commutes:
$$\xymatrix{
\wt{G}^\vee \ar@{^(->}[r]  &{}^L\wt{G}  \ar@{>>}[r] & \W_F \ar@{=}[d]  \\
\wt{M}^\vee \ar@{^(->}[u]^-{\varphi^\vee} \ar@{^(->}[r]  &{}^L\wt{M} \ar@{^(->}[u]^-{{}^L\varphi} \ar@{>>}[r] &\W_F .
}$$

The case $\wt{M}=\wt{T}$ is easy and crucial to us, and therefore we give the proof. Note ${}^L\wt{T}=\Rec^*(E_{\wt{T}})$ since $Z(\wt{T}^\vee)=\wt{T}^\vee$. 

\begin{prop} \label{L-T=p.o.}
Consider $\wt{G}$ and $\wt{T}$ the covering torus of $\wt{G}$. Let $\xymatrix{j^{\wt{T}^\vee}: Z(\wt{G}^\vee) \ar@{^(->}[r] & \wt{T}^\vee}$ be the inclusion. Then we have a canonical isomorphism
$${}^L\wt{T} \simeq j^{\wt{T}^\vee}_* \circ \Rec^*(E_{\wt{G}}),$$
where ${}^L\wt{T}=\Rec^*(E_{\wt{T}})$ by definition. Therefore it induces a canonical ${}^L\varphi$ from ${}^L\wt{T}$ to ${}^L\wt{G}$ giving the desired commutative diagram above.
\end{prop} 

\begin{proof}
It suffices to show 
$$E_{\wt{T}}\simeq j^{\wt{T}^\vee}_*(E_{\wt{G}}).$$
Since $E_{\wt{T}}=E_{1, \wt{T}} \oplus_B E_{2, \wt{T}}$ and $E_{\wt{G}}=E_{1, \wt{G}} \oplus_B E_{2, \wt{G}}$, we are reduced to show
$$E_{i, \wt{T}}\simeq j^{\wt{T}^\vee}_*(E_{i, \wt{G}}), i=1, 2.$$

The case $i=1$ can be easily checked to be true. Consider the case $i=2$ and the diagram as in (\ref{germ: E_2}):
$$\xymatrix{
F^\times/n \ar@{^(->}[r] \ar@{=}[d] &\mca{E}_{Q,n} \ar@{>>}[r] \ar@{>>}[d] &Y_{Q,n} \ar@{>>}[d] \\
F^\times/n \ar@{^(->}[r] &\mca{E}_{Q,n}\big/ \phi_{D,\eta}\circ \s(Y_{Q,n}^{sc}) \ar@{>>}[r] &Y_{Q,n}/Y_{Q,n}^{sc}.
}$$
Taking $\Hom(-,\C^\times)$ and pull-back by $\xymatrix{F^\times \ar[r] & F^\times/n}$, we obtain
$$\xymatrix{
\wt{T}^\vee \ar@{^(->}[r] &E_{2,\wt{T}} \ar@{>>}[r] &F^\times \\
Z(\wt{G}^\vee) \ar@{^(->}[r] \ar@{^(->}[u]^-{ j^{\wt{T}^\vee} } & E_{2,\wt{G}}  \ar@{>>}[r] \ar[u]^-{\phi^*_{D,\eta}} & F^\times \ar@{=}[u] .
}$$

The universal property of push-out $j^{\wt{T}^\vee}_*$ gives the middle map of the following commutative diagram
$$\xymatrix{
\wt{T}^\vee \ar@{^(->}[r] \ar@{=}[d] & j^{\wt{T}^\vee}_*(E_{2,\wt{G}})  \ar@{>>}[r] \ar[d] & F^\times \\
\wt{T}^\vee \ar@{^(->}[r] &E_{2,\wt{T}} \ar@{>>}[r] & F^\times  \ar@{=}[u] \ ,
}$$
which has identity maps on both $\wt{T}^\vee$ and $F^\times$. Therefore the middle map is a canonical isomorphism $j^{\wt{T}^\vee}_*(E_{2,\wt{G}})  \simeq E_{2,\wt{T}}$. 
\end{proof}

Recall by definition $j^{\wt{T}^\vee}_*(E_{\wt{G}})=\wt{T}^\vee \times E_{\wt{G}} \big/ \nabla Z(\wt{G}^\vee)$. Write $\msc{L}$ for the canonical isomorphism from above proof
$$\xymatrix{
\msc{L}: \quad \wt{T}^\vee \times E_{\wt{G}} \big/ \nabla Z(\wt{G}^\vee) \ar[r] & E_{\wt{T}}.
}$$

What is important for us is the explicit form of $\msc{L}$. From the construction we may identify
$$E_{2,\wt{T}}=\set{ ([P_a], a) \in \Hom(\mca{E}_{Q,n}, \C^\times) \times F^\times: [P_a]|_{F^\times/n}=h_a },$$
where as before $h_a(-)=(-, a)_n$ denotes the $n$-th Hilbert symbol. Here $[P_a]\big((b, y)\big)=h_a(b) \cdot P_a(y)$ is a homomorphism of $\mca{E}_{Q,n}$ with respect to a certain map (not necessarily a morphism) $\xymatrix{P_a: Y_{Q,n} \ar[r] & \C^\times}$. 

Similarly  $E_{2,\wt{G}}$ could be identified with the collection 
$$\set{([P_a], a) \in \Hom(\mca{E}_{Q,n}\big/ \phi_{D,\eta}\circ \s(Y_{Q,n}^{sc}), \C^\times) \times F^\times: [P_a]|_{F^\times/n}=h_a }.$$

Let $\wt{t}^\vee\in \wt{T}^\vee=\Hom(Y_{Q,n}, \C^\times)$.  It gives rise to an element $[\wt{t}^\vee]\in \Hom(\mca{E}_{Q,n}, \C^\times)$ via inflation. Tracing through the proof of Proposition \ref{L-T=p.o.}, we see that the canonical isomorphism takes the explicit form:
\begin{diagram}
\wt{T}^\vee \times E_{2,\wt{G}} \big/ \nabla Z(\wt{G}^\vee) &\rTo & E_{2,\wt{T}} \\
\big(\wt{t}^\vee, ([P_a], a)\big) &\rMapsto & \big( [\wt{t}^\vee] \cdot [P_a], a\big).
\end{diagram}

Thus the desired isomorphism $\msc{L}$ is given by
\begin{diagram}
\msc{L}: & \wt{T}^\vee \times E_{\wt{G}} \big/ \nabla Z(\wt{G}^\vee) &\rTo & E_{\wt{T}} \\
& \big(\wt{t}^\vee, ([P_a], a)\oplus_B (1, a)\big) &\rMapsto & ( [\wt{t}^\vee] \cdot [P_a], a )\oplus_B (1,a).
\end{diagram}

It is important for our purpose of the Gindikin-Karpelevich formula latter to determine the element $\wt{t}_{P_a}^\vee \in \wt{T}$ associated with $\msc{L}^{-1}\big(([P_a], a)\oplus_B (1,a)\big)$, where $([P_a], a)\oplus_B (1,a) \in E_{\wt{T}}$. Moreover, it is easy to see that $\wt{t}_{P_a}^\vee$ could be only determined up to a class modulo $Z(\wt{G}^\vee)$. That is, only the image of $\wt{t}_{P_a}^\vee$ in the quotient map $\xymatrix{\wt{T}^\vee \ar[r] & \wt{T}^\vee/ Z(\wt{G}^\vee)}$ is uniquely determined.

For this purpose, consider the trivial extension
$$\seq{\wt{T}^\vee/Z(\wt{G}^\vee)}{\wt{T}^\vee/Z(\wt{G}^\vee) \times F^\times}{F^\times},$$
which has a canonical splitting $s^\text{Tr}$ of the first map. There is a natural map $q_*$ from $j^{\wt{T}^\vee}_*(E_{\wt{G}})$ to $\wt{T}^\vee/Z(\wt{G}^\vee) \times F^\times$ such that the following diagram commutes
$$\xymatrix{
\wt{T}^\vee \ar@{^(->}[r] \ar@{>>}[d]_-q & j^{\wt{T}^\vee}_*(E_{\wt{G}})  \ar@{>>}[r] \ar[d]^-{q_*} & F^\times \ar@{=}[d] \\
\wt{T}^\vee/Z(\wt{G}^\vee) \ar@{^(->}[r]  & \wt{T}^\vee/Z(\wt{G}^\vee) \times F^\times \ar@/^1.3pc/[l]^-{s^{\text{Tr}}} \ar@{>>}[r] &F^\times.
}$$

Take the composition
$$\xymatrix{
\msc{C}=s^\text{Tr}\circ q_* \circ \msc{L}^{-1}: \quad E_{\wt{T}} \ar[r] & j^{\wt{T}^\vee}_*(E_{\wt{G}}) \ar[r] & \wt{T}^\vee/Z(\wt{G}^\vee) \times F^\times \ar[r] & \wt{T}^\vee/Z(\wt{G}^\vee).
}$$

Now we can describe explicitly the image of this map $\msc{C}$.

\begin{cor} \label{ad-t}
Let $\msc{P}_a=([P_a], a)\oplus_B (1, a) \in E_{\wt{T}}$, and identify $\Hom(Y_{Q,n}^{sc}, \C^\times)=\wt{T}^\vee/Z(\wt{G}^\vee)$. Then  $\msc{C}\big(\msc{P}_a\big)(\alpha^\vee_{[n]}) =[P_a] \circ \phi_{D,\eta}\big(\s(\alpha^\vee_{[n]})\big)$ for all $\alpha\in \Psi$.
That is, the image $\wt{t}_{P_a}^\vee$ of $\msc{L}^{-1}(\msc{P}_a)$ in $\wt{T}^\vee/Z(\wt{G}^\vee)$ is uniquely determined by
\begin{equation}
\wt{t}_{P_a}^\vee(\alpha^\vee_{[n]}) =[P_a] \circ \phi_{D,\eta} \big(\s(\alpha_{[n]}^\vee)\big), \alpha\in \Psi.
\end{equation}
\end{cor}
\begin{proof}
Clearly $\wt{t}_{P_a}^\vee(\alpha^\vee_{[n]})=\msc{C}\big(\msc{P}_a\big)(\alpha^\vee_{[n]})$. Now write 
$$\msc{L}^{-1}\big(\msc{P}_a\big)=\big(\wt{t}_{P_a}^\vee, ([P_a]/[\wt{t}_{P_a}^\vee], a)\oplus_B (1, a)\big) \in j^{\wt{T}^\vee}_*(E_{\wt{G}}),$$
where $[P_a]/[\wt{t}_{P_a}^\vee]\in \Hom(\mca{E}_{Q,n}\big/ \phi_{D,\eta} \circ \s(Y_{Q,n}^{sc}), \C^\times)$. In particular, it vanishes on $\phi_{D,\eta} \circ \s(Y_{Q,n}^{sc})$, and therefore
$$[\wt{t}_{P_a}^\vee] \circ \phi_{D,\eta} \big(\s(\alpha_{[n]}^\vee)\big)=[P_a] \circ \phi_{D,\eta} \big(\s(\alpha_{[n]}^\vee)\big), \ \alpha\in \Psi.$$
But we have $\wt{t}_{P_a}^\vee(\alpha_{[n]}^\vee)=[\wt{t}_{P_a}^\vee] \circ \phi_{D,\eta} \big(\s(\alpha_{[n]}^\vee)\big)$ from the definition of $[\wt{t}_{P_a}^\vee]$. This gives the desired result and completes the proof.
\end{proof}


\subsection{Global $L$-group} \label{global L-grp} \index{$L$-group ! global}
Let $\wm{G} \in \CExt(\mbf{G}, \mbf{K}_2)$ be defined over a number field $F$. We may assume that it is incarnated by $(D, \eta)$. Assume $\mu_n \subseteq F^\times$, one obtains a global covering group $\wm{G}(\A)$ in section \ref{G-bar(A)}:
$$\seq{\mu_n}{\wm{G}(\A)}{\mbf{G}(\A)}.$$

As in the local case, it is essential to have the global $L$-group for the purpose of arithmetic parametrization of automorphic forms on $\wm{G}(\A)$ and problems alike. Though we are largely dealing with local problems, we include the construction of global $L$-group here. For more details, see \cite{We14}. One benefit we gain is that it will enable us to define (partial) automorphic $L$-functions later.
\\

The construction of global ${}^L\wt{G}$ relies on a global fundamental extension 
$$\seq{Z(\wt{G}^\vee)}{E_{\A}}{F^\times\backslash \A^\times},$$
from which one may define the global $L$-group associated with $W_F$ as $\pi^*\circ j^{\wt{G}^\vee}_*(E_{\A})$. Here $\xymatrix{j^{\wt{G}^\vee}: Z(\wt{G}) \ar@{^(->}[r] & \wt{G}^\vee}$ is the natural inclusion and $W_F$ the global Weil group with the natural map (cf. \cite{Tat79}) $\xymatrix{\pi: W_F \ar[r] & F^\times\backslash \A^\times}$ induced from the  reciprocity map. In principle one may also replace $W_F$ by any arithmetic group for $F$ with a natural map to $F^\times\backslash \A^\times$, for example the conjectural Langlands automorphic $L$-group.

We recall the construction of $E_{\A}$ (cf. \cite{We14}), which follows from similar consideration of the local case. From now, assume we are given with the data $(n, Q, \mca{E}, \phi)$. 

The group $E_{\A}$ is constructed as the Baer sum of two extensions $E_{1,\A}$ and $E_{2,\A}$:
$$\seq{Z(\wt{G}^\vee)}{E_{i, \A}}{F^\times\backslash \A^\times} \text{ for }  i=1,2.$$

First of all, the complex dual group $\wt{G}^\vee$ is the one with root data as before in the local case:
$$\big(Y_{Q,n},\  \set{\alpha^\vee_{[n]}}_{\alpha\in \Psi}, \ \Hom(Y_{Q,n}, \Z),\  \set{n_\alpha^{-1}\alpha}_{\alpha\in \Psi}\big).$$
The modification here depends on the data $(n, Q)$. As usual,  identify $Z(\wt{G}^\vee)$ with $\Hom(Y_{Q,n}/Y_{Q,n}^{sc}, \C^\times)$.

\subsubsection{The construction of $E_{1,\A}$}

First, we have an extension 
$$\seq{\Hom(Y_{Q,n}/Y_{Q,n}^{sc}, \C^\times)}{E_{0,\A}}{\A^\times}$$
which is given by the cocycle
\begin{diagram}
\A^\times \times \A^\times & \rTo &\Hom(Y_{Q,n}/Y_{Q,n}^{sc}, \C^\times)  \\
(a, b) &\rMapsto &(y \mapsto (a, b)_n^{Q(y)}),
\end{diagram}
where $a, b\in \A^\times$ and $(-,-)_n$ denotes the global $n$-th Hilbert symbol. Since the Hilbert symbol is trivial on $F^\times \times F^\times$, the extension $E_{0, \A}$ splits over $F^\times$. The quotient of $E_{0,\A}$ by the splitting image of $F^\times$ gives us the extension 

$$\seq{Z(\wt{G}^\vee)}{E_{1,\A}}{F^\times\backslash \A^\times}.$$

\subsubsection{The construction of $E_{2,\A}$}

Start with $(n, Q, \mca{E}, \phi)$, we have the diagram with a canonical splitting $\s$ as in the local case:
$$\xymatrix{
F^\times/n \ar@{^(->}[r] \ar@{=}[d] & \mca{E}_{Q,n}^{sc} \ar[d]^-{\phi_{D,\eta}} \ar@{>>}[r] & Y_{Q,n}^{sc} \ar@{^(->}[d] \ar@/_1.3pc/[l]_-\s \\
F^\times/n \ar@{^(->}[r] & \mca{E}_{Q,n} \ar@{>>}[r] & Y_{Q,n},
}$$
from which we have
$$\xymatrix{
F^\times/n \ar@{^(->}[r] & \mca{E}_{Q,n}\big/\phi_{D,\eta} \circ \s(Y_{Q,n}^{sc}) \ar@{>>}[r] & Y_{Q,n}/Y_{Q,n}^{sc}.
}$$
Apply $\Hom(-,\C^\times)$ and pull-back by the global Hilbert symbol 
$$\xymatrix{h_{\A}: F^\times\backslash \A^\times \ar[r] &\Hom(F^\times/n, \C^\times)}$$
given by $h_{\A}(b)=(b, a)_n$, we obtain our definition of $E_{2,\A}$:
$$\seq{Z(\wt{G}^\vee)}{E_{2, \A}}{F^\times\backslash \A^\times}.$$

\subsubsection{Local-global compatibility}

For each place $v$ of $F$, use the inclusion $\xymatrix{i_v: F \ar@{^(->}[r] & F_v}$ we could push-out the data $(Q, \mca{E}, \phi)$ to $(Q, \mca{E}_v, \phi_v)$, which is then the Brylinski-Deligne data associated with $\wm{G}_{/F_v}$, obtained from the base change of $\wm{G}$ via $i_v$. Here, $\mca{E}_v=(i_v)_*(\mca{E})$ and $\phi_v=i_v \circ \phi$.

Associated with $(Q, \mca{E}_v, \phi_v)$ there is the local fundamental extension $E_{\wt{G}_v}$ which sits in
$$\seq{Z(\wt{G}^\vee)}{E_{\wt{G}_v}}{F^\times_v}.$$

The construction of $E_{\A}$ above gives a canonical map $\xymatrix{E_{\wt{G}_v} \ar[r] & E_{\A}}$ such that the following diagram commutes:
\begin{equation} \label{lg FE}
\xymatrix{
Z(\wt{G}^\vee) \ar@{^(->}[r] \ar@{=}[d] & E_{\wt{G}_v} \ar[d] \ar@{>>}[r] & F^\times_v \ar@{^(->}[d] \\
Z(\wt{G}^\vee) \ar@{^(->}[r] & E_{\A} \ar@{>>}[r] & F^\times\backslash \A^\times,
}
\end{equation}
where the right hand side map is the natural inclusion. Clearly, instead of for the fundamental extensions, the local and global compatibility could be stated at the level of local and global $L$-groups. That is, we define the global $L$ group with respect to $W_F$ to be
$${}^L\wt{G}:= \pi^* \circ j_*^{\wt{G}^\vee}(E_{\A}),$$
where $\pi$ and $j^{\wt{G}^\vee}$ are the aforementioned maps from the beginning of this section. For all $v$, the following diagram commutes:
\begin{equation} \label{lg L-grp}
\xymatrix{
\wt{G}^\vee \ar@{^(->}[r] \ar@{=}[d] & {}^L\wt{G}_v \ar[d] \ar@{>>}[r] & \W_{F_v} \ar[d] \\
\wt{G}^\vee \ar@{^(->}[r] & {}^L\wt{G} \ar@{>>}[r] & W_F,
}
\end{equation}
where the right hand side is a certain natural map $\xymatrix{\W_{F_v} \ar[r] & \W_F}$, see \cite{Tat79}.

\chapter{Admissible splittings of the $L$-group}

\section{Subgroups of $\wt{G}$}

Assume that $F$ is a local field, unless stated otherwise. Let $\wt{G}$ be incarnated by $(D,\eta)$, and let $\wt{T}$ be the covering torus of $\wt{G}$. Then the center $Z(\wt{T})$ allows for a simple description as follows.

Let $\xymatrix{Y_{Q,n} \ar@{^(->}[r] & Y}$ be the inclusion, and let $\mbf{T}_{Q,n}$ be the torus associated with $Y_{Q,n}$ whose $F$-rational point is denoted by $T_{Q,n}$. Thus we have a natural map $\xymatrix{i_{Q,n}: T_{Q,n} \ar[r] & T}$ whose image we denote by $T^\dag$.

\begin{lm}[{\cite{We09}}]
The center $Z(\wt{T})$ is the preimage of $T^\dag$ in $\wt{T}$:
$$\xymatrix{
\mu_n \ar@{^(->}[r] & \wt{T} \ar@{>>}[r] &T \\
\mu_n \ar@{=}[u] \ar@{^(->}[r] & Z(\wt{T}) \ar@{^(->}[u] \ar@{>>}[r] &T^\dag \ar@{^(->}[u]\ .
}$$
\end{lm}

Now we define $\wt{T}_{Q,n}:=i_{Q,n}^*(Z(\wt{T}))$ to be the pull-back of $Z(\wt{T})$ via $i_{Q,n}$:
$$\xymatrix{
\mu_n \ar@{^(->}[r] & Z(\wt{T}) \ar@{>>}[r] &T^\dag \\
\mu_n \ar@{=}[u] \ar@{^(->}[r] & \wt{T}_{Q,n} \ar@{>>}[u]_-{\wt{i}_{Q,n}} \ar@{>>}[r] &T_{Q,n} \ar@{>>}[u]_-{i_{Q,n}}.
}$$

That is, elements of $\wt{T}_{Q,n}$ are of the form $\big((\zeta, i_{Q,n}(t)),  t\big) \in Z(\wt{T}) \times T_{Q,n}$. Since $\wt{T}$ is assumed to be incarnated by $D$, we can write $\wt{T}=\mu_n \times_D T$ and therefore also $Z(\wt{T})=\mu_n \times_D T^\dag$.

From now, we will write $(\zeta, t)$ for $\big((\zeta, i_{Q,n}(t)),  t\big) \in \wt{T}_{Q,n}$. The group law on $\wt{T}_{Q,n}=\mu_n \times_D T_{Q,n}$  inherited from $Z(\wt{T})$ is thus given by
\begin{align} 
& (i) \quad \big[(\zeta_1, y_1\otimes a),  (\zeta_2, y_2\otimes b)\big]=1\text{ for } y_1, y_2 \in Y_{Q,n}; \label{TQn law1} \\
& (ii) \quad (\zeta_1, y_1\otimes a) \cdot (\zeta_2, y_2\otimes a)=\big(\zeta_1\zeta_2(a, a)_n^{D(y_1, y_2)}, (y_1+y_2)\otimes a\big); \label{TQn law2} \\
& (iii) \quad (\zeta_1, y\otimes a) \cdot (\zeta_2, y\otimes b)=\big(\zeta_1\zeta_2(a, b)_n^{Q(y)}, y\otimes (ab)\big). \label{TQn law3}
\end{align}

Clearly, there is a canonical isomorphism $\xymatrix{\text{Ker}(i_{Q,n}) \ar[r] & \text{Ker}(\wt{i}_{Q,n})}$ given by
\begin{equation} \label{kernel i-Qn}
\xymatrix{
t \in \text{Ker}(i_{Q,n}) \ar@{|->}[r] & (1, t) \in \text{Ker}(\wt{i}_{Q,n}).
}
\end{equation}

It is desirable to have a simple description of $\text{Ker}(\wt{i}_{Q,n})$. However, we do not have an explicit one and the following consideration would be helpful, at least for discussions in next section.

Consider the three lattices
$$\xymatrix{
nY \ar@{^(->}[r] & Y_{Q,n} \ar@{^(->}[r] &Y 
}$$
and the induced isogenies
$$\xymatrix{
T_{nY} \ar[r]^-{i_{nY}} & T_{Q,n} \ar[r]^-{i_{Q,n}} &T,
}$$
where $\mbf{T}_{nY}$ is the torus defined over $F$ associated with $nY$ and $T_{nY}:=\mbf{T}_{nY}(F)$.

\begin{lm}
The following inclusion holds:
$$\text{Ker}(i_{Q,n}) \subseteq \text{Im}(i_{nY}).$$
\end{lm}
\begin{proof}
Note that the composition $i_{Q,n} \circ i_{nY}$ is the $n$-th power map of $T$. By elementary divisor theorem, let $\set{e_i}$ be a basis of $Y$ such that $\set{m_i e_i}$ is a basis of $Y_{Q,n}$ where $i=1, 2,..., r$. Then $\set{ne_i}$ is a basis of $nY$, and $m_i|n$.
Then the isogenies $\xymatrix{ T_{nY} \ar[r]^-{i_{nY}} & T_{Q,n} \ar[r]^-{i_{Q,n}} &T }$ could be identified with $\xymatrix{ \prod_i F^\times \ar[r]^-{i_{nY}} & \prod_i F^\times \ar[r]^-{i_{Q,n}} & \prod_i F^\times }$, whose $i$-th component maps are given by
$$\xymatrix{
a_i \ar@{|->}[r] & a_i^{n/m_i} \ar@{|->}[r] & a_i^n, \quad a_i \in F^\times.
}$$
So $\text{Ker}(i_{Q,n})$ is generated by $(m_ie_i)\otimes \zeta_{m_i} \in Y_{Q,n}\otimes F^\times$, where $\zeta_{m_i}$ is some $m_i$-th root of unity. But we see that there always exists $n$-th root of unity $\zeta_n\in F^\times$  such that $(ne_i)\otimes \zeta_{n} \in T_{nY}$ is mapped to $(m_ie_i)\otimes \zeta_{m_i}$. It follows $\text{Ker}(i_{Q,n}) \subseteq \text{Im}(i_{nY})$.
\end{proof}

Define a map $s_n$ on the generators of $T_{nY}$ by
$$\xymatrix{
s_n: \quad T_{nY} \ar[r] & \wt{T}_{Q,n}, \quad  (ny)\otimes a \ar@{|->}[r] &(1, (ny)\otimes a) \in \wt{T}_{Q,n}, \quad y\in Y.
}$$

By checking the relations in (\ref{TQn law1})-(\ref{TQn law3}), it follows from the observation $\text{Ker}(s_n)=\text{Ker}(i_{nY})$ that we have the following.

\begin{lm} \label{s_n}
The map $s_n$ is a homomorphism, and therefore gives a splitting of $\text{Im}(i_{nY})$ into $\wt{T}_{Q,n}$, which extends the canonical isomorphism $\xymatrix{\text{Ker}(i_{Q,n}) \ar[r] & \text{Ker}(\wt{i}_{Q,n})}$ given by (\ref{kernel i-Qn}). That is, the following diagram commutes:
$$\xymatrix{
\mu_n \ar@{^(->}[r] &\wt{T}_{Q,n} \ar@{>>}[r] & T_{Q,n} \\
& \text{Ker}(\wt{i}_{Q,n}) \ar@{^(->}[u] & \text{Im}(i_{nY}) \ar@{^(-->}[lu]^-{s_n} \ar@{^(->}[u] \\
& & \text{Ker}(i_{Q,n}) \ar[lu]^-{\text{can iso}}\ar@{^(->}[u] \ .
}$$ 
\end{lm}

Thus the homomorphism $s_n$ also induces a splitting of $\text{Im}(i_{Q,n}\circ i_{nY}) \subseteq T^\dag$ into $Z(\wt{T})$, which by abuse of notation is also written as $s_n$. Note $\text{Im}(i_{Q,n}\circ i_{nY})=\set{t^n: t\in T}$.

\section{Admissible splittings of the $L$-group}

We are interested in splittings of ${}^L\wt{G}$ over $\W_F$, with respect to which we have ${}^L\wt{G}\simeq \wt{G}^\vee \times \W_F$. For example, we will see that ${}^L\wt{T}$ always splits and we have a local Langlands correspondence between splittings of ${}^L\wt{T}$ and genuine characters of the center $Z(\wt{T})$.

Let $\WD_F:=\mbf{SL}_2(\C)\times \W_F$ be the Weil-Deligne group. For splittings of ${}^L\wt{G}$ over $\W_F$, we have

\begin{dfn} \label{adm splitting}
An $L$-parameter is just a splitting $\rho$ of ${}^L\wt{G}$ over $\W_F$:
$$\xymatrix{
\wt{G}^\vee \ar@{^{(}->}[r]  &{}^L\wt{G} \ar@{>>}[r] & \W_F \ar@/^1pc/[l]^-{\rho},
}$$
\noindent while a Weil-Deligne parameter is a homomorphism $\xymatrix{\WD_F \ar[r] & {}^L\wt{G} }$ such that the following diagram commutes
$$\xymatrix{
\wt{G}^\vee \ar@{^(->}[r] &  {}^L\wt{G} \ar@{>>}[r] &\W_F \\
& & \WD_F \ar[lu] \ar@{>>}[u].
}$$
The vertical surjection is the projection of $\WD_F$ onto its second component. Any $L$-parameter $\rho$ is called admissible if and only if it gives an isomorphism
$${}^L\wt{G}\simeq \wt{G}^\vee \times_\rho \W_F$$
with the inverse map given by $\xymatrix{(g^\vee, x) \ar@{|->}[r] & g^\vee \cdot \rho(x)}$ for $g^\vee \in \wt{G}^\vee, x\in \W_F$.
\end{dfn} \index{parameter ! $L$-} \index{parameter ! Weil-Deligne} \index{admissible splitting}

Write $\mfr{S}^a({}^L\wt{G}, \W_F)$ for the set of admissible splittings. Then 

\begin{lm}
A splitting $\rho$ lies in $\mfr{S}^a({}^L\wt{G}, \W_F)$ if and only if its image lies in the centralizer $Z_{{}^L\wt{G}} (\wt{G}^\vee)\simeq \Rec^* (E_{\wt{G}})$ of $\wt{G}^\vee$ in ${}^L\wt{G}$. That is, we have the isomorphism $\mfr{S}^a({}^L\wt{G}, \W_F) \simeq \mfr{S}(E_{\wt{G}}, F^\times)$ consisting of splittings of $E_{\wt{G}}$ over $F^\times$. It is clearly a $\Hom(\W_F, Z(\wt{G}^\vee))$-torsor.
\end{lm}
\begin{proof}
The assignment $\xymatrix{(g^\vee, x) \ar@{|->}[r] & g^\vee \cdot \rho(x)}$ is a homomorphism, then one must have
$$g_1^\vee \cdot \rho(x_1) g_2^\vee \cdot \rho(x_2)=g^\vee_1g_2^\vee \cdot \rho(x_1)\rho(x_2),$$
where $g_i^\vee \in \wt{G}^\vee$ and $x_i \in \W_F$ are arbitrary elements. In particular, $\rho(x)$ commutes with arbitrary element of $\wt{G}^\vee$. The construction of ${}^L\wt{G}$ is:
$$\xymatrix{
Z(\wt{G}^\vee) \ar@{^(->}[r] & E_{\wt{G}} \ar@{>>}[r] & F^\times \\
Z(\wt{G}^\vee) \ar@{=}[u] \ar@{^(->}[r] \ar@{^(->}[d]_-{j^{\wt{G}^\vee}} & \Rec^*(E_{\wt{G}}) \ar[u] \ar@{>>}[r] \ar[d] & \W_F \ar[u]^{\Rec} \ar@{=}[d] \\
\wt{G}^\vee \ar@{^(->}[r] & {}^L\wt{G} \ar@{>>}[r] & \W_F.
}$$
So admissible splittings are valued in $\Rec^*(E_{\wt{G}})$. Tracing through the diagrams, it is easy to see that we have the canonical isomorphisms
$$\mfr{S}^a({}^L\wt{G}, \W_F) \simeq \mfr{S}(\Rec^*(E_{\wt{G}}), \W_F) \simeq \mfr{S}(E_{\wt{G}}, F^\times).$$
\end{proof}

\subsection{Conditions on the existence of admissible splittings}

The goal of this section is on some subsets of  $\Hom_\epsilon (Z(\wt{T}), \C^\times)\subseteq \Hom_\epsilon (\wt{T}_{Q,n}, \C^\times)$, which are called qualified or distinguished characters (to be defined later) and which give rise to admissible splittings of $\mfr{S}^a({}^L\wt{G}, \W_F)$.

Before we proceed, we give some general discussion on the conditions for the existence of admissible splittings. Recall the construction of $E_{\wt{G}}$ as
$$E_{\wt{G}} =E_{2, \wt{G}} \oplus_B E_{1, \wt{G}}.$$ 
We have identified $E_{2,\wt{G}}$ with
$$\set{([P_a], a) \in \Hom\big(\mca{E}_{Q,n}\big/\phi_{D,\eta}\circ \s(Y_{Q,n}^{sc}), \C^\times\big) \times F^\times: [P_a]|_{F^\times/n} =h_a},$$

\noindent where $[P_a]\big((b, y)\big)=(b, a)_n \cdot P_a(y)$ for some $\C^\times$-valued function $P_a$ on $Y_{Q,n}$. The conditions on $[P_a]$ translate into
\begin{align}
& (c1) \quad [P_a]\circ \phi_{D,\eta} \big(\s(\alpha^\vee_{[n]})\big)=1,  \alpha \in \Psi, \\
& (c2) \quad P_a(y_1 +y_2)=P_a(y_1) \cdot P_a(y_2) \cdot (a, a)_n^{D(y_1, y_2)} .
\end{align}

The cocycle of $E_{1,\wt{G}}$ actually takes value in the $2$-torsion of $Z(\wt{G})$, i.e. $\Hom(Y_{Q,n}/Y_{Q,n}^{sc}, \mu_2)$. 
Therefore, any section of $E_{\wt{G}}$ over $F^\times$ given by
$$\xymatrix{a \ar@{|->}[r] & ([P_a], a)\oplus_B (1, a) \in E_{\wt{G}} }$$
is a splitting if and only if the section $a\mapsto ([P_a], a)$ of $E_{2,\wt{G}}$ over $F^\times$ produces a cocycle the same as that on $E_{1,\wt{G}}$. That is, to translate this into $P_a$, we must have
\begin{align}
& (c3) \quad P_{ac}(y)=P_a(y)\cdot P_c(y) \cdot (a, c)_n^{Q(y)}, \quad y\in Y_{Q,n}.
\end{align}

Note that in particular $(c2)$ and $(c3)$ imply $P_{a^m}(y)=P_a(y^m)$. Consider the covering group $\wt{T}_{Q,n}$ and define a map $\xymatrix{\wt{\chi}_P: \wt{T}_{Q,n} \ar[r] & \C^\times}$ on generators of $\wt{T}_{Q,n}$ by
$$\xymatrix{
\wt{\chi}_P: \quad (\xi, y\otimes a) \ar@{|->}[r] & \xi \cdot [P_a]\big((1, y) \big)=\xi \cdot P_a(y).
}$$

Then $\wt{\chi}_P \in \Hom_\epsilon(\wt{T}_{Q,n}, \C^\times)$ is a well-defined genuine character from $(c2)$ and $(c3)$. That is, admissible splittings of ${}^L\wt{G}$ over $\W_F$ correspond to genuine characters of $\Hom_\epsilon(\wt{T}_{Q,n}, \C^\times)$ satisfying a certain condition $(C1)$ derived from $(c1)$.  We will state the condition $(C1)$ explicitly later.

Before we proceed, we observe that in general characters defined on the external $\wt{T}_{Q,n}$ do not reflect the ambient group $\wt{G}$. Recall the definition of $\wt{T}_{Q,n}$ as the pull-back:
$$\xymatrix{
\mu_n \ar@{^(->}[r] & Z(\wt{T}) \ar@{>>}[r] &T^\dag \\
\mu_n \ar@{=}[u] \ar@{^(->}[r] & \wt{T}_{Q,n} \ar@{>>}[u]_-{\wt{i}_{Q,n}} \ar@{>>}[r] &T_{Q,n} \ar@{>>}[u]_-{i_{Q,n}}.
}$$
If one reverses the discussion above and starts with a character on $\wt{T}_{Q,n}$ which gives rise to an admissible splitting of ${}^L\wt{G}$, it is natural to require that it descends to $Z(\wt{T})$. That is, we are interested in characters $\wt{\chi}_P \in \Hom_\epsilon(\wt{T}_{Q,n}, \C^\times)$ which satisfy the following condition in addition to $(C1)$:
\begin{align}
(C0) \quad \wt{\chi}_P\big( (1, t) \big)=1 \text{ for any } t \in \text{Ker} (i_{Q,n}).
\end{align}

To proceed, we first state a very useful result for explicating the condition $(C1)$ and also for the GK formula later. Let $\wt{\chi} \in \Hom_\epsilon(Z(\wt{T}), \C^\times)$ be an arbitrary genuine character of $Z(\wt{T})$. For any $a\in F^\times$, it gives rise to an element $([P_a], a) \oplus_B (1, a)\in E_{\wt{T}}$ where $[P_a]\in \Hom(\mca{E}_{Q,n}, \C^\times)$ is given by
\begin{equation} \label{wt-chi to [P]}[P_a]
\big( (1, y) \big):=\wt{\chi}\big( (1, y\otimes a)\big).
\end{equation}

Recall the natural map $\xymatrix{\Phi_{D,\eta}: \wm{T}^{sc} \ar[r] & \wm{T}}$, which by abuse of notation is also used to denote the induced map $\xymatrix{\wt{T}^{sc} \ar[r] & \wt{T}}$. For any $a\in F^\times$ and any $\alpha \in \Psi$, let $\wm{h}^{[b]}_\alpha(a) \in \wt{T}^{sc}$ be the element of the canonical section which a priori depends on $b \in F^\times$.

Then we have
\begin{prop} \label{s-eta}
The element $\wm{h}^{[b]}_\alpha(a^{n_\alpha}) \in \wt{T}^{sc}$ is independent of $b\in F^\times$, and thus we could omit $b$ for notational simplicity and just write $\wm{h}_\alpha(a^{n_\alpha})$ for it. More importantly, for any $\alpha \in \Psi$ and any $a\in F^\times$,
$$\wt{\chi} \circ \Phi_{D,\eta}\big(\wm{h}_\alpha(a^{n_\alpha}) \big)=[P_a] \circ \phi_{D,\eta} \big(\s(\alpha^\vee_{[n]})\big),$$
where $[P_a]$ is associated with $\wt{\chi}$ by (\ref{wt-chi to [P]}).
\end{prop}
\begin{proof}
First it follows from (\ref{h ppty2}) that for all $\alpha\in \Psi$,
$$\wm{h}^{[db]}_\alpha(a^{n_\alpha})= \wm{h}^{[b]}_\alpha(a^{n_\alpha}) \cdot (d, a^{n_\alpha})_n^{Q(\alpha^\vee)} = \wm{h}^{[b]}_\alpha(a^{n_\alpha}) \in \wt{T}^{sc}.$$
Thus the first assertion is proved.

It is clear that $\Phi_{D,\eta}\big(\wm{h}_\alpha(a^{n_\alpha}) \big) \in Z(\wt{T})$. Now we show the desired equality by first treating the case $\alpha \in \Delta$. If $\alpha\in \Delta$, by (\ref{Phi}), 
$$\Phi_{D,\eta}\big(\wm{h}_\alpha(a^{n_\alpha}) \big)=\big((\eta(\alpha^\vee), a^{n_\alpha})_n, \mbf{h}_\alpha(a^{n_\alpha}) \big) \in \wt{T}.$$
On the other hand, by (\ref{phi}), $\phi_{D,\eta} \big(\s(\alpha^\vee_{[n]})\big)=\big(\eta(n_\alpha \alpha^\vee), n_\alpha \alpha^\vee   \big) \in \mca{E}_{Q,n}$.
Thus,
\begin{align*}
\wt{\chi} \circ \Phi_{D,\eta}\big(\wm{h}_\alpha(a^{n_\alpha}) \big) &= (\eta(n_\alpha \alpha^\vee), a)_n \cdot \wt{\chi} \big( (1, \alpha^\vee_{[n]} \otimes a) \big) \\
&=(\eta(n_\alpha \alpha^\vee), a)_n \cdot [P_a]\big((1, \alpha^\vee_{[n]}) \big) \\
&=[P_a]\big((\eta(\alpha^\vee_{[n]}), \alpha^\vee_{[n]}) \big) \\
&=[P_a] \circ \phi_{D,\eta} \big(\s(\alpha^\vee_{[n]})\big).
\end{align*}
In general, let $\gamma \in \Psi$, we may use induction on the minimum of the lengths of $w\in W$ such that $w(\gamma) \in \Delta$. Thus assume $\gamma=s_\alpha(\beta)$ with $\alpha \in \Delta$ and the equality hold for $\beta$, i.e. $\wt{\chi} \circ \Phi_{D,\eta}\big(\wm{h}_\beta(a^{n_\beta}) \big)=[P_a] \circ \phi_{D,\eta} \big(\s(\beta^\vee_{[n]})\big).$

Note $Q(\gamma^\vee)=Q(\beta^\vee)$ and $n_\gamma=n_\beta$. By \cite[\S 11.3]{BD01},  we have the following equalities in $\wt{T}^{sc}$
\begin{align*} \label{eqn 1}
\wm{h}_\gamma(a^{n_\gamma}) &=\wm{h}_\beta(a^{n_\beta}) \cdot \wm{h}_\alpha(a^{-n_\beta \angb{\alpha}{\beta^\vee}}) \\
&=\wm{h}_\beta(a^{n_\beta}) \cdot \wm{h}_\alpha(a^{-n_\alpha \angb{\beta}{\alpha^\vee}}) \\
&=\wm{h}_\beta(a^{n_\beta}) \cdot \wm{h}_\alpha(a^{n_\alpha})^{-\angb{\beta}{\alpha^\vee}} \cdot (a^{n_\alpha}, a^{n_\alpha})_n^{\varepsilon(-\angb{\beta}{\alpha^\vee}) Q(\alpha^\vee)} \\
&=\wm{h}_\beta(a^{n_\beta}) \cdot \wm{h}_\alpha(a^{n_\alpha})^{-\angb{\beta}{\alpha^\vee}}.
\end{align*}

On the other hand, $\s(\gamma^\vee_{[n]})=\s(\beta^\vee_{[n]}) \cdot \s(\alpha^\vee_{[n]})^{-\angb{\beta}{\alpha^\vee}} \in \mca{E}_{Q,n}$. By the induction hypothesis on $\beta$ and the proved equality for $\alpha \in \Delta$, we see $\wt{\chi} \circ \Phi_{D,\eta}\big(\wm{h}_\gamma(a^{n_\alpha}) \big)=[P_a] \circ \phi_{D,\eta} \big(\s(\gamma^\vee_{[n]})\big)$. That is, the assertion holds for $\gamma$, and the proof is completed.
\end{proof}

Now we explicate or state the conditions $(C0), (C1)$ in equivalent or stronger forms.

\subsubsection{The condition $(C0)$}

In view of the inclusion $\text{Ker}(\wt{i}_{Q,n})\subset \text{Im}(s_n)$  of Lemma \ref{s_n}, there is the stronger condition
that could be imposed on $\wt{\chi}_P$:
\begin{align}
(C0)_+ \quad \wt{\chi}_P\circ s_n=\mbf{1}, \text{ or equivalently } \wt{\chi}_P(\wt{t})= 1 \text{ for all } \wt{t} \in \text{Im}(s_n) \subset \wt{T}_{Q,n}.
\end{align}

It takes an explicit form
\begin{align}
(C0)_+^\text{eq} \quad \wt{\chi}_P\big((1, (ny)\otimes a)\big)= 1 \text{ for all } (ny)\otimes a \in T_{Q,n} \text{ with } y\in Y.
\end{align}

\subsubsection{The condition $(C1)$}

To write $(C1)$ in explicit form, we would like to impose that $\wt{\chi}_P$ satisfies $(C0)$ in the first place. Then we could consider $\wt{\chi}_P$ as a genuine character of $Z(\wt{T})$.

Thus assuming $(C0)$ of $\wt{\chi}_P$, the condition $(c1)$ on $[P_a]$ translates into 
\begin{align}
& (C1)\quad \wt{\chi}_P \circ \Phi_{D,\eta}\big(\wm{h}_\alpha(a^{n_\alpha}) \big) =1, \text{ for all } \alpha \in \Psi,
\end{align}
or equivalently (from above proof) the same statement but only for $\alpha\in \Delta$, which takes the explicit form
\begin{align}
& (C1)^\text{eq} \quad \wt{\chi}_P \big( (1, \alpha^\vee_{[n]}\otimes a)\big) =\big(a, \eta(\alpha_{[n]}^\vee)\big)_n, \text{ for all } \alpha \in \Delta.
\end{align}

Here $(1, \alpha^\vee_{[n]}\otimes a)$ could be viewed either in $\wt{T}_{Q,n}$ or $Z(\wt{T})$ as we have assumed $(C0)$.

\begin{dfn} \label{key dfn} \index{character ! qualified} \index{character ! distinguished}
A genuine character $\wt{\chi}$ of $\wt{T}_{Q,n}$ is called \emph{qualified} if it satisfies the two conditions 
\begin{align}
& (C0)  \quad \wt{\chi}\big((1, t)\big)=1 \text{ for any } t \in \text{Ker} (i_{Q,n}), \\
& (C1)^\text{eq} \quad \wt{\chi} \big( (1, \alpha^\vee_{[n]}\otimes a)\big) =\big(a, \eta(\alpha_{[n]}^\vee)\big)_n, \text{ for all } \alpha \in \Delta.
\end{align}
Since $(C0)$ implies that $\wt{\chi}$ descends to $Z(\wt{T})$, thus an admissible character is a genuine character $\wt{\chi}$ of $Z(\wt{T})$ satisfying $(C1)^\text{eq}$. In this case, $(1, y\otimes a)$ and $(1, \alpha^\vee_{[n]}\otimes a)$ above are then viewed as elements in $Z(\wt{T})$.

A genuine character $\wt{\chi}$ of $\wt{T}_{Q,n}$ is called \emph{distinguished} if it satisfies
\begin{align}
& (C0)_+^\text{eq} \quad \wt{\chi}\big((1, (ny)\otimes a)\big)= 1 \text{ for all } (ny)\otimes a \in T_{Q,n} \text{ with } y\in Y, \\
& (C1)^\text{eq} \quad \wt{\chi} \big( (1, \alpha^\vee_{[n]}\otimes a)\big) =\big(a, \eta(\alpha_{[n]}^\vee)\big)_n, \text{ for all } \alpha \in \Delta.
\end{align}

Equivalently, a distinguished character is a genuine character of $Z(\wt{T})$ satisfying the same conditions $(C0)_+^\text{eq}$ and $(C1)^\text{eq}$, where $(1, y\otimes a)$ and $(1, \alpha^\vee_{[n]}\otimes a)$ above are then viewed as elements in $Z(\wt{T})$.
\end{dfn}

Write $\Hom^q_\epsilon(Z(\wt{T}), \C^\times)$ and $\Hom^d_\epsilon(Z(\wt{T}), \C^\times)$ for the set of qualified and distinguished characters respectively. Clearly, any distinguished character is qualified, and qualified characters give rise to admissible splittings of ${}^L\wt{G}$. We have the following relations:
$$\xymatrix{
\Hom^d_\epsilon(Z(\wt{T}), \C^\times) \ar@{^(->}[r] & \Hom^q_\epsilon(Z(\wt{T}), \C^\times) \ar@{^(->}[r] & \mfr{S}^a({}^L\wt{G}, \W_F) \ar@{^(->}[r] & \Hom_\epsilon(\wt{T}_{Q,n}, \C^\times).
}$$

There are obstructions to the existence of qualified or distinguished characters. For example, for the existence of distinguished characters one has the necessary condition that for all $\alpha\in \Delta$,
\begin{equation}
\begin{cases}
(a, \eta(\alpha^\vee_{[n]}))_n=1 \text{ if } \alpha^\vee_{[n]}\otimes a=1 \in T_{Q,n}, \\
(a, \eta(\alpha^\vee_{[n]}))_n=1 \text{ if } \alpha^\vee_{[n]} \otimes a= (ny)\otimes b \in T_{Q,n} \text{ for some } y\in Y, b\in F^\times. 
\end{cases}
\end{equation}

However, these conditions may not be sufficient, and in general we do not have an explicit description of the necessary and sufficient condition for the existence of distinguished characters.

Nevertheless, we could restrict to look at certain coverings $\wt{G}_{D,\eta}$ which arise from nicer incarnation object.

\begin{dfn} \index{bisector!fair}
A bisector $D$ is called fair if for all $\alpha^\vee \in \Delta^\vee$,
$$2|Q(\alpha^\vee) \text{ implies } 2|D(\alpha^\vee, y) \text{ for all } y\in Y.$$
We also call $(D, \eta)$ fair if $D$ is fair.
\end{dfn}

Regarding the existence, we have
\begin{prop}[{\cite[\S 2.5]{We13}}]
Any Weyl-invariant quadratic form $Q$ on $Y$ possesses a fair bisector defined above.
\end{prop}

Fix a fair $D^\text{fair}$ for $Q$. In view of Example \ref{D1 to D2}, any general $(D,\eta)$ is isomorphic to $(D^\text{fair}, \eta')$ for some $\eta'$. Therefore, there is no loss of generalities to assume that the bisector $D$ is always fair and allow $\eta$ to vary.

\begin{eg} \label{fair D}
Let $\mbf{G}$ be a reductive group with simply-connected derived group $\mbf{G}^{der}$. The quotient $Y/Y^{sc}$ is a free $\Z$-module. A fair bisector $D$ could be defined in the following way.

Let $\set{\mfr{b}_1, \mfr{b}_2, ..., \mfr{b}_r}\cup \set{\mfr{b}_{r+1}, ..., \mfr{b}_k}$ be a basis of $Y$ such that $\mfr{b}_i=\alpha_i^\vee, 1\le i\le r$ for a certain order of $\alpha_i^\vee \in \Delta^\vee$, and $\mfr{b}_i$ for $r+1\le i\le k$ belongs to $Y$. Then define $D^\text{fair}$ by
\begin{equation*}
    D^\text{fair}(\mfr{b}_i, \mfr{b}_j) =
    \begin{cases}
      0 & \text{ if } i< j, \\
Q(\mfr{b}_i) &\text{ if } i=j, \\
      B(\mfr{b}_i, \mfr{b}_j) & \text{ if } i>j.
    \end{cases}
\end{equation*}

Clearly $D^\text{fair}$ defined in this manner is fair. For general $\wt{G}_{D,\eta}$ associated to $\wm{G}$ of $\mbf{G}$ of this type, it follows from (\ref{D-eta free}) that $(D, \eta)\simeq (D^\text{fair}, \mbf{1})$ for arbitrary $(D,\eta)$. Therefore, for such groups, there is no loss of generalities (up to isomorphism classes) in considering fair incarnation objects of the form $(D^\text{fair}, \mbf{1})$.
\end{eg}

Let $T_{Q,n}^{sc}$ be the torus associated with $Y_{Q,n}^{sc}$ and $\xymatrix{i_{Q,n}^{sc}: T_{Q,n}^{sc} \ar[r] & T_{Q,n}}$ the isogeny induced from $\xymatrix{Y_{Q,n}^{sc} \ar@{^(->}[r] & Y_{Q,n}}$ . Here $T_{Q,n}^{sc}$ is generated by $\alpha^\vee_{[n]}\otimes a$, $\alpha^\vee \in \Delta^\vee, a\in F^\times$. The fairness enables us to have the following
\begin{lm}
If $D$ is fair, then the map $s_\phi$ generated by 
$$\xymatrix{
s_\phi: \quad T_{Q,n}^{sc} \ar[r] & \wt{T}_{Q,n}, \quad \alpha_{[n]}^\vee\otimes a \ar@{|->}[r] & \Phi_{D,\eta}\big(\wm{h}_\alpha(a^{n_\alpha})\big), \alpha^\vee \in \Delta^\vee,
}$$
where $\Phi_{D,\eta}\big(\wm{h}_\alpha(a^{n_\alpha})\big)=\big( (\eta(\alpha_{[n]}^\vee), a)_n, \alpha_{[n]}^\vee\otimes a\big)$ for $\alpha^\vee \in \Delta^\vee$, is a well-defined homomorphism.
\end{lm}
\begin{proof}
Note that $T_{Q,n}^{sc}$ is generated by $\alpha^\vee_{[n]}\otimes a$, $\alpha\in \Delta^\vee, a\in F^\times$. Thus by considering the group law on $\wt{T}_{Q,n}^{sc}$ it suffices to check the following. First we check for $a, b\in F^\times$
\begin{align*}
(1, \alpha^\vee_{[n]}\otimes a) \cdot (1, \alpha^\vee_{[n]}\otimes b)= &\big((a, b)_n^{Q(n_\alpha \alpha^\vee)}, \alpha^\vee_{[n]}\otimes (ab)\big) \\
=& \big(1, \alpha^\vee_{[n]}\otimes (ab)\big).
\end{align*}
Second, for $\alpha^\vee$ and $\beta^\vee$ in $\Delta^\vee$, we need to check 
$$(1, \alpha^\vee_{[n]}\otimes a) \cdot (1, \beta^\vee_{[n]}\otimes a)=\big(1, (\alpha^\vee_{[n]}+ \beta_{[n]}^\vee)\otimes a \big). $$
However, direct computation gives
\begin{align*}
(1, \alpha^\vee_{[n]}\otimes a) \cdot (1, \beta^\vee_{[n]}\otimes a)=& \big( (a, a)_n^{D(\alpha^\vee_{[n]}, \beta_{[n]}^\vee)}, (\alpha^\vee_{[n]}+ \beta_{[n]}^\vee)\otimes a \big) \\
=& \big( (a, a)_n^{n_\alpha n_\beta \cdot D(\alpha^\vee, \beta^\vee)}, (\alpha^\vee_{[n]}+ \beta_{[n]}^\vee)\otimes a \big).
\end{align*}

If $n$ is odd $(a, a)_n=1$. If $n$ is even then either $2|n_\alpha$ or $2|Q(\alpha^\vee)$, and in the latter case $2|D(\alpha^\vee, \beta^\vee)$ since we have assumed $D$ fair. That is, in any case we have 
$$(a, a)_n^{n_\alpha n_\beta \cdot D(\alpha^\vee, \beta^\vee)}=1.$$
The proof is completed.
\end{proof}

By using the induction as in the proof of Proposition \ref{s-eta} and simple computation, we have:
\begin{cor}
For $D$ fair, the homomorphism $s_\phi$ has the property that
$$s_\phi\big(\alpha_{[n]}^\vee\otimes a\big)=\Phi_{D,\eta}\big(\wm{h}_\alpha(a^{n_\alpha})\big) \text{ for all } \alpha^\vee \in \Psi^\vee.$$
Also in this case $s_\phi$ takes the explicit form
\begin{equation} \label{s-phi}
s_\phi\big(y\otimes a\big)=\big( (\eta(y), a)_n, y\otimes a\big), \quad y\in Y_{Q,n}^{sc}.
\end{equation}
\end{cor}

It follows that for $D$ fair, $(C1)^{eq}$ is equivalent to 
$$\wt{\chi} \circ s_\phi=\mbf{1}.$$

Thus, in this case a character $\wt{\chi}$ of $\wt{T}_{Q,n}$ is distinguished if and only if the following two conditions hold:
\begin{align}
& (C0)_+ \quad \wt{\chi} \circ s_n =\mbf{1}, \\
& (C1)^{eq}   \quad \wt{\chi} \circ s_\phi=\mbf{1}.
\end{align}

A character satisfying $(C1)^{eq}$ can exist if and only if 
$$\text{(Obs1)} \quad \text{Ker}(s_\phi) = \text{Ker}(i_{Q,n}^{sc}),$$
where the inclusion $\text{Ker}(s_\phi) \subseteq\text{Ker}(i_{Q,n}^{sc})$ is automatic. If there exists $t\in \text{Ker}(i_{Q,n}^{sc}) \backslash \text{Ker}(s_\phi)$, then $s_\phi(t) \in \mu_n$ and $s_\phi(t)\ne 1$, in which case there does not exist $\wt{\chi}$ such that $\wt{\chi}\circ s_\phi (t)=1$. Conversely, if the equality in (Obs1) is satisfied, then $s_\phi$ gives a splitting of $\wt{T}_{Q,n}$ over the image $\text{Im}(i_{Q,n}^{sc})\subseteq T_{Q,n}$. By Pontrjagin duality, there exist $\wt{\chi}$ satisfying $(C1)^{eq}$.

We could also rephrase the condition $(C1)^{eq}$ in explicit terms. By elementary divisor theorem, $\text{Ker}(i_{Q,n}^{sc})$ is generated by pure tensors $y\otimes a \in T_{Q,n}^{sc}$ such that $y\otimes a =1 \in T_{Q,n}$. Thus in view of (\ref{s-phi}), an equivalent formation for (Obs1) is
$$\text{(Obs1)}^{eq} \quad (\eta(y), a)_n=1 \text{ for any } y\otimes a=1 \in T_{Q,n}, \ y\in Y_{Q,n}^{sc}, a\in F^\times.$$

Moreover, characters satisfying both $(C0)_+$ and $(C1)^{eq}$ can exist if and only if
$$\text{(Obs2)} \quad (\eta(y), a)_n=1 \text{ for any } y\in nY\cap Y_{Q,n}^{sc}, a\in F^\times.$$

These obstructions can not be removed automatically. However, they can be removed if the character $\xymatrix{\eta_n: Y^{sc} \ar[r]^\eta & F^\times \ar@{>>}[r] & F^\times/n}$ is extendable to $Y$.


To summarize, we have
\begin{prop} \label{dis char exis}
Suppose $D$ is fair and the conditions in (Obs1) and (Obs2) hold, in particular when $\eta_n$ is extendable to $Y$. Let $J=nY +Y_{Q,n}^{sc}$, and 
$$Z(\wt{G}^\vee)[J]=\Hom(Y_{Q,n}/J, \C^\times) \subseteq Z(\wt{G}^\vee).$$

Then 

1) The set of distinguished genuine characters of $\wt{T}_{Q,n}$ is nonempty, and is a torsor over $\Hom(\W_F, Z(\wt{G}^\vee)[J])$.

2) Each distinguished character $\wt{\chi}$ gives an admissible splitting $\rho_{\wt{\chi}}$ in $\mfr{S}^a({}^L\wt{G}, \W_F)$, with respect to which we have
$${}^L\wt{G} \simeq_{\rho_{\wt{\chi}}} \wt{G}^\vee \times \W_F.$$
\end{prop}
\begin{proof}
It suffices to prove 1). 



The absence of (Obs1) and (Obs2) implies that the two conditions $\wt{\chi}\circ s_\phi=\mbf{1}$ and $\wt{\chi} \circ s_n=\mbf{1}$ are compatible. Thus, by the Pontrjagin duality, there exists $\wt{\chi}\in \Hom_\epsilon(Z(\wt{T}), \C^\times)$ satisfying both conditions.

Finally, let $\wt{\chi}_i, i=1, 2$ be two distinguished characters which give rise to $\rho_i \in \mfr{S}(E_{\wt{G}}, F^\times)$. For $a\in F^\times$, $\rho_1/\rho_2 \in \Hom(F^\times, Z(\wt{G}))$. As before let $[P_a]_i \in \Hom(\mca{E}_{Q,n}, \C^\times)$ be given by 
$$[P_a]_i\big((1, y)\big)=\wt{\chi}_i\big( (1, y\otimes a)).$$

With $Z(\wt{G}^\vee)=\Hom(Y_{Q,n}/Y_{Q,n}^{sc}, \C^\times)$, then in fact
$$\Big(\frac{\rho_1}{\rho_2}(a)\Big)(y)=[P_a]_1\big((1, y)\big)\big/[P_a]_2\big((1, y)\big), \quad y\in Y_{Q,n}.$$
 
Since $\wt{\chi}_i$ both satisfy $(C0)_+$ and $(C1)^{eq}$, $[\rho_1/\rho_2](a)$ vanishes on $J$ and the result follows.
\end{proof}

In general, there may exist no qualified or distinguished characters, and above obstructions to their existence do exist.

\begin{eg}
Consider $\mbf{PGL}_2$ with cocharacter $Y=\Z\cdot e$ and coroot $\alpha^\vee=2e$. Let $Q$ be the unique quadratic form such that
$$Q(e)=1.$$
Let $n=2$. Then $Q(\alpha^\vee)=4$ and $n_\alpha=1$, which gives
$$Y_{Q,2}=Y, \quad Y_{Q,2}^{sc}=Y^{sc}=2Y.$$

For $(C1)^\text{eq}$ to be satisfied, one necessarily has 
$$\text{(Obs1)} \quad \big(-1, \eta(\alpha^\vee)\big)_2=1 \text{ since } \alpha^\vee \otimes (-1)=1 \in T_{Q,n}=T,$$
which is an obstruction to the existence of both qualified and distinguished characters.
This obstruction can be removed if and only if $-1 \in (F^\times)^2$, which in general may not be satisfied.

This shows that obstruction does exist.
\end{eg}

\begin{rmk}
In general, the condition $(C0)_+$ or its equivalent $(C0)_+^\text{eq}$ is certainly not the most minimum requirement and it may be an overkill. However, we do not know any simple characterization of the condition $(C0)$. Thus to replace $(C0)$ by $(C0)_+^\text{eq}$ does not seem too restrictive. In particular, if $\eta_n$ is extendable to $Y$, there is no problem.

Moreover, we may similarly define extensions $E_i^{[n]}, i=1, 2$ of $F^\times/n$ by $Z(\wt{G})$ whose pull-back via the quotient map $\xymatrix{F^\times \ar[r] & F^\times/n}$ are just $E_i$. For $E_1^{[n]}$, the definition is clear. For $E_2^{[n]}$, it is just the pull-back of $\Hom(\mca{E}_{Q,n}\big/\phi_\eta \circ \s(Y_{Q,n}^{sc}), \C^\times)$ via the Hilbert symbol $\xymatrix{h: F^\times/n \ar[r] & \Hom(F^\times, \C^\times)}$.

It is also natural to ask for the splitting of the Baer sum of $E_1^{[n]}\oplus_B E_2^{[n]}$ over $F^\times/n$:
$$\seq{Z(\wt{G}^\vee)}{E_1^{[n]}\oplus_B E_2^{[n]}}{F^\times/n}.$$
The previous argument for the splitting of $E_{\wt{G}}$ could be applied, provided that there is an additional condition besides $C(0)$ and $C(1)$:
$$ \wt{\chi}(1, y\otimes a)=1, \text{ for all } y\in nY_{Q,n}, a\in F^\times.$$
\end{rmk}

Since this condition is subsumed by $(C0)^{eq}_+$ in the definition of distinguished characters, we see that any distinguished character actually gives a splitting of $E_1^{[n]}\oplus_B E_2^{[n]}$ over $F^\times/n$. However, in general a qualified character may not give rise to a splitting of $E_1^{[n]}\oplus_B E_2^{[n]}$ over $F^\times/n$.

\subsection{The case for $\wt{G}=\wt{T}$: the local Langlands correspondence} \index{local Langlands correspondence}

In the case $\wt{G}=\wt{T}$, any splitting of ${}^L\wt{T}$ is admissible, i.e. $\mfr{S}^a({}^L\wt{T}, \W_F)=\mfr{S}({}^L\wt{T}, \W_F)$. The condition $(C1)^\text{eq}$ for qualified character $\wt{\chi}$ of $\wt{T}_{Q,n}$ is vacuous. That is, any genuine character $\wt{\chi}$ of $Z(\wt{T})$ is qualified. Also $\wt{\chi} \in \Hom_\epsilon(Z(\wt{T}), \C^\times)$ is distinguished if and only if it satisfies $(C0)_+^\text{eq}$.

Thus,

\begin{prop} \label{LLC tori}
There is a natural injective homomorphism as compositions
$$\xymatrix{
\Hom_\epsilon(Z(\wt{T}), \C^\times) \ar@{^(->}[r] & \mfr{S}(E_{\wt{T}}, F^\times) \ar[r]^-{\simeq} & \mfr{S}^a({}^L\wt{T}, \W_F),
}$$
where the first map is explicitly given by
$$\xymatrix{
\wt{\chi} \ar@{|->}[r] & \rho_{\wt{\chi}} \text{ with }  \rho_{\wt{\chi}}(a)=\big([\wt{\chi}(1, -\otimes a)], a)\oplus_B (1, a) \big).
}$$
\end{prop}

We call it the local Langlands correspondence (LLC).\\

Now back to the case of general $\wt{G}$ with covering torus $\wt{T}$. Recall we have the canonical homomorphism $\xymatrix{\msc{C}: E_{\wt{T}} \ar[r] &\wt{T}^\vee/Z(\wt{G}^\vee)}$ defined right before Corollary \ref{ad-t}. It gives an induced map $\xymatrix{\msc{C}_*: \mfr{S}(E_{\wt{T}}, F^\times) \ar[r] &\Hom(F^\times, \wt{T}^\vee/Z(\wt{G}^\vee)}$ by post composition with $\msc{C}$.

Consider the composition 
$$\xymatrix{
\Hom_\epsilon(Z(\wt{T}), \C^\times) \ar@{^(->}[r]^-{\text{LLC}} & \mfr{S}(E_{\wt{T}}, F^\times) \ar[r]^-{\msc{C}_*} &\Hom(F^\times, \wt{T}^\vee/Z(\wt{G}^\vee))
}$$
 which is given by
$$\xymatrix{
\wt{\chi} \ar@{|->}[r] & \rho_{\wt{\chi}} \ar@{|->}[r] &  \msc{C} \circ \rho_{\wt{\chi}}.
}$$

The following result is of fundamental importance to the GK formula.
\begin{cor} \label{key iden}
With notations as above, identify $\wt{T}^\vee/Z(\wt{G}^\vee)$ with $\Hom(Y_{Q,n}^{sc}, \C^\times)$. Then for any $a\in F^\times$,
$$\big(\msc{C} \circ \rho_{\wt{\chi}}(a)\big)(\alpha^\vee_{[n]})=\wt{\chi}\circ \Phi_{D,\eta}\big(\wm{h}_\alpha(a^{n_\alpha})\big) \text{ for all } \alpha\in \Psi.$$
\end{cor}
\begin{proof}
Just combine Corollary \ref{ad-t} and Proposition \ref{s-eta}.
\end{proof}
\subsection{Weyl group invariance for qualified characters} \label{Weyl inv}

The normalizer $N(\mbf{T})$ acts on $\wt{\mbf{T}}$, which gives rise to an action of $N(\mbf{T})$ on $\wt{T}$. The action does not descend to $N(\mbf{T})/\mbf{T}$ in general.

However, on the other hand $N(\mbf{T})$ preserves the center $Z(\wt{T})$ since it preserves $T^\dag$. Also $\mbf{T}$ acts trivially on $Z(\wt{T})$. Therefore we obtain a well-defined action of the Weyl group $W=N(\mbf{T})/\mbf{T}$ on $Z(\wt{T})$.

Let $\alpha\in \Delta$. For any genuine character $\wt{\chi}$ of $Z(\wt{T})$, we have an action of $\w_\alpha \in W$ on $\wt{\chi}$ given by
$${}^{\w_\alpha}\wt{\chi}(\wt{t}):= \wt{\chi}(\w_\alpha^{-1} \wt{t} \w_\alpha).$$

We may ask for whether $\wt{\chi}$ is $W$-invariant, for which it is sufficient to check for the simple reflections $\w_\alpha$ for $\alpha\in \Delta$. For this purpose, we have the following useful result.

\begin{lm}
For any root $\alpha\in \Psi$ and any $y\in Y_{Q,n}$, write $\angb{\alpha}{y}$ for the pairing between $X$ and $Y$. Then
$$n_\alpha \text{ divides } \angb{\alpha}{y}.$$
\end{lm}
\begin{proof}
Using the Weyl-invariance of $B_Q$ we have shown (cf. Lemma \ref{BQ<>}) $\angb{\alpha}{y}\cdot Q(\alpha^\vee) = B_Q(y, \alpha^\vee)$ which is divisible by $n$ with $y\in Y_{Q,n}$. It follows $n_\alpha| \angb{\alpha}{y}$.
\end{proof}

\begin{prop}
Let $\wt{\chi}$ be a qualified genuine character of $Z(\wt{T})$, then it is $W$-invariant. That is, for all $\w_\alpha \in W$ with $\alpha\in \Delta$,
$$^{\w_\alpha}\wt{\chi}=\wt{\chi}.$$
\end{prop}
\begin{proof}
We use the notation $\mbf{h}_\alpha(a)$ for $\alpha^\vee(a), a\in F^\times$. Let $t$ be the image of $\wt{t}$ in $T^\dag$. It suffices to show $^{\w_\alpha}\wt{\chi}(\wt{t})=\wt{\chi}(\wt{t})$ on the generators of $Z(\wt{T})$. Hence, we may assume $t=y\otimes b$ with $y\in Y_{Q,n}$ and $b\in F^\times$. From \cite[(11.9.1)]{BD01} we have
\begin{align*}
\w_\alpha^{-1} \wt{t} \w_\alpha= & \wt{t} \cdot \Phi_{D,\eta} \big( \wm{h}_\alpha(\alpha(t)^{-1}) \big) \\
 =& \wt{t} \cdot \Phi_{D,\eta} \big( \wm{h}_\alpha( b^{-\angb{\alpha}{y}} ) \big).
\end{align*}

We need to show $\wt{\chi} \circ \Phi_{D,\eta} \big( \wm{h}_\alpha( b^{-\angb{\alpha}{y}} ) \big)=1$. However, from \cite[11.1.5]{BD01}, we have
\begin{align*}
\Phi_{D,\eta} \big( \wm{h}_\alpha( b^{-\angb{\alpha}{y}} ) \big) &=\Phi_{D,\eta} \big( \wm{h}_\alpha( b^{n_\alpha}) \big)^{-\angb{\alpha}{y}/n_\alpha} \cdot (b^{n_\alpha}, b^{n_\alpha})_n^{\varepsilon(-\angb{\alpha}{y}/n_\alpha)Q(\alpha^\vee)} \\
&=\Phi_{D,\eta} \big( \wm{h}_\alpha( b^{n_\alpha}) \big)^{-\angb{\alpha}{y}/n_\alpha} \\
\end{align*}

Since $\wt{\chi}$ is a qualified character, we have $\wt{\chi} \circ \Phi_{D,\eta} \big( \wm{h}_\alpha( b^{n_\alpha} ) \big)=1$ by the condition $(C1)$ (which is equivalent to $(C1)^{eq}$). It follows $\wt{\chi} \circ \Phi_{D,\eta} \big( \wm{h}_\alpha( b^{-\angb{\alpha}{y}} ) \big)=1$. Therefore,  $^{\w_\alpha}\wt{\chi}=\wt{\chi}$ and this concludes the proof.
\end{proof}

\section{Construction of distinguished characters for fair $(D, \mbf{1})$} \label{cons dis char}
Consider fair $(D, \mbf{1})$. By Proposition \ref{dis char exis}, there exist distinguished characters of $\wt{T}_{Q,n}$. In this section, we will give an explicit construction.

Recall $J=nY+Y_{Q,n}^{sc}$. Using the fairness of $D$, it is easy to see that the map given by its image of the generators of $J\otimes F^\times$ as
\begin{equation} \label{sJ}
\xymatrix{
s_J: \quad J\otimes F^\times \ar[r] & \wt{T}_{Q,n}, \quad y\otimes a \ar@{|->}[r] & \big(1, y\otimes a\big) \in \wt{T}_{Q,n}
}
\end{equation}
is a well-defined homomorphism.

Let $\wt{\chi}$ be a genuine character of $\wt{T}_{Q,n}$, then the condition $\wt{\chi} \circ s_J=\mbf{1}$ implies both $(C0)_+$ and $(C1)^\text{eq}$. Now we give an explicit construction of genuine characters of $\wt{T}_{Q,n}$ which satisfy $\wt{\chi} \circ s_J=\mbf{1}$ using the Weil index and Hilbert symbol. These characters will be distinguished characters.

\index{Weil index}
Recall for any additive character $\psi$ of $F$, the Weil index $\xymatrix{\gamma_\psi: F^\times \ar[r] & \mu_4\subseteq \C^\times}$ is a map satisfying the following properties:
\begin{align}
\gamma_\psi(a) \cdot \gamma_\psi(b) &=\gamma_\psi(ab) \cdot (a, b)_2 \\
\gamma_{\psi_{c^2}} &=\gamma_\psi,
\end{align}
where for any $c\in F^\times$ we define $\psi_c(x):=\psi(cx), x\in F$. The Weil index plays an important role in describing the representations of the classical double cover of $\mbf{Sp}_{2r}(F)$, in particular the construction of genuine characters of its abelian covering torus. Our construction below will recover this.

\subsubsection{Reduction to dimension one tori}

First of all, by elementary divisor theorem, let $\set{e_i}$ for $ 1\le i\le r$ be a basis of $Y_{Q,n}$ such that $\set{k_i e_i}$ is a basis of the
lattice $J=nY +Y_{Q,n}^{sc}$.

Let $\mbf{T}_J$ be the torus defined over $F$ associated with $J$, and let $T_J:=\mbf{T}_{J}(F)$. We may write $T_J=J\otimes F^\times \simeq \prod_{i} F^\times$ and $T_{Q,n}=Y_{Q,n}\otimes F^\times \simeq \prod_i F^\times$. Thus we obtain the map
\begin{diagram}
T_J &\rTo & T_{Q,n} \\
\prod_i (k_i e_i)\otimes a  &\rMapsto & \prod_i e_i \otimes a^{k_i}.
\end{diagram}

On the $i$-th component of the product, the map is the $k_i$-power. Write $T_{Q,n, i}$ and $T_{J, i}$ for the one-dimensional tori for the lattices generated by $e_i$ and $k_i e_i$ respectively. Since $\wt{T}_{Q,n}$ is abelian, we have 
$$\wt{T}_{Q,n}=\wt{T}_{Q,n, 1} \times ... \times \wt{T}_{Q,n, r}\big/ Z,$$
where $\wt{T}_{Q,n, i}$ is the preimage of $T_{Q,n,i}$ in $\wt{T}_{Q,n}$ and
$$Z=\set{(\zeta_i)\in \prod_i \mu_n: \prod_i \zeta_i=1 }.$$


To construct a genuine character on $\wt{T}_{Q,n}$ such that $\wt{\chi} \circ s_J=\mbf{1}$, it suffices to do so for each $\wt{T}_{Q, n, i}$ by requiring that it is trivial on the image of the splitting
\begin{diagram}
s_{J,i}: & T_{J, i} &\rTo & \wt{T}_{Q,n, i} \\
& (k_ie_i) \otimes a  &\rMapsto & \big(1, e_i \otimes a^{k_i}\big).
\end{diagram}

Note the group law on $\wt{T}_{Q, n, i}$ is given by
$$(1, y_i\otimes a) \cdot (1, y_i \otimes b) =((a, b)_n^{Q(y_i)}, y_i\otimes (ab)).$$

\subsubsection{The definition of $\wt{\chi}$}

We now attempt to define a character $\wt{\chi}_i$ of $\wt{T}_{Q, n, i}$ by
$$\wt{\chi}_i(1, e_i\otimes a)=\gamma_\psi(a)^{f_i},$$
where $f_i$ is to be determined. There are several requirements on $f_i$:

1). First, the relation
$$\wt{\chi}_i((1, e_i\otimes a)) \cdot \wt{\chi}_i((1, e_i\otimes b)) =(a, b)_n^{Q(e_i)} \cdot \wt{\chi}_i((1, e_i\otimes (ab)))$$
gives
$$ (\text{R1}) \quad f_i = \frac{2Q(e_i)}{n} \mod 2.$$

2). Write $A_i=2Q(e_i)/n \in \Z$. Simple computation gives 
$$\wt{\chi}_i((1, e_i \otimes a^{k_i})) =\gamma_\psi(a)^{k_i f_i + k_i(k_i-1)A_i}.$$

We require this to be trivial, which gives
$$ (\text{R2}) \quad k_i f_i + k_i(k_i-1)A_i =0 \mod 4.$$

\begin{lm} \label{constr dis char}
There exists $f_i$ such that both (R1) and (R2) hold. In fact, the assignment $f_i=\pm (k_i-1)A_i$ works.
\end{lm}
\begin{proof}
Clearly, if $f_i=-(k_i-1)A_i$, then R2 is satisfied. If $f_i=(k_i-1)A_i$, then we have
$$k_i f_i + k_i(k_i-1)A_i=2k_i(k_i-1)A_i=0 \mod 4,$$
which follows from the fact $k_i(k_i-1) =0 \mod 2.$

Now it suffices to show
$$\pm (k_i -1)A_i = A_i \mod 2.$$
Equivalently, $$k_iA_i =0 \mod 2.$$
If $k_i$ is even we are done. Now assume $k_i$ is odd.

We have $k_i e_i \in J=nY +Y_{Q,n}^{sc}$, and thus clearly $n| Q(k_i e_i)$. Since $k_i$ is odd,
$$k_iA_i =\frac{2k_i^2 Q(e_i)}{n} =\frac{2Q(k_ie_i)}{n} =0 \mod 2.$$

This completes the proof.
\end{proof}

Let $y=\sum_i n_i e_i \in Y_{Q,n}$ and $a\in F^\times$ be arbitrary. Then we define

\begin{align}
\wt{\chi}_\psi\big((1, y\otimes a)\big) & =\prod_i \wt{\chi}_i(1, e_i\otimes a^{n_i}) \cdot (a, a)_n^{\sum_{j<j'} n_j n_{j'} D(e_j, e_{j'})} \\
& = \prod_i \gamma_\psi(a^{n_i})^{f_i} \cdot (a, a)_n^{\sum_{j<j'} n_j n_{j'} D(e_j, e_{j'})},
\end{align}

where $f_i=(k_i -1)A_i$. Then $\wt{\chi}_\psi$ is a distinguished character of $\wt{T}_{Q,n}$.

\section{Explicit distinguished characters and compatibility}
In this section, we consider a simply-connected group $\mbf{G}$ of arbitrary type. By Theorem \ref{torsor/sc}, there is up to unique isomorphism a $\mbf{K}_2$-torsor $\wm{G}$ associated to a Weyl-invariant quadratic form on $Y^{sc}=Y$. Consider $\wt{G}$ incarnated by $(D, \eta)$, then there is no loss of generality in assuming $D$ fair and $\eta=\mbf{1}$. We will do so in the following.

We apply the construction in previous section and explicate the distinguished characters case by case. 

For simplicity we assume $n=2$ except for the case of the exceptional $G_2$ where the computation is very simple for general $n$. We also assume that $Q$ is the unique Weyl-invariant quadratic form which takes value 1 on the short coroots of $\mbf{G}$. The general case follows from similar computations.

In the simply-laced case and the case $C_r$ for the Dynkin diagram, we shall see that these explicit distinguished characters are compatible with those given by Savin (cf. \cite{Sav04}) and those for the classical double cover of $\mbf{Sp}_{2r}(F)$ respectively (cf. \cite{Kud96} \cite{Rao93}).

Note that since we have assume $n=2$, we will write $Y_{Q,2}$ and $Y_{Q,2}^{sc}$ for the lattices $Y_{Q,n}$ and $Y_{Q,n}^{sc}$ which are of interest. We also have $J=2Y+Y_{Q,2}^{sc}=Y_{Q,2}^{sc}$ since $Y=Y^{sc}$.

\subsection{The simply-laced case $A_r, D_r, E_6, E_7, E_8$ and compatibility}
Now let $\mbf{G}$ be a simply-laced simply-connected group of type $A_r$ for $ r\ge 1$, $D_r$ for $r\ge 3$, and $E_6, E_7, E_8$. Let $\Delta=\set{\alpha_1, ..., \alpha_r}$ be a fixed set of simple roots of $\mbf{G}$. Let $\wt{\mbf{G}}$ be the extension of $\mbf{G}$ determined by the quadratic form $Q$ with $Q(\alpha_i^\vee)=1$ for all coroots $\alpha_i^\vee$. 

We obtain the two-fold cover $\wt{G}$ of $G$. We show that our distinguished genuine character in previous section agrees with the one given by Savin.

Clearly we have $n_\alpha=2$ for all  $\alpha \in \Psi$ in this case. As mentioned, we also have
$$J=2Y^{sc} + Y_{Q,2}^{sc}=2Y^{sc}=Y_{Q,2}^{sc}.$$

Let $\alpha_i^\vee \in \Delta^\vee$ for $i=1, ..., r$ be the simple coroots of $\mbf{G}$. It is easy to compute the bilinear form $B_Q$ associated with $Q$: 
\begin{equation} \label{sl B}
    B_Q(\alpha_i^\vee, \alpha_j^\vee) =
    \begin{cases}
      -1 & \text{if } \alpha_i \text{ and } \alpha_j \text{ connected in the Dynkin diagram,} \\
      0 & \text{otherwise}.
    \end{cases}
\end{equation}

In order to show compatibility with Savin, we may further assume that $\wt{G}$ is incarnated by the following fair bisector $D$ associated with $B_Q$ as given in \cite{Sav04},
\begin{equation} \label{S-fair D}
    D(\alpha_i^\vee, \alpha_j^\vee) =
    \begin{cases}
      0 & \text{ if } i< j, \\
Q(\alpha^\vee_i) &\text{ if } i=j, \\
      B_Q(\alpha_i^\vee, \alpha_j^\vee) & \text{ if } i>j.
    \end{cases}
\end{equation}

The following lemma is in \cite{Sav04} and reproduced here for convenience. The stated result can also be checked by straightforward computation.
\begin{lm} \label{geom lm}
Let $\Omega$ be a subset of the vertices in the Dynkin diagram of $\mbf{G}$ satisfying:
\begin{enumerate}
\item[(i)] No two vertices in $\Omega$ are adjacent,

\item[(ii)] Every vertex not in $\Omega$ is adjacent to an even number of vertices in $\Omega$. 
\end{enumerate}

Then the map given by $\xymatrix{\Omega \ar@{|->}[r] & e_\Omega}$ with $e_\Omega:=\sum_{\alpha_i \in \Omega} \alpha_i^\vee$ gives a well-defined correspondence between such sets $\Omega$ and the cosets of $Y_{Q,2}/J$. In particular, the empty set $\varnothing$ corresponds to the trivial coset $J$.
\end{lm}
By properties of $B_Q$ and (i) of $\Omega$ above, it follows that
$$Q(e_\Omega)=|\Omega|.$$


We now give a brief case by case discussion.

\subsubsection{The $A_r$ case.} 
There are two situations according to the parity of $r$.

\cu{Case 1:} $r$ is even.  As an illustration, we first do the straightforward computation. Let $\sum_i k_i \alpha_i^\vee \in Y_{Q,2}$ for proper $k_i\in \Z$. Then $B_Q(\sum_i k_i \alpha_i^\vee, \alpha^\vee_j) \in 2\Z$ for all $1\le j\le r$ by the definition of $Y_{Q,2}$. In view of (\ref{sl B}), it is equivalent to 
\begin{equation} \label{comp}
\begin{cases}
2k_1 + (-1) k_2 &\in 2\Z, \\
(-1)k_1 + 2 k_2 + (-1) k_3 &\in 2\Z, \\
(-1)k_2 + 2 k_3 + (-1) k_4 &\in 2\Z, \\
\vdots & \ \\
(-1)k_{r-2} + 2 k_{r-1} + (-1) k_r &\in 2\Z, \\
(-1)k_{r-1} + 2k_r &\in 2\Z.
\end{cases}
\end{equation}
It follows that $k_2$ is even and so are the successive $k_4, ..., k_r$ (we have assumed $r$ to be even). Similarly, $k_{r-1}$ is even, and therefore all $k_{r-3}, ..., k_1$ are also even. That is $Y_{Q,2}=J$.

Note that we could simply apply the lemma to get $Y_{Q,2}=J$, which corresponds to the fact there is only the empty set for $\Omega$ satisfying properties $(i)$ and $(ii)$. There is nothing to check in this case, and the character $\wt{\chi}$ with $\wt{\chi}\big(s_J(t)\big)=1, t\in T_J=J\otimes F^\times$ will be a distinguished character.

\cu{Case 2:} $r=2k-1$ is odd. The consideration as in (\ref{comp}) works. However, for convenience we will apply the lemma to get $[Y_{Q,2}, J]=2$ with a basis of $Y_{Q,2}$ given by
$$\set{\alpha_{r,[2]}^\vee, \alpha_{r-1,[2]}^\vee, ..., \alpha_{2,[2]}^\vee, e_\Omega=\sum_{m=1}^k \alpha_{2m-1}^\vee}.$$
The nontrivial coset corresponds to the set $\Omega$ indicated by alternating bold circles below
$$
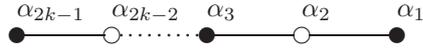
\begin{picture}(4.7,0.2)(0,0)
\put(1,0){\circle*{0.08}}
\put(1.5,0){\circle{0.08}}
\put(2,0){\circle*{0.08}}
\put(2.5,0){\circle{0.08}}
\put(3,0){\circle*{0.08}}
\put(1.04,0){\line(1,0){0.42}}
\multiput(1.55,0)(0.05,0){9}{\circle*{0.02}}
\put(2.04,0){\line(1,0){0.42}}
\put(2.54,0){\line(1,0){0.42}}
\put(1,0.1){\footnotesize $\alpha_{2k-1}$}
\put(1.5,0.1){\footnotesize $\alpha_{2k-2}$}
\put(2,0.1){\footnotesize $\alpha_3$}
\put(2.5,0.1){\footnotesize $\alpha_2$}
\put(3,0.1){\footnotesize $\alpha_1$}
\end{picture}
$$
A basis for $J=2Y^{sc}$ is given by
$$\set{\alpha_{r,[2]}^\vee, \alpha_{r-1,[2]}^\vee, ..., \alpha_{2,[2]}^\vee, 2 e_\Omega}.$$

Now the construction of the distinguished $\wt{\chi}_\psi$ in previous section is given by
\begin{equation}
\begin{cases}
\wt{\chi}_\psi(1, \alpha_{i,[2]}^\vee \otimes a) =1 \\
\wt{\chi}_\psi(1, e_\Omega\otimes a)=\gamma_\psi(a)^{(2-1)Q(e_\Omega)}=\gamma_\psi(a)^{|\Omega|}.
\end{cases}
\end{equation}

This agrees with the formula in Savin when we substitute $a=\varpi$ in $\gamma_\psi(a)$ for $\psi$ of conductor $\msc{O}_F$. See \cite[pg. 118]{Sav04}.

\subsubsection{The $D_r$ case.}
We also have two cases.

\cu{Case 1:} $r=2k-1$ is odd with $k\ge 2$. Then $[Y_{Q,2}, J]=2$ with the nontrivial $\Omega=\set{\alpha_1, \alpha_2}$:
$$
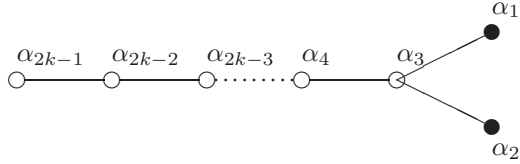
\begin{picture}(4.7,0.4)(0,0)
\put(1,0){\circle{0.08}}
\put(1.5,0){\circle{0.08}}
\put(2,0){\circle{0.08}}
\put(2.5,0){\circle{0.08}}
\put(3,0){\circle{0.08}}
\put(3.5, 0.25){\circle*{0.08}}
\put(3.5, -0.25){\circle*{0.08}}
\put(1.04,0){\line(1,0){0.42}}
\put(1.54,0){\line(1,0){0.42}}
\multiput(2.05,0)(0.05,0){9}{\circle*{0.02}}
\put(2.54,0){\line(1,0){0.42}}
\put(3.00,0){\line(2,1){0.46}}
\put(3.00,0){\line(2,-1){0.46}}
\put(1,0.1){\footnotesize $\alpha_{2k-1}$}
\put(1.5,0.1){\footnotesize $\alpha_{2k-2}$}
\put(2,0.1){\footnotesize $\alpha_{2k-3}$}
\put(2.5,0.1){\footnotesize $\alpha_4$}
\put(3,0.1){\footnotesize $\alpha_3$}
\put(3.5,0.35){\footnotesize $\alpha_1$}
\put(3.5,-0.4){\footnotesize $\alpha_2$}
\end{picture}
$$
\vspace{0.5cm}

Consider the basis $\set{\alpha_{i,[2]}^\vee: 2\le i \le r} \cup \set{e_\Omega}$ of $Y_{Q,2}$, then the construction of distinguished $\wt{\chi}_\psi$ in previous section is determined by
\begin{align*}
\wt{\chi}_\psi(1, e_\Omega\otimes a)=\gamma_\psi(a)^{(2-1)Q(e_\Omega)}=\gamma_\psi(a)^{|\Omega|}.
\end{align*}

\cu{Case 2:} $r=2k$ is even. Then $[Y_{Q,2}: J]=4$. There are three nontrivial sets $\Omega_i$ for $i=1, 2, 3$ as indicated by the bold circles below.
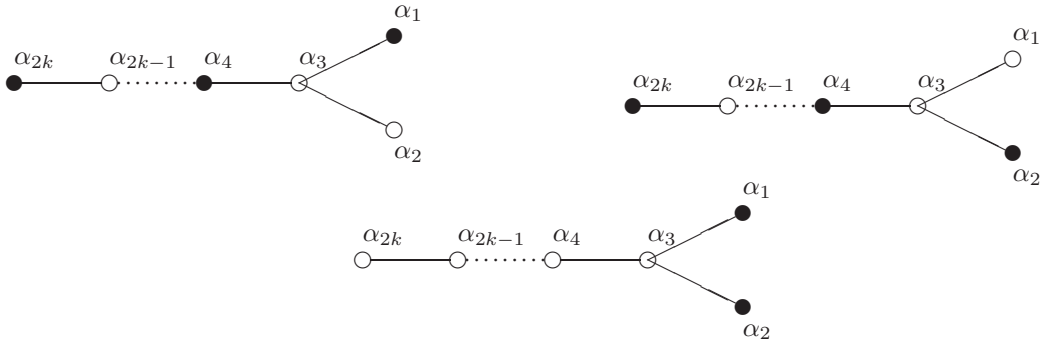
\begin{multicols}{2}
$$
\begin{picture}(4.7,0.1)(1,0)
\put(1.5,0){\circle*{0.08}}
\put(2,0){\circle{0.08}}
\put(2.5,0){\circle*{0.08}}
\put(3,0){\circle{0.08}}
\put(3.5, 0.25){\circle*{0.08}}
\put(3.5, -0.25){\circle{0.08}}
%
\put(1.54,0){\line(1,0){0.42}}
\multiput(2.05,0)(0.05,0){9}{\circle*{0.02}}
\put(2.54,0){\line(1,0){0.42}}
\put(3.00,0){\line(2,1){0.46}}
\put(3.00,0){\line(2,-1){0.46}}
%
\put(1.5,0.1){\footnotesize $\alpha_{2k}$}
\put(2,0.1){\footnotesize $\alpha_{2k-1}$}
\put(2.5,0.1){\footnotesize $\alpha_4$}
\put(3,0.1){\footnotesize $\alpha_3$}
\put(3.5,0.35){\footnotesize $\alpha_1$}
\put(3.5,-0.4){\footnotesize $\alpha_2$}
\end{picture}
$$
\goodbreak
$$
\begin{picture}(4.7,0.1)(1,0)
\put(1.5,0){\circle*{0.08}}
\put(2,0){\circle{0.08}}
\put(2.5,0){\circle*{0.08}}
\put(3,0){\circle{0.08}}
\put(3.5, 0.25){\circle{0.08}}
\put(3.5, -0.25){\circle*{0.08}}
%
\put(1.54,0){\line(1,0){0.42}}
\multiput(2.05,0)(0.05,0){9}{\circle*{0.02}}
\put(2.54,0){\line(1,0){0.42}}
\put(3.00,0){\line(2,1){0.46}}
\put(3.00,0){\line(2,-1){0.46}}
%
\put(1.5,0.1){\footnotesize $\alpha_{2k}$}
\put(2,0.1){\footnotesize $\alpha_{2k-1}$}
\put(2.5,0.1){\footnotesize $\alpha_4$}
\put(3,0.1){\footnotesize $\alpha_3$}
\put(3.5,0.35){\footnotesize $\alpha_1$}
\put(3.5,-0.4){\footnotesize $\alpha_2$}
\end{picture}
$$
\end{multicols}

$$
\begin{picture}(4.7,0.4)(0,0)
\put(1.5,0){\circle{0.08}}
\put(2,0){\circle{0.08}}
\put(2.5,0){\circle{0.08}}
\put(3,0){\circle{0.08}}
\put(3.5, 0.25){\circle*{0.08}}
\put(3.5, -0.25){\circle*{0.08}}
%
\put(1.54,0){\line(1,0){0.42}}
\multiput(2.05,0)(0.05,0){9}{\circle*{0.02}}
\put(2.54,0){\line(1,0){0.42}}
\put(3.00,0){\line(2,1){0.46}}
\put(3.00,0){\line(2,-1){0.46}}
%
\put(1.5,0.1){\footnotesize $\alpha_{2k}$}
\put(2,0.1){\footnotesize $\alpha_{2k-1}$}
\put(2.5,0.1){\footnotesize $\alpha_4$}
\put(3,0.1){\footnotesize $\alpha_3$}
\put(3.5,0.35){\footnotesize $\alpha_1$}
\put(3.5,-0.4){\footnotesize $\alpha_2$}
\end{picture}
$$
\vspace{0.5cm}

That is, $\Omega_1=\set{\alpha_1}\cup \set{\alpha_{2m}: 2\le m\le k}$, $\Omega_2=\set{\alpha_2}\cup \set{\alpha_{2m}: 2\le m\le k}$ and $\Omega_3=\set{\alpha_1, \alpha_2}$. Note $|\Omega_1|=|\Omega_2|=k$ and $|\Omega_2|=2$.

A basis of $Y_{Q,2}$ is given by
$$\set{\alpha_{i,[2]}^\vee: 3\le i\le 2k-1}\cup \set{e_{\Omega_1}, e_{\Omega_2}, e_{\Omega_3}}.$$
However, the construction of distinguished characters in previous section utilized the elementary divisor theorem. Thus we have to provide bases for $Y_{Q,2}$ and $J$ aligned in a proper way. To achieve this, consider the alternative basis of $Y_{Q,2}$ given by
$$\set{\alpha_{i,[2]}^\vee: 3\le i\le 2k-1}\cup \set{e_{\Omega_1}+e_{\Omega_2} + e_{\Omega_3}, e_{\Omega_2} +e_{\Omega_3}, e_{\Omega_3}}.$$
Then it is easy to check that the set
$$\set{\alpha_{i,[2]}^\vee: 3\le i\le 2k-1}\cup \set{e_{\Omega_1}+e_{\Omega_2} + e_{\Omega_3}, 2(e_{\Omega_2} +e_{\Omega_3}), 2e_{\Omega_3}}$$
is a basis for $J$.
Note 
$$Q(e_{\Omega_2} +e_{\Omega_3})=|\Omega_2|+Q(2\alpha^\vee)=|\Omega_2|+4.$$

Thus a distinguished character could be determined by
\begin{equation*}
\begin{cases}
\wt{\chi}_\psi(1, (e_{\Omega_2} +e_{\Omega_3})\otimes a)=\gamma_\psi(a)^{Q(e_{\Omega_2} +e_{\Omega_3})}= \gamma_\psi(a)^{|\Omega_2|}, \\
\wt{\chi}_\psi(1, e_{\Omega_3} \otimes a)=\gamma_\psi(a)^{|\Omega_3|}.
\end{cases}
\end{equation*}

However, since we have assumed that $D$ takes the special form given by (\ref{S-fair D}), we have
\begin{align*}
& D(e_{\Omega_1}, e_{\Omega_2} + e_{\Omega_3}) = |\Omega_1|\\
& D(e_{\Omega_2}, e_{\Omega_3}) =Q(\alpha^\vee) =1.
\end{align*}

Thus
\begin{align*}
\wt{\chi}_\psi(1, e_{\Omega_1}\otimes a) \cdot \wt{\chi}_\psi(1, (e_{\Omega_2} +e_{\Omega_3})\otimes a) &= (a, a)_2^{|\Omega_1|} \cdot
\wt{\chi}_\psi(1, (e_{\Omega_1}+ e_{\Omega_2} +e_{\Omega_3})\otimes a) \\
\wt{\chi}_\psi(1, e_{\Omega_2}\otimes a) \cdot \wt{\chi}_\psi(1, e_{\Omega_3}\otimes a) &= (a, a)_2 \cdot
\wt{\chi}_\psi(1, (e_{\Omega_2} +e_{\Omega_3})\otimes a) 
\end{align*}

Recall $\gamma_\psi(a)^2=(a, a)_2$. This combined with the above results gives
\begin{equation*}
\begin{cases}
\wt{\chi}_\psi(1, e_{\Omega_1}\otimes a)=\gamma_\psi(a)^{|\Omega_1|}, \\
\wt{\chi}_\psi(1, e_{\Omega_2}\otimes a)=\gamma_\psi(a)^{|\Omega_2|}, \\
\wt{\chi}_\psi(1, e_{\Omega_3}\otimes a)=\gamma_\psi(a)^{|\Omega_3|}.
\end{cases}
\end{equation*}

It agrees with the character given by Savin.

\subsubsection{The $E_6, E_7, E_8$ case.} For $E_6$ and $E_8$, $Y_{Q,2}=J$ and so the situation is trivial. Consider $E_7$, then $[Y_{Q,n}: J]=2$. The nontrivial $\Omega$ is given by $\Omega=\set{\alpha_4, \alpha_6, \alpha_7}$.
$$
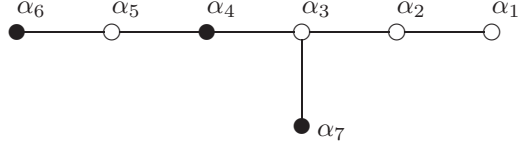
\begin{picture}(4.7,0.2)(0,0)
\put(1,0){\circle*{0.08}}
\put(1.5,0){\circle{0.08}}
\put(2,0){\circle*{0.08}}
\put(2.5,0){\circle{0.08}}
\put(3,0){\circle{0.08}}
\put(3.5,0){\circle{0.08}}
\put(2.5,-0.5){\circle*{0.08}}
\put(1.04,0){\line(1,0){0.42}}
\put(1.54,0){\line(1,0){0.42}}
\put(2.04,0){\line(1,0){0.42}}
\put(2.5,-0.04){\line(0,-1){0.42}}
\put(2.54,0){\line(1,0){0.42}}
\put(3.04,0){\line(1,0){0.42}}
\put(1,0.1){\footnotesize $\alpha_6$}
\put(1.5,0.1){\footnotesize $\alpha_5$}
\put(2,0.1){\footnotesize $\alpha_4$}
\put(2.5,0.1){\footnotesize $\alpha_3$}
\put(3,0.1){\footnotesize $\alpha_2$}
\put(3.5,0.1){\footnotesize $\alpha_1$}
\put(2.58,-0.55){\footnotesize $\alpha_7$}
\end{picture}
$$
\vspace{0.7cm}

The set $\set{\alpha_{i,[2]}^\vee: 1\le i\le 6} \cup \set{e_\Omega}$ is a basis of $Y_{Q,2}$, while $\set{\alpha_{i,[2]}^\vee: 1\le i\le 6} \cup \set{2 e_\Omega}$ a basis for $J$.

Our distinguished character is determined by
$$\wt{\chi}_\psi(1, e_\Omega \otimes a)=\gamma_\psi(a)^{|\Omega|}.$$
This agrees with Savin also.

\subsection{The case $C_r$ and compatibility with the classical metaplectic double cover $\wm{Sp}_{2r}(F)$}

Let $\mbf{Sp}_{2r}$ be the simply-connected simple group with Dynkin diagram:

$$
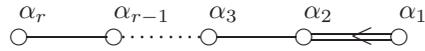
\begin{picture}(4.7,0.2)(0,0)
\put(1,0){\circle{0.08}}
\put(1.5,0){\circle{0.08}}
\put(2,0){\circle{0.08}}
\put(2.5,0){\circle{0.08}}
\put(3,0){\circle{0.08}}
\put(1.04,0){\line(1,0){0.42}}
\multiput(1.55,0)(0.05,0){9}{\circle*{0.02}}
\put(2.04,0){\line(1,0){0.42}}
\put(2.54,0.015){\line(1,0){0.42}}
\put(2.54,-0.015){\line(1,0){0.42}}
\put(2.74,-0.04){$<$}
\put(1,0.1){\footnotesize $\alpha_r$}
\put(1.5,0.1){\footnotesize $\alpha_{r-1}$}
\put(2,0.1){\footnotesize $\alpha_3$}
\put(2.5,0.1){\footnotesize $\alpha_2$}
\put(3,0.1){\footnotesize $\alpha_1$}
\end{picture}
$$
\

Let $\set{\alpha_1^\vee, \alpha_2^\vee, ..., \alpha_r^\vee}$ be the set of simple coroots with $\alpha_1^\vee$ the short one. Let $n=2$. As mentioned, $\wm{Sp}_{2r}(F)=\Hs_{\mbf{Sp}_{2r}}(\wt{\mbf{Sp}_{2r}})$ is the two-fold cover of $\mbf{Sp}_{2r}(F)$. Here $\wt{\mbf{Sp}}_{2r}$ is determined by the unique Weyl-invariant quadratic form $Q$ on $Y$ with $Q(\alpha_1^\vee)=1$.

It follows $n_{\alpha_1}=2$. Also $Q(\alpha_i^\vee)=2$ and $n_{\alpha_i}=1$ for $2\le i\le  r$. Moreover, 
\begin{align*}
Y_{Q,2}=Y^{sc}=\langle \alpha_1^\vee,  \alpha_2^\vee, ..., \alpha_r^\vee \rangle_\Z, \\
Y_{Q,2}^{sc} =\langle 2\alpha_1^\vee,  \alpha_2^\vee, ..., \alpha_r^\vee \rangle_\Z.
\end{align*}

Since $J=Y_{Q,2}^{sc}$, by the construction of distinguished character $\wt{\chi}_\psi$, it is determined by
\begin{equation}
\begin{cases}
\wt{\chi}_\psi(1, \alpha^\vee_i \otimes a) =1, \text{ if } i=2, 3, ..., r; \\
\wt{\chi}_\psi(1, \alpha^\vee_1 \otimes a) =\gamma_\psi(a)^{(2-1)Q(\alpha^\vee_1)}=\gamma_\psi(a).
\end{cases}
\end{equation}

This uniquely determined a genuine character of $\wt{T}$ which is abelian. It can be checked  that this agrees with the classical one (cf. \cite{Rao93} or \cite{Kud96} for example).

\subsection{The $B_r$, $F_4$ and $G_2$ case} 
For completeness, we also give the explicit form of the distinguished character constructed in previous section for the double cover  $\wt{G}$ of the simply connected group $G$ of type $B_r$, $F_4$ and $G_2$. Recall that when $n=2$ we have $J=Y_{Q,2}^{sc}$.

\subsubsection{The $B_r$ case}

Consider the Dynkin diagram of $B_r$:
$$
\begin{picture}(4.7,0.2)(0,0)
\put(1,0){\circle{0.08}}
\put(1.5,0){\circle{0.08}}
\put(2,0){\circle{0.08}}
\put(2.5,0){\circle{0.08}}
\put(3,0){\circle{0.08}}
\put(1.04,0){\line(1,0){0.42}}
\multiput(1.55,0)(0.05,0){9}{\circle*{0.02}}
\put(2.04,0){\line(1,0){0.42}}
\put(2.54,0.015){\line(1,0){0.42}}
\put(2.54,-0.015){\line(1,0){0.42}}
\put(2.74,-0.04){$>$}
\put(1,0.1){\footnotesize $\alpha_1$}
\put(1.5,0.1){\footnotesize $\alpha_2$}
\put(2,0.1){\footnotesize $\alpha_{r-2}$}
\put(2.5,0.1){\footnotesize $\alpha_{r-1}$}
\put(3,0.1){\footnotesize $\alpha_r$}
\end{picture}
$$

Let $Q$ be the unique Weyl-invariant quadratic form with $Q(\alpha^\vee_i)=1$ for $1\le i\le r-1$. It gives $Q(\alpha^\vee_r)=2$. We have also assumed that the double cover $\wt{G}$ is incarnated by a fair bisector $D$. The discussion now will be split into two cases according to the parity of $r$. 

\cu{Case 1:} $r$ is odd. Direct computation gives $Y_{Q,2}^{sc}=Y_{Q,n}$ and therefore this case is trivial.

\cu{Case 2:} $r$ is even. It is not difficult to compute the index $[Y_{Q,2}: Y_{Q,2}^{sc}]=2$. In fact, a basis of $Y_{Q,2}$ is given by
$$\set{\alpha^\vee_1+\alpha^\vee_3 +... +\alpha^\vee_{r-1}} \cup \set{2\alpha^\vee_i: 2\le i\le r-1} \cup \set{\alpha_r^\vee}.$$
This gives a basis of $J=Y_{Q,2}^{sc}$:
$$\set{2(\alpha^\vee_1+\alpha^\vee_3 +... +\alpha^\vee_{r-1})} \cup \set{2\alpha^\vee_i: 2\le i\le r-1} \cup \set{\alpha_r^\vee}.$$

We have 
$$Q(\alpha^\vee_1+\alpha^\vee_3 +... +\alpha^\vee_{r-1})=r/2.$$
By the construction of distinguished character $\wt{\chi}_\psi$, it is determined by
\begin{equation}
\begin{cases}
 \wt{\chi}_\psi\big((1, (\alpha^\vee_1+\alpha^\vee_3 +... +\alpha^\vee_{r-1})\otimes a)\big) =\gamma_\psi(a)^{r/2}; \\
 \wt{\chi}_\psi\big((1, (2\alpha_i^\vee)\otimes a)\big) =1, \text{ for } 2\le i\le r-1; \\
 \wt{\chi}_\psi\big((1, \alpha_r^\vee\otimes a)\big) =1.
\end{cases}
\end{equation}

\subsubsection{The $F_4$ case}
Consider the Dynkim diagram of $F_4$:

$$
\begin{picture}(4.7,0.2)(0,0)
\put(2,0){\circle{0.08}}
\put(2.5,0){\circle{0.08}}
\put(3,0){\circle{0.08}}
\put(3.5,0){\circle{0.08}}
%
\put(2.04,0){\line(1,0){0.42}}
\put(2.54,0.015){\line(1,0){0.42}}
\put(2.54,-0.015){\line(1,0){0.42}}
\put(2.74,-0.04){$>$}
\put(3.04,0){\line(1,0){0.42}}
%
\put(2,0.1){\footnotesize $\alpha_1$}
\put(2.5,0.1){\footnotesize $\alpha_2$}
\put(3,0.1){\footnotesize $\alpha_3$}
\put(3.5,0.1){\footnotesize $\alpha_4$}
\end{picture}
$$

Let $Q$ be such that $Q(\alpha^\vee_i)=1$ for $i=1, 2$. It implies $Q(\alpha_i^\vee)=2$ for $i=3, 4$. Clearly $n_{\alpha_i}=2$ for $i=1, 2$ and $n_{\alpha_i}=1$ for $i=3, 4$. We can compute 
$$B_Q(\alpha^\vee_1, \alpha^\vee_2)=-1,\quad B_Q(\alpha^\vee_2, \alpha^\vee_3)=-2, \quad B_Q(\alpha^\vee_3, \alpha^\vee_4)=-2.$$ 
Also $B_Q(\alpha^\vee_i, \alpha^\vee_j)=0$ if $\alpha_i$ and $\alpha_j$ are not adjacent in the Dynkin diagram.

Moreover, any $\sum_i k_i\alpha^\vee_i \in Y^{sc}$ with certain $k_i \in \Z$ belongs to $Y_{Q,2}$ if and only if $2|B_Q(\sum_i k_i\alpha^\vee_i, \alpha_j^\vee)$ for all $1\le j\le 4$, i.e., in explicit terms:
\begin{equation*}
\begin{cases}
2k_1 + (-1) k_2 \in 2\Z, \\
(-1) k_1 + 2k_2 +(-2) k_3\in 2\Z, \\
(-2)k_2 +4k_3 +(-2)k_4 \in 2\Z, \\
(-2)k_3 +4k_4 \in 2\Z.
\end{cases}
\end{equation*}

Equivalently, $k_1, k_2\in 2\Z$. This shows $Y_{Q, 2}=Y_{Q,2}^{sc}$, and thus the situation is trivial. Note that this agrees with the fact that the dual group $\wt{G}^\vee$ in this case has to be of type $F_4$, and there is a unique one. This gives a priori the equality $Y_{Q,2}=Y_{Q,2}^{sc}$.

\subsubsection{The $G_2$ case}

Consider the Dynkin diagram of $G_2$:
$$
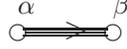
\begin{picture}(4.7,0.2)(0,0)
\put(2.5,0){\circle{0.08}}
\put(3,0){\circle{0.08}}
%
\put(2.53,0.018){\line(1,0){0.44}}
\put(2.54,0){\line(1,0){0.42}}
\put(2.53,-0.018){\line(1,0){0.44}}
\put(2.74,-0.035){$>$}
%
\put(2.5,0.1){\footnotesize $\alpha$}
\put(3,0.1){\footnotesize $\beta$}
\end{picture}
$$

Let $Q$ be such that $Q(\alpha^\vee)=1$. This determines $Q(\beta^\vee)=3$. Note $B_Q(\alpha^\vee, \beta^\vee)=-Q(\alpha^\vee)=-3$.

Since the computation is straightforward, we may assume $n \in \N_{\ge 1}$ is general instead of $2$. It follows $n_\alpha=n$ and $n_\beta=n/\text{gcd}(n, 3)$. Then $k_1\alpha^\vee + k_2\beta^\vee $ lies in $Y_{Q,n}$ if and only if
\begin{equation*}
\begin{cases}
2k_1 -3k_2 \in n\Z, \\
-3k_1 + 6k_2 \in n\Z.
\end{cases}
\end{equation*}

Equivalently, $k_1\in n\Z$ and $k_2$ divisible by $n/\text{gcd}(n,3)$. This exactly shows $Y_{Q,n}=Y_{Q,n}^{sc}$ for arbitrary $n$. This also agrees with the a priori fact that the dual group $\wt{G}^\vee$ of $\wt{G}$ must be of type $G_2$ and there is a unique one, which enforces the equality $Y_{Q,n}=Y_{Q,n}^{sc}$ to hold.

Also in this case, it is trivial to define the distinguished character for the fair $D$.

\begin{rmk}
Fix $n=2$ and $Q$ such that $Q(\alpha^\vee)=1$ for short coroots $\alpha^\vee$. In retrospect we see that for type $B_r, F_4, G_2$ one can have ad hoc descriptions of the cosets representative of $Y_{Q,n}/J$ in a similar way as Lemma \ref{geom lm}. For example, for $B_r$ and $F_4$ we modify the Dynkin diagram of these two by removing the short roots, then there is a correspondence between coset representative of $Y_{Q,n}/J$ and subsets of nodes in the modified Dynkin diagram satisfying $(i)$ and $(ii)$ as in Lemma \ref{geom lm}. For $G_2$, it is easy to have similar description as well.
\end{rmk}

\section{An equivalent construction of ${}^L\wt{T}$ and LLC by Deligne} \index{local Langlands correspondence!by P. Deligne}

Consider the case when $\wt{G}=\wt{T}$ over a local field $F$, Deligne gives a canonical construction ${}^\mca{D}\wt{T}$ with $\seq{\wt{T}^\vee}{{}^\mca{D}\wt{T}}{\W_F}$, which is isomorphic to ${}^L\wt{T}$ (cf. \cite{We14} also). In fact ${}^\mca{D}\wt{T}$ is defined to be ${}^\mca{D}\wt{T}=\Rec^*({}^\mca{D}E_{\wt{T}})$ for a certain ${}^\mca{D}E_{\wt{T}}$ in
$$\seq{Z(\wt{G}^\vee)}{{}^\mca{D}E_{\wt{T}}}{F^\times}.$$

The construction of ${}^\mca{D}E_{\wt{T}}$ is canonical and along the way one obtains an analogous version of the local Langlands correspondence (${}^\mca{D}\text{LLC}$). After recalling the construction of ${}^\mca{D}E_{\wt{T}}$ and the ${}^\mca{D}\text{LLC}$, we will show that there is a natural isomorphism $E_{\wt{T}} \simeq {}^\mca{D}E_{\wt{T}}$, which gives the compatibility of LLC and ${}^\mca{D}\text{LLC}$.

\subsubsection{The construction of ${}^\mca{D}E_{\wt{T}}$ and ${}^\mca{D}\text{LLC}$ }

Recall for $\wt{T}\in \CExt(T,\mu_n)$ of BD type, the center $Z(\wt{T})$ sits in the exact sequence
$$ \seq{\mu_n}{Z(\wt{T})}{T^\dag}. $$

Moreover, we have the map $\xymatrix{i_{Q,n}: T_{Q,n}  \ar@{>>}[r] & T^\dag}$ and the pull-back $\wt{T}_{Q,n}$:
\begin{equation} \label{T-Qn}
\xymatrix{
\mu_n \ar@{^(->}[r] &Z(\wt{T}) \ar@{>>}[r] & T^\dag \\
\mu_n \ar@{=}[u] \ar@{^(->}[r] &\wt{T}_{Q,n}  \ar@{>>}[u] \ar@{>>}[r] & T_{Q,n} \ar@{>>}[u]^-{i_{Q,n}}.
}
\end{equation}

Use the fixed embedding $\xymatrix{\epsilon: \mu_n \ar[r] & \C^\times}$ to obtain the push-out $\epsilon_*(Z(\wt{T}))$. At the same time, any $\wt{\chi} \in \Hom_\epsilon(Z(\wt{T}), \C^\times)$ gives rise to a splitting of $\epsilon_*(Z(\wt{T}))$ into $\C^\times$.

\begin{equation}
\xymatrix{
\mu_n \ar@{^{(}->}[r] \ar@{^(->}[d]_-{\epsilon} &Z(\wt{T}) \ar@{>>}[r] \ar[ld]^-{\wt{\chi}}   \ar[d] & T^\dag  \ar@{=}[d]  \\
\C^\times \ar@{^{(}->}[r]  &\epsilon_*(Z(\wt{T})) \ar@{>>}[r] \ar@/^1.4pc/@{.>}[l]^-{s_{\wt{\chi}}} & T^\dag .
}
\end{equation}

Here the splitting $\xymatrix{s_{\wt{\chi}}: \epsilon_*(Z(\wt{T})) \ar[r] & \C^\times}$ is given by
$$\xymatrix{
s_{\wt{\chi}}: \quad [(z, \wt{t})] \ar@{|->}[r] & z\cdot \wt{\chi}^{-1}(\wt{t}), 
}$$
where $[(z, \wt{t})]$ denote the class of $(z, \wt{t}) \in \C^\times \times Z(\wt{T})$ in $\epsilon_*(Z(\wt{T}))$. Clearly $s_{\wt{\chi}}$ gives a splitting $\rho^\mca{D}_{\wt{\chi}}$ of $\epsilon_*(Z(\wt{T}))$ over $T^\dag$ given by
$$\xymatrix{
\rho^\mca{D}_{\wt{\chi}}: \quad t \ar@{|->}[r] & [(\wt{\chi}(\wt{t}) , \wt{t})], 
}$$
where $\wt{t}\in Z(\wt{T})$ is any preimage of $t$.

Note by definition $\wt{T}^\vee =X_{Q,n} \otimes \C^\times$ and also $T_{Q,n}=Y_{Q,n}\otimes F^\times$, where $X_{Q,n}:=\Hom(Y_{Q,n}, \Z)$ is the lattice dual to $Y_{Q,n}$. By abuse of notation, we still use $i_{Q,n}$ to denote the naturally induced map $\xymatrix{i_{Q,n}: X_{Q,n}\otimes T_{Q,n} \ar[r] & X_{Q,n} \otimes T^\dag}$.

Let $m$ be the map 
$$\xymatrix{
m: F^\times \ar[r] & X_{Q,n}\otimes T_{Q,n}
}$$
given by 
$$\big(m(a)\big)(y)=y\otimes a, \ y\in Y_{Q,n},$$
where we have identified $X_{Q,n}\otimes T_{Q,n}\simeq \Hom(Y_{Q,n}, T_{Q,n})$. Then ${}^\mca{D}E_{\wt{T}}$ is defined to be the pull-back of $X_{Q,n}\otimes \epsilon_*(Z(\wt{T}))$ via $i_{Q,n} \circ m$:

\begin{equation}\label{Del's L}
\xymatrix{
\wt{T}^\vee \ar@{^{(}->}[r]  &X_{Q,n} \otimes \epsilon_*(Z(\wt{T})) \ar@{>>}[r] & X_{Q,n}\otimes T^\dag \ar@/_1.3pc/[l]_-{\rho^\mca{D}_{\wt{\chi}}}  \\
\wt{T}^\vee \ar@{=}[u] \ar@{^{(}->}[r]  & {}^\mca{D}E_{\wt{T}} \ar[u] \ar@{>>}[r] & F^\times \ar[u]_-{i_{Q,n}\circ m} \ar@/^1.3pc/[l]^-{\rho^\mca{D}_{\wt{\chi}}}.
}
\end{equation}

Meanwhile, there is the inherited splitting of ${}^\mca{D}E_{\wt{T}}$ over $F^\times$, still denoted by $\rho^\mca{D}_{\wt{\chi}}$.

\begin{dfn}
We define $^\mca{D}\wt{T}$ to be $\Rec^*(^\mca{D}E_{\wt{T}})$. We will call ${}^\mca{D}\text{LLC}$ the canonical map $\xymatrix{ \Hom_\epsilon(Z(\wt{T}), \C^\times) \ar[r] & \mfr{S}( {}^\mca{D}E_{\wt{T}}, F^\times)}$ given by
$$\xymatrix{
\wt{\chi} \ar@{|->}[r] & \rho^\mca{D}_{\wt{\chi}}
}$$
from above discussion. Since we have the canonical isomorphism $ \mfr{S}( {}^\mca{D}E_{\wt{T}}, F^\times) \simeq  \mfr{S}( {}^\mca{D}\wt{T}, \W_F)$, $\rho^\mca{D}_{\wt{\chi}}$ could be viewed as a splitting of $^\mca{D}\wt{T}$ over $\W_F$.
\end{dfn}

Note that we could work over $\wt{T}_{Q,n}$ and obtain similarly $\epsilon_*(\wt{T}_{Q,n})$ and also $X_{Q,n} \otimes \epsilon_*(\wt{T}_{Q,n})$ in the extension
$$\seq{X_{Q,n} \otimes \C^\times}{X_{Q,n} \otimes \epsilon_*(\wt{T}_{Q,n})}{X_{Q,n}\otimes T_{Q,n}}.$$

The the pull-back $m^*\big(X_{Q,n} \otimes \epsilon_*(\wt{T}_{Q,n})\big)$, by tracing through the compatible construction from (\ref{T-Qn}), is canonically isomorphic to ${}^\mca{D}E_{\wt{T}}$:

$$\xymatrix{
\wt{T}^\vee \ar@{^{(}->}[r]  &{}^\mca{D}E_{\wt{T}} \ar@{>>}[r] & F^\times \\
\wt{T}^\vee \ar@{=}[u] \ar@{^(->}[r]  & m^*\big(X_{Q,n} \otimes \epsilon_*(\wt{T}_{Q,n})\big) \ar[u]^-{\simeq} \ar@{>>}[r] & F^\times \ar@{=}[u]\ .
}$$

Since any $\wt{\chi}\in \Hom_\epsilon(Z(\wt{T}), \C^\times)$ could be viewed as a character of $\wt{T}_{Q,n}$, we will have a splitting of $m^*\big(X_{Q,n} \otimes \epsilon_*(\wt{T}_{Q,n})\big)$ over $F^\times$. It is canonically identified with $\rho^\mca{D}_{\wt{\chi}}$ via the isomorphism in above commutative diagram.


To summarize, we see that for any genuine character $\wt{\chi}$ of $Z(\wt{T})$, there is a canonically defined splitting of ${}^\mca{D}E_{\wt{T}}$ over $F^\times$, and therefore also a canonical splitting of ${}^\mca{D}\wt{T}$ over $\W_F$. In fact, we have

\begin{prop}[{\cite{We14}}]
Above construction gives an isomorphism $\Hom_\epsilon(\wt{T}_{Q,n}, \C^\times) \simeq \mfr{S}({}^\mca{D} E_{\wt{T}}, F^\times)$. Therefore, the ${}^\mca{D}\text{LLC}$ could be viewed as a composition of the following canonical maps:
$$\xymatrix{
\Hom_\epsilon(Z(\wt{T}), \C^\times) \ar@{^(->}[r] & \Hom_\epsilon(\wt{T}_{Q,n}, \C^\times) \ar[r]^-{\simeq} &\mfr{S}({}^\mca{D} E_{\wt{T}}, F^\times) \ar[r]^-{\simeq} &  \mfr{S}({}^\mca{D}\wt{T}, \W_F).
}$$ 
\end{prop}

\subsubsection{Compatibility of LLC and ${}^\mca{D}\text{LLC}$}

We now describe a natural isomorphism $\xymatrix{\Theta: E_{\wt{T}} \ar[r] & {}^\mca{D}E_{\wt{T}} }$, and check that ${}^\mca{D}\text{LLC}$ and LLC are compatible with respect to $\Theta$.

We view the tensor functor $X_{Q,n}\otimes -$ as $\Hom(Y_{Q,n}, -)$. By definition in (\ref{Del's L}), elements of  ${}^\mca{D}E_{\wt{T}}$ are of the form $(^\mca{D}[R_a], a) \in \Hom(Y_{Q,n}, \epsilon_*(Z(\wt{T}))) \times F^\times$, given by
$$\xymatrix{
^\mca{D}[R_a]: \quad y \ar@{|->}[r] & \big[\big( R_a(y), (1, y\otimes a) \big)\big] \in \epsilon_*(\wt{T}_{Q,n}),
}$$
where $\xymatrix{R_a: Y_{Q,n} \ar[r] &\C^\times}$ is a certain function. Since $^\mca{D}[R_a]$ is a homomorphism, this gives
$$R_a(y_1 +y_2)=R_a(y_1) \cdot R_a(y_2) \cdot (a, a)_n^{D(y_1, y_2)},$$

\noindent which is a necessary and sufficient condition for $^\mca{D}[R_a]$ to be a homomorphism. We mention in passing that the inverse of $(^\mca{D}[R_a], a)$ is $(^\mca{D}[R_a^{-1}\cdot (a, a)_n^{Q(-)}], a^{-1})$.

Recall the definition of  $E_{\wt{T}}$ as the Baer sum $E_{2,\wt{T}} \oplus_B E_{1,\wt{T}}$. The group $E_{2,\wt{T}}$ consists of $([P_a], a) \in \Hom(\mca{E}_{Q,n}, \C^\times) \times F^\times$ such that $[P_a]|_{F^\times/n}=h_a$, where $h_a(b)=(b, a)_n$. Meanwhile we have written
$$\xymatrix{
[P_a]: \quad (b, y) \ar@{|->}[r] & (b, a)_n \cdot P_a(y),
}$$
where $\xymatrix{P_a: Y_{Q,n} \ar[r] &\C^\times}$ is some function. That fact of $[P_a]$ being a homomorphism is equivalent to the following property of the map $P_a$:
$$P_a(y_1 +y_2)=P_a(y_1) \cdot P_a(y_2) \cdot (a, a)_n^{D(y_1, y_2)}.$$

It is now clear that we can define a map $\xymatrix{\Theta: E_{\wt{T}} \ar[r] & {}^\mca{D}E_{\wt{T}}}$ by
$$\xymatrix{
\Theta: \quad ([P_a], a)\oplus_B (1, a)_{E_1} \ar@{|->}[r]  & (^\mca{D}[P_a], a).
}$$
It is easy to check that $\Theta$ is an isomorphism.

Let $\wt{\chi}$ be a genuine character of $Z(\wt{T})$. The ${}^\mca{D}\text{LLC}$ gives the splitting $\rho_{\wt{\chi}}^\mca{D}$. Tracing through the construction of $\rho_{\wt{\chi}}^\mca{D}$ we see that it takes the explicit form: for $a\in F^\times$,
$$\rho_{\wt{\chi}}^\mca{D}(a)=\big(^\mca{D}[\wt{\chi}(1,  - \otimes a)], a\big) \in {}^\mca{D}E_{\wt{T}}.$$
That is, $R_a(y)=\wt{\chi}\big((1, y\otimes a)\big)$ for all $y\in Y_{Q,n}$.

\begin{cor}
We have the equality
$$\rho_{\wt{\chi}}^\mca{D}=\Theta \circ \rho_{\wt{\chi}},$$
where $\rho_{\wt{\chi}}$ is the splitting given by LLC  as in Proposition (\ref{LLC tori}). That is, the LLC and ${}^\mca{D}\text{LLC}$ are compatible with respect to $\Theta$.

Equivalently, if by abuse of notation we use $\Theta$ to denote the induced isomorphism $\xymatrix{ {}^L\wt{T} \ar[r] & {}^\mca{D}\wt{T} }$, $\rho_{\wt{\chi}}$ and $\rho_{\wt{\chi}}^\mca{D}$ for the induced splittings in $\mfr{S}({}^L\wt{T}, \W_F)$, $\mfr{S}({}^\mca{D}\wt{T}, \W_F)$ respectively, then the following diagram commutes:
$$\xymatrix{
{}^L\wt{T} \ar[rr]^-\Theta \ar@{>>}[rd] & &  {}^\mca{D}\wt{T} \ar@{>>}[ld] \\
&     \W_F \ar@/^1pc/@{-->}[lu]^-{\rho_{\wt{\chi}}} \ar@/_1pc/@{-->}[ru]_-{\rho^\mca{D}_{\wt{\chi}}}
}$$

\end{cor}

\section{Discussion on the global situation} \label{splt glb L}

Although in previous sections the discussion has been for local $L$-groups, some of the results can be extended easily to the global situation.

Let $F$ be a number field with $\mu_n\subseteq F^\times$. Let $\wm{G}$ be a BD-type $\mbf{K}_2$-torsor over $F$ incarnated by $(D,\eta)$. It gives rise to $\wm{G}(\A)$ whose $L$-group ${}^L\wt{G}$ is defined in section \ref{global L-grp}.

Following the Definition \ref{adm splitting} we may similarly define admissible splittings of ${}^L\wt{G}$ over $\W_F$ to be those which gives an isomorphism ${}^L\wt{G} \simeq \wt{G}^\vee \times \W_F$. Since ${}^L\wt{G}$ is derived from the fundamental extension $E_{\A}$ by pull-back, admissible splittings of the former are just splittings of
$$\seq{Z(\wt{G}^\vee)}{E_{\A}}{F^\times\backslash \A^\times}.$$

Consider $(D, \mbf{1})$ fair, i.e. with $D$ fair. As in the local case, to split $E_{\A}$ it suffices to find a genuine automorphic character
$$\xymatrix{
\wt{\chi}: \quad \mbf{T}_{Q,n}(F) \big\backslash \wm{T}_{Q,n}(\A) \ar[r] & \C^\times,
}$$
which is trivial on the image of $\mbf{T}_J(\A)$ in $\wm{T}_{Q,n}(\A)$ by the map locally given by (\ref{sJ}). Here as before $J=nY+ Y_{Q,n}^{sc}$. We use $s_\phi$ to denote the map given by
\begin{equation}
\xymatrix{
s_\phi: \mbf{T}_J(\A) \ar[r] & \wm{T}_{Q,n}(\A), \quad \prod_v y\otimes a_v \ar@{|->}[r] & \prod_v (1, y\otimes a_v).
}
\end{equation}

A character $\wt{\chi}$ satisfying $\wt{\chi} \circ s_\phi=\mbf{1}$ will descend to an automorphic character of the center $Z(\wm{T}(\A))$.

The construction in section \ref{cons dis char} generalizes. More precisely, we fix an additive character $\psi=\bigotimes_v \psi_v$ of $\A$. As in section \ref{cons dis char}, we obtain a local distinguished character $\wt{\chi}_{\psi_v}$ of $\wm{T}_{Q,n}(F_v)$. Then
$$\wt{\chi}_{\psi}:=\bigotimes_v \wt{\chi}_{\psi_v}$$
will be a well-defined automorphic character of $\wm{T}_{Q,n}(\A)$ such that $\wt{\chi}_{\psi} \circ s_\phi =\mbf{1}$. Such a global character $\wt{\chi}_{\psi}$ will be Weyl-invariant.

\chapter{The Gindikin-Karpelevich formula and the local Langlands-Shahidi $L$-functions}
\chaptermark{The GK formula and local $L$-functions}
\section{Satake isomorphism and unramified representations}

We recall basic facts on the Satake isomorphism as in \cite{McN12}, \cite{Li12}, \cite{We14-2} or \cite{We14}, and refer to the papers for detailed discussions and proofs.

Let $\wt{G} \in \CExt (\mbf{G}(F), \mu_n)$ be an object incarnated by some $(D,\eta)\in \Bis_{\mbf{G}}^Q$ over a local field $F$. We are interested in $\epsilon$-genuine representation of $\wt{G}$, i.e. with $\mu_n \subseteq \wt{G}$ acting by the fixed faithful character $\xymatrix{\epsilon: \mu_n \ar@{^(->}[r] & \C^\times}$. 

Recall from Definition \ref{G-bar unram} that the group $\wt{G}$ is called unramified if  and only if 
\begin{enumerate}
\item[(i)] $\text{gcd}(n, p)=1$ and,
\item[(ii)] there exists a splitting $\xymatrix{s_K: K_G \ar@{^(->}[r] &\wt{G}}$ of the maximal compact subgroup $K_G$.
\end{enumerate}

To simplify notation, we may omit $s_K$ whenever necessary, and no confusion will arise. From now, unless otherwise stated, we  assume $\wt{G}$ is unramified with respect to a fixed $s_K$.

\begin{dfn} \index{unramified!representation}
An irreducible genuine representation $\sigma \in \text{Irr}_\epsilon(\wt{G})$ is called unramified if and only if 
$$\sigma^{K_G}\ne 0.$$
That is, the space of $K_G$-fixed vectors is nonzero.
\end{dfn}

The key to understanding unramified representation is the Satake isomorphism as in the linear algebraic group case. Let $\mca{H}_\epsilon(\wt{G}, K_G)$ be the $\C$-algebra consists of locally constant and compactly supported functions $\xymatrix{f: \wt{G} \ar[r] &\C}$ satisfying
$$f(\xi k_1\wt{g} k_2 )=\epsilon(\xi) f(\wt{g}), \text{ for all } \xi \in \mu_n, k_1, k_2 \in K_G \text{ and } \wt{g}\in \wt{G}.$$

Similarly, we can define the Hecke algebra for $\mca{H}_\epsilon(\wt{T}, K_T)$ for the covering torus $\wt{T}$ with respect to $K_T:=T\cap K$.

The Satake transform
$$\xymatrix{
\mca{S}: \quad \mca{H}_\epsilon(\wt{G}, K_G) \ar[r] & \mca{H}_\epsilon(\wt{T}, K_T)
}$$
is given by
$$\mca{S}(f)(\wt{t}):=\delta(\wt{t})^{1/2} \int_U f(\wt{t}u) du \text{ for all } f\in  \mca{H}_\epsilon(\wt{G}, K_G). $$

Any function $f\in \mca{H}_\epsilon(\wt{T}, K_T)$ has support in the centralizer $C_{\wt{T}}(K_T)$ of $K_T$ in $\wt{T}$. This follows from the chain of equalities
$$f(\wt{t})=f(\wt{t} k) =[\wt{t}, k] \cdot f(k \wt{t})=[\wt{t}, k] \cdot f(\wt{t}),$$
where $\wt{t}\in \wt{T}$ and $k\in K_T$ are arbitrary, and $[-,-]$ is the commutator on $\wt{T}$.

Clearly, $Z(\wt{T}) K_T \subseteq C_{\wt{T}}(K_T)$. We refer to \cite[Lm. 1]{McN12}, \cite[\S 3.2]{We14-2} and \cite{We14} for the following.
\begin{lm} \label{A}
The centralizer $C_{\wt{T}}(K_T)$ of $K_T$ in $\wt{T}$ is equal to $Z(\wt{T}) K_T$. Moreover, $C_{\wt{T}}(K_T)=Z(\wt{T}) K_T$ is a maximal abelian subgroup of $\wt{T}$ containing $Z(\wt{T})$.
\end{lm}

We may identify the Weyl group $W$ with $(N(T)\cap K_G)\big/K_T$. Then $W$ acts on $Z(\wt{T})$ and also $K_T$ and therefore on $Z(\wt{T})K_T$, which induces a well-defined action on $\mca{H}_\epsilon(\wt{T}, K_T)$. Moreover, the group $K_T$ is invariant under $W$.

As in the linear algebraic group case, we have:

\begin{thm} \index{Satake isomorphism}
The Satake transform $\mca{S}$ gives an isomorphism of algebras:
$$\xymatrix{
\mca{S}: \quad \mca{H}_\epsilon(\wt{G}, K_G) \ar[r] & \mca{H}_\epsilon(\wt{T}, K_T)^W.
}$$
\end{thm}

The proof of above theorem could be found in \cite[\S 13.10]{McN12} where $F$ is assumed to contain $2n$-th root of unity. See also \cite[\S 3.2]{Li12} \cite{We14-2} and \cite{We14} for a discussion and proof in general.

From above one has by restriction of functions the isomorphism
$$\xymatrix{
\mca{V}: \quad \mca{H}_\epsilon(\wt{T}, K_T) \ar[r]^-\simeq & \mca{H}_\epsilon(Z(\wt{T})K_T, K_T),
}$$
where the right hand side consists of $\epsilon$-genuine compactly supported functions on $Z(\wt{T})K_T$ invariant under $K_T$. Clearly, it is isomorphic to $\mca{H}_\epsilon(Z(\wt{T}), T^\dag \cap K_T)$ since $Z(\wt{T})K_T/K_T \simeq Z(\wt{T})/(T^\dag \cap K_T)$.

For this purpose we define $\wt{Y}_{Q,n}:=\Big( Z(\wt{T})/(T^\dag \cap K_T)\simeq Z(\wt{T})K_T/K_T \Big)$ as in

$$\xymatrix{
& & T^\dag \cap K_T \ar@{^(->}[ld]^-{s_K} \ar@{^(->}[d] \\
\mu_n \ar@{^(->}[r] \ar@{=}[d] & Z(\wt{T}) \ar@{>>}[r] \ar@{>>}[d] & T^\dag \ar@{>>}[d]\\
\mu_n \ar@{^(->}[r] &\wt{Y}_{Q,n} \ar@{>>}[r] &Y_{Q,n}.
}$$

The abelian group $\wt{Y}_{Q,n}$ inherits a $W$-action from that of $W$ on $Z(\wt{T})K_T$. Then the algebra $\mca{H}_\epsilon(Z(\wt{T}), T^\dag \cap K_T)$ is isomorphic to
$$\C_\epsilon[\wt{Y}_{Q,n}]:=\frac{\C[\wt{Y}_{Q,n}]}{\big\langle \zeta-\epsilon(\zeta)^{-1}: \zeta \in \mu_n \big\rangle},$$
which in turn gives a $W$-equivariant isomorphism
$$\xymatrix{
 \mca{H}_\epsilon(\wt{T}, K_T)  \ar[r]^-\simeq & \C_\epsilon[\wt{Y}_{Q,n}].
}$$

\begin{cor}[{\cite{We14}}] \label{SV corres}
There is a natural isomorphism of algebras
$$\xymatrix{
\mca{H}_\epsilon(\wt{G}, K_G) \ar[r]^-{\mca{S}} &  \mca{H}_\epsilon(\wt{T}, K_T)^W \ar[r]^-{\mca{V}} & \C_\epsilon[\wt{Y}_{Q,n}]^W.
}$$
The Satake isomorphism thus gives the natural bijections between the following isomorphism classes:
$$\xymatrix{
\set{   \text {irred.  unramified genuine representations of $\wt{G}$}               } \ar@{<~>}[d]^-{\mca{S}^*} \\
\set{   \text {$W$-orbits of irred.  unramified genuine representations of $\wt{T}$}               } \ar@{<~>}[d]^-{\mca{V}^*} \\
\set{    \text{$W$-orbits of unramified genuine characters of $Z(\wt{T})$}            }.
}$$
\end{cor}

Here a character $\wt{\chi}$ of $Z(\wt{T})$ is called unramified if it is trivial on $T^\dag\cap K_T$, or equivalently, it is the pull-back of a certain character on $\wt{Y}_{Q,n}$. Also the $W$-action on unramified representations of $\wt{T}$ is via the correspondence between unramified representations of $\wt{T}$ and unramified characters of $Z(\wt{T})$ induced by the isomorphism $\mca{V}$.

We also have
\begin{cor}
For any unramfied representation $\sigma \in \text{Irr}_\epsilon(\wt{G})$,
$$\text{dim}_\C\  \sigma^{K_G}=1.$$
\end{cor}

Given $\wt{\chi}\in \Hom_\epsilon(Z(\wt{T}), \C^\times)$ unramified, we will give an elaborate discussion on the correspondences $\mca{S}^*$ and $\mca{V}^*$ and the normalized unramified vector in the resulting unramified representation of $\wt{G}$.

\subsection{Unramified representations of $\wt{T}$}
Let $\wt{T}$ be a degree $n$ cover of $T$ of BD type. Then $\wt{T}$ is an example of a Heisenberg group, i.e. a group whose commutator subgroup lies in the center. We first state the following general Stone-von Neumann theorem which works for general $\wt{T}$, unramified or not.

\begin{prop}[{\cite[Thm. 3.1]{We09}}] \label{Stone-Von} \index{Stone von-Neumann theorem}
Let $Z(\wt{T})$ be the center of $\wt{T}$, and let $A$ be any maximal abelain subgroup of $\wt{T}$ which contains $Z(\wt{T})$. Let $\xymatrix{\wt{\chi}: Z(\wt{T}) \ar[r] & \C^\times}$ be a genuine character of $Z(\wt{T})$, and let $\xymatrix{\wt{\chi}': A \ar[r] & \C^\times}$ be any extension to $A$. Write $\text{Ind}_A^{\wt{T}} (\wt{\chi}')$ for the induced representation on $\wt{T}$. 

The construction $\xymatrix{\wt{\chi} \ar@{|~>}[r] & \text{Ind}_A^{\wt{T}} (\wt{\chi}')}$ gives, up to isomorphism, a bijection between
$$\xymatrix{
\Hom_\epsilon\big(Z(\wt{T}), \C^\times\big) \ar@{~>}[r] & \text{Irr}_\epsilon(\wt{T}).
}$$
In particular, the isomorphism class of $\text{Ind}_A^{\wt{T}} (\wt{\chi}')$ does not depend on the choice of $A$ and the extension $\wt{\chi}'$, and therefore we may write $i(\wt{\chi})$ for the induced representation instead.
\end{prop}

For $\wt{T}$ unramified, there is a natural choice of $A$ which we take to be $C_{\wt{T}}(K_T)=Z(\wt{T})K_T$ by Lemma \ref{A}. It is a maximal abelian subgroup of $\wt{T}$ and contains $Z(\wt{T})$ by Lemma \ref{A}.

As mentioned, to give an unramified genuine character $\wt{\chi}$ on $Z(\wt{T})$ is equivalent to giving a genuine character on $\wt{Y}_{Q,n}$. Therefore, there is a natural extension $\wt{\chi}'$ to $C_{\wt{T}}(K_T)$ which is trivial on $K_T$ by inflation, since we also have $\wt{Y}_{Q,n}\simeq C_{\wt{T}}(K_T)/K_T$. 

Consider $i(\wt{\chi}):=\text{Ind}_A^{\wt{T}}(\wt{\chi}')$. It is clear that $Z(\wt{T})$ acts on $i(\wt{\chi})$ by $\wt{\chi}$. Moreover, $i(\wt{\chi})$ is unramified and $\text{dim}_\C\ i(\wt{\chi})^{K_T}=1$.

Conversely, let $\pi$ be an unramified representation of $\wt{T}$. Then, by the Stone von-Neumann theorem $\pi$ is isomorphic to $i(\wt{\chi})$ for some unramified character $\wt{\chi}$ of $Z(\wt{T})$. Therefore, we see that the correspondence $\mca{V}^*$ in Corollary \ref{SV corres} is basically realized as
$$\xymatrix{
\wt{\chi} \ar@{~>}[r] & i(\wt{\chi}):=\text{Ind}_A^{\wt{T}}(\wt{\chi}').
}$$

\subsection{Unramified principal series representations of $\wt{G}$}

Given any $\wt{G}$, let $\wt{B}$ be the Borel $B\subseteq G$ with the decomposition $\wt{B}=\wt{T}N$. Let $i(\wt{\chi})$ be a genuine irreducible unramified representation of $\wt{T}$ obtained above. We now define the principal series representations of $\wt{G}$.

\begin{dfn} \index{principal series}
Consider any $i(\wt{\chi}) \in \text{Irr}_\epsilon(\wt{T})$. We define the normalized induced representation $I_{\wt{B}}^{\wt{G}}(i(\wt{\chi}))$ as follows.

First let 
$$L(i(\wt{\chi}))=\set{f: \wt{G} \to i(\wt{\chi}) | \ f(\wt{b}\wt{g})=\delta_{\wt{B}}^{1/2}(\wt{b}) \cdot i(\wt{\chi})(\wt{b})f(\wt{g})}.$$ 
Then $I_{\wt{B}}^{\wt{G}}(i(\wt{\chi}))$ are the smooth vectors of $L(i(\wt{\chi}))$. Here $\delta_{\wt{B}}$ is the modular character of $\wt{B}$ and we view $i(\wt{\chi})$ as a representation of $\wt{B}$ by the inflation $\xymatrix{\wt{B} \ar@{>>}[r] & \wt{T}}$.
\end{dfn}

For simplicity we may write $I(\wt{\chi})$ for $I_{\wt{B}}^{\wt{G}}(i(\wt{\chi}))$ and call it a principal series representation of $\wt{G}$. Since the modular character $\delta_{\wt{B}}$ factors through that of $B$, i.e. $\delta_{\wt{B}}(\wt{b})=\delta_B(b)$ where $b\in B$ is the image of $\wt{b}$, we may use interchangeably both notations $\delta_{\wt{B}}$ and $\delta_B$.

\begin{lm}
Suppose $\wt{\chi}$ is unramified. Then the principal series $I(\wt{\chi})$ has a one-dimensional space of $K_G$-fixed vectors.
\end{lm}
\begin{proof}
The proof can be taken verbatim from \cite[Lm. 2]{McN12}. As it is important to have an explicit description of the unramified vectors of $I(\wt{\chi})$, we give a sketch here.

The key is that we have the isomorphism of vector spaces
$$\xymatrix{
I(\wt{\chi})^{K_G} \ar[r] & i(\wt{\chi})^{K_T} , \text{ given by } f \ar@{|->}[r] & f(1_{\wt{G}}).
}$$
Write $\mbf{f}=f(1_{\wt{G}}) \in i(\wt{\chi})^{K_T}$. Then $\mbf{f}(\wt{a}\wt{t})=\wt{\chi}'(\wt{a}) \mbf{f}(\wt{t})$ for all $\wt{t}\in \wt{T}$ and $\wt{a}\in A=C_{\wt{T}}(K_T)=Z(\wt{T})K_T$,  where $\wt{\chi}'$ is the extension of $\wt{\chi}$ to $A$ which is trivial on $K_T$.

Moreover, for all $\wt{t}$ and $k\in K_T$,
$$ \mbf{f}(\wt{t})=\big(i(\wt{\chi})(k)\mbf{f}\big)(\wt{t}) =\mbf{f}(\wt{t} k)=[\wt{t}, k] \wt{\chi}'(k) \mbf{f}(\wt{t})= [\wt{t}, k] \wt{\chi}'(k) \mbf{f}(\wt{t})=[\wt{t}, k] \mbf{f}(\wt{t}).$$

Thus the support of $\mbf{f}$ is $A$, and since $\mbf{f}|_A$ transform under $\wt{\chi}'$, $\mbf{f}$ is uniquely determined by $\mbf{f}(1_{\wt{T}})$. That is, $\text{dim}_\C \ I(\wt{\chi})^{K_G}=\text{dim}_\C \ i(\wt{\chi})^{K_T}=1$.
\end{proof}

Thus, we can argue as in the linear case that the correspondence $\mca{S}^*$ in Corollary \ref{SV corres} is basically given by
$$\xymatrix{i(\wt{\chi}) \ar@{|~>}[r] & \text{ the unramified component of } I(\wt{\chi}).
}$$

\section{Intertwining operators} \index{intertwining operator}
\subsection{Notations and basic set-up} \label{W notation}
As before, we use the boldface $\w$ to denote an element of $W$. We have defined in section \ref{BD section} the following elements of $\mbf{SL}_2$:
\begin{equation*}
    \mbf{e}_+(a) = \left(
      \begin{array}{cccc}
        1 & a \\
        0 & 1
      \end{array} \right), \quad
\mbf{e}_-(a)=\left(
      \begin{array}{cc}
        1 & 0 \\
        -a & 1
      \end{array} \right),
  \end{equation*}
\begin{equation*}
    \mbf{w}_o(a) =\mbf{e}_+(a) \mbf{e}_-(a^{-1}) \mbf{e}_+(a)= \left(
      \begin{array}{cccc}
        0 & a \\
        -a^{-1} & 0
      \end{array} \right), \quad
\mbf{h}_o(a) =\mbf{w}_o(a) \mbf{w}_o(-1)=\left(
      \begin{array}{cc}
        a & 0 \\
        0& a^{-1}
      \end{array} \right).
  \end{equation*}

From section \ref{BD section} there is the associated morphism $\xymatrix{\varphi_\alpha: \mbf{SL}_2 \ar[r] & \mbf{G}^{sc}}$ for any coroot $\alpha^\vee\in \Psi^\vee$. We also have the elements in $\mbf{G}^{sc}$: $\mbf{e}_\alpha(a)=\varphi_\alpha(\mbf{e}_+(a))$,    
$\mbf{e}_{-\alpha}(t)=\varphi_\alpha(\mbf{e}_-(a))$. Similarly, $\mbf{w}_\alpha(a):=\varphi_\alpha(\mbf{w}_o(a))$, $\mbf{h}_\alpha(a):=\varphi_\alpha(\mbf{h}_o(a))$. Clearly $\mbf{w}_\alpha(a) \in N(\mbf{T})(F)$.

More importantly, for $\wm{G}$ and its pull-back to $\wm{G}^{sc}$ via $\xymatrix{\mbf{G}^{sc} \ar[r] & \mbf{G}}$, there are the natural Brylinski-Deligne liftings of above elements in $\wm{G}^{sc}$ (cf. section \ref{BD section}), which we have denoted by
$$\wm{e}_\alpha, \quad \wm{w}_\alpha, \quad \wm{h}^{[b]}_\alpha.$$

We are interested in the image of the $F$-rational points of these elements in $\wt{G}=\wm{G}(F)$. Recall we have used  $\Phi_{D,\eta}$ to denote the natural map as in the following diagram:
$$\xymatrix{
\mu_n \ar@{^(->}[r] \ar@{=}[d] & \wt{G}^{sc} \ar@{>>}[r]  \ar[d]^-{\Phi_{D,\eta}} & G^{sc} \ar[d]^-\Phi\\
\mu_n \ar@{^(->}[r] & \wt{G} \ar@{>>}[r]  & G, 
}$$
which is just the pull-back of $\wt{G}$ to $G^{sc}$.

From now on, to simplify and fix our notations we will write for any root $\alpha \in \Psi$ 
$$ e_\alpha(a):=\Phi(\mbf{e}_\alpha(a)), \quad w_\alpha(a):=\Phi(\mbf{w}_\alpha(a)), \quad h_\alpha(a):=\Phi(\mbf{h}_\alpha(a)) \in G,$$
which all lie in $G$.

For any $a, b\in F^\times$, we have the elements $\wm{e}_\alpha(a)$, $\wm{w}_\alpha(a)$ and $\wm{h}^{[b]}_\alpha(a)\in \wt{G}^{sc}$. We thus introduce the following notation
\begin{equation} \label{wm to wt}
\wt{e}_\alpha(a):=\Phi_{D,\eta}(\wm{e}_\alpha(a)), \quad \wt{w}_\alpha(a):=\Phi_{D,\eta}(\wm{w}_\alpha(a)), \quad \wt{h}^{[b]}_\alpha(a):=\Phi_{D,\eta}(\wm{h}^{[b]}_\alpha(a)) \in \wt{G}.
\end{equation}
They all lie in $\wt{G}$ for any root $\alpha \in \Psi$. As in Proposition \ref{s-eta}, we will write $\wm{h}_\alpha(a)$ and $\wt{h}_\alpha(a)$ for $\wm{h}^{[b]}_\alpha(a) \in \wt{G}^{sc}$ and $\wt{h}^{[b]}_\alpha(a) \in \wt{G}$ respectively, whenever the latters are independent of the choice $b\in F^\times$. For example, we have used the notation $\wm{h}_\alpha(a^{n_\alpha})$ in Proposition \ref{s-eta}.

Furthermore, for convenience the following notations may also be used
$$e_\alpha:=e_\alpha(1), \quad w_\alpha:=w_\alpha(1) \ \in G.$$
Similarly,
$$\wt{e}_\alpha:=\wt{e}_\alpha(1), \quad \wt{w}_\alpha:=\wt{w}_\alpha(1), \quad \wt{h}^{[b]}_\alpha:=\wt{h}^{[b]}_\alpha(1) \ \in \wt{G}.$$
In particular, $\wt{h}_\alpha^{[1]}=1_{\wt{G}}$.

We will follow \cite{Sav04} and consider the group $W^{K_G}\subseteq K_G$ generated by $w_\alpha(-1)$ for $\alpha\in \Psi$, which lies in the exact sequence
$$\seq{T^{sgn}}{W^{K_G}}{W}, \qquad w_\alpha(-1) \mapsto \w_\alpha,$$
where $T^{sgn}\subseteq T$ is the finite group generated by $h_\alpha(-1)$ for $\alpha\in \Psi$. For application of global purpose, there is no loss of generalities to choose representatives for the Weyl group $W$ which lie in $K_G$.

There is a preferred section of $W^{K_G}$ over $W$. Let $\w\in W$, write $\w=\w_{\alpha_k}... \w_{\alpha_2}... \w_{\alpha_1} \in W$ in the form of a minimum decomposition. We would choose the following as a representative of $\w$:
$$s_W(\w):=w_{\alpha_k}\cdot  w_{\alpha_{k-1}} ... \cdot w_{\alpha_2} \cdot w_{\alpha_1} \in W^{K_G}.$$

One property of $s_W$ is that the representative $s_W(\w)$ is independent of the minimum decomposition of $\w$. Moreover, it is multiplicative (cf. \cite[\S 29.4]{Hum75}). That is, if $l(\w \w')=l(\w) + l(\w')$, then
$$s_W(\w \w')=s_W(\w) \cdot s_W(\w').$$

For simplicity, we may also write $w$ for $s_W(\w)$, i.e. by definition
$$w:=s_W(\w).$$

The finite subgroup $T^{sgn}$ does not lie in $T^\dag$, and this follows from considering the commutator $[h_\alpha(-1), y\otimes \varpi]=(-1, \varpi)_n^{B_Q(\alpha^\vee, y)}$ for any $y\in Y$. The obstruction is given by $(-1, \varpi)_n$. Thus the conjugation action of $W^{K_G}$ on $\wt{T}$ does not descend to $W$. However, we have seen that the action of $W$ on the isomorphism class of representations of $\wt{T}$ is well-defined.

We may consider $W^{K_G}$ as a subgroup of $\wt{G}$ via the splitting $s_K$ of $K_G$; then its finite subgroup $T^{sgn}$ viewed as a subgroup of $\wt{G}$ does not lie in the center $Z(\wt{T})$ from above consideration. Since the splitting $s_K$ agrees with the unipotent splitting $i_u$ as in Corollary \ref{sK=iu}, we have 
\begin{align*}
s_K(w_{\alpha_i})=& s_K(e_{\alpha_i}\cdot  e_{-\alpha_i} \cdot e_{\alpha_i})\\
=& s_K(e_{\alpha_i}) \cdot s_K(e_{-\alpha_i}) \cdot s_K(e_{\alpha_i}) \\
=& i_u(e_{\alpha_i}) \cdot i_u(e_{-\alpha_i}) \cdot i_u(e_{\alpha_i}) \\
=& \wt{e}_{\alpha_i} \cdot \wt{e}_{-\alpha_i} \cdot \wt{e}_{\alpha_i} \\
= & \wt{w}_{\alpha_i}
\end{align*}

It follows in general
\begin{align*}
s_K(w) & =s_K(w_{\alpha_k}) \cdot s_K(w_{\alpha_{k-1}}) \cdot ... \cdot s_K(w_{\alpha_2}) \cdot s_K(w_{\alpha_1}) \\
&= \wt{w}_{\alpha_k} \cdot \wt{w}_{\alpha_{k-1}} \cdot ... \cdot \wt{w}_{\alpha_2} \cdot \wt{w}_{\alpha_1}.
\end{align*}
The expression is independent of the minimum decomposition of $\w$. Therefore, we may also write $\wt{w}$ for $s_K(w)=s_K\circ s_W (\w)$ without any ambiguity. 

To summarize for the notations, for every $\w\in W$, we have a well-defined representative $w \in W^{K_G} \subseteq K_G$ and $\wt{w}\in s_K(W^{K_G})\subseteq \wt{W^{K_G}}$ (the preimage of $W^{K_G}$ in $\wt{G}$):
$$\xymatrix{
&   \mu_n \ar@{^(->}[d] \\
&    \wt{W^{K_G}} \ar@{>>}[d] \\
T^{sgn} \ar@{^(->}[r] & W^{K_G} \ar@{>>}[r] \ar@/^1.3pc/[u]^-{s_K} & W \ar@/_1.3pc/[l]_-{s_W}. \\
}$$

\subsubsection{The canonical isomorphism $^{w}i(\wt{\chi}) \simeq i(^{\w}\wt{\chi})$}

Let $i(\wt{\chi})$ be an irreducible representation of $\wt{T}$. Then we define an action of $w \in W^{K_G}$ (or equivalently that of $\wt{w}$) on $i(\wt{\chi})$ by
$$^{w}i(\wt{\chi})(\wt{t})=i(\wt{\chi})( w^{-1} \wt{t} w).$$

We have defined $^{\w}\wt{\chi}$ (see section \ref{Weyl inv}) by also considering the conjugation action of $w$ on the genuine character $\wt{\chi}$ of $Z(\wt{T})$. Since it actually only depends on $\w \in W$, we have written ${}^{\w}\wt{\chi}$ previously with no ambiguity.

Now we can compare $^w i(\wt{\chi})$ with $i(^\w \wt{\chi})$.

\begin{lm} \label{rep iden}
The two representations are isomorphic:
$$^wi(\wt{\chi}) \simeq i(^\w\wt{\chi}).$$

If $i(\wt{\chi})$ is unramified with unramified character $\wt{\chi}$, then both ${}^w i(\wt{\chi})$and $i(^\w\wt{\chi})$ are unramified. Since any isomorphism between the two spaces preserves unramified vectors, there is a canonical isomorphism which identifies the normalized unramified vectors of the two sides.
\end{lm}
\begin{proof}
Note that both spaces are irreducible with $Z(\wt{T})$ acting by the same character $^\w\wt{\chi}$. Therefore, by the Stone von-Neumann theorem in Proposition \ref{Stone-Von}, they are isomorphic.

Now assume $i(\wt{\chi})$ is unramified with $\wt{\chi}$ a unramified character of $Z(\wt{T})$. Then $^\w\wt{\chi}$ is unramified as well since $\w$ preserves $T^\dag\cap K_T$. Therefore $i(^\w\wt{\chi})$ as well as $^wi(\wt{\chi})$ are both unramified since $Z(\wt{T})$ in both representations acts by the unramified character ${}^\w\wt{\chi}$.

It is easy to see that any isomorphism between ${}^w i(\wt{\chi})$and $i(^\w\wt{\chi})$ preserves unramified vectors, and thus the normalization in the lemma could be achieved.
\end{proof}

We could give some analysis on the unramified vectors on both sides of above lemma. 

Identify $i(\wt{\chi})$ with $\text{Ind}_A^{\wt{T}}(\wt{\chi}')$, where $A=C_{\wt{T}}(K_T)$ and $\wt{\chi}'$ the natural unramified extension to $A$ introduced before. Let $\xymatrix{\mbf{f}_{i(\wt{\chi})}: \wt{T} \ar[r] &\C}$ be the normalized unramified vector of $i(\wt{\chi})$. Then the support of $\mbf{f}_{i(\wt{\chi})}$ is $A$ and it takes value 1 at $1_{\wt{T}}$. We claim $\mbf{f}_{i(\wt{\chi})} \in {}^wi(\wt{\chi})$ is unramified under the $^wi(\wt{\chi})$ action as well.

Let $k \in K_T$. Then $w^{-1} k w \in K_T$ since $w\in K_G$, and this gives
$$^wi(\wt{\chi})(k) \mbf{f}_{i(\wt{\chi})}(\wt{t}) =\mbf{f}_{i(\wt{\chi})}(\wt{t} \cdot w^{-1} k w )=\mbf{f}_{i(\wt{\chi})}(\wt{t}). $$
This shows that indeed $\mbf{f}_{i(\wt{\chi})}$ is $^wi(\wt{\chi})$-unramified.

Moreover, we see that  $\mbf{f}_{i(\wt{\chi})}$ transforms under $^\w(\wt{\chi}')$ when restricted to $A$. This can be seen as follows. Let $a \in A$, then $w^{-1} a w \in A$ and thus
\begin{equation*}
^wi(\wt{\chi})(a) \mbf{f}_{i(\wt{\chi})}(\wt{t}) = \mbf{f}_{i(\wt{\chi})}(\wt{t} \cdot w^{-1} a w) =
\begin{cases}
0=  {}^\w(\wt{\chi}')(a) \cdot \mbf{f}_{i(\wt{\chi})}(\wt{t}) & \text{ if } \wt{t} \notin A, \\
\mbf{f}_{i(\wt{\chi})}(w^{-1} a w \cdot \wt{t}) = {}^\w(\wt{\chi}')(a) \cdot \mbf{f}_{i(\wt{\chi})}(\wt{t}) & \text{ if } \wt{t} \in A,
\end{cases}
\end{equation*}
where the first case holds since the support of $f$ is in $A$. 

On the other hand, consider $i(^\w\wt{\chi})$ and identify it with $\text{Ind}_A^{\wt{T}} \big((^\w\wt{\chi})'\big)$, where $(^\w\wt{\chi})'$ is the natural unramified extension of $^\w\wt{\chi}$ to $A$. The normilized unramified vector of $ i(^\w\wt{\chi})$ is a function $\xymatrix{\mbf{f}_{i(^\w\wt{\chi})}: \wt{T} \ar[r] &\C^\times }$ with support in $A$ that transforms under $(^\w\wt{\chi})'$ when restricted to $A$. Also we require $\mbf{f}_{i(^\w\wt{\chi})}(1_{\wt{T}})=1$.

Note however, since $\wt{\chi}$ is unramified, we have $^\w(\wt{\chi}')=(^\w\wt{\chi})'$. Therefore, the rigidified isomorphism in above lemma is requiring the normalized vector $\mbf{f}_{i(\wt{\chi})}$ sent to the normalized vector $\mbf{f}_{i(^\w\wt{\chi})}$. 

In fact, it is not hard to check that the canonical isomorphism is given by
\begin{equation}\label{r_w}
\xymatrix{
r_w: {}^w i(\wt{\chi}) \ar[r]^-\simeq & i({}^\w \wt{\chi}), \quad \mbf{f} \ar@{|->}[r] & r_w(\mbf{f})(\wt{t}):=\mbf{f}(w^{-1} \wt{t} w).
}
\end{equation}

Meanwhile, if we take another representative $w' \in W^{K_G}$ of $\w$, then we have the following commutative diagram
\begin{equation} \label{lower tri}
\xymatrix{
                         & {}^w i(\wt{\chi}) \ar[ld]_-{r_{w, w'}} \ar[d]^-{r_w} \\
{}^{w'} i(\wt{\chi}) \ar[r]^-{r_{w'}} & i({}^{\w} \wt{\chi}),
}
\end{equation}
where $r_{w, w'}$ is given by $r_{w, w'}(\mbf{f})(\wt{t})=\mbf{f}(w^{-1} w'\cdot \wt{t} \cdot w'^{-1} w)$. Note that the unramified vectors $\mbf{f}_{ {}^wi({\wt{\chi}})}$ and $\mbf{f}_{{}^{w'}i({\wt{\chi}})}$ are the same functions and equal to the unramified vector $\mbf{f}_{i(\wt{\chi})} \in i(\wt{\chi})$. Moreover, the map $r_{w, w'}$ restricts to the identity map from $\C \cdot \mbf{f}_{ {}^wi({\wt{\chi}})}$ to $\C \cdot \mbf{f}_{{}^{w'}i({\wt{\chi}})}$ (with both identified with $\C \cdot \mbf{f}_{i(\wt{\chi})}$).

\vskip 10pt

As a consequence of the above discussion, it follows that for all $h_\alpha(-1) \in T^{sgn}$,
$${}^{h_\alpha(-1)} i(\wt{\chi}) \simeq i(\wt{\chi}).$$

From now, we will fix the isomorphism $^wi(\wt{\chi}) \simeq i(^\w\wt{\chi})$ by requiring that the normalized unramified functions on two sides correspond to each other as in the proof of the above lemma.

\subsection{Intertwining operators and cocycle relations}

For linear case, we have the cocycle relations for intertwining operators for induced representations. As before let $w=w_{\alpha_k} w_{\alpha_{k-1}}  ...  w_{\alpha_2} w_{\alpha_1}$ be the element of $W^{K_G}$ representing $\w=\w_{\alpha_k} \w_{\alpha_{k-1}} ... \w_{\alpha_2}\w_{\alpha_1}$ in a minimum expansion.  We have also defined the element $\wt{w}=s_K(w) \in \wt{G}$ which is independent of the factorization of $\w$ as well.

We are interested in the map:

\begin{equation} \label{inter op}
T\big(w, i(\wt{\chi})\big):\quad I(i(\wt{\chi}))\to I({}^wi(\wt{\chi})), \quad f\mapsto \int_{N^w} f(\wt{w}^{-1}\wt{u}g) du.
\end{equation}

and its factorization properties. Here $N^w=N\cap w N^- w^{-1}$, $N$ the unipotent radical of the Borel $\wt{B}=\wt{T}N$ and $N^-$ the unipotent radical opposite to $N$. 

First of all,

\begin{lm}
The map $\xymatrix{T\big(w, i(\wt{\chi})\big): I(i(\wt{\chi})) \ar[r] & I(^wi(\wt{\chi}) ) }$  is well-defined. That is, $T\big(w, i(\wt{\chi})\big)(f) \in I(^wi(\wt{\chi}))$. It intertwines the two unramified representations and sends unramified vector to unramified.
\end{lm}
\begin{proof}
The proof is routine and follows the same argument as in the linear case, see \cite[\S 4.1]{Sha10}. 

For example, to check $T\big(w, i(\wt{\chi})\big)(f) \in I(^wi(\wt{\chi}))$, let $u_o\wt{t}\in \wt{B}=N\wt{T}$, we compute
\begin{align*}
T\big(w, i(\wt{\chi})\big)(f)(u_o\wt{t} \cdot \wt{g}) &=  \int_{N^w} f(\wt{w}^{-1}\wt{u} \cdot u_o\wt{t} \cdot \wt{g} ) du \\
&= \int_{N^w} f(\wt{w}^{-1}\wt{t} \cdot \wt{t}^{-1} u\wt{t} \cdot \wt{g} ) du\\
&= \prod_{\substack{\alpha\in \Psi^+ \\ \w^{-1}\alpha<0}} |\alpha(t)|_F \cdot \int_{N^w} f(\wt{w}^{-1}\wt{t} \cdot u \cdot \wt{g} ) du\\
&=\prod_{\substack{\alpha\in \Psi^+ \\ \w^{-1}\alpha<0}} |\alpha(t)|_F \cdot \int_{N^w} f(\wt{w}^{-1}\wt{t} \wt{w}  \cdot \wt{w}^{-1} u \cdot \wt{g} ) du\\
&=\prod_{\substack{\alpha\in \Psi^+ \\ \w^{-1}\alpha<0}} |\alpha(t)|_F \cdot \delta_{\wt{B}}(w^{-1}\wt{t} w)^{1/2} \cdot {}^wi(\wt{\chi})(\wt{t})  \int_{N^w} f(\wt{w}^{-1} u \cdot \wt{g} ) du.\\
&=\prod_{\substack{\alpha\in \Psi^+ \\ \w^{-1}\alpha<0}} |\alpha(t)|_F \cdot \delta_{B}(w^{-1}t w)^{1/2} \cdot {}^wi(\wt{\chi})(\wt{t})  T\big(w, i(\wt{\chi})\big)(f)(\wt{g}).
\end{align*}
Here any root $\alpha\in \Hom(\mbf{T}, \mbf{G}_\text{mul})$ is viewed as a character of $\Hom(T, F^\times)$. Also we have the following equality (cf. \cite[pg. 83]{Sha10} for example),
$$\delta_{\wt{B}}(\wt{t})^{1/2}=\prod_{\substack{\alpha\in \Psi^+ \\ \w^{-1}\alpha<0}} |\alpha(t)|_F \cdot \delta_{B}(w^{-1}t w)^{1/2}.$$
That is, the intertwining operator $T\big(w, i(\wt{\chi})\big)$ is well-defined. It is easy to see that it sends an unramified vector to an unramified one.
\end{proof}

Let $f_{i({\wt{\chi}})}$ and $f_{{}^w i({\wt{\chi}})}$ be the normalized unramified vectors in $I(i(\wt{\chi}))$ and  $I(^wi(\wt{\chi}))$ respectively. Write $c(w, \wt{\chi}) \in \C$ for the coefficient such that
$$T\big(w, i(\wt{\chi})\big)f_{i({\wt{\chi}})}= c(w, \wt{\chi})  f_{{}^wi({\wt{\chi}})}.$$

The coefficient $c(w, \wt{\chi})$, which depends a priori on the preferred representative $w$ of $\w$, does not in the following sense. Take any other representative $w' \in K_G$ of $\w$, we have the intertwining operator $\xymatrix{T\big(w', i(\wt{\chi})\big): I(i(\wt{\chi})) \ar[r] & I({}^{w'}i(\wt{\chi})) }$ and 
$$T\big(w', i(\wt{\chi})\big)f_{ i({\wt{\chi}}) }= c(w', \wt{\chi})  f_{{}^{w'}i({\wt{\chi}})}.$$
We have the following commutative diagram
\begin{equation} \label{full tri}
\xymatrix{
I(i(\wt{\chi})) \ar[d]_-{T(w', i(\wt{\chi}))} \ar[rr]^-{T(w, i(\wt{\chi}))}  & & I({}^w i(\wt{\chi}))  \ar[lld]^{r_{w, w'}^*} \ar[d]^-{r_w^*} \\
I({}^{w'}i(\wt{\chi})) \ar[rr]_-{r_{w'}^*} & & I({}^{\w} \wt{\chi}),
}
\end{equation}
where the maps from the right lower triangle are induced from those in (\ref{lower tri}). It is easy to see $r_{w, w'}^*(f_{{}^w i({\wt{\chi}})})=f_{{}^{w'}i({\wt{\chi}})}$, and it follows 
$$c(w, \wt{\chi})=c(w', \wt{\chi}).$$

Therefore, we may write $c(\w, \wt{\chi})$ instead of $c(w, \wt{\chi})$. To take another view of $c(\w, \wt{\chi})$, we may consider the compositions from above diagram:

\begin{equation} \label{can inter op}
\xymatrix{
T(\w, \wt{\chi})=r_w^* \circ T(w, i(\wt{\chi})): \quad I(i(\wt{\chi})) \ar[r] & I(^wi(\wt{\chi})) \ar[r] & I(i(^\w\wt{\chi})).
}
\end{equation}

It is justified from above diagram that this map is independent of the choice of representative in $W^{K_G}$ for fixed $\w\in W$, hence the notation. In brief, we may write
$$\xymatrix{
T(\w, \wt{\chi}): \quad I(\wt{\chi}) \ar[r] & I(^\w\wt{\chi}).
}$$

Thus $T(\w, \wt{\chi})$ is the intertwining map between unramified spaces uniquely determined by 
$$T(\w, \wt{\chi}) f_{i({\wt{\chi}})}=c(\w, \wt{\chi}) f_{i(^\w{\wt{\chi}})}. $$

\vskip 5pt

\begin{rmk}
Following a suggestion of Wee Teck Gan, we could consider representatives of $W$ in a larger group. From the Bruhat-Tits theory, we have obtained a model $\mathbf{G}$ of $\mbf{G}$ over $\msc{O}_F$, from which $K_G:=\mathbf{G}(\msc{O}_F)$. This gives models $\mathbf{T}$ and $N(\mathbf{T})$ for $\mbf{T}$ and $N(\mbf{T})$ respectively, and we have an exact sequence
$$\seq{\mathbf{T}(\msc{O}_F)}{N(\mathbf{T})(\msc{O}_F)}{W}.$$

The discussion in this (and the previous) subsection applies with representatives of $W$ chosen from $N(\mathbf{T})(\msc{O}_F)$. For example, given any $n \in N(\mathbf{T})(\msc{O}_F)$ mapping to $\w \in W$, as in (\ref{r_w}) the morphism 
$$ \xymatrix{
r_n: {}^n i(\wt{\chi}) \ar[r]^-\simeq & i({}^\w \wt{\chi}), \quad f \ar@{|->}[r] & r_n(f)(\wt{t}):=f(n^{-1} \wt{t} n)
}
$$
gives the desired canonical isomorphism. We can define $T(n, \wt{\chi})$ as in (\ref{inter op}), and also the composition $\xymatrix{r_n^* \circ T(n, \wt{\chi}): I(\wt{\chi}) \ar[r] & I({}^\w \wt{\chi})}$ as in (\ref{can inter op}). The latter does not depend on the choice of $n$ and thus can be denoted by $T(\w, \wt{\chi})$ as well.

This is certainly a more natural and less restrictive approach, as the more restrictive group $W^{K_G}$ is essentially $N(\mathbf{T})(\Z)$. The consideration of $N(\mathbf{T})(\msc{O}_F)$ would allow for a treatment for general quasi-split $\mbf{G}$, which may not have a canonical integral model over $\Z$.
\end{rmk}

\vskip 10pt

Keep notations as above, and write $\Psi_\w:=\set{\alpha\in \Psi^+: \ \w (\alpha) \in \Psi^-}$.

\begin{prop} \index{cocycle relation}
Let $\w=\w_{\alpha_k}... \w_{\alpha_2}\w_{\alpha_1} \in W$ be an expansion of minimum length into simple reflections, and let $w=w_{\alpha_k} ... w_{\alpha_2} w_{\alpha_1}$ be defined as above. Then the intertwining operator factorizes as
$$T\big(w, i(\wt{\chi})\big)=T\big(w_{\alpha_k}, {}^{w_{\alpha_{k-1}} w_{\alpha_{k-2}} ... w_{\alpha_1}}i(\wt{\chi}))\big) \circ ... \circ T\big(w_{\alpha_2}, {}^{w_{\alpha_1}}i(\wt{\chi})\big) \circ T\big(w_{\alpha_1}, i(\wt{\chi})\big).$$
\end{prop}
\begin{proof} The proof is as in the linear case, see \cite[pg. 139-140]{CKM04}. We will apply induction. For convenience, we write $w_o$ for $w_{\alpha_{k-1}} w_{\alpha_{k-2}} ... w_{\alpha_1}$ and also $\wt{w}_o$ for $s_K(w_o)$, i.e. the element $\wt{w}_{\alpha_{k-1}} \wt{w}_{\alpha_{k-2}} ... \wt{w}_{\alpha_1}$. Since $w=w_{\alpha_k} w_o$, by induction it suffices to show
$$T\big(w, i(\wt{\chi})\big)=T\big(w_{\alpha_k}, ^{w_o}i(\wt{\chi})\big) \circ T\big(w_o, i(\wt{\chi})\big).$$

Note $U^w=\prod_{\alpha \in \Psi_\w} e_\alpha (u_\alpha), u_\alpha \in F$. In fact, 
\begin{align*}
T\big(w, i(\wt{\chi})\big)f(\wt{g}) &= \int_{U^w} f(\wt{w}^{-1} \wt{u} \ \wt{g}) du \\
 &=\int_{U^w} f(\wt{w}^{-1} \wt{u} \ \wt{w} \cdot \wt{w}^{-1} \wt{g}) du \\
&=\int_{U^{-,w}} f(\wt{u} \ \wt{w}^{-1} \wt{g}) du,
\end{align*}
where $U^{-,w}=\prod_{\alpha \in \Psi_\w} e_{-\alpha}(u_\alpha), u_\alpha \in F$. Similar equality holds for $T(w_o, i(\wt{\chi}))$ with replacing $U^{-,w}$ by $U^{-,w_o}$.

The set $\Psi_\w$ is equal to the union of two disjoint sets:
$$\Psi_\w=\set{\w_o^{-1}(\alpha_k)} \cup \Psi_{\w_o}.$$

Thus $U^{-, w}= U^{-, w_o} \cdot U_{-\w_o^{-1}(\alpha_k)}$, which gives
\begin{align*}
T\big(w, i(\wt{\chi})\big)f(\wt{g}) &= \int_{U_{-\w_o^{-1}(\alpha_k)}} \int_{U^{-,w_o}}  f(\wt{u}_o \wt{e}_{-\w_o^{-1}(\alpha_k)}(u) \cdot \wt{w}_o^{-1} \wt{w}_{\alpha_k}^{-1} \wt{g}) du_o du \\
 &= \int_{U_{-\w_o^{-1}(\alpha_k)}} \int_{U^{-, w_o}}  f(\wt{u}_o \wt{w}_o^{-1} \cdot \wt{w}_o \wt{e}_{-\w_o^{-1}(\alpha_k)}(u) \wt{w}_o^{-1} \cdot \wt{w}_{\alpha_k}^{-1} \wt{g}) du_o du \\
&= \int_{U_{-\alpha_k}} \int_{U^{-,w_o}}  f(\wt{u}_o \wt{w}_o^{-1} \cdot \wt{e}_{-\alpha_k}(u) \cdot \wt{w}_{\alpha_k}^{-1} \wt{g}) du_o du \\
&= \int_{U_{-\alpha_k}} T\big(w_o, i(\wt{\chi})\big) f (\wt{e}_{-\alpha_k}(u) \cdot \wt{w}_{\alpha_k}^{-1} \wt{g}) du \\
&= T\big(w_{\alpha_k}, {}^{w_o}i(\wt{\chi})\big) \circ T\big(w_o, i(\wt{\chi})\big) f(\wt{g}).
\end{align*}
\end{proof}

Immediately it follows:

\begin{cor}
With notations above, one has
\begin{equation} \label{cocycle rel}
c(\w, \wt{\chi})=\prod_{m=1}^k c(\w_{\alpha_m}, ^{\w_{\alpha_{m-1}}... \w_{\alpha_2} \w_{\alpha_1}}\wt{\chi}),
\end{equation}
which will be referred to as the cocycle relation.
Equivalently, we have
$$T(\w, \wt{\chi})=T(\w_{\alpha_k}, {}^{\w_{\alpha_{k-1}} \w_{\alpha_{k-2}} ... \w_{\alpha_1}}\wt{\chi}) \circ ... \circ T(\w_{\alpha_2}, {}^{\w_{\alpha_1}}\wt{\chi}) \circ T(\w_{\alpha_1}, \wt{\chi}).$$
\end{cor}

\section{The crude Gindikin-Karpelevich formula}
Let $\wt{G} \in \CExt(G,\mu_n)$ be an unramified central cover of BD type over a local field $F$. Choose a uniformizer $\varpi$ of $F$. The Gindikin-Karpelevich formula is obtained in \cite[Thm. 6.4]{McN11} by using a crystal basis decomposition of the domain of integration. Recently, as a consequence of the Casselman-Shalika formula computed, the GK formula is obtained as in \cite[Thm. 12.1]{McN14}. However, we will compute directly below using a straightforward method as in the linear case and remove restrictions such as $\mu_{2n} \subseteq F^\times$.    
Moreover, the GK formula here is expressed in terms of naturally defined terms and could be considered as a refinement of \cite{McN11} and above. This allows us to interpret it as local Langlands-Shahidi $L$-function later.

Before we proceed, we state the following simple but useful results.

\begin{lm} \label{add meas}
Let $du$ be an additive measure of $F$ such that the measure of $\msc{O}_F$ is equal 1, or equivalently the measure of $\msc{O}_F^\times$ with respect to $du$ is $1-1/q$. If $k\in \Z$, then
\begin{equation}
\int_{\msc{O}_F^\times} (u, \varpi)_n^k\ du=
\begin{cases}
1-1/q & \text{ if } n| k, \\
0 & \text{otherwise}.
\end{cases}
\end{equation}
\end{lm}

\begin{lm} \label{BD elements}
For any root $\alpha\in \Psi$ and any $u\in F^\times$, the following relations hold:
$$\wt{w}_\alpha^{-1} \wt{e}_\alpha(u) =\wt{h}^{[1]}_\alpha(u^{-1}) \wt{e}_\alpha(-u) \wt{e}_{-\alpha}(-u^{-1}) \in \wt{G}.$$
\end{lm}
\begin{proof}
It suffices to show the following
$$\wm{w}_\alpha^{-1} \wm{e}_\alpha(u) =\wm{h}^{[1]}_\alpha(u^{-1}) \wm{e}_\alpha(-u) \wm{e}_{-\alpha}(-u^{-1}) \in \wt{G}^{sc},$$
from which we could apply the morphism $\xymatrix{\Phi_{D,\eta}: \wt{G}^{sc} \ar[r] & \wt{G}}$ to get the desired result by the definition of the elements in above equality, see section \ref{W notation}. Note $\wm{w}_\alpha=\wm{w}_\alpha(1)$ by definition, and therefore $\wm{w}_\alpha^{-1}=\wm{w}_\alpha(-1)$.

We start with
\begin{align*}
\text{RHS} &= \wm{h}^{[1]}_\alpha(u^{-1}) \cdot \wm{e}_\alpha(-u) \wm{e}_{-\alpha}(-u^{-1}) \wm{e}_\alpha(-u) \cdot \wm{e}_\alpha(u) \\
&=\wm{h}^{[1]}_\alpha(u^{-1}) \wm{w}_\alpha(-u) \cdot \wm{e}_\alpha(u) \\
&=\wm{h}^{[1]}_\alpha(u^{-1}) \cdot \wm{w}_\alpha(-u) \wm{w}_\alpha(-1) \cdot \wm{w}_\alpha(1) \wm{e}_\alpha(u) \\
&=\wm{h}^{[1]}_\alpha(u^{-1}) \wm{h}^{[1]}_\alpha(-u) \cdot \wm{w}_\alpha(1) \wm{e}_\alpha(u) \\
&=(u^{-1}, -u)_n^{Q(\alpha^\vee)} \cdot \wm{h}^{[1]}_\alpha(-1) \wm{w}_\alpha(1) \wm{e}_\alpha(u) \\
&=\wm{w}_\alpha(-1) \wm{w}_\alpha(-1) \wm{w}_\alpha(1) \wm{e}_\alpha(u) \\
&=\wm{w}_\alpha(-1) \wm{e}_\alpha(u)
\end{align*}
The proof is completed since $\wm{w}_\alpha(-1)=\wm{w}_\alpha^{-1}$.
\end{proof}

Recall the convention on notations in section \ref{W notation}: we write $\wt{h}_\alpha(a^{n_\alpha})$ for $\wt{h}^{[b]}_\alpha(a^{n_\alpha})$ since the latter does not depend on $b\in F^\times$. Now we state the one-dimensional computation of $T\big(w_\alpha, i(\wt{\chi})\big)$.
 
\begin{prop} \label{rank1 coeff} 
Let $i(\wt{\chi})$ be an unramified representation of $\wt{T}$ for some unramified character $\wt{\chi}$ of $Z(\wt{T})$. Let $\alpha\in \Delta$, consider the intertwining operator 
$$\xymatrix{
T\big(w_\alpha, i(\wt{\chi})\big):\quad  I(i(\wt{\chi})) \ar[r] & I({}^{w_\alpha}i(\wt{\chi}))
}$$
between unramified principal series respectively of $\wt{G}$.
Let $f_{i(\wt{\chi})}$ and $f_{^{w_\alpha}i(\wt{\chi})}$ be the normalized unramified vectors of $I(i(\wt{\chi}))$ and $I({}^{w_\alpha}i(\wt{\chi}))$ respectively. Write $T\big(w_\alpha, i(\wt{\chi})\big)f_{i(\wt{\chi})}=c(\w_\alpha, \wt{\chi}) f_{{}^{w_\alpha}i(\wt{\chi})}$. Then
$$c(\w_\alpha, \wt{\chi})=\frac{1- q^{-1} \wt{\chi} (\wt{h}_\alpha(\varpi^{n_\alpha})) }{1-\wt{\chi} (\wt{h}_\alpha(\varpi^{n_\alpha})) },$$
which is independent on the uniformizer $\varpi$ chosen.
\end{prop}

\begin{proof}
Write $\wt{\pi}:= i(\wt{\chi})$. The unities $1_{\wt{G}}, 1_{\wt{T}}$ of all subgroups $\wt{G}, \wt{T}$ etc of $\wt{G}$ are all the same; but for convenience, we use the subindex to remind us the group on which the representation lives.

By the definition of $T\big(w_\alpha, i(\wt{\chi})\big)$, it suffices to compute
\begin{align*}
& T\big(w_\alpha, i(\wt{\chi})\big)(f_{i(\wt{\chi})})(1_{\wt{G}})(1_{\wt{T}}) \\
=& \int_{U_\alpha} f\big( \wt{w}_\alpha^{-1} \wt{e}_\alpha(u) \big)(1_{\wt{T}}) du\\
              =&\int_{0<|u|\le 1} f\big( \wt{w}_\alpha^{-1} \wt{e}_\alpha(u) \big)(1_{\wt{T}})du + \int_{|u|>1} f\big( \wt{w}_\alpha^{-1} \wt{e}_\alpha(u) \big)(1_{\wt{T}}) du.
\end{align*}
Here $du$ is the additive Haar measure as in Lemma \ref{add meas}. For the first integral, $\wt{w}_\alpha^{-1} \wt{e}_\alpha(u) \in K_G\subseteq \wt{G}$ for all $0<|u|\le 1$. Since $f$ is $K_G$-invariant, the integrand is equal to $f(1_{\wt{G}})(1_{\wt{T}})=1$. Therefore, the first integral is equal to 1.

For the second integral, we apply the previous lemma to get it equal to
\begin{align*}
&\ \quad \int_{|u|>1} f\big( \wt{h}^{[1]}_\alpha(u^{-1}) \wt{e}_\alpha(-u) \wt{e}_{-\alpha}(-u^{-1}) \big)(1_{\wt{T}}) du \\
&=\int_{|u|>1} f\big( \wt{h}^{[1]}_\alpha(u^{-1}) \wt{e}_\alpha(-u)\big)(1_{\wt{T}}) du\\
&=\int_{|u|>1} \delta_{B}^{1/2}(h_\alpha(u^{-1})) \cdot \Big( \wt{\pi} \big( \wt{h}^{[1]}_\alpha(u^{-1})\big) f(1_{\wt{G}}) \Big) (1_{\wt{T}}) du
\end{align*}
Note $$\delta_{B}^{1/2}(h_\alpha(u))=|u|_F^{\angb{\rho_B}{\alpha^\vee}}=|u|_F.$$

Use the partition $\set{u: |u|>1}=\bigcup_{k\ge 1} \msc{O}_F^\times \varpi^{-k}$, the integral is then equal to
\begin{align*}
&\ \quad \sum_{k\ge 1}\int_{u\in \varpi^{-k}\msc{O}^\times} \delta_{B}^{1/2}(h_\alpha(\varpi^k)) \cdot (u, \varpi^k)_n^{Q(\alpha^\vee)}  \cdot \Big( \wt{\pi}\big( \wt{h}^{[1]}_\alpha(\varpi^{k}) \wt{h}_\alpha(u) \big) f(1_{\wt{G}}) \Big) (1_{\wt{T}}) du.\\
&=\sum_{k\ge 1} \int_{u\in \msc{O}_F^\times} |\varpi|_F^k  \cdot (u, \varpi)_n^{kQ(\alpha^\vee)} \Big( \wt{\pi}\big( \wt{h}^{[1]}_\alpha(\varpi^{k})\big) f(1_{\wt{G}}) \Big) (1_{\wt{T}}) \cdot |\varpi^{-k}|_F du \\
&=\sum_{k\ge 1, n_\alpha |k} \int_{u\in \msc{O}_F^\times} \Big( \pi\big( \wt{h}^{[1]}_\alpha(\varpi^{k})\big) f(1_{\wt{G}}) \Big) (1_{\wt{T}}) du \quad \text{ by Lemma } \ref{add meas}\\
&=\sum_{k\ge 1, n_\alpha |k} \wt{\chi}\big( \wt{h}_\alpha(\varpi^{k}) \big) \cdot f(1_{\wt{G}})(1_{\wt{T}}) \cdot (1-q^{-1})\\
&=(1-q^{-1})\frac{\wt{\chi}\big( \wt{h}_\alpha(\varpi^{n_\alpha}) \big) }{1-\wt{\chi}\big( \wt{h}_\alpha(\varpi^{n_\alpha}) \big)},
\end{align*}
where the last equality is due to the fact that $\wt{\chi}\big( \wt{h}_\alpha(\varpi^{r\cdot n_\alpha}) \big)=\wt{\chi}^r\big( \wt{h}_\alpha(\varpi^{n_\alpha}) \big)$ for $r\in \N_{\ge 1}$. Now combine the first and second integral, we obtain the desired formula.

It remains to show the independence of the chosen uniformizer, we have for $u\in \msc{O}_F^\times$,
\begin{align*}
\wt{h}_\alpha\big((\varpi u)^{n_\alpha}\big) =& \wt{h}_\alpha(\varpi^{n_\alpha}) \cdot \wt{h}_\alpha(u^{n_\alpha}) \cdot (\varpi, u)_n^{n_\alpha^2 Q(\alpha^\vee)} \\
=& \wt{h}_\alpha(\varpi^{n_\alpha}) \cdot \wt{h}_\alpha(u^{n_\alpha}).
\end{align*}
However, since $\wt{h}_\alpha(u^{n_\alpha}) \in Z(\wt{T}) \cap K_T \subseteq s_K(K_G)$, it follows that $\wt{\chi}$ takes the value 1 at it and the independence on uniformizer could be concluded.
\end{proof}

Now we come back to general $\w \in W$ and consider the intertwining operator 
$$\xymatrix{T(\w, \wt{\chi}): I(\wt{\chi})  \ar[r] & I(^\w\wt{\chi})}.$$
We can use the cocycle relation to deduce:

\begin{cor} \label{gen coeff} \index{Gindikin-Karpelevich formula}
Let $f_{i(\wt{\chi})}$ and $f_{i(^\w\wt{\chi})}$ be the normalized unramified vectors in $I(\wt{\chi})$ and $I(^\w\wt{\chi})$ respectively. Then one has $T(\w, \wt{\chi}) f_{i(\wt{\chi})}=c(\w, \wt{\chi}) f_{i(^\w\wt{\chi})}$ with
$$c(\w, \wt{\chi})=\prod_{\alpha \in  \Psi_\w} \frac{1- q^{-1} \wt{\chi}\big( \wt{h}_\alpha(\varpi^{n_\alpha}) \big) }{1- \wt{\chi}\big( \wt{h}_\alpha(\varpi^{n_\alpha}) \big) },$$
where $\Psi_\w=\set{\alpha \in \Psi^+: \ \w\alpha \in \Psi^-}$.
\end{cor}
\begin{proof}
Let $\w_k \w_{k-1}... \w_2 \w_1$ be a minimum decomposition of $\w$, where $\w_i$ represents a simple reflection $\w_{\alpha_i}$ with $\alpha_i \in \Delta$.

By the cocycle condition in Corollary \ref{cocycle rel}, it suffices to compute the general coefficient $c(\w_m, {}^{\w_{m-1}... \w_1}\wt{\chi})$ for $m=1, 2, ..., k$.
First note, since $Q$ is Weyl-invariant, it gives
\begin{align*}
n_{\w_{1}... \w_{m-1} \alpha_m} =n_{\alpha_m}.
\end{align*}

Meanwhile, we have
\begin{align*}
& ^{\w_{m-1}... \w_1}\wt{\chi}\big(\wt{h}_{\alpha_m}(\varpi^{n_{\alpha_m}}) \big) \\
= &\wt{\chi}\big ( (w_{m-1}... w_1)^{-1} \cdot \wt{h}_{\alpha_m}(\varpi^{n_{\alpha_m}}) \cdot w_{m-1}... w_1) \big) \\
=& \wt{\chi}\big ( (w_1^{-1}... w_{m-1}^{-1} \cdot \wt{h}_{\alpha_m}(\varpi^{n_{\alpha_m}})  \cdot w_{m-1}... w_1) \big).
\end{align*}

To proceed, consider in general the element $w_\beta^{-1} \wt{h}_\alpha^{[b]}(a) w_\beta$ for $\alpha, \beta \in \Psi$ and $a, b\in F^\times$. We have
\begin{align*}
& w_\beta^{-1} \wt{h}_\alpha^{[b]}(a) w_\beta \\
=& w_\beta^{-1} \cdot \wt{w}_\alpha(ab)  \wt{w}_\alpha(-b) \cdot w_\beta \\
=& w_\beta^{-1} \wt{w}_\alpha(ab) w_\beta \cdot w_\beta^{-1}\wt{w}_\alpha(-b)  w_\beta.
\end{align*}
At the same time, for general $c\in F^\times$, the following equalities hold
\begin{align*}
& w_\beta^{-1} \wt{w}_\alpha(c) w_\beta \\
=& w_\beta^{-1} \cdot \wt{e}_\alpha(c) \wt{e}_{-\alpha}(c^{-1})  \wt{e}_\alpha(c)  \cdot w_\beta \\
=& \wt{e}_{\w_\beta (\alpha)} (\epsilon c) \cdot \wt{e}_{-\w_\beta (\alpha)} (\epsilon c^{-1}) \cdot \wt{e}_{\w_\beta (\alpha)} (\epsilon c), \\
=& \wt{w}_{\w_\beta (\alpha)}(\epsilon c),
\end{align*}
where the second last equality follows from the fact that the unipotent splitting is $G$-equivariant as from Proposition \ref{unip splitting}, and $\epsilon \in \set{\pm1}$ is a certain sign (depending on $\alpha, \beta$) associated with the Chevalley system of \'epinglage, see \cite[\S 3.2.2]{BrTi84}.

Now it follows
$$w_\beta^{-1} \wt{h}_\alpha^{[b]}(a) w_\beta = \wt{w}_{\w_\beta (\alpha)}(\epsilon ab) \cdot \wt{w}_{\w_\beta (\alpha)}(-\epsilon b)= \wt{h}_{\w_\beta(\alpha)}^{[\epsilon b]}(a).$$

In the case of $a=\varpi^{n_\alpha}$, the element $\wt{h}_{\w_\beta(\alpha)}^{[\epsilon b]}(a)$ is independent of $\epsilon$ and $b$. Compute inductively we obtain
$$^{\w_{m-1}... \w_1}\wt{\chi}\big(\wt{h}_{\alpha_m}(\varpi^{n_{\alpha_m}}) \big) = \wt{\chi}\big( \wt{h}_{\w_1... \w_{m-1}\alpha_m}(\varpi^{n_{\w_1... \w_{m-1}\alpha_m}}) \big).$$
Lastly, we have the equality $\Psi_\w=\set{\w_1... \w_{m-1}\alpha_m: m=1, 2, ..., k},$ from which the result follows from combining all $c(\w_m, {}^{\w_{m-1}... \w_1}\wt{\chi})$'s.
\end{proof}

\begin{rmk}
The usage of the element $\wt{h}(\varpi^{n_\alpha})$ (or in general $\wt{h}^{[1]}(a)$) from the Brylinski-Deligne section enables us to remove the assumption $\mu_{2n} \subseteq F^\times$ as in \cite{McN11} and \cite{McN12} for example. In fact, the computation of the metaplectic Casselman-Shalika formula in \cite{McN14} could be carried over using such naturally defined elements. It can be checked that in the case of double cover of $\mbf{Sp}_{2r}(F)$, McNamara's formula \cite[Thm. 13.1]{McN14} recovers that of Szpruch in \cite[Thm. 8.1]{Szp11} which does not reply on the assumption that $F$ contains $2n$-th root of unity, provided we make use of these naturally defined elements as in Lemma \ref{BD elements}: $\wt{w}_\alpha, \wt{e}_\alpha(u), \wt{h}^{[1]}(u)$ etc. 
\end{rmk}

\begin{rmk}
Recall that for any root $\alpha\in \Psi$ the morphism $\xymatrix{\varphi_\alpha: \mbf{SL}_2 \ar[r] & \mbf{G}}$ induces a covering $\wt{SL}_2^\alpha \in \CExt(\SL_2,\mu_n)$ from any given $\wt{G}\in \CExt(\wt{G}, \mu_n)$ of BD type. Let $T_o$ and $T$ be the tori of $SL_2$ and $G$, and let $\wt{T}$ and $\wt{T}^\alpha$ be the covering tori of $\wt{G}$ and $\wt{SL}_2^\alpha$ respectively. Let $Z(\wt{T})$ and $Z(\wt{T}^\alpha)$ be the centers of the two covering tori respectively. Then we have the following commutative diagram 
$$\xymatrix{
                        &\wt{T}   \ar@{>>}[r]          &T\\
Z(\wt{T}) \ar@{^(->}[ru]               &\wt{T}^\alpha \ar[u]^-{\wt{\varphi}_\alpha} \ar@{>>}[r]       & T_o \ar[u]^-{\varphi_\alpha} \\
\wt{\varphi}_\alpha^*(Z(\wt{T}))  \ar[u]^-{\wt{\varphi}_\alpha} \ar@{^(->}[ru] \ar@{^(->}[r] & Z(\wt{T}^\alpha) \ar@{^(->}[u] \ ,
}$$
where $\wt{\varphi}_\alpha^*(Z(\wt{T}))$ is the pull-back. By definition, $Z(\wt{T}^\alpha)$ and $\wt{\varphi}_\alpha^*(Z(\wt{T}))$ are closely related to $\alpha^\vee_{[n]}/\text{gcd}(2,n_\alpha)$ and  $\alpha^\vee_{[n]}$. More precisely, define
\begin{equation*}
\mathbf{n}=
\begin{cases}
1 & \text{ if } \alpha^\vee_{[n]}/\text{gcd}(2,n_\alpha) \in Y_{Q,n} \\
2 & \text{ otherwise}.
\end{cases}
\end{equation*}
Then it is not hard to see 
$$Z(\wt{T}^\alpha)\big/\wt{\varphi}_\alpha^*(Z(\wt{T}))\simeq F^\times/\mathbf{n}.$$

That is, in general $\wt{\varphi}_\alpha^*(Z(\wt{T}))$ is not equal to the whole group $Z(\wt{T}^\alpha)$ and $Z(\wt{T}^\alpha)\big/\wt{\varphi}_\alpha^*(Z(\wt{T}))$ is a torsion $2$ group. In fact, since we have assumed $\text{gcd}(n, p)=1$, $F^\times/\mathbf{n}\simeq \Z/2\Z \times \Z/2\Z$.

This has the following implication. For any genuine character $\wt{\chi}$ on $Z(\wt{T})$ we may write $\wt{\chi}_\alpha:=\wt{\chi}\circ \wt{\varphi}_\alpha$, which is a genuine character on $\wt{\varphi}_\alpha^*(Z(\wt{T}))$. In the unramified case, the rank one intertwining operator $T(\w_\alpha, \wt{\chi})$, or equivalently the scalar $c(\w_\alpha, \wt{\chi})$ such that $T(\w_\alpha, \wt{\chi})f_{i(\wt{\chi})}=c(\w_\alpha, \wt{\chi}) f_{i({}^{\w_\alpha}\wt{\chi})}$, can be determined from computing the following intertwining operator on $\wt{SL}_2^\alpha$:
$$\xymatrix{
T(\w_\alpha, \wt{\chi}): \quad I\Big(\text{Ind}_{\wt{\varphi}_\alpha^*(Z(\wt{T}))}^{Z(\wt{T}^\alpha)} (\wt{\chi}_\alpha )\Big) \ar[r] & I\Big( {}^{\w_\alpha}\big(\text{Ind}_{\wt{\varphi}_\alpha^*(Z(\wt{T}))}^{Z(\wt{T}^\alpha)} (\wt{\chi}_\alpha ) \big)\Big).
}$$

Note however, if in general we have $Z(\wt{T}^\alpha)\big/\wt{\varphi}_\alpha^*(Z(\wt{T}))\simeq (\Z/2\Z)^2$, then $\text{Ind}_{\wt{\varphi}_\alpha^*(Z(\wt{T}))}^{Z(\wt{T}^\alpha)} (\wt{\chi}_\alpha )=\bigoplus_{i=1}^4\wt{\chi}_{\alpha,i}$ is a $4$-dimensional representation of $Z(\wt{T}^\alpha)$.

Thus $T(\w_\alpha, \wt{\chi})$ is given by
$$\xymatrix{
T(\w_\alpha, \wt{\chi}) =\bigoplus_{i=1}^4 T(\w_\alpha, \wt{\chi}_{\alpha, i}): \quad \bigoplus_{i=1}^4 I(\wt{\chi}_{\alpha,i}) \ar[r] & \bigoplus_{i=1}^4 I ({}^{\w_\alpha}\wt{\chi}_{\alpha,i}).
}$$
Write $T(\w_\alpha, \wt{\chi}_{\alpha, i})(f_{i(\wt{\chi}_{\alpha,i})})=c(\w_\alpha, \wt{\chi}_{\alpha,i}) f_{i({}^{\w_\alpha}\wt{\chi}_{\alpha,i})}$. Then it follows from the computation in this section that
$$c(\w_\alpha, \wt{\chi}_{\alpha,i})=c(\w_\alpha, \wt{\chi}_{\alpha, j}) \text{ for any } i \text{ and } j,$$
since it is shown that $c(\w_\alpha, \wt{\chi}_{\alpha,i})$ depends only on $\wt{\chi}_{\alpha, i}$ restricted to $\wt{\varphi}_\alpha^*(Z(\wt{T}))$ and these characters are all equal there.

\end{rmk}

\begin{eg}
Consider $\wt{\mbf{GL}}_2$ with $(Q, \mca{E}, \phi)$, where $Q$ is given by
$$\xymatrix{
Q: \quad Y \ar[r] & \Z, \quad (y_1, y_2) \mapsto -y_1y_2.
}$$
Let $n=2$ and one obtains $\wt{GL}_2\in \CExt(GL_2, \mu_2)$.
 
Here we identity $Y=\Z^2$, with respect to which the element $(1,-1) \in \Z^2$ gives the coroot $\alpha^\vee$ of $\mbf{SL}_2\subseteq \mbf{GL}_2$. So $Q(\alpha^\vee)=1$ and $n_\alpha=2$.

One thus has the degree two covers $\wt{SL}_2$ and $\wt{GL}_2$ of $SL_2$ and $GL_2$ respectively. It is easy to check $\alpha^\vee_{[n]}/\text{gcd}(2,n_\alpha)=\alpha^\vee \notin Y_{Q,n}$, and therefore $Z(\wt{T}^\alpha)$ is not equal to $\wt{\varphi}_\alpha^*(Z(\wt{T}))$.
\end{eg} 

\section{The GK formula as local Langlands-Shahidi $L$-functions}

\subsection{Adjoint action and the GK formula for principal series} \label{ad rep}

Locally for $\wt{G}\in \CExt(G,\mu_n)$ of BD type, the $L$-group ${}^L\wt{G}$ sits in the exact sequence
$$\seq{\wt{G}^\vee}{{}^L\wt{G}}{\W_F}.$$

One can define the adjoint action  of ${}^L\wt{G}$ on its Lie algebra which is simply the Lie algebra $\wt{\mfr{g}}^\vee$ of $\wt{G}^\vee$. Thus we obtain the adjoint representation
$$\xymatrix{
Ad: \quad {}^L\wt{G} \ar[r] &  GL(\wt{\mfr{g}}^\vee).
}$$
By definition ${}^L\wt{G}= j^{\wt{G}^\vee}_* \circ \text{Rec}^*(E_{\wt{G}})$, where $E_{\wt{G}}$ is the fundamental extension over $F^\times$ by $Z(\wt{G}^\vee)$. This implies that the adjoint action $Ad$ depends only on the first coordinate in $\wt{G}^\vee$, where by definition we are viewing ${}^L\wt{G}$ as
$${}^L\wt{G}= \frac{\wt{G}^\vee \times \text{Rec}^*(E_{\wt{G}})}{\nabla Z(\wt{G}^\vee)}.$$

More precisely, define $\wt{G}^\vee_{ad}:=\wt{G}^\vee/Z(\wt{G}^\vee) $ and consider the extension $\wt{G}_{ad}^\vee \times W_F$ over $W_F$. Then there is a natural map $q*$ such that the following commutes:

$$\xymatrix{
\wt{G}^\vee \ar@{^(->}[r] \ar@{>>}[d]^-q & {}^L\wt{G} \ar@{>>}[r] \ar[d]^-{q^*} & \W_F \ar@{=}[d] \\
\wt{G}^\vee_{ad} \ar@{^(->}[r]  & \wt{G}_{ad}^\vee \times W_F \ar@{>>}[r] \ar@/^1.3pc/[l]^-{s^\text{Tr}}  & \W_F,
}$$
which is equipped with a canonical splitting $s^\text{Tr}$.

The adjoint representation $Ad$ factors through the usual complex adjoint representation $Ad^\C$ of $\wt{G}^\vee_{ad}$ on $\wt{\mfr{g}}^\vee$ with respect to the map $s^\text{Tr}\circ q^*$:
\begin{equation} \label{G:adC}
\xymatrix{
{}^L\wt{G} \ar[r]^-{Ad} \ar[d]_-{s^\text{Tr}\circ q^*} &  GL(\wt{\mfr{g}}^\vee) \\
                \wt{G}^\vee_{ad} \ar@{-->}[ru]_-{Ad^\C}.
}
\end{equation}

Now for any parabolic $\mbf{P}=\mbf{M}\mbf{U}$ of $\mbf{G}$ we obtain $\wt{M}$ and therefore ${}^L\wt{M}$ as well. Recall that there exists a canonical map ${}^L\varphi$  such that the following diagram commutes:

$$\xymatrix{
\wt{M}^\vee \ar@{^(->}[r] \ar@{^(->}[d]_-{\varphi^\vee} & {}^L\wt{M} \ar@{>>}[r] \ar@{^(->}[d]_-{{}^L\varphi} & \W_F \ar@{=}[d] \\
\wt{G}^\vee \ar@{^(->}[r] & {}^L\wt{G} \ar@{>>}[r]  & \W_F.
}$$

By restriction, this gives rise to a representation of ${}^L\wt{M}$ which factors through the complex representation $Ad^\C$ of $\wt{M}^\vee /Z(\wt{G}^\vee)$:
$$\xymatrix{
Ad^\C: \quad \wt{M}^\vee /Z(\wt{G}^\vee) \ar[r]       &GL(\wt{\mfr{g}}^\vee).
}$$

Since $\wt{M}^\vee$ is a Levi subgroup of $\wt{G}^\vee$, we can define a complex unipotent $\wt{U}^\vee$ such that it is the unipotent radical of the parabolic subgroup $\wt{M}^\vee \wt{U}^\vee$ of $\wt{G}^\vee$.

The group $\wt{M}^\vee /Z(\wt{G}^\vee)$ and therefore ${}^L\wt{M}$ act on the Lie algebra $\wt{\mfr{u}}^\vee$ of $\wt{U}^\vee$, which is a invariant space under the adjoint action of $\wt{M}^\vee$. That is, we have as in (\ref{G:adC})
\begin{equation}
\xymatrix{
{}^L\wt{M} \ar[r]^-{Ad} \ar[d] &  GL(\wt{\mfr{u}}^\vee) \\
                  \wt{M}^\vee/Z(\wt{G}^\vee) \ar@{-->}[ru]_-{Ad^\C},
}
\end{equation}
where the vertical map is the composition ${}^L \varphi \circ s^\text{Tr} \circ q^*$.

Now we specialize to the case where $\mbf{P}=\mbf{B}$ is the Borel subgroup of $\mbf{G}$. To be consistent with previous notations, we write $\mbf{B}=\mbf{T}\mbf{N}$. Then $\wt{\mfr{n}}^\vee$ is generated by eigenvectors of the form $E_{\alpha^\vee_{[n]}}$ for all $\alpha \in \Psi^+$. For any $\w\in W$, recall we have defined $\Psi_\w=\set{\alpha \in \Psi^+: \ \w\alpha \in \Psi^-}$. We are interested in the space
$$\wt{\mfr{n}}_\w^\vee=\bigoplus_{\alpha \in \Psi_\w} \C \cdot E_{\alpha^\vee_{[n]}} \subseteq \wt{\mfr{n}}^\vee.$$

The space $\wt{\mfr{n}}_\w^\vee$ is invariant under the adjoint action $Ad^\C$ of $\wt{T}^\vee /Z(\wt{G}^\vee)$ and this allows us to write $Ad_\w:=Ad|_{\wt{\mfr{n}}_\w^\vee}$. Clearly $Ad_\w$ has the decomposition
$$\xymatrix{
Ad_\w=\bigoplus_{\alpha\in \Psi_\w} Ad_\alpha: \quad {}^L\wt{T}  \ar[r] & GL(\wt{\mfr{n}}_\w^\vee),
}$$
where each $Ad_\alpha$ is the one-dimensional representation on $\wt{\mfr{n}}_\alpha^\vee:=\C\cdot E_{\alpha^\vee_{[n]}}$.

Let $I(\wt{\chi})=I(i(\wt{\chi}))$ be an unramified principal series of $\wt{G}$. Let $\xymatrix{\rho_{\wt{\chi}}: \W_F \ar[r] &{}^L\wt{T}}$ be the splitting of ${}^L\wt{T}$ over $\W_F$ associated with $\wt{\chi}$ by the local Langlands correspondence.  We could identify $\rho_{\wt{\chi}}$ with the splitting $\xymatrix{\rho_{\wt{\chi}}: F^\times \ar[r] & E_{\wt{T}}}$, which arises from the canonical isomorphism $\mfr{S}(E_{\wt{T}}, F^\times) \simeq \mfr{S}({}^L\wt{T}, \W_F)$. Note we may view $F^\times$ as the abelianization $\W^{ab}_F$ via the Artin reciprocity map.

We have the following ad hoc definition of unramified $\rho_{\wt{\chi}}$.
\begin{dfn}
The splitting (or $L$-parameter) $\rho_{\wt{\chi}}$ associated with $\wt{\chi}$ is called unramified if and only if $\wt{\chi}$ is unramified.
\end{dfn}

Note that for unramified character $\wt{\chi}$, the element $\rho_{\wt{\chi}}(\varpi) \in E_{\wt{T}}$ may still depend on the choice of the uniformizer $\varpi$ as $\rho_{\wt{\chi}}$ is a splitting of $F^\times$ into $E_{\wt{T}}$. However, recall the canonical map $\msc{C}$ from $E_{\wt{T}}$ to $\wt{T}^\vee/Z(\wt{G}^\vee)$ as in Corollary \ref{ad-t}. Then the element $\msc{C}\circ \rho_{\wt{\chi}}(\varpi) \in \msc{C}(E_{\wt{T}})$ is independent of the uniformizer.

\begin{prop} \label{unram rho}
Let $\rho_{\wt{\chi}} \in \mfr{S}(E_{\wt{T}}, F^\times)$ be an unramified splitting. Identify $\wt{T}^\vee/Z(\wt{G}^\vee)$ with $\Hom(Y_{Q,n}^{sc}, \C)$. Then for all $\alpha \in \Psi$, we have that
$$\msc{C}\circ \rho_{\wt{\chi}}(\varpi)(\alpha^\vee_{[n]})=\wt{\chi}\big(\wt{h}(\varpi^{n_\alpha})\big) \in \C$$
is independent of the uniformizer $\varpi$ chosen.
\end{prop}
\begin{proof}
The identify is from Corollary \ref{key iden}, while the independence of the uniformizer follows from by Proposition \ref{rank1 coeff}.
\end{proof}

Write $\wt{\pi}:=i(\wt{\chi})$. Then the local Langlands $L$-function $L(s, \wt{\pi}, Ad_\alpha)$ is given by the Artin $L$-function associated with
$$\xymatrix{
Ad_\alpha \circ \rho_{\wt{\chi}}:\quad \W_F \ar[r]  & {}^L\wt{T}  \ar[r]     &GL(\wt{\mfr{n}}_\alpha^\vee),
}.$$ 
That is,
$$L(s, \wt{\pi}, Ad_\alpha):=\frac{1}{\text{det}\big(1-q^{-s} \cdot Ad_\alpha \circ \rho_{\wt{\chi}}(\text{Frob})|_{{\wt{\mfr{n}}_\alpha^\vee}^I} \big)}.$$

For unramified $\rho_{\wt{\chi}}$ which we view as an element in $\mfr{S}(E_{\wt{T}}, F^\times)$, is given by
$$L(s, \wt{\pi}, Ad_\alpha)=\frac{1}{\text{det}\big(1-q^{-s} \cdot Ad_\alpha \circ \rho_{\wt{\chi}}(\varpi)|_{{\wt{\mfr{n}}_\alpha^\vee}^I} \big)}.$$

Now fix $\w \in W$. 
\begin{thm} \label{L rank1 coeff} \index{$L$-function ! local Langlands-Shahidi}
The eigenvalue of $\rho_{\wt{\chi}}(\varpi)$ of the representation $Ad_\alpha$ on the one-dimensional invariant space $\wt{\mfr{n}}_\alpha^\vee$ is given by
\begin{equation} \label{eigenvalue}
Ad_\alpha(\rho_{\wt{\chi}}(\varpi))(E_{\alpha^\vee_{[n]}})= \wt{\chi}\big( \wt{h}_\alpha(\varpi^{n_\alpha})\big) \cdot E_{\alpha^\vee_{[n]}},
\end{equation}
which is independent of the uniformizer chosen.

It follows that the GK formula for the intertwining operator $T(\w, \wt{\chi})$ acting on the unramified representation $I(\wt{\chi})$ can be rewritten as $T(\w, \wt{\chi}) f_{i(\wt{\chi})}=c(\w, \wt{\chi}) (f_{i({}^\w\wt{\chi})})$ with

$$c(\w, \wt{\chi})=\prod_{\alpha \in \Psi_\w} \frac{L(0, \wt{\pi}, \text{Ad}_\alpha)}{L(1, \wt{\pi}, \text{Ad}_\alpha)}.$$
\end{thm}
\begin{proof}
It suffices to show the equality in (\ref{eigenvalue}) regarding the eigenvalue.

Identify the splitting $\rho_{\wt{\chi}}$ as that of $E_{\wt{T}}$ over $F^\times$. Then the adjoint action of ${}^L\wt{T}$ naturally factor through that of $E_{\wt{T}}$:
$$\xymatrix{
{}^L\wt{T} \ar[r]^-{Ad} \ar[d] &  GL(\wt{\mfr{n}}_\alpha^\vee) \\
                  E_{\wt{T}} \ar@{-->}[ru]_-{Ad}.
}$$
Further more, the adjoint representation of $E_{\wt{T}}$ on $\wt{\mfr{n}}_\alpha^\vee$ also factors through $\wt{T}^\vee/Z(\wt{G}^\vee)$:

$$\xymatrix{
 E_{\wt{T}}  \ar[r]^-{Ad_\alpha} \ar[d]_-{\msc{C}}      &GL(\wt{\mfr{n}}_\alpha^\vee) \\
 \wt{T}^\vee/Z(\wt{G}^\vee) \ar[ru]_-{Ad_\alpha^\C}.
}$$

Therefore, with $\rho_{\wt{\chi}}$ viewed as a splitting of $E_{\wt{T}}$ over $F^\times$, we have the composition
$$\xymatrix{
Ad_\alpha^\C \circ \msc{C} \circ \rho_{\wt{\chi}}: \quad F^\times \ar[r] & E_{\wt{T}} \ar[r] & \wt{T}^\vee/Z(\wt{G}^\vee) \ar[r] & GL(\wt{\mfr{n}}_\alpha^\vee).
}$$

Now it suffices to compute the eigenvalue of $Ad_\alpha^\C \circ \msc{C} \circ \rho_{\wt{\chi}} (\varpi)$ on $E_{\alpha^\vee_{[n]}}$, which is equal to $\msc{C}\circ \rho_{\wt{\chi}}(\varpi)(\alpha_{[n]}^\vee)$ since $E_{\alpha^\vee_{[n]}}$ is the unipotent vector corresponding to the root $\alpha^\vee_{[n]}$ of $\wt{G}^\vee$. Here we view $\msc{C}\circ \rho_{\wt{\chi}}(\varpi)\in \wt{T}^\vee/Z(\wt{G}^\vee)\simeq \Hom(Y_{Q,n}^{sc}, \C^\times)$.

The proof is thus completed in view of the equality $\msc{C} \circ \rho_{\wt{\chi}}(\varpi)(\alpha_{[n]}^\vee)=\wt{\chi}\big(\wt{h}(\varpi^{n_\alpha})\big)$ and the independence statement from the previous proposition.
\end{proof}

\begin{rmk} \label{nonsplit GK}
When $\mbf{G}$ is not necessarily split over $F$ but over an unramified extension $E$ of $F$, the computation in \cite[Thm. 12.1]{McN14} gives the rank one GK coefficient as
$$\frac{\big(1+ (-1)^{n_\alpha} \cdot q^{-1} \wt{h}_\alpha(\varpi^{n_\alpha})\big) \big(1- (-1)^{n_\alpha} \cdot q^{-2} \wt{h}_\alpha(\varpi^{n_\alpha})\big)}{1 - \wt{h}_\alpha(\varpi^{n_\alpha})^2},$$
where $\wt{h}_\alpha(\varpi^{n_\alpha})^2=\wt{h}_\alpha(\varpi^{2n_\alpha})$. We believe that a proper understanding of the $L$-group construction for BD covers of nonsplit $\mbf{G}$ in \cite{We14} will enable us to
express the GK coefficient obtained as Langlands-Shahidi type $L$-function as well. We leave the investigation of this to a future work.
\end{rmk}

\subsection{The GK formula for induction from maximal parabolic} \label{max parab}
In this subsection, we consider parabolic induction from a maximal parabolic of $\wt{G}$. Let $\mbf{G}$ be a reductive group split over $F$ with root datum $(X, \Psi, Y, \Psi^\vee)$. As before, we have fixed a set of simple roots $\Delta\subseteq \Psi^+$ with $\Psi=\Psi^+ \sqcup \Psi^-$. Consider any simple root $\beta \in \Delta$, and let $\mbf{P}=\mbf{M}\mbf{U}$ be a maximal parabolic of $\mbf{G}$ associated with $\Delta \backslash \set{\beta}$. 

Let $2\rho_P$ be the sum of positive roots in $\mbf{U}$, define
$$\beta_P=\frac{\rho_P}{\angb{\rho_P}{\beta^\vee}},$$
where $\angb{-}{-}$ is the pairing between $X$ and $Y$. Then $\beta_P \in X\otimes \Q$ is the fundamental weight associated with $\beta$.

From the $\mu_n$-extension $\wt{G}$, we obtain a $\mu_n$ extension $\wt{M}$ of $M=\mbf{M}(F)$.  Since the construction of $\mu_n$ extension $\wt{G}$ is functorial, the extension $\wt{M}$ over the Levi $M$
$$\seq{\mu_n}{\wt{M}}{M}$$
is obtained from the data inherited from those associated with $\wt{\mbf{G}}$.

Let $\wt{\pi}$ be an  irreducible unramified genuine representation of $\wt{M}$. Consider the normalized induced representation $I_{\wt{P}}^{\wt{G}}(\wt{\pi} \otimes \mbf{1}_U)$. For simplicity we may just write $I(\wt{\pi})$ for it.

Since $\wt{\pi}$ is unramified, by the Satake isomorphism in Corollary \ref{SV corres} there is a $\epsilon$-genuine character $\wt{\chi}$ of $Z(\wt{T})$ such that 
$$\xymatrix{
\wt{\pi} \ar@{^(->}[r] & I_{\wt{B}_M}^{\wt{M}}(\wt{\chi})
},$$
where $\wt{B}_M=\wt{T} N_M$ is the Borel of $\wt{M}$ whose Levi  factor is just $\wt{T}$. Then the representation $\text{Ind}_{\wt{P}}^{\wt{G}} \big (I_{\wt{B}_M}^{\wt{M}}(\wt{\chi}\big) \otimes \mbf{1}_{U} $, by induction in stages, is just the unramified principle series $I(\wt{\chi})=I(i(\wt{\chi}))$ of $\wt{G}$ introduced before. Moreover, we have
$$\xymatrix{
I(\wt{\pi}) \ar@{^(->}[r] & I(\wt{\chi}).
}$$

It is known (cf. \cite[pg. 122]{CKM04}) that there exists a unique $\w\in W$ such that 
$$\w(\Delta \backslash \set{\beta})\subseteq \Delta \text{ and } \w(\beta)\in \Psi^-.$$
In fact, $\w(\beta_P)=-\beta_P$. From now on, we fix this $\w$ whenever we consider intertwining operators for induction from maximal parabolic.

Let $w =s_W(\w) \in W^{K_G}$ and $\wt{w}=s_K(w)$ be the representatives of $\w$ defined in section \ref{W notation}. We are interested in the intertwining operator
$$\xymatrix{
T(w, \wt{\pi}): \quad I(\wt{\pi}) \ar[r] & I({}^w\wt{\pi})
}$$
given by
$$ f\mapsto  T(w, \wt{\pi})f(\wt{g})=\int_{U^w} f(\wt{w}^{-1}\wt{u}\ \wt{g}) du, $$
where $U^w=N\cap w U^- w^{-1}$ with $U^-$ the unipotent radical opposed to $U$. 

As in the linear algebraic case, the following diagram commutes:
$$\xymatrix{
I(i(\wt{\chi}))   \ar[rr]^-{T(w, i(\wt{\chi}))}    & & I({}^w i(\wt{\chi})) \\
I(\wt{\pi})  \ar@{^(->}[u] \ar[rr]^-{T(w, \wt{\pi})}    & &I({}^w \wt{\pi}) \ar@{^(->}[u] \ .
}$$

Let $f_{\wt{\pi}}, f_{{}^w\wt{\pi}}$ be the normalized unramified vectors of $I(\wt{\pi}), I({}^w\wt{\pi})$ respectively. We view them as vectors in the unramified principal series $I(i(\wt{\chi}))$ and $I({}^wi(\wt{\chi}))$, normalized such that $f_{\wt{\pi}}(1_{\wt{G}})(1_{\wt{T}})=f_{{}^w\wt{\pi}}(1_{\wt{G}})(1_{\wt{T}})=1$. We want to compute the constant $c(\w, \wt{\pi})$ that appears in $T(w, \wt{\pi}) f_{\wt{\pi}}=c(\w, \wt{\pi}) f_{{}^w\wt{\pi}}$.

Coupled with the computation of the GK formula in Theorem \ref{L rank1 coeff}, we see that
$$c(\w, \wt{\pi})=\prod_{\alpha \in \Psi_\w} \frac{L(0, Ad_\alpha \circ \rho_{\wt{\chi}})}{L(1, Ad_\alpha \circ \rho_{\wt{\chi}})}, $$
where each $\xymatrix{Ad_\alpha \circ \rho_{\wt{\chi}}:  \W_F \ar[r] & \C^\times}$ is a character.

Consider the adjoint action
$$\xymatrix{
Ad_{\wt{\mfr{u}}^\vee}: \quad {}^L\wt{M} \ar[r] & GL(\wt{\mfr{u}}^\vee),
}$$
where $\wt{\mfr{u}}^\vee$ is the Lie algebra of $\wt{U}^\vee$ such that $\wt{M}^\vee \wt{U}^\vee$ is a parabolic subgroup of $\wt{G}^\vee$. It factors through $Ad^\C_{\wt{\mfr{u}}^\vee}$:
$$\xymatrix{
Ad_{\wt{\mfr{u}}^\vee}: & {}^L\wt{M}  \ar[r] \ar[d]    &GL(\wt{\mfr{u}}^\vee) \\
& \wt{M}^\vee/Z(\wt{G}^\vee) \ar[ru]_-{Ad^\C_{\wt{\mfr{u}}^\vee}}.
}$$
Therefore, irreducible subspaces of $\wt{\mfr{u}}^\vee$ for $Ad_{\wt{\mfr{u}}^\vee}$ are in correspondence with irreducible subspaces with respect to $Ad^\C_{\wt{\mfr{u}}^\vee}$, which are familiar (cf. \cite{Lan71}). More precisely, we consider the decomposition of $Ad_{\wt{\mfr{u}}^\vee}$ into  irreducibles
$$Ad_{\wt{\mfr{u}}^\vee}=\bigoplus_{i=1}^m Ad_i.$$
Let $V_i \subseteq \wt{\mfr{u}}^\vee$ be the irreducible space for $Ad_i$. Then as observed by Langlands, $V_i$ is given by
$$V_i=\bigoplus_{\langle \beta_P/n_\beta, \alpha_{[n]}^\vee \rangle=i } \C \cdot E_{\alpha_{[n]}^\vee}.$$

Moreover the following equality holds:
$$\Psi_\w=\bigsqcup_{i=1}^m \set{ \alpha\in \Psi^+: \angb{\beta_P/n_\beta}{\alpha^\vee_{[n]}}=i }.$$

Recall the local Artin $L$-function $L(s, \wt{\pi}, Ad_i)$ is by definition $L(s, Ad_i \circ \rho_{\wt{\chi}})$ associated with $Ad_i \circ \rho_{\wt{\chi}}$:
$$L(s, \wt{\pi}, Ad_i)=\frac{1}{\text{det}\big(1-q^{-s} Ad_i \circ \rho_{\wt{\chi}}(\textsf{Frob})|_{V_i^I} \big)},$$
where we also write $\rho_{\wt{\chi}}$ for the composition $\xymatrix{\W_F \ar[r] &{}^L\wt{T} \ar[r] &{}^L\wt{M}}$. For unramified $\wt{\pi}$, if we identify $\rho_{\wt{\chi}}$ with an unramified splitting of $E_{\wt{T}}$ over $F^\times$, then the inertia group $I$ acts trivially on $V_i$ by Theorem \ref{L rank1 coeff}. We thus have
$$L(s, \wt{\pi}, Ad_i)=\frac{1}{\text{det}\big(1-q^{-s} Ad_i \circ \rho_{\wt{\chi}}(\varpi)|_{V_i}\big)}.$$

In view of Theorem \ref{L rank1 coeff}, we could summarize our discussion above as:

\begin{thm} \label{L gen coeff}
The G-K formula takes the form $T(w, \wt{\pi}) f_{\wt{\pi}}=c(\w, \wt{\pi}) f_{{}^w\wt{\pi}}$ with
$$c(\w, \wt{\pi})=\prod_{i=1}^m \frac{L(0, \wt{\pi}, Ad_i)}{L(1, \wt{\pi}, Ad_i)},$$
where in this case we have the equality
$$L(s, \wt{\pi}, Ad_i)=\prod_{\substack{\alpha \in \Psi^+ \\ \angb{\beta_P/n_\beta}{\alpha_{[n]}^\vee}=i}} L(s, Ad_\alpha \circ \rho_{\wt{\chi}}).$$
\end{thm}


\chapter{Automorphic $L$-function, constant term of Eisenstein series and residual spectrum}
\chaptermark{$L$-functions and residual spectrum}
The aim of this chapter is to compute the constant term of Eisenstein series, and to write the coefficient of global intertwining operators in terms of certain Langlands-Shahidi type $L$-functions. The main tool is the theory of Eisenstein series and the computation of its constant terms in terms of intertwining operators as given in \cite{MW95}. It is important to bridge from the global situation to the local ones and thus to utilize the local results which we have obtained in the previous chapter.

Since the global theory is developed systematically in the book by Moeglin-Waldspurger (\cite{MW95}), we will just give a brief review below and refer to the book for properties of Eisenstein series, e.g. meromorphic continuation and functional equations etc. We also refer to \cite{MW95} for detailed introduction on automorphic forms on $\wm{G}(\A)$ and the spectral decomposition of $L^2(\mbf{G}(F)\big\backslash \wt{\mbf{G}}(\A))$.

In this chapter we fix $F$ to be a number field, and ${}^L\wt{G}$ the global $L$-group.

\section{Automorphic $L$-function} \label{auto L-ftn} \index{$L$-function ! automorphic}

Let $\wt{\sigma}=\bigotimes_v \wt{\sigma}_v$ be a genuine automorphic representation of $\wm{G}(\A)$. Then for almost all $v$, $\wt{\sigma}_v$ is unramified and $\xymatrix{\wt{\sigma}_v \ar@{^(->}[r] & I(\wt{\chi}_v)}$. It gives rise to a splitting of the local $L$-group $\wt{}^L\wt{G}_v$ as the composition
$$\xymatrix{
\rho_v: \quad \W_{F_v} \ar[r] & {}^L\wt{T}_v \ar@{^(->}[r] & {}^L\wt{G}_v,
}$$
where the first map is given by $\rho_{\wt{\chi}_v}$ from the LLC in Proposition \ref{LLC tori}. Recall that for all $v\in |F|$, there is the canonical map $\xymatrix{{}^L\wt{G}_v \ar[r] & {}^L\wt{G}}$.

\begin{dfn} 
Let $\xymatrix{R: {}^L\wt{G} \ar[r] & GL(V)}$ be any finite dimensional representation. For any $v$, let $\xymatrix{R_v: {}^L\wt{G}_v \ar[r] & {}^L\wt{G} \ar[r]^-R & GL(V)}$ be the post composition with $R$. Then the global partial $L$-function of $\wt{\sigma}$ with respect to $R$ is defined to be
$$L^S(s, \wt{\sigma}, R)=\prod_{v\notin S} L(s, \wt{\sigma}_v, R),$$
where $L(s, \wt{\sigma}_v, R)$ is the local Artin $L$-function associated with the unramified representation $\xymatrix{\rho_{v,R}: \W_{F_v} \ar[r]^-{\rho_v} & {}^L\wt{G}_v \ar[r]^-{R_v} & GL(V)}$. More precisely,
$$L(s, \wt{\sigma}_v, R):=\frac{1}{\text{det}\big(1- q_v^{-s} \cdot \rho_{v, R} (\textsf{Frob}_v)|_{V^{I_v}} \big)}.$$
\end{dfn}

Since ${}^L\wt{G}$ is a disconnected reductive complex Lie group, it is not easy to give its irreducible representations. However, there is a natural family of representations which are of interest to us, namely the adjoint type $L$-functions.
As noted in section \ref{ad rep}, we have the commutative diagram

$$\xymatrix{
\wt{G}^\vee \ar@{^(->}[r] \ar@{>>}[d]^-q & {}^L\wt{G} \ar@{>>}[r] \ar[d]^-{q^*} & W_F \ar@{=}[d] \\
\wt{G}^\vee_{ad} \ar@{^(->}[r] & \wt{G}^\vee_{ad} \times W_F   \ar@{>>}[r] \ar@/^1.3pc/[l]^-{s^{\text{Tr}}}   &  W_F.
}$$

Then any representation of $\wt{G}^\vee_{ad}$ pulls back to a representation of ${}^L\wt{G}$. In particular, the adjoint representation $Ad$ of $\wt{G}^\vee_{ad}$ gives the adjoint representations of ${}^L\wt{G}$ on the Lie algebra $\wt{\mfr{g}}^\vee$. More generally, for $\mbf{P}=\mbf{M}\mbf{U}$ a parabolic subgroup of $\mbf{G}$, we obtain the dual group $\wt{M}^\vee$ embedded in $\wt{G}^\vee$. Then the adjoint representation of $\wt{M}^\vee/Z(\wt{G}^\vee)$ on the Lie algebra $\wt{\mfr{u}}^\vee$ can be pulled back to give representation of ${}^L\wt{M}$.

Moreover, from $\xymatrix{{}^L\wt{G}_v \ar[r] & {}^L\wt{G}}$ the adjoint representation $Ad$ of ${}^L\wt{G}$ can be pulled back to ${}^L\wt{G}_v$ to give the adjoint representation of the local $L$-group discussed in previous chapter.  One thus obtains the Langlands-Shahidi type $L$-functions with respect to these adjoint representations, which are of the main interest in the following sections.

\section{Eisenstein series and its constant terms}
For simplicity, we concentrate on the maximal parabolic case, while the general case follows from similar treatment despite the complication in notations. Before we proceed, we recall two equivalent realizations of induced representations.

Follow notations in section \ref{max parab}, let $\mbf{P}=\mbf{M}\mbf{U}$ be a maximal parabolic of $\mbf{G}$ corresponding to $\Delta\backslash \set{\beta}$.  Let $2\rho_P$ be the sum of positive roots in $\mbf{U}$ and $\beta_P$ be the fundamental weight corresponding to $\beta$. 

Consider the character group $X^*(\mbf{M})$ of $\mbf{M}$, and also the real and complex vector space
$$X^*(\mbf{M})_\R =X^*(\mbf{M})\bigotimes_\Z \R, \quad X^*(\mbf{M})_\C =X^*(\mbf{M})\bigotimes_\Z \C.$$

Any $\nu_o \in X^*(\mbf{M})$ could be viewed as a character on $\mbf{M}(\A)$ valued in $\A^\times$. Further composition with the valuation of $\A^\times$ gives us a character of $\mbf{M}(\A)$ valued in $\C^\times$.

Similarly, for any $\nu=\nu_o\otimes s \in X^*(\mbf{M})_\C$ with $\nu_o\in X^*(\mbf{M})$ and $s\in \C$, we denote by $\pmb{\delta}^\nu$ the following character of $\mbf{M}(\A)$:
$$\xymatrix{
\pmb{\delta}^\nu: \quad \mbf{M}(\A) \ar[r] & \C, & m \ar@{|->}[r] & |\nu_o(m)|_{\A}^s.
}$$

The relation between $\pmb{\delta}$ and the modular character $\delta_P$ is that
$$\pmb{\delta}^{\rho_P\otimes 1}=\delta_P^{1/2}.$$

In the case of maximal parabolic, $X^*(\mbf{M}/Z(\mbf{G}))\bigotimes \C$ is of dimension one over $\C$ with $\beta_P\otimes 1$ or $\rho_P\otimes 1$ as a basis vector. Henceforth, we will write
$$\pmb{\delta}^s:=\pmb{\delta}^{\beta_P \otimes s}, \quad s\in \C.$$

For example for $\mbf{SL}_2$ with positive root $\beta$, $\rho_P=\beta/2$ and $\beta_P=\rho_P$. Then $\pmb{\delta}^s=\delta_P^{s/2}$, with $\delta_P$ the modular character of the Borel subgroup $P$.

Let $\wt{\pi}$ be a genuine cuspidal automorphic representation of $\wm{M}(\A)$, i.e., $\wt{\pi}$ occurs as a direct summand $V_{\wt{\pi}}$ in the decomposition of $L^2_\text{cusp}(\mbf{M}(F)\big\backslash \wt{\mbf{M}}(\A))$. Here $\wt{\pi}$ is a unitary representation. 

We take $\pmb{\delta}^s$ to be a character of the covering $\wm{M}(\A)$ by the inflation via the surjection $\xymatrix{\wm{M}(\A) \ar@{>>}[r] & \mbf{M}(\A)}$. Now we consider the induction $I(s, \wt{\pi}):=Ind_{\wt{P}}^{\wt{G}}\ (\pmb{\delta}^s \wt{\pi}) \otimes \mbf{1}$. We have the tensor product decomposition
$$I(s, \wt{\pi})=\bigotimes_v I(s, \wt{\pi}_v),$$
where $I(s, \wt{\pi}_v)$ is unramified for almost all $v$. \\

For the purpose of defining Eisenstein series valued in $\C$, one would like to consider the following alternative description of $I(s,\wt{\pi})$ (cf. \cite[\S I.2.17]{MW95}).

We have the Iwasawa decomposition
$$\wt{\mbf{G}}(\A)=\mbf{U}(\A) \wt{\mbf{M}}(\A) \wt{K},$$
where $\wt{K}\subseteq \wm{G}(\A)$ is the preimage of $K=\prod_v K_v$, which is a fixed maximal compact subgroup of $\mbf{G}(\A)$ such that $K_v=\mbf{G}(\msc{O}_v)$ for almost all $v$.

It follows
$$\mbf{U}(\A) \mbf{M}(F)\big\backslash \wt{\mbf{G}}(\A) \simeq \big( \mbf{M}(F)\big\backslash \wt{\mbf{M}}(\A) \big) \cdot \wt{K}.$$

We thus define a space $V_{P,\wt{\pi}}$ of functions
$$\xymatrix{
\phi: \quad \mbf{U}(\A) \mbf{M}(F)\big\backslash \wt{\mbf{G}}(\A) \ar[r] & \C,
}$$
satisfying
\begin{enumerate}
\item[(i)] $\phi$ is right $\wt{K}$-finite. That is, the space spanned by $\phi_k$ where $\phi_k(\wt{g})=\phi(\wt{g}k)$ is finite dimensional.
\item[(ii)] for each $k\in \wt{K}$, the function given by $\wt{m}\mapsto \phi(\wt{m}k), \wt{m} \in \wt{\mbf{M}}(\A)$, lies in $V_{\wt{\pi}}$ which we recall is a realization of $\wt{\pi}$ in $L^2_\text{cusp}(\mbf{M}(F)\backslash \wm{M}(\A))$.
\end{enumerate}

 Consider the space
$$\mathbf{I}(s, \wt{\pi})=\set{\phi \cdot \pmb{\delta}^{s+\rho_P}: \phi \in V_{P,\wt{\pi}}},$$
where we have written $\pmb{\delta}^{s+\rho_P}:=\pmb{\delta}^{\beta_P\otimes s+\rho_P \otimes 1}$ in abbreviation. Here $\pmb{\delta}^s$ is also used to denote the map
$$\xymatrix{
\pmb{\delta}^s: \mbf{U}(\A) \mbf{M}(F)\big\backslash \wt{\mbf{G}}(\A) \ar[r] &\C^\times, \quad \wt{g}=u\wt{m} k \ar@{|->}[r] &\pmb{\delta}^s(\wt{m}),
}$$
which is well-defined as $\pmb{\delta}^s$ is trivial on $\wm{M}(\A)\cap \wt{K}$. Roughly speaking, the space $\mathbf{I}(s, \wt{\pi})$ is endowed with an action of right translation, i.e. for all $\wt{g} \in \wt{\mbf{G}}(\A)$, 
$$\xymatrix{
\mathbf{I}(s, \wt{\pi})(\wt{g}): \quad \phi(x) \cdot \pmb{\delta}^{s+\rho_P}(x) \ar@{|->}[r] &\phi(x\wt{g}) \cdot \pmb{\delta}^{s+\rho_P}(x\wt{g}), x\in \wt{\mbf{G}}(\A).
}$$
Note that in a more rigorous way $\wm{G}(\A)$ does not act directly on it. Rather, the representation space (or rather module) we are interested in is endowed with a structure of $\big(\text{Lie}(\wm{G}(\mbf{A}_\infty))\otimes_\R \C, \wt{K}\big) \times \wm{G}(\mbf{A}_\text{fin})$-module. Here we just fix ideas and refer to \cite{MW95} for a more careful and rigorous treatment.

Now we can define a map
$$\xymatrix{
I(s, \wt{\pi}) \ar[r] &\mathbf{I}(s, \wt{\pi}), & f_s \ar@{|->}[r] & \text{ev}_{1_{\wt{P}}} \circ f_s,
}$$
where $\xymatrix{f_s: \wt{\mbf{G}}(\A) \ar[r] & \wt{\pi}}$ is an element in $I(s,\wt{\pi})$, and $\text{ev}_{1_{\wt{P}}}$ is the evaluation map at the identity $1_{\wt{P}}=1_{\wt{G}}$. Regarding the well-definedness one can check that the function 
$$\wt{g} \quad \mapsto \quad \frac{\text{ev}_{1_{\wt{P}}} \circ f_s}{\pmb{\delta}^{-(s+\rho_P)}}(\wt{g}) $$
is independent of $s$ and lies in $V_{P,\wt{\pi}}$. Equivalently, $\text{ev}_{1_{\wt{P}}} \circ f_s \in \mathbf{I}(s, \wt{\pi})$. In fact,  above map is an isomorphism, with respect to which $I(s, \wt{\pi})$ and $\mathbf{I}(s, \wt{\pi})$ can be identified.

Let $\phi_s=\phi \cdot \pmb{\delta}^{s+\rho_P} \in \mathbf{I}(s,\wt{\pi})$ be corresponding to a certain $f_s\in I(s, \wt{\pi})$. 
Define the Eisenstein series 
$$E(s, \wt{\pi}, \phi_s, \wt{g})=\sum_{\gamma \in \mbf{P}(F)\backslash \mbf{G}(F)} \phi_s(\gamma \wt{g})=\sum_{\gamma \in \mbf{P}(F)\backslash \mbf{G}(F)} \phi(\gamma \wt{g}) \cdot \pmb{\delta}^{s+\rho_P}(\gamma \wt{g}).$$
\index{Eisenstein series}

We may also write $E(s, \wt{\pi}, \phi, \wt{g})$ for $E(s, \wt{\pi}, \phi_s, \wt{g})$.

Recall we have the unique $\w\in W$ such that $\w(\Delta\backslash \set{\beta}) \subseteq \Delta$ and $\w(\beta)\in \Psi^-$. 

\begin{dfn} \label{self-associated P}
The parabolic $\mbf{P}$ is called self-associated if $\w(\Delta\backslash \set{\beta})=\Delta\backslash \set{\beta}$.
\end{dfn}

Pick the representative $w \in \mbf{G}(F)$ as in section \ref{W notation}, which implies $w\in \mbf{G}(F_v)$ for all $v$.
That is, the embedding $\xymatrix{\mbf{G}(F) \ar@{^(->}[r] & \mbf{G}(\A)}$ gives $\xymatrix{w \ar@{^(->}[r] &(w_v)_v}$ such that $w_v=w \in \mbf{G}(F_v)$ for all $v$. 

Moreover, the splitting $\xymatrix{\mbf{G}(F) \ar@{^(->}[r] & \wm{G}(\A)}$ gives $\xymatrix{w \ar@{^(->}[r] &(\widetilde{w}_v)_v}$ with $\widetilde{w}_v \in \wm{G}(F_v)$ for all $v$.
Since $w$ is generated by unipotent element; for almost all $v$, the element $\widetilde{w}_v$ is in fact equal to the lifting of $w_v \in \mbf{G}(F_v)$ by $s_K$. That is, for almost all $v$ we have $w_v=w \in W^{K_{G_v}}$ and furthermore $\widetilde{w}_v=\wt{w}_v \in s_K(W^{K_{G_v}})$.
\\

Consider the parabolic associated with $\w(\Delta\backslash \set{\beta}) \subseteq \Delta$ whose Levi subgroup is then given by $\mbf{M}':=w \mbf{M} w^{-1}$. As in \cite[\S II.1.6]{MW95}, denote by ${}^w\wt{\pi}$ the representation on $\wt{\mbf{M}'}(\A)$. Then we have the global intertwining operator 
$$\xymatrix{
T(w, s, \wt{\pi})=\bigotimes_v T(\widetilde{w}_v, s, \wt{\pi}_v) :\quad  I(s, \wt{\pi}) \ar[r] & I(-s, {}^w\wt{\pi})
}$$
induces an intertwining operator from $\mathbf{I}(s, \wt{\pi})$ to $\mathbf{I}(-s, {}^w\wt{\pi})$, still denoted by $T(w, s, \wt{\pi})$.

The main properties of Eisenstein series could be summarized as
\begin{thm}[{\cite[\S I.1.5, \S IV.1.10-11]{MW95}}] \label{ppty Eisen}
The Eisenstein series $E(s, \wt{\pi}, \phi_s, \wt{g})$ has the following properties:
\begin{enumerate}
\item $E(s, \wt{\pi}, \phi_s, \wt{g})$ is absolutely convergent for $\text{Re}(s)>\angb{\rho_P}{\beta^\vee}$.
\item $E(s, \wt{\pi}, \phi_s, \wt{g})$ and $T(w, s, \wt{\pi})$ both have meromorphic continuation to $s\in \C$ and satisfy a functional equation
$$E(s, \wt{\pi}, \phi_s, \wt{g})=E(-s, {}^w\wt{\pi}, T(w, s, \wt{\pi})\phi_s, \wt{g}), \quad T(w, -s, {}^w\wt{\pi}) T(w, s, \wt{\pi})=id.$$
\item $E(s, \wt{\pi}, \phi, \wt{g})$ and $T(w, s, \wt{\pi})$ are holomorphic on $\text{Re}(s)=0$.
\item The singularities of $E(s, \wt{\pi}, \phi, \wt{g})$ and $T(w, s, \wt{\pi})$ are the same. In the region $\text{Re}(s)>0$, there are only finitely many of them, all are simple and on the interval $(0, \angb{\rho_P}{\beta^\vee})$.
\end{enumerate}
\end{thm}

The last part of the above theorem is due to the fact that the intertwining operator $T(w, s, \wt{\pi})$ figures itself in the computation of the constant term of Eisenstein series, which has the same singularities as the Eisenstein series $E(s, \wt{\pi}, \phi, \wt{g})$.

\index{Eisenstein series!constant term of}
By definition, the constant term of $E(s, \wt{\pi}, \phi, \wt{g})$ along a parabolic subgroup $\mbf{P}_1=\mbf{M}_1 \mbf{U}_1$ in general is defined by
$$E_{P_1}(s, \wt{\pi}, \phi, \wt{g})=\int_{\mbf{U}_1(F)\backslash \mbf{U}_1(\A)} E(s, \wt{\pi}, \phi, u\wt{g}) du,$$
where for the global integration we use the Tamagawa measure normalized such that $F\backslash \A$ has measure 1.

Recall that the maximal parabolic $\mbf{P}$ associated with $\Delta\backslash \set{\beta}$ is called self-associated if $\w(\Delta\backslash \set{\beta})=\Delta\backslash \set{\beta}$ (cf. Definition \ref{self-associated P}).

\begin{thm}[{\cite[\S II.1.7]{MW95}}] \label{ppty CTE}
We have the following
\begin{enumerate}
\item[(1)] If $\mbf{P}_1$ is neither $\mbf{P}$ nor the parabolic $\mbf{P}'$ associated with $\w(\Delta\backslash \set{\beta}) \subseteq \Delta$ , then
$$E_{P_1}(s, \wt{\pi}, \phi, \wt{g})=0.$$
\end{enumerate}

Therefore, the interesting cases are $E_{P}(s, \wt{\pi}, \phi, \wt{g})$ and $E_{P'}(s, \wt{\pi}, \phi, \wt{g})$. For this we have
\begin{enumerate}
\item[(2)] If $\mbf{P}$ is self-associated, i.e. $\mbf{P}=\mbf{P}'$, then
$$E_P(s, \wt{\pi}, \phi, \wt{g})=E_{P'}(s, \wt{\pi}, \phi, \wt{g})=\phi_s(\wt{g}) + T(w, s, \wt{\pi}) \phi_s(\wt{g}),$$
where $\phi_s(\wt{g})=\phi(\wt{g}) \cdot \pmb{\delta}^{s+\rho_P}$.
\item[(2)'] If $\mbf{P}$ is not self-associated, i.e. $\mbf{P}\ne \mbf{P}'$, then the two cases of interest are given by
\begin{align*}
E_P(s, \wt{\pi}, \phi, \wt{g}) & =\phi_s(\wt{g}), \\
E_{P'}(s, \wt{\pi}, \phi, \wt{g})& =\phi_s(\wt{g}) + T(w, s, \wt{\pi}) \phi_s(\wt{g}).
\end{align*}
\end{enumerate}
\end{thm}

The poles of $E(s, \wt{\pi}, \phi, \wt{g})$ agree with those of $T(w, s, \wt{\pi})$ by Theorem \ref{ppty Eisen}. The intertwining operator has a tensor product decomposition
$$\xymatrix{
T(w, s, \wt{\pi})=\bigotimes_v T(\widetilde{w}_v, s, \wt{\pi}_v): \quad \bigotimes_v I(s, \wt{\pi}_v) \ar[r] & \bigotimes_v I(s, {}^{\widetilde{w}_v}\wt{\pi}_v).
}$$

For finite set $S\subseteq |F|$ big enough and for all $v\notin S$, we have $\widetilde{w}_v=\wt{w}_v$ and thus the operator $T(\wt{w}_v, s, \wt{\pi}_v)$ intertwins between unramified representations. We computed in previous chapter the coefficient $c(\wt{w}_v, s, \wt{\pi}_v)$ such that $T(\wt{w}_v, s, \wt{\pi}_v) f_{\wt{\pi}_v} =c(\wt{w}_v, s, \wt{\pi}_v) f_{{}^{\wt{w}_v}\wt{\pi}_v}$. By applying the results in Theorem \ref{L gen coeff} to $\pmb{\delta}^{s}_v \otimes \wt{\pi}_v$ we get:

\begin{thm} \label{L/L for T} \index{$L$-function ! global Langlands-Shahidi}
Let $f=\bigotimes_{v\notin S} f_{\wt{\pi}_v} \otimes \bigotimes_{v\in S} f_v \in \wt{\sigma}$. The global intertwining operator $T(w, s, \wt{\pi}) f$ is then given by
$$ T(w, s, \wt{\pi}) f=\prod_{i=1}^m \frac{L^S( n_\beta i \cdot s, \wt{\pi}, Ad_i)}{L^S(1+ n_\beta i \cdot s, \wt{\pi}, Ad_i)} \bigotimes_{v\notin S} f_{{}^{\wt{w}_v}\wt{\pi}_v} \otimes \bigotimes_{v\in S} T(\widetilde{w}_v, s, \wt{\pi}_v) f_v.$$
Here $L^S( s, \wt{\pi}, Ad_i)$ is the automorphic partial $L$-function (cf. section \ref{auto L-ftn}) associated with the adjoint representations $Ad_{\wt{\mfr{u}}^\vee}=\bigoplus_{i=1}^m Ad_i$ of ${}^L\wt{M}$ on $\wt{\mfr{u}}^\vee$. More explicitly, it is given by
$$L^S( s, \wt{\pi}, Ad_i)=\prod_{v\notin S} L( s, \wt{\pi}_v, Ad_i),$$
where the local $L$-function $L( s, \wt{\pi}_v, Ad_i)$ is the one determined in Theorem \ref{L gen coeff}.
\end{thm}
\begin{proof}
Observe that the adjoint representation $Ad_{\wt{\mfr{u}}^\vee}$ of the global ${}^L\wt{M}$ on $\wt{\mfr{u}}^\vee$ restricts to a representation of the local ${}^L\wt{M}_v$, which is just the adjoint representation of ${}^L\wt{M}_v$ on $\wt{\mfr{u}}^\vee$. They both factor through the complex adjoint representation $Ad^\C$ of $\wt{M}^\vee$. We will use $Ad=\bigoplus_{i=1}^m Ad_i$ to represent both local and global situations, and no confusion will arise.

Thus, in view of Theorem \ref{L gen coeff} it suffices to show that  for almost all $v$ the equality $L(0, \pmb{\delta}^s_v \otimes \wt{\pi}_v, Ad_i)=L(n_\beta i \cdot s, \wt{\pi}_v, Ad_i)$ holds.
Write $\mbf{T}^\dag(\A)$ for the image of $Z(\wm{T}(\A))$ in $\mbf{T}(\A)$. Consider the character 
$$\xymatrix{
\chi_{\pmb{\delta}^s}: \quad \mbf{T}^\dag(\A) \ar@{^(->}[r] & \mbf{T}(\A) \ar@{^(->}[r] & \mbf{M}(\A) \ar[r]^-{\pmb{\delta}^s} & \C^\times,
}$$
which can be decomposed as $\chi_{\pmb{\delta}^s}=\bigotimes_v \chi_{\pmb{\delta}^s_v}$. We then have $I(s, \wt{\chi}) \simeq I(\wt{\chi}\otimes \chi_{\pmb{\delta}^s})$ and thus $\xymatrix{ I(s, \wt{\pi}) \ar@{^(->}[r] & I(\wt{\chi}\otimes \chi_{\pmb{\delta}^s})}$. Therefore we have locally for $v\notin S$,
\begin{align*}
L(0, \pmb{\delta}^s_v \otimes \wt{\pi}_v, Ad_i) &= \prod_{\substack{\alpha\in \Psi^+ \\ \angb{\beta_P/n_\beta}{\alpha_{[n]}^\vee}=i}} L(0, Ad_\alpha \circ \rho_{\wt{\chi}_v \otimes \chi_{\pmb{\delta}^s_v}}) \\
&= \prod_{\substack{\alpha\in \Psi^+ \\ \angb{\beta_P/n_\beta}{\alpha_{[n]}^\vee}=i}} \frac{1}{1-\wt{\chi}_v\otimes \chi_{\pmb{\delta}^s_v}\big( \wt{h}_\alpha(\varpi_v^{n_\alpha})\big)} \\
&=\prod_{\substack{\alpha\in \Psi^+ \\ \angb{\beta_P/n_\beta}{\alpha_{[n]}^\vee}=i}}  \frac{1}{1-\wt{\chi}_v\big( \wt{h}_\alpha(\varpi^{n_\alpha})\big) \cdot |\varpi_v^{\angb{\beta_P}{\alpha_{[n]}^\vee}}|_F^s} \\
&=\prod_{\substack{\alpha\in \Psi^+ \\ \angb{\beta_P/n_\beta}{\alpha_{[n]}^\vee}=i}} L(n_\beta i \cdot s, Ad_\alpha \circ \rho_{\wt{\chi}_v}) ,
\end{align*}
which is clearly equal to $L(n_\beta i\cdot s, \wt{\pi}_v, Ad_i)$. The proof is completed.
\end{proof}

In view of Theorem \ref{ppty Eisen}, we immediately have
\begin{cor}
The product of the partial $L$-functions
$$\prod_{i=1}^m \frac{L^S( n_\beta i \cdot s, \wt{\pi}, Ad_i)}{L^S(1+ n_\beta i \cdot s, \wt{\pi}, Ad_i)}$$
has meromorphic continuation to the whole complex plane $\C$. In particular, if $m=1$,
the $L$-function $L^S(s, \wt{\pi}, Ad)$ has meromorphic continuation to $\C$.
\end{cor}

\begin{rmk} \label{ppty of L/L}
If $m\ge 2$, one would like to show the meromorphic continuation of each $L^S(s, \wt{\pi}, Ad_i)$. However, in this case, the induction method for linear algebraic groups as in \cite[Lm. 4.2]{Sha88} or \cite[Prop. 4.1]{Sha90} can not be applied directly. The proof of the induction lemma mentioned in these references is largely due to a classification of the pair of reductive group and the Levi subgroups of its maximal parabolic subgroups. Such classification is lacking in the covering group setting. It would be interesting to see how this and other analytic properties of the individual $L$-functions above could be achieved, by either using an analogous method or alternatives.
\end{rmk}

\section{The residual spectrum for $\wt{\mbf{SL}}_2(\A)$} \index{residual spectrum}
In this section, we determine completely the residual spectrum for $\wm{SL}_2(\A)$ associated with arbitrary $n\in \N_{\ge 1}$ and quadratic form $Q$ on $Y=Y^{sc}$ of $\mbf{SL}_2$.

Therefore, we fix $n$ such that $\mu_n\subseteq F^\times$. Let $\alpha^\vee$ be the positive coroot, then $Q$ is uniquely determined by $Q(\alpha^\vee) \in \Z$. We can readily compute the complex dual group $\wt{SL}_2^\vee$ and the $L$-group ${}^L\wt{SL}_2$. There are two cases accordingly.
\begin{enumerate}
\item[(1)] $\text{gcd}(n_\alpha, 2)=1$, then $\wt{T}^\vee \simeq \C^\times$ and $\wt{SL}_2^\vee=\mbf{PGL}_2(\C)$. By construction there are canonical isomorphisms ${}^L\wt{T}\simeq \C^\times \times \W_F$ and ${}^L\wt{SL}_2\simeq \mbf{PGL}_2(\C) \times \W_F$.
\item[(2)] $\text{gcd}(n_\alpha, 2)=2$, then $\wt{SL}_2^\vee=\mbf{SL}_2(\C)$. In this case, the $L$-group ${}^L\wt{SL}_2$ is isomorphic to $\mbf{SL}_2(\C) \times \W_F$, but not canonically so.
\end{enumerate}

We obtain the global $n$-covering:
$$\seq{\mu_n}{\wt{\mbf{SL}}_2(\A)}{\mbf{SL}_2(\A)}.$$

Let $\mbf{P}=\mbf{B}=\mbf{T}\mbf{N}$ be the parabolic (Borel) subgroup of $\mbf{SL}_2$. Let $\wt{\pi}$ be a genuine irreducible representation of $\wt{\mbf{T}}(\A)$. By \cite{We14} or \cite[Thm. 4.15]{We14-2}, the group  $\mbf{T}(F)Z(\wt{\mbf{T}}(\A))$ is a maximal abelian subgroup of $\wt{\mbf{T}}(\A)$, and it follows that
\begin{equation}\label{pi and chi}
\wt{\pi} \simeq \text{Ind}_{\mbf{T}(F)Z(\wt{\mbf{T}}(\A))}^{\wt{\mbf{T}}(\A)} \wt{\chi},
\end{equation}
where $\wt{\chi}$ is a global genuine unitary character of $\mbf{T}(F)Z(\wt{\mbf{T}}(\A))$ that is trivial on the intersection $\mbf{T}(F) \cap Z(\wt{\mbf{T}}(\A))$. Since there is the surjection
$$\xymatrix{
\wt{\mbf{T}}_{Q,n}(\A) \ar@{>>}[r] &Z(\wt{\mbf{T}}(\A)),
}$$
the character $\wt{\chi}$ could be viewed as one of $\wt{\mbf{T}}_{Q,n}(\A)$ which descends.

We have $\wt{\pi}=\bigotimes_v \wt{\pi}_v$, where $\wt{\pi}_v=\text{Ind}_{Z(\wt{T}_v)}^{\wt{T}_v} \wt{\chi}_v$, and for almost all $v$ it is unramified with unramified character $\wt{\chi}_v$. We could compute the Satake parameter $\rho_{\wt{\chi}_v}(\text{Frob}_v)\in {}^L\wt{SL}_2$ and there are two cases as above. 

\begin{enumerate}
\item[(1)] $\text{gcd}(n_\alpha, 2)=1$, then the Satake parameter is given by
$$\rho_{\wt{\chi}_v}(\text{Frob}_v)=\Bigg(
\begin{bmatrix}
\wt{\chi_v}\big(\wt{h}(\varpi_v^{n_\alpha})\big)  & 0 \\
0  & 1
\end{bmatrix}, \text{Frob}_v\Bigg) \in {}^L\wt{T}_v \subseteq \mbf{PGL}_2(\C)\times \W_{F_v}.$$
\item[(2)] $\text{gcd}(n_\alpha, 2)=2$. Identify $\rho_{\wt{\chi}_v}$ with a splitting of $E_{\wt{T}_v}$ over $F_v$, and the Satake parameter is determined by $\rho_{\wt{\chi}_v}(\varpi_v) \in E_{\wt{T}_v}$. Meanwhile, we may identify $E_{\wt{T}_v}$ with $\wt{T}^\vee_v \times E_{\wt{G}_v}\big/ \nabla Z(\wt{SL}_2^\vee)$ as in Proposition \ref{L-T=p.o.}, with respect to which we write $\rho_{\wt{\chi}_v}(\varpi_v)=(s_v, e_v) \in \wt{T}^\vee_v \times E_{\wt{G}_v}\big/ \nabla Z(\wt{SL}_2^\vee)$. By tracing through the discussion from Proposition \ref{L-T=p.o.} to Corollary \ref{ad-t}, it is possible to obtain explicit form of $s_v$ and $e_v$. However, for our purpose, it suffices to note that the element (not uniquely determined, only so in the quotient of $\mbf{SL}_2(\C)$ modulo its center)
$$s_v=\begin{bmatrix}
z_v  & 0 \\
0  & z^{-1}_v
\end{bmatrix} \in \mbf{SL}_2(\C)$$
always has the property that $z_v^2=\wt{\chi_v}\big(\wt{h}(\varpi_v^{n_\alpha})\big)$.
\end{enumerate}

The fundamental weight for $\mbf{P}$ is $\beta_P=\alpha/2$. Considering the induced representation $\wt{\sigma}=I(s, \wt{\pi})$, we form the Eisenstein series $E(s, \wt{\pi}, \phi, \wt{g})$ as before. 

Clearly in our case $\mbf{P}$ is self-associated, and thus the pole of $E(s, \wt{\pi}, \phi, \wt{g})$ agrees with that of $E_P(s, \wt{\pi}, \phi, \wt{g})$, which is given by (cf. Theorem \ref{ppty CTE})
$$E_P(s, \wt{\pi}, \phi, \wt{g}) =\phi_s(\wt{g}) + T(w, s, \wt{\pi}) \phi_s(\wt{g}).$$

Let $f=\bigotimes_{v\notin S} f_{\wt{\pi}_v} \otimes \bigotimes_{v\in S} f_v \in \wt{\sigma}$, by Theorem \ref{L/L for T}, we have
$$T(w, s, \wt{\pi}) f=\frac{L^S( n_\alpha s, \wt{\pi}, Ad)}{L^S(1+n_\alpha s, \wt{\pi}, Ad)} \bigotimes_{v\notin S} f_{{}^{\wt{w}_v}\wt{\pi}_v} \otimes \bigotimes_{v\in S} T(\widetilde{w}_v, s, \wt{\pi}_v) f_v.$$

Also,
$$L^S( n_\alpha s, \wt{\pi}, Ad) =\prod_{v\notin S} L( n_\alpha s, \wt{\pi}_v, Ad)=\prod_{v\notin S} \frac{1}{1-q_v^{-n_\alpha s} \wt{\chi}(\wt{h}_\alpha(\varpi_v^{n_\alpha}))}.$$

To state it in another way, consider the diagram below
$$\xymatrix{
\mu_n \ar@{^(->}[r] & Z(\wt{\mbf{T}}(\A)) \ar@{>>}[r] & \mbf{T}^\dag(\A) \\
& & \mbf{T}^\ddag(\A) \ar@{^(-->}[lu]^-{\s_{\A}} \ar@{^(->}[u] \\
& & \mbf{T}_{Q,n}^{sc}(\A) \ar@{>>}[u]_-{(-)^{n_\alpha}}.
}$$
First we explain the groups and maps in the diagram. The group $\mbf{T}^\dag(\A)$ (respectively $\mbf{T}^\ddag(\A)$) is the image of $\mbf{T}_{Q,n}(\A)$ (resp. $\mbf{T}_{Q,n}^{sc}(\A)$) in $\mbf{T}(\A)$ induced from the embedding $\xymatrix{i_{Q,n}: Y_{Q,n} \ar@{^(->}[r] & Y}$ (resp. $\xymatrix{i_{Q,n}^{sc}: Y_{Q,n}^{sc} \ar@{^(->}[r] & Y}$) of the rank-one lattices. If we identify $\mbf{T}_{Q,n}^{sc}(\A)$ and $\mbf{T}(\A)$ both with $\A^\times$, then the induced map $\xymatrix{\mbf{T}_{Q,n}^{sc}(\A) \ar[r] & \mbf{T}(\A)}$ is simply the $n_\alpha$-power, whence the notation $(-)^{n_\alpha}$ in the diagram.

Moreover, the extension $Z(\wm{T}(\A))$ of $\mbf{T}^\dag(\A)$ splits over $\mbf{T}^\ddag(\A) \subseteq \mbf{T}^\dag(\A)$, which is given by
$$\xymatrix{
\s_{\A}: \quad \big((a_v)_v\big)^{n_\alpha} \ar@{|->}[r] & \big(\wt{h}_\alpha(a_v^{n_\alpha}) \big)_v
}$$
for any $(a_v)_v\in \A^\times \simeq \mbf{T}_{Q,n}^{sc}(\A)$.

Using this splitting,  let $\chi^{sc}=\bigotimes_v \chi^{sc}_v$ be the following composition
$$\xymatrix{
\chi^{sc}=\wt{\chi} \circ s_{\A} \circ (-)^{n_\alpha}: \quad \A^\times \ar@{>>}[r] & \mbf{T}^\ddag(\A) \ar@{^(->}[r] & Z(\wm{T}(\A)) \ar[r]^-{\wt{\chi}} &\C^\times.
}$$

Then $\chi^{sc}$ is a unitary Hecke character and it follows $L(n_\alpha s, \wt{\pi}_v, Ad)=L(n_\alpha s, \chi^{sc}_v)$ for $v\notin S$. Fix an additive character $\xymatrix{\psi=\bigotimes_v \psi_v: \A \ar[r] & \C}$. Then we can normalize and rewrite above formula for $T(w, s, \wt{\pi})$ as
$$T(w, s, \wt{\pi}) f=\frac{L( n_\alpha s, \chi^{sc})}{L( 1+ n_\alpha s, \chi^{sc})} \bigotimes_{v\notin S} f_{{}^{\wt{w}_v}\wt{\pi}_v} \otimes \bigotimes_{v\in S} N(\widetilde{w}_v, s, \wt{\pi}_v) f_v,$$
where for all $v\in |F|$,
$$N(\widetilde{w}_v, s, \wt{\pi}_v) f_v=\frac{L( 1+n_\alpha s, \chi^{sc}_v)}{L(n_\alpha s, \chi^{sc}_v) \varepsilon(n_\alpha s, \chi^{sc}_v, \psi_v)} T(\widetilde{w}_v, s, \wt{\pi}_v) f_v.$$

To determine the residual spectrum, we require
\begin{lm}
For all $v\in S$, the normalized operator $N(\widetilde{w}_v, s, \wt{\pi}_v)$ is holomorphic and nonvanishing for $\text{Re}(s)>0$.
\end{lm}
\begin{proof}
It is easy to check the nonvanishing of $T(\widetilde{w}_v, s, \wt{\pi}_v)$ and $L(s, \chi^{sc}_v)$ for $\text{Re}(s)>0$. Moreover, since $\chi_v^{sc}$ is unitary, the local $L(s, \chi^{sc}_v)$ contains no poles.

The gives the desired result. 
\end{proof}

Also,
\begin{lm}
For all $v\in |F|$, the images of $N(\widetilde{w}_v, 1/n_\alpha, \wt{\pi}_v)$ and $T(\widetilde{w}_v, 1/n_\alpha, \wt{\pi}_v)$ are both irreducible and nonzero.
\end{lm}
\begin{proof}
The normalizing factor $L( 1+n_\alpha s, \chi^{sc}_v)L(n_\alpha s, \chi^{sc}_v)^{-1} \varepsilon(n_\alpha s, \chi^{sc}_v, \psi_v)^{-1}$ has no pole or zero at $s=1/n_\alpha$, therefore it suffices to prove the lemma for $T(\widetilde{w}_v, 1/n_\alpha, \wt{\pi}_v)$.
However, since $s=1/n_\alpha>0$, it follows from the Langlands classification theorem (cf. \cite[Thm. 4.1]{BaJa13}) that the image of $T(\widetilde{w}_v, 1/n_\alpha, \wt{\pi}_v)$ is irreducible and equal to the Langlands quotient of $I(s, \wt{\pi})$.
\end{proof}

In fact, one can show that $s=1/n_\alpha$ is a reducibility point for the local induced representation,  though we do  not need such fact here. Now denote by $\msc{J}(1/n_\alpha, \wt{\pi}_v)$ the irreducible images of $N(\widetilde{w}_v, 1/n_\alpha, \wt{\pi}_v)$. If $\wt{\chi}$ is such that $\chi^{sc}=\mbf{1}$, then there is a simple pole of $E_P(s, \wt{\pi}, \phi, \wt{g})$ at $s=1/n_\alpha$ which arises from the Hecke $L$-series $L(n_\alpha s, \chi^{sc})$.

Under this condition,
$$\text{Res}_{s=1/n_\alpha} E_P(s, \wt{\pi}, \phi, \wt{g})=\bigotimes_v \msc{J}(1/n_\alpha, \wt{\pi}_v).$$
Taking constant terms commutes with taking residues:
$$\xymatrix{
I(s, \wt{\pi}) \ar[rrr]^-{\phi_s \mapsto \text{Res}_{s_o} E(\phi_s)} \ar[rrrd]_-{\phi_s \mapsto \text{Res}_{s_o} E_P(\phi_s)} & & & \mca{A}^2(\mbf{SL}_2(F)\backslash \wt{\mbf{SL}}_2(\A)) \ar[d]^-{\text{take const. term}} \\
& & & \mca{A}(\mbf{N}(\A) \mbf{T}(F)\backslash \wt{\mbf{SL}}_2(\A)).
}$$
Moreover, the right vertical map is injective on the image of the top horizontal map (cf. \cite[pg. 45]{MW95}). Thus, we may identify $\text{Res}_{s=1/n_\alpha} E(s, \wt{\pi}, \phi, \wt{g})$ with $\text{Res}_{s=1/n_\alpha} E_P(s, \wt{\pi}, \phi, \wt{g})$ as abstract representations of $\wm{SL}_2(\A)$. Since $\w(s\beta_P)=\w(\alpha/2n_\alpha)=-\alpha/2n_\alpha$, it follows from the Langlands' criterion (cf. \cite[\S I.4]{MW95}) that these residues are square integrable.

Let $\mfr{A}$ be the collection of unitary characters $\xymatrix{\wt{\chi}: Z(\wt{\mbf{T}}(\A)) \ar[r] & \C^\times}$ such that
\begin{equation*}
\begin{cases}
(1) \quad \wt{\chi}(\wt{g}) =1 \text{ for all }\wt{g} \in \mbf{T}(F) \cap Z(\wt{\mbf{T}}(\A)), \\
(2) \quad \chi^{sc} =\mbf{1}.
\end{cases}
\end{equation*}

Let $\wt{\pi}=i(\wt{\chi})$ be the globall induced representation of $\wm{T}(\A)$ as in (\ref{pi and chi}). Write $\msc{J}(1/n_\alpha, \wt{\pi})=\bigotimes_v \msc{J}(1/n_\alpha, \wt{\pi}_v)$. Let $L^2_{res}(\mbf{SL}_2(F)\backslash \wt{\mbf{SL}}_2(\A))$ denote the residual spectrum. Then we have

\begin{thm}
The representation $\msc{J}(1/n_\alpha, \wt{\pi})$ occurs in the residual spectrum of $\wt{\mbf{SL}}_2(\A)$ for such $\wt{\mbf{SL}}_2$. In fact, we have the decomposition
$$L^2_{res}(\mbf{SL}_2(F)\backslash \wt{\mbf{SL}}_2(\A))=\bigoplus_{\substack{ \wt{\pi}=i(\wt{\chi}) \\ \wt{\chi} \in \mfr{A} }} \msc{J}(1/n_\alpha, \wt{\pi}).$$
\end{thm}

In view of the existence of global Weyl-group invariant characters as discussed in section \ref{splt glb L}, we could have an alternative description of the condition $\mfr{A}$ and thus also the residual spectrum.

For simplicity, we restrict ourselves to consider the case when $n| 2Q(\alpha^\vee)$, under which assumption $n_\alpha$ could be equal to either 1 or 2. This covers the linear case when $n=1$ and the classical metaplectic double cover of $\mbf{SL}_2(\A)$ when $n=2, Q(\alpha^\vee)=1$.

Though the general case could be considered in a similar way, with above assumption we see that $Y_{Q,n}=Y$ and thus $\wt{\mbf{T}}(\A)$ is abelian, i.e. $Z(\wt{\mbf{T}}(\A))=\wt{\mbf{T}}(\A)$. Therefore the first condition on $\wt{\chi} \in \mfr{A}$ is equivalent to
$$(1)' \quad \wt{\chi} \text{ is a unitary Hecke character on } \mbf{T}(F)\backslash \wt{\mbf{T}}(\A).$$

For $(2)$, we fix an additive character $\xymatrix{\psi=\bigotimes_v \psi_v: \A \ar[r] &\C^\times}$. Then we obtain a Weyl invariant genuine character $\xymatrix{\wt{\chi}_\psi=\bigotimes \wt{\chi}_{\psi_v}: \wt{\mbf{T}}(\A) \ar[r] & \C^\times}$ from section \ref{cons dis char}. In our case, the local Weyl invariant genuine character $\wt{\chi}_{\psi_v}$ is determined by
$$\xymatrix{
\wt{\chi}_{\psi_v}: \quad \wt{T}_v \ar[r] &\C^\times, & (1, \alpha^\vee \otimes a_v) \ar@{|->}[r] & \gamma_{\psi_v}(a_v)^{2(n_\alpha-1)Q(\alpha^\vee)/n},
}$$
where $\wt{T}_v:=\wm{T}(F_v)$ and $\gamma_{\psi_v}$ is the Weil index. In fact, $\wt{\chi}_\psi$ is an automorphic character, i.e. trivial on $\mbf{T}(F)$. Then with respect to $\wt{\chi}_\psi$, any unitary genuine character $\wt{\chi}$ can be written as
$$\wt{\chi} =\wt{\chi}_\psi \cdot \chi \text{ for some unitary } \chi \in \Hom(\mbf{T}(F)\backslash \mbf{T}(\A), \C^\times).$$
If we identify $\mbf{T}(\A)$ with $\A^\times$, then we could write $\xymatrix{\chi: F^\times\backslash \A^\times \ar[r] &\C^\times}$ which is a unitary Hecke-character.

Keep notations as before, in the local setting the splitting of $\wt{T}_v$ over $T^\ddag_v$ is given by
$$\xymatrix{
T^\ddag_v \ar[r] &\wt{T}_v, & \alpha_{[n]}^\vee \otimes a_v \in T^\ddag_v \ar@{|->}[r] & \wt{h}_\alpha(a_v^{n_\alpha}) \in \wt{T}_v,
}$$
where in fact $\wt{h}_\alpha(a_v^{n_\alpha})=(1, \alpha_{[n]}^\vee \otimes a_v)$ in terms of the coordinates on $\wt{T}_v$.

Note that by the defining property of $\wt{\chi}_{\psi_v}$, it is trivial on $\wt{h}_\alpha(a_v^{n_\alpha})$. Therefore for all $a_v\in F_v^\times$,
\begin{align*}
\chi^{sc}_v(a_v)&= \wt{\chi}_v\big(\wt{h}_\alpha(a_v^{n_\alpha})\big)\\
 & =\wt{\chi}_{\psi_v}\big(\wt{h}_\alpha(a_v^{n_\alpha})\big) \cdot \chi_v\big(h_\alpha(a_v^{n_\alpha})\big) \\
&= \chi_v\big(h_\alpha(a_v^{n_\alpha})\big) \\
&=\chi_v^{n_\alpha}(a_v)
\end{align*}
Thus globally $\chi^{sc}=\chi^{n_\alpha}$. The second condition for $\wt{\chi}\in \mfr{A}$ is then equivalent to
$$(2)' \quad \chi^{n_\alpha}=\mbf{1}.$$

To summarize,
\begin{thm}
Suppose $n|2Q(\alpha^\vee)$. Keep notations as above, and denote by $\mfr{A}'$ characters $\chi$ of $F^\times\backslash \A^\times=\mbf{T}(F)\backslash \mbf{T}(\A)$ satisfying $\chi^{n_\alpha}=\mbf{1}$. Then we have the decomposition of the residual spectrum
$$L^2_{res}(\mbf{SL}_2(F)\backslash \wt{\mbf{SL}}_2(\A))=\bigoplus_{\chi \in \mfr{A}' } \msc{J}(1/n_\alpha, \wt{\chi}_\psi \otimes \chi).$$
\end{thm}

\section{The residual spectrum of $\wm{GL}_2(\A)$}
Since the derived group of $\mbf{GL}_2$ is $\mbf{SL}_2$ which is simply-connected, any $(D, \eta) \in \Bis_{\mbf{GL}_2}^Q$ is isomorphic to a fair $(D, \mbf{1})$, cf. (\ref{D-eta free}). Therefore, without loss of generality we work with a fair $(D, \mbf{1})$ and consider $\wm{GL}_2$ incarnated by such fair object.

Let $e_1$ and $e_2$ be two $\Z$-basis of the cocharacter group $Y$ of $\mbf{GL}_2$ such that the coroot is $\alpha^\vee =e_1 -e_2$. Any Weyl-invariant bilinear form $B_Q$ is uniquely determined by the two numbers $B_Q(e_1, e_1)=2Q(e_1)=2Q(e_2)=B_Q(e_2, e_2)$ and $B_Q(e_1, e_2)=B_Q(e_2, e_1)$. Write $Q(e_1)=\textbf{p} \in \Z, B_Q(e_1, e_2)=\textbf{q} \in \Z$, then $B_Q$ is determined by the matrix $B(e_i, e_j), i, j=1, 2$:
$$B_Q(e_i, e_j)=\begin{bmatrix}
2\p  & \q \\
\q  & 2\p
\end{bmatrix}.$$

It follows $Q(\alpha^\vee)=2\p-\q$. Fix a natural number $n\in \N_{\ge 1}$, and thus by definition
$$n_\alpha=\frac{n}{\text{gcd}(n, 2\p-\q)}.$$
Define $Y_{Q,n}$ and $Y_{Q,n}^{sc}$ as before with $Y_{Q,n}^{sc}$ generated by $\alpha^\vee_{[n]}$.

Note that in general the complex dual group $\wt{GL}_2^\vee$ may not be equal to $\mbf{GL}_2(\C)$. For instance, consider the case where $\q=2\p$ and $n=2\q$. Then $Q(\alpha^\vee)=0$ and thus $n_\alpha=1$. We see that
$$Y_{Q,n}=\set{k_1e_1 + k_2 e_2: k_1 \equiv k_2 \mod 2}.$$
The root data of $\wt{GL}_2^\vee$ is given by $\big(Y_{Q,n},\ \set{\alpha^\vee_{[n]}}, \ \Hom(Y_{Q,n}, \Z), \ \set{\alpha/n_\alpha}\big)$. However, for this example it is not difficult to see that $(2n_\alpha)^{-1}\alpha \in \Hom(Y_{Q,n}, \Z)$. In fact, the complex dual group in this case is $\wt{GL}_2^\vee= \mbf{GL}_2(\C)/\mu_2$.

To proceed with the general case, we start with the following lemma.
\begin{lm}
There exists an element $y_o \in Y_{Q,n}$ such that $\alpha^\vee_{[n]}$ and $y_o$ form a $\Z$-basis of the lattice $Y_{Q,n}$.
\end{lm}
\begin{proof}
First, we show that if $k\alpha^\vee \in Y_{Q, n}$ for some $k\in \Z$, then $n_\alpha|k$. If $k\alpha^\vee \in Y_{Q,n}$, then
$$k \cdot B_Q(\alpha^\vee, e_i)\in n\Z \text{ for }  i=1, 2.$$
It follows that $n$ divides $\pm k(2\p-\q)$. Therefore $n_\alpha | k$.

Now let $y_1, y_2$ be a basis of $Y_{Q,n}$ and let $\alpha^\vee_{[n]} =a_1 y_1 + a_2 y_2$ for some $a_i\in \Z$. By above observation, $\text{gcd}(a_1, a_2)=1$. Write
$$a_1 b_1 + a_2 b_2=1, b_i \in \Z.$$ 
Let $y_o=b_2 y_1 + (-b_1)y_2$. Consider the set $\set{\alpha^\vee_{[n]}, y_o}$. We claim that it forms a basis for $Y_{Q,n}$.
It suffices to show $\set{\alpha^\vee_{[n]}, y_o}$ can generate $y_1, y_2$, and this follows from the following equalities which could be verified easily
\begin{equation*}
\begin{cases}
y_1= b_1 \alpha^\vee_{[n]} + a_2 y_o, \\
y_2= b_2 \alpha^\vee_{[n]} + (-a_1) y_o.
\end{cases}
\end{equation*}

The proof is completed.
\end{proof}

With respect to the one-dimensional lattice $Y_{Q,n}^o$ spanned  by $y_o$, we have
$$Y_{Q,n}=Y_{Q,n}^{sc} \oplus Y_{Q,n}^o.$$

Denote by $\mbf{T}_{Q,n}^{sc}$ and $\mbf{T}_{Q,n}^o$ the tori corresponding to $Y_{Q,n}^{sc}$ and $Y_{Q,n}^o$ respectively. Then $Z(\wm{T}(\A))$ is equal to the image of $\wm{T}_{Q,n}^{sc}(\A) \times \wm{T}_{Q,n}^o(\A) /\nabla \mu_n$ in $\wm{T}(\A)$. To give a genuine character of $Z(\wm{T}(\A))$ is equivalent to give $\wt{\chi}^{sc} \otimes \wt{\chi}^o$, where $\wt{\chi}^{sc}$ and $\wt{\chi}^o$ are genuine characters of $ \wm{T}_{Q,n}^{sc}(\A)$ and $ \wm{T}_{Q,n}^o(\A)$ respectively, which both descend to $\wm{T}(\A)$.
\\

Clearly the fundamental weight $\alpha_P$ is equal to $\rho_P=\alpha/2$. Identify $s$ with $\alpha_P\otimes s \in X^*(\mbf{T})_\C$ as in section \ref{max parab}. Write $\wt{\pi}=i(\wt{\chi}^{sc} \otimes \wt{\chi}^o)$. Define the Eisenstein series $E(s, \wt{\pi}, \phi, \wt{g})$ for the representation $I(s, \wt{\pi})$. 

We see that from Theorem \ref{L/L for T} the residue is determined by
$$T(w,s, \wt{\pi}) f =\frac{L^S(n_\alpha s, i(\wt{\chi}^{sc} \otimes \wt{\chi}^o), Ad)}{L^S(1+ n_\alpha s, i(\wt{\chi}^{sc} \otimes \wt{\chi}^o), Ad)} \bigotimes_{v\notin S} f_{{}^{\wt{w}_v}\wt{\pi}_v} \otimes \bigotimes_{v\in S} T(\widetilde{w}_v, s, \wt{\pi}_v) f_v,$$
where $f=\bigotimes_{v\notin S} f_{\wt{\pi}_v} \otimes \bigotimes_{v\in S} f_v$.

More explicitly,
$$L^S(s, i(\wt{\chi}^{sc} \otimes \wt{\chi}^o), Ad)=\prod_{v\notin S} \frac{1}{1-q_v^{-s} \cdot \wt{\chi}^{sc}_v(\wt{h}_\alpha(\varpi_v^{n_\alpha}))}.$$

As in the $\mbf{SL}_2$ case in previous section, let $\chi^{sc}$ be the linear character:
$$\xymatrix{
\chi^{sc} =\bigotimes_v \chi^{sc}_v: \quad \A^\times \ar@{>>}[r] & \mbf{T}^\ddag(\A) \ar@{^(->}[r]^-{\s_{\A}} & \wm{T}(\A) \ar[rr]^-{\wt{\chi}^{sc}\otimes \wt{\chi}^o} & &\C^\times,
}$$
where as before we identify $\mbf{T}_{Q,n}^{sc}(\A)$ with $\A^\times$ and $\mbf{T}^\ddag(\A)$ denotes its image in $\wm{T}(\A)$.

It follows,
$$T(w,s, \wt{\pi}) f=\frac{L^S(n_\alpha s, \chi^{sc})}{L^S(1+ n_\alpha s, \chi^{sc})} \bigotimes_{v\notin S} f_{{}^{\wt{w}_v}\wt{\pi}_v} \otimes \bigotimes_{v\in S} T(\widetilde{w}_v, s, \wt{\pi}_v) f_v.$$

To determine the residues of $E(s, i(\wt{\chi}^{sc} \otimes \wt{\chi}^o), \phi, \wt{g})$, we follow the proof of the previous section exactly, and details may be omitted here. In particular, we denote by $\msc{J}(1/n_\alpha, i(\wt{\chi}^{sc}_v \otimes \wt{\chi}^o_v))$ the irreducible and nonzero image of the certain normalized operator $N(\widetilde{w}_v, 1/n_\alpha, i(\wt{\chi}^{sc}_v \otimes \wt{\chi}^o_v))$. Write $\msc{J}(1/n_\alpha, i(\wt{\chi}^{sc} \otimes \wt{\chi}^o))=\bigotimes_v \msc{J}(1/n_\alpha, i(\wt{\chi}^{sc}_v \otimes \wt{\chi}^o_v))$.

Finally, let $\mfr{B}$ be the collection of characters $\wt{\chi}^{sc}\otimes \wt{\chi}^o$ of $Z(\wm{T}(\A))$ trivial on $\mbf{T}(F)\cap Z(\wm{T}(\A))$ such that $\chi^{sc}$ defined above is trivial. Then
\begin{thm} The residual spectrum $L^2_{res}(\mbf{GL}_2(F)\backslash \wm{GL}_2(\A))$ has a decomposition of the form
$$L^2_{res}(\mbf{GL}_2(F)\backslash \wm{GL}_2(\A)) =\bigoplus_{\wt{\chi}^{sc} \otimes \wt{\chi}^o \in \mfr{B}} \msc{J}(1/n_\alpha, i(\wt{\chi}^{sc} \otimes \wt{\chi}^o)) .$$
\end{thm}

\section{The residual spectrum of $\wt{\mbf{Sp}}_4(\A)$}
Let $\Delta=\set{\alpha_1, \alpha_2}$ be two simple roots of $\mbf{Sp}_4$ with $\alpha_1$ the the long root. Let $Q$ be the Weyl-invariant quadratic form on $Y=Y^{sc}$ uniquely determined by $Q(\alpha_1^\vee)=1$. Let $n=2$, then we obtain the classical metaplectic group
$$\seq{\mu_2}{\wt{\mbf{Sp}}_4(\A)}{\mbf{Sp}_4(\A)}.$$
The residual spectrum $L^2_{res}(\mbf{Sp}_4(F)\backslash \wt{\mbf{Sp}}_4(\A))$ is completely determined in \cite{Gao12}, and therefore we will not give any elaborate discussion here.

However, as an example, we will show that the partial $L$-functions appearing in the constant terms of Eisenstein series induced from the two maximal parabolic subgroups as in \cite{Gao12} agree with the ones given by Theorem \ref{L/L for T}.

Let $\mbf{P}_j=\mbf{M}_j \mbf{U}_j$ be the maximal parabolic subgroups generated by $\alpha_j$. We may call $\mbf{P}_2$ and $\mbf{P}_1$ the Siegel and non-Siegel parabolic subgroups respectively. 

In this case, the complex dual group is $\wt{Sp}_4^\vee=\mbf{Sp}_4(\C)$. The complex dual group $\wt{M}_j^\vee$ is contained in some parabolic $\wt{P}_j^\vee=\wt{M}_j^\vee \wt{U}_j^\vee$ generated by the two simple roots $\alpha_{j,[2]}^\vee:=2\alpha_j^\vee/\text{gcd}(2, Q(\alpha_j^\vee))$ of $\wt{Sp}_4^\vee$ respectively, with $\alpha_{1,[2]}^\vee$ being the long root of $\wt{Sp}_4^\vee$.
That is, $\wt{P}_1^\vee$ is the non-Siegel parabolic subgroup of $\wt{Sp}_4^\vee$, while $\wt{P}_2^\vee$ the Siegel parabolic.

\subsubsection{The non-Siegel parabolic $\mbf{P}_1$ case}

For $j=1$, i.e. the non-Siegel parabolic $\mbf{P}_1$. Write $\mbf{M}_1=\mbf{GL}_1 \times \mbf{Sp}_2$, we have
$$\wm{M}_1(\A) \simeq \wm{GL}_1(\A) \times \wm{Sp}_2(\A)\big/ \nabla \mu_2.$$

Any genuine cuspidal representation of $\wm{M}_1(\A)$, by using certain global Weyl-invariant character $\wt{\chi}_\psi$ defined as before, could be identified with $\wt{\chi} \boxtimes \wt{\pi}$, where $\wt{\chi}=\wt{\chi}_\psi \otimes \chi $ is a genuine character of $\wm{GL}_1(\A)$ with $\chi$ being a unitary Hecke character and $\wt{\pi}$ a cuspidal representation of the degree two cover $\wm{Sp}_2(\A)$. 

Let $I(s, \wt{\chi} \boxtimes \wt{\pi})$ be the induced representation, where $s:=(\alpha_1/2+\alpha_2)\otimes s \in X^*(\mbf{M}_1)_\C$. The Weyl group element of interest is $\w=\w_{\alpha_2} \w_{\alpha_1} \w_{\alpha_2}$. 

We have $n_{\alpha_2}=1$. By Theorem \ref{L/L for T}, the partial $L$-functions which appear in the constant term of Eisenstein series in this case is given by
\begin{align*}
\prod_{i=1}^{m=2} \frac{L^S(n_{\alpha_2}i \cdot s, \wt{\chi} \boxtimes \wt{\pi}, Ad_i)}{L^S(1+ n_{\alpha_2}i \cdot s, \wt{\chi} \boxtimes \wt{\pi}, Ad_i)}
=\frac{L^S( s, \chi \times \wt{\pi})}{L^S( 1+ s, \chi \times \wt{\pi}))} \cdot \frac{L^S( 2s, \chi^2)}{L^S( 1+ 2s, \chi^2)}.
\end{align*}
Here the Rankin-Selberg product $L^S( s, \chi \times \wt{\pi})$, or more precisely its local counterpart, is given in \cite[\S 7]{Szp11}. It agrees with \cite[\S 4.2]{Gao12}.

\subsubsection{The Siegel parabolic $\mbf{P}_2$ case}

For $j=2$, the Siegel parabolic $\mbf{P}_2$ has Levi $\mbf{M}_2\simeq \mbf{GL}_2$. Therefore,
$$\wm{M}_2(\A) \simeq \wm{GL}_2(\A).$$

Using an additive character $\psi$ of $\A$, there is a genuine character $\wm{GL}_2(\A)$ which is also denoted by $\wt{\chi}_\psi$ by abuse of notation. Any cuspidal representation $\wt{\pi}$ of $\wm{GL}_2(\A)$ could be written as $\wt{\pi}=\pi \otimes \wt{\chi}_\psi$, where $\pi$ is a cuspidal representation of $\mbf{GL}_2(\A)$.

Identify $s$ with $(\alpha_1+\alpha_2)\otimes s\in X^*(\mbf{M}_2)_\C$ and Let $I(s, \wt{\pi})$ be the induced representation. We will consider the intertwining operator for $\w=\w_{\alpha_1} \w_{\alpha_2} \w_{\alpha_1}$.

Note $n_{\alpha_1}=2$. By Theorem \ref{L/L for T}, the partial $L$-function that appears in the constant term of Eisenstein series in this case is given by
\begin{align*}
\prod_{i=1}^{m=1} \frac{L^S(n_{\alpha_1}i \cdot s, \wt{\pi}, Ad_i)}{L^S(1+ n_{\alpha_1}i \cdot s, \wt{\pi}, Ad_i)}
=\frac{L^S( 2s, \pi, \text{Sym}^2)}{L^S( 1+ 2s, \pi, \text{Sym}^2)},
\end{align*}
which also agrees with \cite[\S 3.2]{Gao12}.

\chapter{Discussions and future work}
There are questions which have been left with inconclusiveness in our discussion on the splitting of $L$-groups, the computation of the GK formula and the interpretation in terms of Langlands-Shahidi type $L$-functions. For instance, one may wonder about the characterization of conditions on the existence of distinguished characters. Also, as in Remark \ref{ppty of L/L}, it is not completely satisfactory to have only the meromorphic continuation of the whole product of partial $L$-functions in Theorem \ref{L/L for T}.

There are also problems and questions that could be imposed immediately based on the discussions in previous chapters. It is expected that some of these problems could be addressed without much difficulty by extending our argument, while others might stimulate for investigations which can be completed only as a long-time project. For the latter, one could readily impose problems by taking an analogy and comparing with the linear algebraic case, as in some sense every question that exists for linear algebraic groups could be asked for BD-type covering groups with proper modifications.

In the following, we will list some questions and problems, which are by no means exhaustive.  We only give a glimpse of such problems and questions reflecting our current interest.

Note that we only dealt with the case where $\mbf{G}$ is split, whereas the construction of $L$-group in \cite{We14} works more generally for nonsplit groups as well. It is thus natural to consider (see Remark \ref{nonsplit GK} also)
\begin{enumerate}
\item[] Problem[1]. To compute the GK formula for coverings of quasi-split groups and to interpret it analogously in terms of adjoint $L$-functions.
\end{enumerate}

One of the most ambitious goals for the theory of genuine automorphic representation  would be to postulate and prove functoriality in the spirit of the Langlands' program. There are subtleties for this matter even for some simple covers. A more restrictive question is on how genuine automorphic forms on $\wm{G}(\A)$ are related to automorphic forms on some linear algebraic $\mbf{G}'(\A)$. A complete answer would have to be able to recover existing links and yield potentially new results. It is not only the answer that is important, any machinery and delicate analysis which could shed light on the question would have extensive applications as well. This could already be seen from the theory of theta correspondence between the degree two covering $\wm{Sp}_{2r}(\A)$ and the orthogonal $\mbf{SO}_{2k+1}(\A)$, for which there has been a rich literature.

For example, we may even restrict ourselves to a more specific question:

\begin{enumerate}
\item[] Problem[2]. Determine whether the (completed in some way) $L$-functions which appear in the GK formula for the global intertwining operators (cf. Theorem \ref{L/L for T}) are equal to the Langlands-Shahidi $L$-functions for certain linear algebraic groups.
\end{enumerate}

First of all, to make sense of the problem which addresses the global \emph{completed} $L$-function, a theory of local $L$-functions or local $\gamma$-factors is to be developed. For generic representations of the double cover $\wm{Sp}_{2r}(F_v)$, a Langlands-Shahidi method is developed in \cite{Szp11}, where a theory of local $\gamma$-factor is developed.

There is no need to emphasize the importance of local $L$-function. In particular, one way to attack Problem[2] is to apply the converse theorem for $L$-functions for admissible representations of global groups. It is clear that the partial $L$-functions in the GK formula in Theorem \ref{L/L for T} could be associated with \emph{admissible} representations of global linear groups; however, to prove that they are \emph{automorphic} $L$-functions, there is the essential ingredient of local $\gamma$-factors in order to apply the converse theory. See \cite{CKM04} for a very readable introduction to converse theorems in the linear case.

Thus, a problem closed related to Problem[2] is
\begin{enumerate}
\item[] Probelm[3]. To develop a theory of local $L$-functions, based on which one could develop converse theorems and give answers to problems including but not restricted to the previous one.
\end{enumerate}

As an example, one might like to look at the Kazhdan-Patterson coverings $\wm{GL}_r(\A)$ of $\mbf{GL}_r(\A)$, which arise from the Brylinski-Deligne framework, see \cite[\S 13.2]{GaG14}.
\begin{enumerate}
\item[] Problem[4]. To work out the Kazhdan-Patterson coverings from the BD perspective in details.
\end{enumerate}

Beyond the theory of $L$-functions, one may also explore other facets of the BD type covering groups from different angles. For instance, it would be very rewarding to understand the harmonic analysis on BD covering groups which is relevant to representation theory. We especially refer to the analysis of orbital integrals, and more foundationally the (stable) conjugacy classes of $\wt{G}_v$. Therefore, one may consider as the first step the following problem, which we do not phrase in precise terms.

\begin{enumerate}
\item[] Problem[5]. Work out the geometry of the local covering groups $\wt{G}_v$.
\end{enumerate}

Again, for double cover $\wm{Sp}_{2r}(F_v)$,  a theory of endoscopy is developed in the thesis of W.-W. Li. In a sequel of works starting with \cite{Li12}, he also gives local analysis on more general central covering groups.

Beyond these, there is also the central question on the applications of automorphic forms on covering groups to arithmetic. Certainly this question is closely related to the consideration stated before Problem[2], which concerns how genuine automorphic forms on covering groups can be understood in terms of linear algebraic groups, in which case arithmetic applications of the former would be available via the latter by a detour consideration. 

However, one may ask for direct applications which would in turn gives hint on the relations between BD covering and linear algebraic groups. For instance, in the work of Lieman \cite{Lie94}, it is shown that the $L$-function of a certain twisted elliptic curves is closely related to the Whittaker coefficient of certain Eisenstein series on degree three covers of $GL_3(\A)$. 

More generally, along such direction there has been the deep study of general Whittaker coefficients of automorphic forms on covering groups by many mathematicians, notably B. Brubaker, D. Bump, S. Friedberg and J. Hoffstein etc. See for instance \cite{BBFH07}, \cite{BBF11} and \cite{BBF11-2}. They showed that certain Weyl group multiple Dirichlet series arising from those coefficients satisfy the nice expected properties such as meromorphic continuation and functional equation. Moreover, it is also shown that such series make constant appearance in the theory of crystal graph, quantum groups etc (cf. \cite{BBCFG}, \cite{BBF11-2}). 


For arithmetic application of $L$-functions of automorphic forms or more generally their coefficients, one will inevitably mention the study of $L$-functions arising from prehomogeneous vectors spaces with actions by a reductive group. We refer to the wok of T. Shintani, F. Sato, A. Yukie and M. Bhargava for expositions, especially for the applications in determining distribution of number fields etc. A program along such line has been undertaken, and this motivates us to ask what role covering groups might play.

We summarize the discussion from several paragraphs above with the following generic problem.
\begin{enumerate}
\item[] Problem[6]. To explore any possible arithmetic applications of automorphic forms and representations of BD-type coverings, which preferably allow for a systematic treatment.
\end{enumerate}

\printindex

\end{document}